%% file: HOCKING-breakpointError.tex
\documentclass{article}
\usepackage{fullpage}
\usepackage{graphicx}

\usepackage[latin1]{inputenc}
\usepackage[T1]{fontenc}

\usepackage{amsmath,amsthm,amsfonts,amssymb}
\newtheorem{definition}{Definition}

\newtheorem{proposition}{Proposition}
\usepackage{url}
\usepackage{natbib}

\newcommand{\figpdf}[3][H]{
  \begin{figure}[#1]
    \hskip -1cm
    \includegraphics[width=\textwidth]{figure-#2}
    \caption{#3}
    \label{fig:#2}
  \end{figure}
}

\usepackage{tikz}
\usepackage{float}
\renewcommand{\r}{ \mathbf{ r} }
\newcommand{\rileft}[1][i]{\underline r_{#1}}
\newcommand{\riright}[1][i]{\overline r_{#1}}
\DeclareMathOperator*{\argmin}{arg\,min}

\newcommand{\RR}{\mathbb{R}}

\begin{document}


  


  \title{A breakpoint detection error function for segmentation model
    selection and evaluation}
  \author{Toby Dylan Hocking}

\maketitle

  \begin{abstract}
    We consider the multiple breakpoint detection problem, which is
    concerned with detecting the locations of several distinct changes
    in a one-dimensional noisy data series. We propose the
    breakpointError, a function that can be used to evaluate estimated
    breakpoint locations, given the known locations of true
    breakpoints. We discuss an application of the breakpointError for
    finding optimal penalties for breakpoint detection in simulated
    data. Finally, we show how to relax the breakpointError to obtain
    an annotation error function which can be used more readily in
    practice on real data. A fast C implementation of an algorithm
    that computes the breakpointError is available in an R package on
    R-Forge.
  \end{abstract}

\tableofcontents




  
 
\newpage

\section{Introduction to segmentation models}

The goal of a segmentation model or algorithm is to divide a series of
data into distinct segments. A major application of segmentation
models is in detecting changes in copy number in cancer, using
technologies such as array comparative genomic hybridization
\citep{PSS98}. In these noisy biological data sets, the goal of
segmentation is to detect the precise base pairs or genomic positions after
which there are changes in copy number.

How to evaluate the accuracy of a segmentation model? A new method for
supervised segmentation of copy number data was proposed by
\citet{HOCKING-breakpoints}, who quantified the segmentation model
accuracy using an annotation database containing visually-determined
regions with or without breakpoints.
This method depends critically on the
definition of the visually-determined annotated region database, which
is used to compute an annotation error function. 

This paper continues
this line of research by defining the breakpointError function, which
uses the true breakpoints to precisely compute the accuracy of a
segmentation model. Also in this paper we demonstrate that the
breakpointError is closely related to the annotation error, thus
giving a theoretical foundation to the very practical new methods
based on visually-determined annotated region databases.

In this introduction, we first discuss a few motivating examples with
figures. In Section~2 we discuss related work, and in Section~3 we
define the breakpointError. In Section~4 we show an
application of the breakpointError, and in Section~5 we discuss its
relationship to the annotation error. In general we use bold to denote
vectors ($\mathbf x, \mathbf{\hat y}^k$) and plain text to denote
elements of those vectors ($x_i, y_i^k$) and scalars ($p, \hat
\sigma^2_k$).

\subsection{Definition of breakpoints}

Assume there are $P$ distinct positions in a series at
which data could be gathered. Let $\mathcal P= \{1,\dots,P\}$ be the
set of all such positions. For every position $p\in\mathcal P$, we
assume there is some true probability distribution $D_p$. Let $\mathbb
B=\{1,\dots,P-1\}$ be all bases after which a breakpoint is possible.

\begin{definition}
  A \textbf{breakpoint} is any position $p\in\mathbb B$ for which the
  next position does not have the same distribution: $D_p \neq
  D_{p+1}$. 
\end{definition}

For a series with $P$ positions, there is a minimum of 0 breakpoints
($D_1=\cdots=D_P$) and a maximum of $P-1$ breakpoints ($D_1 \neq
\cdots \neq D_P$). Note that the changes in distribution may be in
mean, variance, or any other parameters that affect the distribution.

The segmentation algorithm is given a sample of size $d \leq P$ of
data $(p_1, y_1), \dots, (p_d, y_d)$, with positions $p_i\in\mathcal
P$ and noisy observations $y_i\sim D_{p_i}$ for all samples $i\in\{1,
\dots, d\}$.

\newpage

For example, consider the normal distributions and simulated data
shown in Figure~\ref{fig:motivation}. The two panels show different,
separate segmentation problems. The top panel shows a problem with two
changes in mean, and the bottom panel shows two changes in
variance. For both panels in the figure, there are $P=500$ distinct
positions, $d=100$ simulated samples, and two breakpoints:
$\{300, 400\}$.

\begin{figure}[h!]
  \centering
  \includegraphics[width=\textwidth]{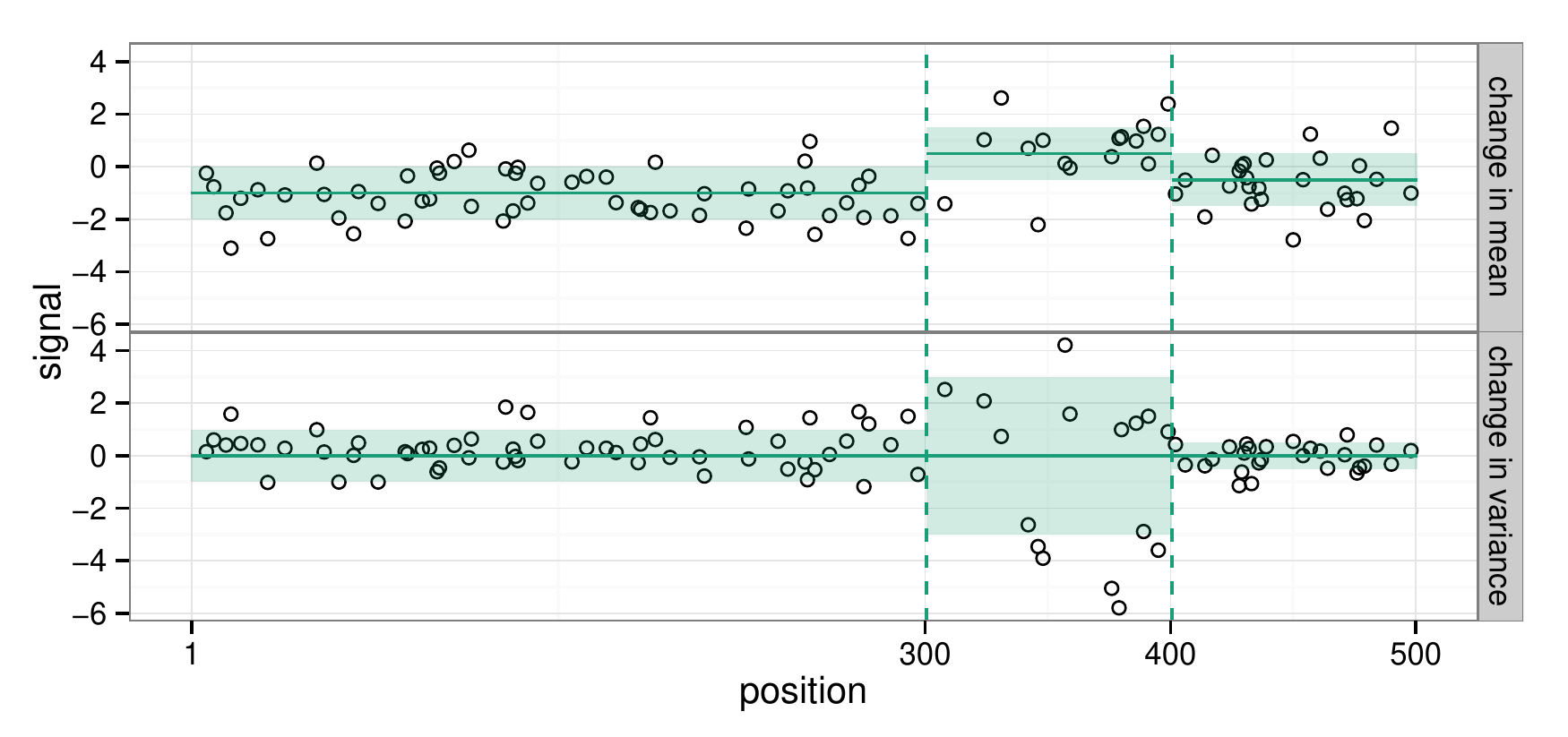}
  \vskip -0.5cm
  \caption{Two signals with the same breakpoints (vertical dashed
    lines) but different distributional changes. Black circles show
    $d=100$ sampled data points drawn from normal distributions
    defined on positions $\mathcal P=\{1, \dots, 500\}$. Horizontal
    line segments and shaded bands show mean $\pm$ one standard
    deviation of the true normal distributions $D_p$. The goal of segmentation
    is to recover the distributions and/or breakpoints, using only the
    sampled data points.}
  \label{fig:motivation}
\end{figure}

\subsection{Maximum likelihood segmentation algorithms}
\label{sec:max-lik}

A segmentation algorithm takes the $d$ sampled data points as input,
and returns a list of estimated distributions and/or breakpoints. In
this section, we will review one class of segmentation algorithms
called maximum likelihood segmentation.

A maximum likelihood segmentation model for multiple breakpoints in
the mean of a normal distribution was proposed by
\citet{statistical-approach}. Let $\mathbf y = \left[
  \begin{array}{ccc}
    y_1 & \cdots & y_d
  \end{array}
\right]^\intercal \in \RR^d$ be the vector formed by the $d$
sampled data points, and let  $\mathbf p = \left[
  \begin{array}{ccc}
    p_1 & \cdots & p_d
  \end{array}
\right]^\intercal \in \mathcal P^d$ be the corresponding vector of
positions, ordered such that $p_1 < \cdots < p_d$. Then for any number
of segments $k\in\{1, \dots, d\}$, the estimated mean vector $\mathbf{\hat
  y}^k\in\RR^d$ is defined as
\begin{equation}
\label{eq:yhat^k}
\begin{aligned}
\mathbf{\hat  y}^k = &\argmin_{\mathbf x \in \RR^d} &&  ||\mathbf y - \mathbf x||^2_2
\\
&\text{subject to} && k-1=\sum_{j=1}^{d-1} 1_{x_j\neq x_{j+1}},
\end{aligned}
\end{equation}
where $||\mathbf x||^2_2=\sum_{j=1}^d x_j^2$ is the squared $\ell_2$
norm. Note that the optimization objective of minimizing the squared
error is equivalent to maximizing the Gaussian likelihood with uniform
variance \citep{statistical-approach}. For a fixed $k_{\text{max}}\leq
d$, we can quickly calculate $\mathbf{\hat y}^k$ for all
$k\in\{1,\dots,k_{\text{max}}\}$ using pruned dynamic programming
\citep{pruned-dp}. For any model size $k$, the estimated variance
$\hat \sigma^2_k\in\RR^+$ is defined as the mean of the squared
residuals:
\begin{equation}
  \label{eq:sigmahat}
  \hat \sigma_k^2 = ||\mathbf y - \mathbf{\hat y}^k||^2_2/d.
\end{equation}

The derivation is similar for the model of multiple breakpoints in the
variance of a normal distribution \citep{lavielle2005}, and can be
computed using the methods of \citet{pelt} or \citet{segmentor}. For
both models, we visually represent the true distribution and estimates
for $k\in\{2, \dots, 5\}$ in Figure~\ref{fig:motivation-modelOnly}.

\begin{figure}[H]
  \centering
  \includegraphics[width=\textwidth]{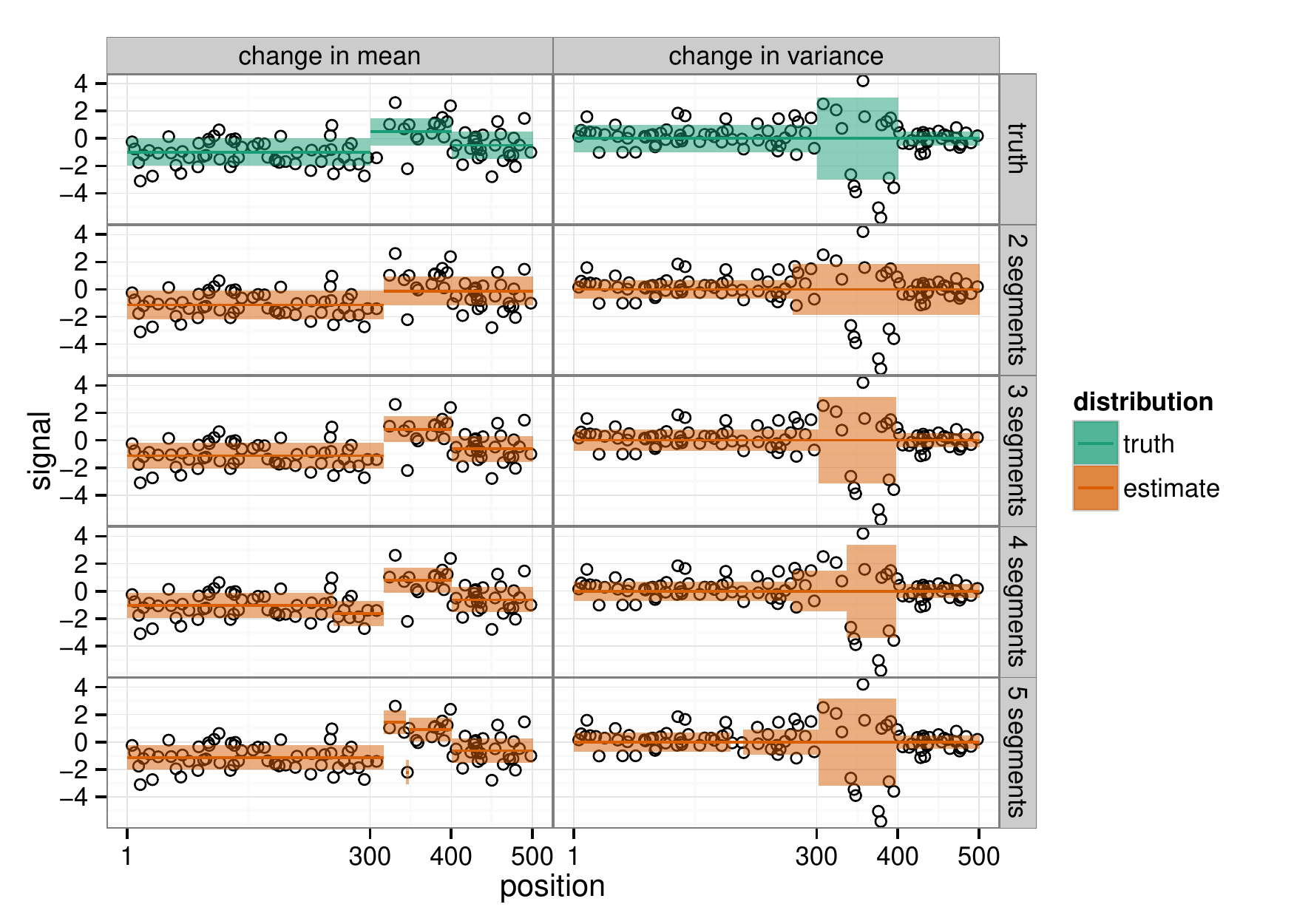}
  \vskip -0.5cm
  \caption{Comparing true and estimated distributions is one approach
    to segmentation model selection/evaluation but is not the subject
    of this paper. \textbf{Top 2 panels}: the same reference/true
    signals as in Figure~\ref{fig:motivation}. \textbf{Others}:
    estimated maximum likelihood models $\mathbf{\hat y}^k,\hat
    \sigma^2_k$ for $k\in\{2, \dots, 5\}$ segments.}
  \label{fig:motivation-modelOnly}
\end{figure}

\subsection{Model selection}

The segmentation model selection problem may be posed as follows. Of
the 4 estimated segmentation models $k\in\{2, \dots, 5\}$, which is
the closest to the true model?

\newpage

One method for segmentation model selection is to compare the true
distribution with the estimated distributions
(Figure~\ref{fig:motivation-modelOnly}), and choose the estimate whose
distribution is closest to the true distribution. Assuming the true
probability distributions $D_p$ are available, one could compare them
with the estimates using a distance function such as the earth mover's
distance \citep{earth-mover}, or some other distance
function. However, the true distribution is not available in practice
on real data, so in this paper we will not explore segmentation model
selection via comparing distributions.

Instead, we propose a method for comparing the true and estimated
breakpoints. For any $d$-vectors of data and positions $(\mathbf x,
\mathbf p)$, we estimate the breakpoint locations using
\begin{equation}
  \label{eq:breaks_phi}
\phi(\mathbf{x}, \mathbf p)
= \big\{
\lfloor 
(p_j+p_{j+1})/2
\rfloor
\text{ for all }j\in\{1,\dots,d-1\}\text{ such that }
\mathbf x_j\neq \mathbf x_{j+1}
\big\}.
\end{equation}
Thus for any model size $k$, we estimate the breakpoint positions
using $\phi(\mathbf{\hat y}^k, \mathbf p)$.
In Figure~\ref{fig:motivation-breaksOnly}, we compare
these estimated breakpoints to the true set of breakpoints
\begin{equation}
  \label{eq:breaks_B}
  B = \{j\in\mathbb B:D_j\neq D_{j+1}\}.
\end{equation}

\begin{figure}[H]
  \centering
  \includegraphics[width=\textwidth]{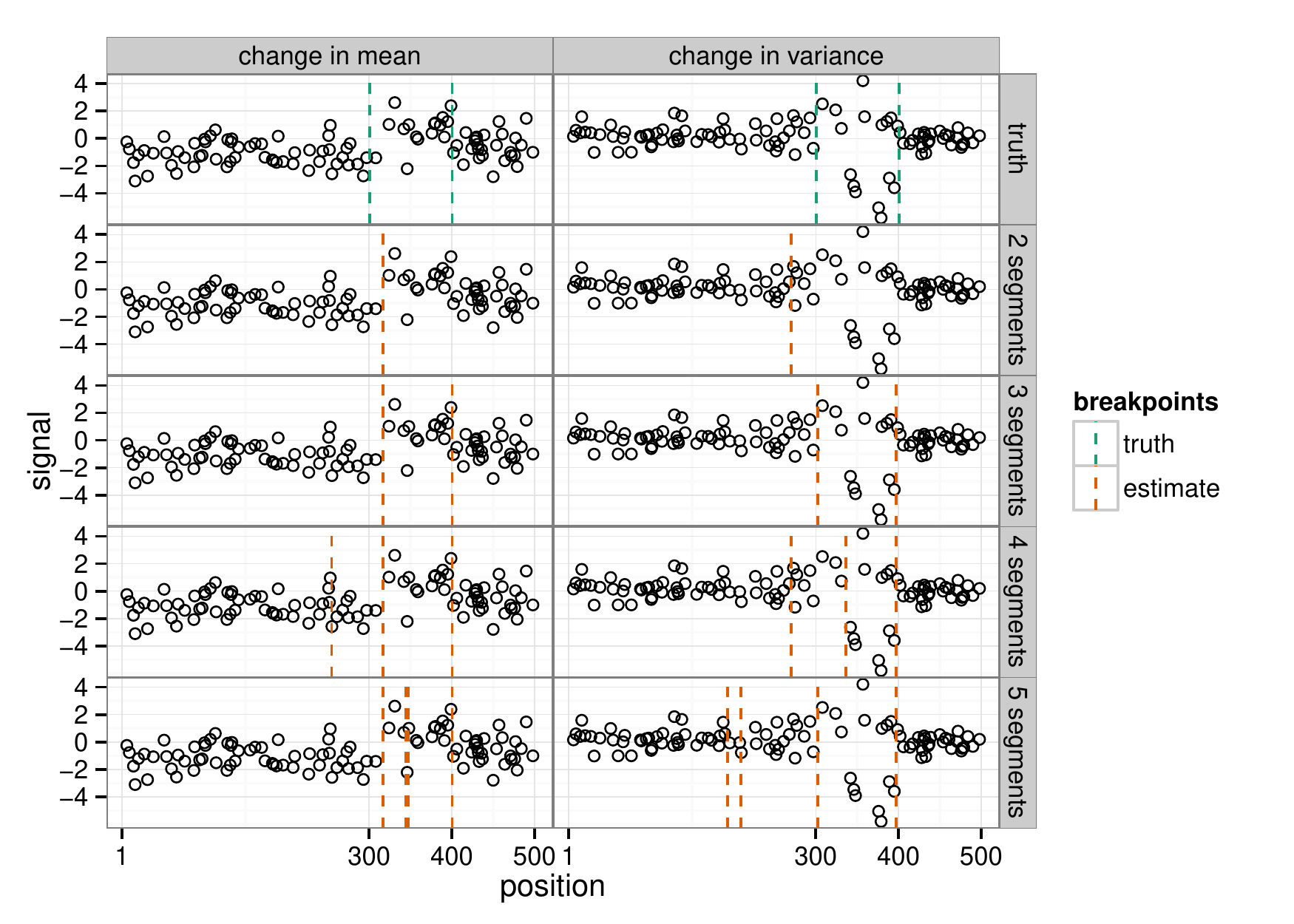}
  \vskip -0.5cm
  \caption{Same as Figure~\ref{fig:motivation-modelOnly}, but showing
    breakpoints instead of distributions. This paper proposes to
    compare the true and estimated breakpoints with the
    breakpointError function, which can be used for any reference/true
    distribution and any type of change. }
  \label{fig:motivation-breaksOnly}
\end{figure}

\newpage

Figure~\ref{fig:motivation-breaksOnly} clearly shows 3 distinct types
of errors that are possible in estimating the breakpoint positions:
\begin{description}
\item[False negative (FN)] for both data sets, the models with 2
segments are suboptimal because they only detect 1 of the 2 true
breakpoints.
\item[False positive (FP)] for both data sets, the models with 4
  segments are suboptimal since they detect 3 rather than 2
  breakpoints. The models with 5 segments are even worse since they
  detect 4 breakpoints.
\item[Imprecision (I)] of the two models with 3 segments, the
  breakpoints estimated for the change in variance data are more
  precise (closer to the true breakpoint positions).
\end{description}

This paper proposes the breakpointError function
(Figure~\ref{fig:motivation-breakpointError}), which can be used to
quantify these intuitive observations. The breakpointError can be
computed to quantify how well a set of estimated breakpoint positions
matches a true or reference set of breakpoints.

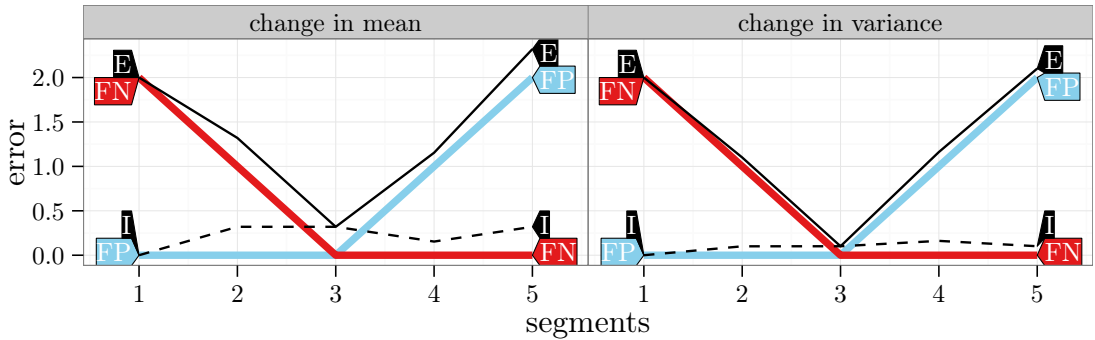
\begin{figure}[H]
  \centering
  \input{figure-motivation-breakpointError}
  \vskip -0.7cm
  \caption{For the same data shown in
    Figure~\ref{fig:motivation-breaksOnly}, we computed the
    breakpointError function (E) and its components. False positives
    (FP) occur when there are more estimated than true breakpoints,
    and false negatives (FN) are the opposite. For the correct model
    size (3 segments = 2 breakpoints), the imprecision function (I)
    quantifies the distance between the true and estimated breakpoint
    positions. The breakpointError is the sum of the other components
    (E=FP+FN+I).}
  \label{fig:motivation-breakpointError}
\end{figure}

\newpage

\section{Related Work}

This paper has been revised and expanded from Chapter~4 of the
doctoral thesis of \citet{HOCKING-phd-ch4}, which has not been
previously published elsewhere. Differences include minor changes
in notation, an expanded introduction, and more complete references.

The main subject of this paper is the breakpointError (defined in
Section~\ref{sec:definition}), which is a function for precisely
measuring the breakpoint detection accuracy of a segmentation
model. There are several other approaches for evaluating segmentation
models. \citet{zaid-lasso} compared the number of detected breakpoints
with the number of true breakpoints, ignoring the positions of the
breakpoints. A more precise method was proposed by
\citet{perf-eval-framework}, who checked if the detected breakpoints
appear in regions of arbitrary size around the true breakpoints. In
contrast, the breakpointError we propose in this paper has no
arbitrary region size parameter. \citet{group-fused} used exact
equality of the estimated and true breakpoint location in their
asymptotic theoretical analysis. The breakpointError function is more
precise since it is able to quantify that a guess close to a true
breakpoint is better than a guess far from a true breakpoint. A final
class of methods uses an annotated region database to quantify false
positive and false negative breakpoint detections
\citep{HOCKING-breakpoints, HOCKING-penalties}. An annotation database
can be created by drawing regions on scatterplots of the data using a
graphical user interface \citep{SegAnnDB}. Evaluating a segmentation
model via annotated regions is similar to the breakpointError function
we propose in this paper, and the precise link between these methods
will be explored in Section~\ref{sec:relaxation}.

Section~\ref{sec:simulations} shows one example application of the
breakpointError function, for determining the optimal form of penalty
functions in segmentation models for simulated data. Many related
penalties have been proposed for the change-point detection
problem. The standard AIC or BIC criteria are not well adapted in this
context since the model collection is exponential
\citep{BM04,BIC,Akaike73,BGH09}, and also because change-points are
discrete parameters \citep{mBIC}.  Many criteria specifically adapted
to change-point models have been proposed.  For example, there are
many different variants of the BIC \citep{Yao88,Lee95,mBIC}, and the
model selection theory of Birg\'e and Massart suggest other penalties
\citep{lavielle2005,lebarbier,BM04,calibration}. The precise
differences between these penalties and the penalties that we find
will be discussed in Section~\ref{sec:simulations}, but the main
difference is that the penalties discussed in this paper are
specifically designed to minimize the breakpointError (rather than
some other function, e.g. the squared error or negative log likelihood
of the data).

\newpage

\section{Definition of the breakpointError}
\label{sec:definition}

Let us recall the notation of Section~1. Assume there are $P$ distinct
positions in a series at which data could be gathered. Depending on
the desired application, these positions could be indices in a data
vector, genomic positions, or time points. Let $\mathcal P=
\{1,\dots,P\}$ be the set of all such positions. For every position
$p\in\mathcal P$, we assume there is some true probability
distribution $D_p$. Let $\mathbb B=\{1,\dots,P-1\}$ be all bases after
which a breakpoint is possible, and let $B = \{p\in \mathbb B : D_p
\neq D_{p+1}\}$ be the set of true breakpoints.

The segmentation algorithm is given a sample of size $d \leq P$ of
data $(p_1, y_1), \dots, (p_d, y_d)$, with positions $p_i\in\mathcal
P$ and noisy observations $y_i\sim D_{p_i}$ for all samples $i\in\{1,
\dots, d\}$. The job of the segmentation algorithm is to return a
breakpoint guess $G\subseteq \mathbb B$. The object of this section is to
define the breakpointError $E^B_{\text{exact}}(G)$, which quantifies
the accuracy of the guess $G$ with respect to the true breakpoints
$B$.

\subsection{Desired properties of the breakpointError function}
\label{sec:desired-properties}

We would like the breakpointError function $E^B_{\text{exact}}:
2^{\mathbb B}\rightarrow \RR^+$ to satisfy the following properties:

\begin{itemize}
\item \textbf{(correctness)} Guessing exactly right costs nothing:
  $E^B_{\text{exact}}(B)=0$.
\item \textbf{(precision)} A guess closer to a real breakpoint is less
  costly:\\if $B=\{b\}$ and $0\leq i<j$, then
  $E^B_{\text{exact}}(\{b+i\})\leq E^B_{\text{exact}}(\{b+j\})$ and
  $E^B_{\text{exact}}(\{b-i\})\leq E^B_{\text{exact}}(\{b-j\})$.
\item \textbf{(FP)} False positive breakpoints are
  bad: if $b\in B$ and $g\not\in B$, then $E^B_{\text{exact}} (\{b\}) <
  E^B_{\text{exact}} (\{b,g\})$.
\item \textbf{(FN)} Undiscovered breakpoints are bad:
  $b\in B\Rightarrow E^B_{\text{exact}}(\{b\}) < E^B_{\text{exact}} (\emptyset)$.
\end{itemize}

In the next section we define the breakpointError, which satisfies all
4 properties.

\newpage

\subsection{Definition of the breakpointError function}
\label{sec:breakpoint_error}

In this section, we use the exact breakpoint locations $B=\{B_1,
\dots, B_n\}$ to define the breakpointError function.

We define the error of a breakpoint location guess $g\in\mathbb
B$ as a function of the closest breakpoint in $B$. So
first we put the breaks in order, by writing them as $B_1<\cdots<
B_n$. Then, we define a set of intervals
$R_B=\{\r_1,\dots,\r_n\}$ that form a partition of $\mathbb B$. For each
breakpoint $B_i$ we define the region
${\r}_i=[\rileft,\riright]\in\mathbb I \mathbb B$, where $\mathbb
I\mathbb B\subset 2^{\mathbb B}$ denotes the set of all intervals of
$\mathbb B$. We take the notation conventions from the interval
analysis literature \citep{intervals}.

We define the upper limit of region $i$ as
\begin{equation}
  \label{eq:R_i}
\riright
=
  \begin{cases}
    P-1 & \text{if } i=n \\
    \lfloor (B_{i+1}+B_i)/2 \rfloor & \text{otherwise}
  \end{cases}
\end{equation}
and
the lower limit as
\begin{equation}
  \label{eq:L_i}
  \rileft =
  \begin{cases}
    1 & \text{if } i=1 \\
    \riright[i-1]+1 & \text{otherwise}.
  \end{cases}
\end{equation}

The breakpoints $B_i$ and regions $\r_i$ are labeled for a small
signal in Figure~\ref{fig:exact_imprecision}.

\begin{figure}[H]
  \centering
  \input{figure-breakpoint-error-pieces}
  \vskip -0.5cm
  \caption{For a small signal with 2 breakpoints, and for breakpoints
    $i\in\{1, 2\}$, we plot the $\ell_i$ functions that measure the
    precision of a guess in $\mathbf r_i = [\underline r_i, \overline
    r_i]$.  The blue signal $\mathbf m\in\RR^{22}$ has 2 breakpoints:
    $B=\{4,14\}$. To emphasize the discrete nature of the data, N is
    drawn at each of the $P=22$ distinct positions at which data
    could be gathered.}
  \label{fig:exact_imprecision}
\end{figure}
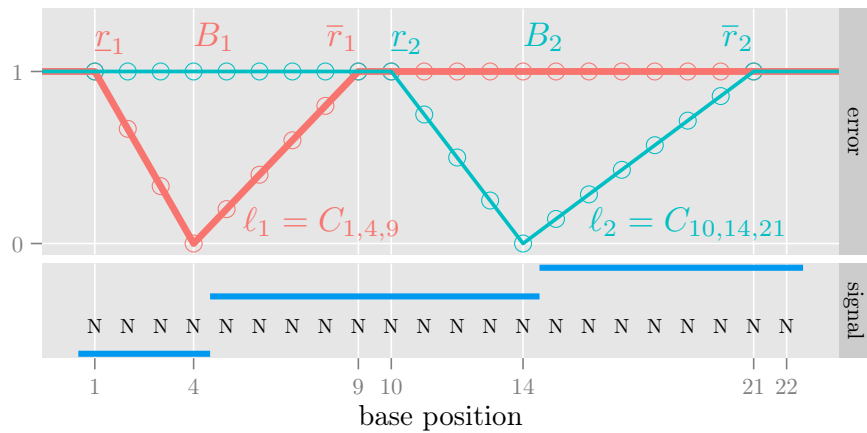

Intuitively, if we observe a breakpoint guess $g\in \r_i$, then its
closest breakpoint is $B_i$. To define the best guess in each region,
we use piecewise affine functions $C_{\underline r,b,\overline
  r}:\RR\rightarrow[0,1]$ defined as follows:
\begin{equation}
  \label{eq:cLxR}
  C_{\underline r,b,\overline r}(g) =
  \begin{cases}
    0 & \text{if }g=b \\
    (b-g)/(x-\underline r) & \text{if } \underline r< g< b \\
    (g-b)/(\overline r-x) & \text{if } b< g< \overline r\\
    1 & \text{ otherwise}.
  \end{cases}
\end{equation}
For each breakpoint $i$ we measure the precision of a guess
$g\in\mathbb B$ using
\begin{equation}
  \label{eq:ell_i_exact}
  \ell_i(g)=C_{\rileft,B_i,\riright}(g).
\end{equation}
These piecewise affine functions are shown in
Figure~\ref{fig:exact_imprecision} for a small signal with 2
breakpoints. Note that there is some degree of arbitrary choice in the
definition of the $\ell_i$ functions. For example, properly defined
piecewise quadratic functions could also satisfy the
\textbf{precision} property desired of the breakpointError
(Section~\ref{sec:desired-properties}).

Now, we are ready to define the exact breakpointError of a set of
guesses $G\subseteq\mathbb B$.
First, let $G \cap\r$ be the subset of guesses $G$ that fall in
region~$\r$. 

Then, we define the false negative rate for region $\r$ as 
\begin{equation}
  \label{eq:FN_i}
  \text{FN}(G,\r) = 
  \begin{cases}
    1 & \text{if } G\cap\r = \emptyset\\
    0 & \text{otherwise}
  \end{cases}
\end{equation}
and the false positive rate for region $\r$ as
\begin{equation}
  \label{eq:FP_i}
  \text{FP}(G,\r) =
  \begin{cases}
    0 & \text{if }G\cap\r = \emptyset\\
    |G\cap\r|-1 &\text{otherwise}
  \end{cases}
\end{equation}
and the imprecision of the best guess in region $r$ as
\begin{equation}
  \label{eq:imprecision}
  I(G,\r,\ell) =
  \begin{cases}
    0 & \text{if } G\cap\r = \emptyset\\
    \min_{g\in G\cap\r} \ell(g) & \text{otherwise}.
  \end{cases}
\end{equation}
When there are no breakpoints, we have $B=\emptyset$ and
$R_B=\emptyset$. But we still would like to quantify the false
positives, so let $G\setminus\big( \cup R_B\big) $ be the set of
guesses $G$ outside of the breakpoint regions $R_B$. 

\begin{definition}
  \label{def:exact_breakpoint_cost}
  The \textbf{breakpointError} of set of breakpoint guesses $G$ with
  respect to the true breakpoints $B$ is the sum of the False
  Positive, False Negative, and Imprecision functions:

\begin{equation*}
  {E }_{\text{exact}}^B(G) =
  \big|G\setminus(\cup R_B)\big|
 + \sum_{i=1}^{|B|}\textrm{FP}(G,\r_i)+\textrm{FN}(G,\r_i)+I(G,\r_i,\ell_i).
\end{equation*}
\end{definition}

\subsection{Implementation}

To compute the exact breakpointError, we first sort lists of $n=|B|$
and $m=|G|$ items. Using the quicksort algorithm, this requires
$O(n\log n + m\log m)$ operations in the average case
\citep{clrs}. Once sorted, the components of the cost can be computed
in linear time $O(n + m)$. So, overall the computation of the error
can be accomplished in best case $O(n + m)$, average case $O(n\log n +
m\log m)$ operations. Its computation is implemented in efficient C
code in the \verb|breakpointError| R package on R-Forge, which can be
installed in R using

\begin{verbatim}
install.packages("breakpointError", repos="http://r-forge.r-project.org")
\end{verbatim}

\newpage

\section{Penalties with minimal breakpointError in simulations}
\label{sec:simulations}
In this section, we show several examples of how to use the
breakpointError function to determine penalties which minimize the
train and test breakpointError in simulated data sets. In all cases,
we will assume that there is a database of several piecewise constant
signals with Gaussian noise. The goal is to learn a penalty constant
that can be shared between signals with different properties. In each
of the following sections, we will first present an empirical analysis
of several simulated signals using the breakpointError. Then, we will
discuss the relationship of our results to relevant theoretical
results.

\subsection{Sampling density normalization}
\label{variable_density}
The first problem we consider is finding a penalty that is invariant
to sampling density. This is important because sampling density is
often not uniform in real data sets. In fact, we see a sampling
density between 40 and 4400 kilobases per probe in the neuroblastoma
data set of \citet{HOCKING-breakpoints}. We would like to construct a
single algorithm or penalty function that can be used for each of
these segmentation problems.

So to determine the form of the penalty function that can best adapt
to this variation, we analyze the following simulation. We create a
true piecewise constant signal $\mathbf m\in\RR^P$ over $P=70000$ base
pairs, with breakpoints every 10000 base pairs, shown as the blue line
in Figure~\ref{fig:variable-density-signals}. Then, we define a signal
sample size $d_i\in\{70,\dots,70000\}$ for every noisy signal
$i\in\{1,\dots,z\}$. Let $y_i\in\RR^{d_i}$ be noisy signal $i$,
sampled at positions $p_i\in\mathcal P^{d_i}$, with
$p_{i1}<\cdots<p_{i,d_i}$. We sample every probe $j$ from the
$y_{ij}\sim N(m_{p_{ij}},1)$ distribution. These samples are shown as
the black points in Figure~\ref{fig:variable-density-signals}.

We would like to learn some model complexity parameter $\lambda$ on
the first noisy signal, and use it for accurate breakpoint detection
on the second noisy signal. In other words, we are looking for a model
selection criterion which is invariant to sampling density. 

\begin{figure}[H]
\includegraphics[width=\linewidth]{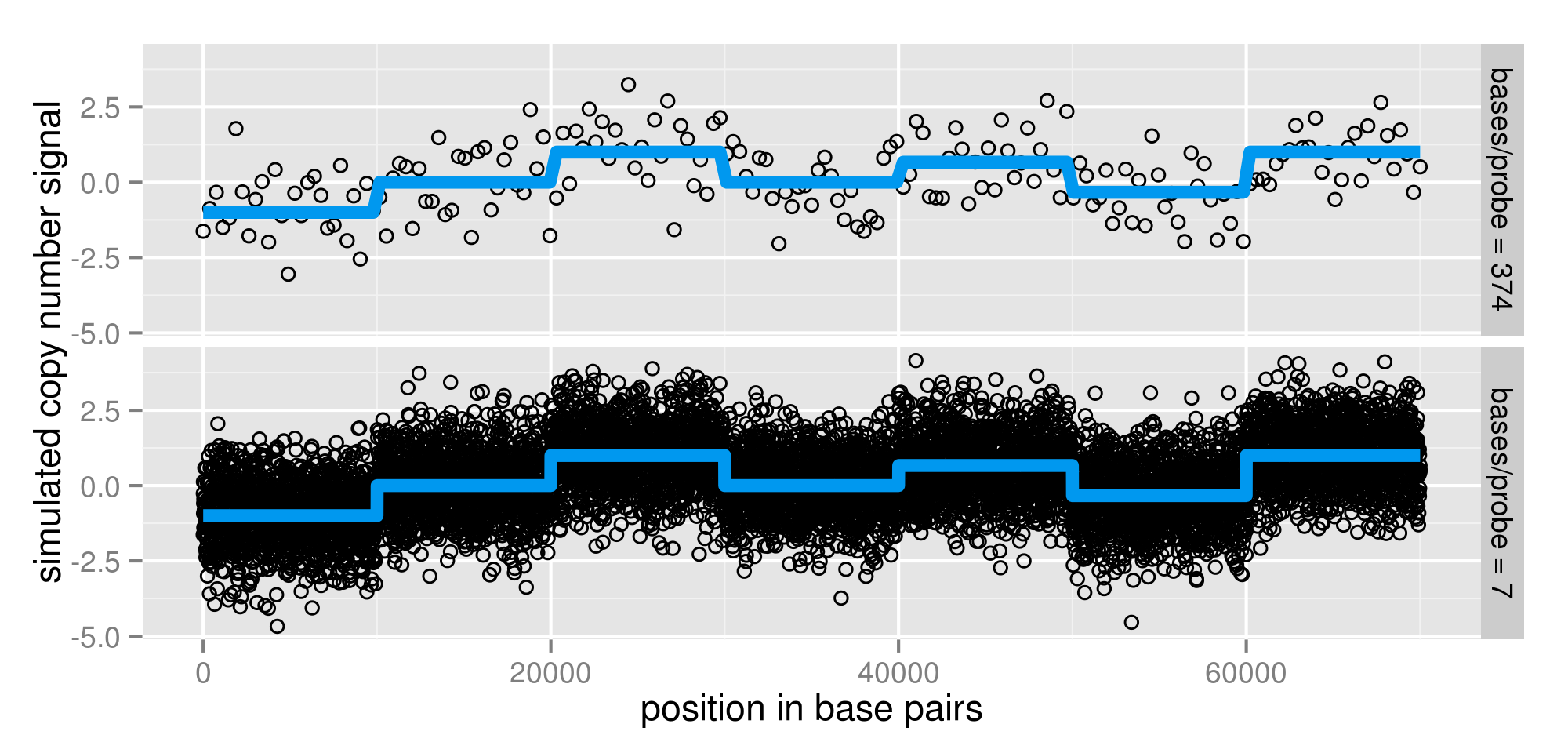}
\vskip -0.3cm
  \caption{Two noisy signals (black) sampled from
  a true piecewise constant signal (blue). 
Note that these are the same signals that appear in
  Figure~\ref{fig:variable-density-annotation-cost}.}
\label{fig:variable-density-signals}
\end{figure}

\newpage

To attack this problem, we proceed as follows. For every signal $i$,
we use pruned dynamic programming to calculate the maximum likelihood
estimator $\mathbf{\hat y}^k_i\in\RR^{d_i}$ (\ref{eq:yhat^k}), for several model sizes
$k\in\{1,\dots,k_{\text{max}}\}$ \citep{pruned-dp}. Then, we define
the model selection criteria
\begin{equation}
  \label{eq:kstar_density}
  k^\alpha_i(\lambda) =\argmin_k \lambda k d_i^\alpha + 
  ||\mathbf{y_i}-\mathbf{\hat y}^k_i||_2^2.
\end{equation}
Each of these is a function $k_i^\alpha:\RR^+\rightarrow
\{1,\dots,k_{\text{max}}\}$ that takes a model complexity tradeoff
parameter $\lambda$ and returns the optimal number of segments for
signal $i$. The goal is to find a penalty exponent $\alpha\in\RR$ that
lets us generalize $\lambda$ between different signals $i$.

To quantify the accuracy of a segmentation for signal $i$, let
$\text{BErr}_i(k)$ be the breakpointError of the model with
$k$ segments. This is a function
$\text{BErr}_i:\{1,\dots,k_{\text{max}}\}\rightarrow\RR^+$, defined as
\begin{equation}
  \label{eq:berr}
  \text{BErr}_i(k) = E_{\text{exact}}^{B}\left[
\phi(\mathbf{\hat y}_i^k,p_i)
\right].
\end{equation}
where $B$ is the set of real breakpoints in the true piecewise constant
signal $\mathbf m$.

In Figure~\ref{fig:variable-density-berr-k}, we plot $\text{BErr}_i$
for the 2 simulated signals $i$ shown previously.  As expected, the
model recovers more accurate breakpoints from the signal sampled at a
higher density.

  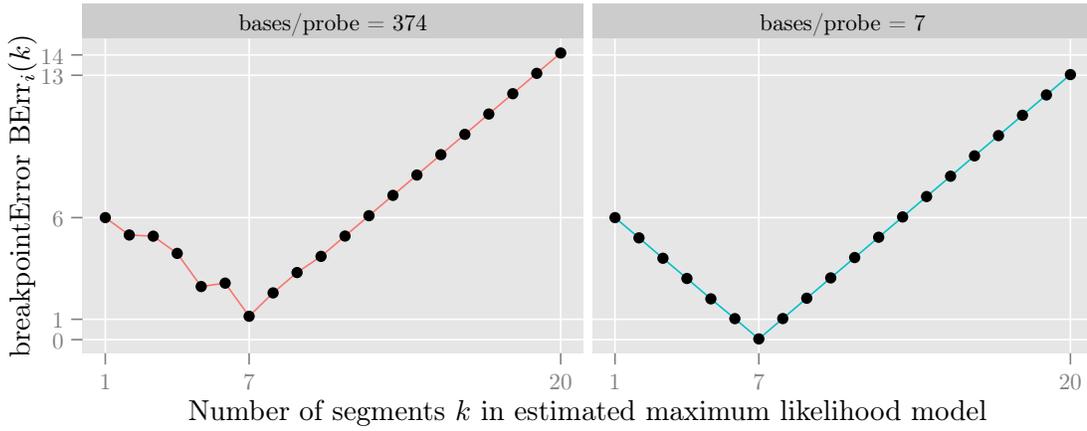
\begin{figure}[H]
    \hskip -1cm
    \input{figure-variable-density-berr-k}
    \caption{Exact breakpoint error
  $\text{BErr}_i(k)$ for two signals $i$ and several model
  sizes $k$. Note that these are the same error curves that appear
in the Breakpoint panels of Figure~\ref{fig:variable-density-sigerr-small}.}
    \label{fig:variable-density-berr-k}
  \end{figure}

\newpage

Now, let us define the penalized
model breakpoint error $E^\alpha_i:\RR^+\rightarrow\RR^+$ as
\begin{equation}
  \label{eq:lerr}
E^\alpha_i(\lambda) = \text{BErr}_i\left[
k^\alpha_i(\lambda)
\right].
\end{equation}
In Figure~\ref{fig:variable-density-berr}, we plot these functions for the
two signals $i$ shown previously, and for several penalty exponents $\alpha$.

The dots in Figure~\ref{fig:variable-density-berr} show the optimal
$\lambda$ found by minimizing the penalized model breakpoint detection
error:
\begin{equation}
  \label{eq:lambda_hat}
  \hat \lambda^\alpha_i = \argmin_{\lambda\in\RR^+}  E^\alpha_i(\lambda)
\end{equation}

Figure~\ref{fig:variable-density-berr} suggests that $\alpha\approx1/2$
defines a penalty with aligned error curves, which will result in
$\hat \lambda_i^\alpha$ values that can be generalized between
profiles. 

  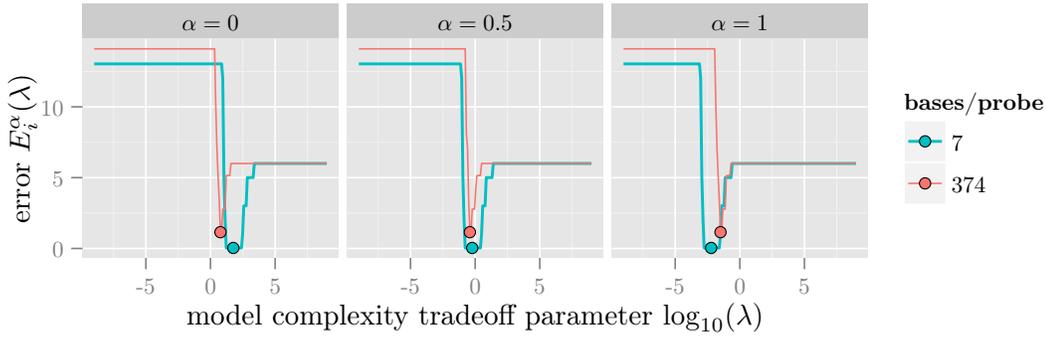
\begin{figure}[H]
    \hskip -1cm
    \input{figure-variable-density-berr}
    \caption{Model selection error curves
  $E_i^\alpha(\lambda)$ for 2 signals $i$ and several exponents
  $\alpha$. The penalty contains a term for the number of points sampled $d_i^\alpha$.}
    \label{fig:variable-density-berr}
  \end{figure}

\newpage

Now, we are ready to define 2 quantities that will be able to help us
choose an optimal penalty exponent $\alpha$.

First, let us consider the training error over the entire database:
\begin{equation}
  \label{eq:lerr_train}
  E^\alpha(\lambda) = \sum_{i=1}^z E_i^\alpha(\lambda),
\end{equation}
and we define the minimal value of this function as
\begin{equation}
  \label{eq:lerr_train_min}
  E^*(\alpha) = \min_\lambda E^\alpha(\lambda).
\end{equation}
In Figure~\ref{fig:variable-density-error-train}, we plot these
training error functions $E^\alpha$ and their minimal values $E^*$ for
several values of $\alpha$. It is clear that the minimum training
error is found for some penalty exponent $\alpha$ near 1/2, and we
would like to find the precise $\alpha$ that results in the lowest
possible minimum $E^*(\alpha)$.

  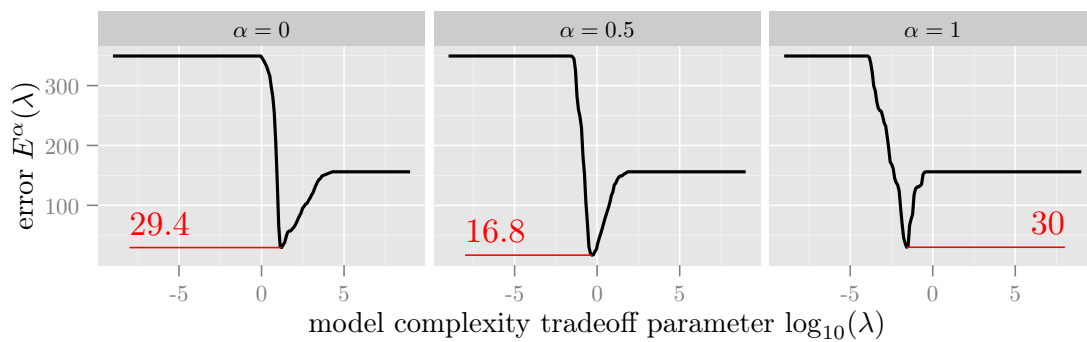
\begin{figure}[H]
    \hskip -1cm
    \input{figure-variable-density-error-train}
    \caption{Training error functions $E^\alpha$
  in black and their minimal values $E^*(\alpha)$ in red. The penalty
  contains a term for the number of points sampled $d_i^\alpha$.}
    \label{fig:variable-density-error-train}
  \end{figure}

\newpage

We also consider the test error over all pairs of signals when
training on one and testing on another:
\begin{equation}
  \label{eq:lerr_test}
  \text{TestErr}(\alpha) = 
  \sum_{i\neq j} E^\alpha_i(\hat \lambda_j^\alpha).
\end{equation}

In Figure~\ref{fig:variable-density-error-alpha}, we plot $E^*$ and
TestErr for a grid of $\alpha$ values.  It is clear that the optimal
penalty is given by $\alpha=1/2$. This corresponds to the following
model selection criterion which is invariant to the number of data
points sampled $d_i$ (for different simulated signals $i$ with the
same true breakpoints):
\begin{equation}
  \label{eq:var_dens_opt_pen}
  k_i(\lambda) = 
  \argmin_k \lambda k \sqrt{d_i}+||\mathbf y_i-\mathbf{\hat y}_i^k||^2_2
\end{equation}

  \begin{figure}[H]
    \hskip -1cm
    \input{figure-variable-density-error-alpha}
    \caption{Train and test breakpoint detection
  error as a function of penalty exponent $\alpha$. The penalty
  contains a term for the number of points sampled $d_i^\alpha$. Mean
  error is drawn as a black line, with one standard deviation shown as
  a grey band.}
    \label{fig:variable-density-error-alpha}
  \end{figure}
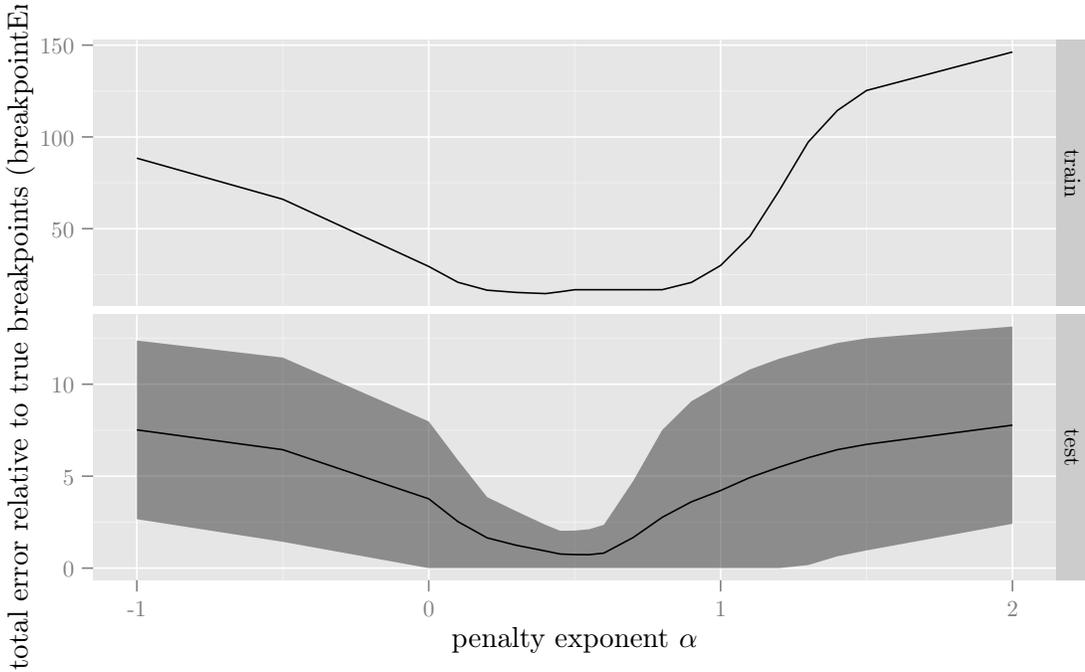

\newpage

As explained by \citet{sylvain-survey}, a model selection procedure
can be either efficient or consistent. An efficient procedure for
model estimation accurately recovers the true piecewise constant signal, whereas a
consistent procedure for model identification accurately recovers the
breakpoints. Since we attempt to minimize the breakpointError, we are
attempting to construct a consistent penalty, not an efficient
penalty.

In general terms, the fact that we find a nonzero exponent $\alpha$
for our $d_i^\alpha$ penalty term agrees with other results. In
particular, \citet{vfold} proposed an optimal procedure to select model
complexity parameters in cross-validation by normalizing by the sample
size $d_i$. 

The $\sqrt{d_i}$ term that we find here using simulations is in
agreement with \citet{aurelie}, who used finite sample model selection
theory to find a $\sqrt{d_i}$ term in a penalty optimal for
clustering.

When theoretically deriving an efficient penalty for segmentation
model estimation in the non-asymptotic setting, \citet{lebarbier}
obtained a $\log d_i$ term. This contrasts our result, which attempts
to find a consistent penalty, and uses the breakpointError to find a
$\sqrt{d_i}$ penalty term. But in fact this is in agreement with
classical results that the efficient AIC underpenalizes with respect
to the consistent BIC, as shown in Table~\ref{tab:AIC-BIC}.

\begin{table}[H]
  \centering
  \begin{tabular}{cc|cc}
     Efficient & Penalty & Consistent & Penalty \\
     Model & Term & Model & Term\\
     \hline
     AIC & 2 & BIC & $\log d_i$\\
     Lebarbier & $\log d_i$ & This work & $\sqrt{d_i}$\\
  \end{tabular}
  \caption{Comparing our results with Lebarbier, 
in the context of classical results involving AIC and BIC. 
The BIC is designed for model identification and penalizes more than the AIC.
Likewise, our penalty examines model identification using the breakpoint
detection error, and penalizes more than the efficient penalty proposed
by Lebarbier.}
  \label{tab:AIC-BIC}
\end{table}

\newpage
\subsection{Signal length normalization}
\label{variable_size}
In real array CGH data, we need to analyze chromosomes of varying length
in base pairs. For example, human chromosome 1 is the largest at about
250 mega base pairs, and chromosome 22 is the smallest with only about
36 mega base pairs. But we expect that the number of breakpoints is
proportional to the length of the chromosome in base pairs, and we would
like to design a model selection criterion that is invariant to the
signal length.

So as a first step toward constructing a penalty that is invariant to
the number of breakpoints, we consider the following simulation where
we fix the number of points sampled at $d_i=2000$, and vary the length
of the signal sampled. In
Figure~\ref{fig:variable-breaks-constant-size}, we show samples of 2
different lengths $l_i$, for the same true piecewise constant signal
$\mathbf m$. This simulation is somewhat unrealistic since the number
of data points $d_i$ in real data sets is usually proportional to the
signal length $l_i$. We will consider a more realistic simulation
model and a more complicated penalty in the next section.

\figpdf{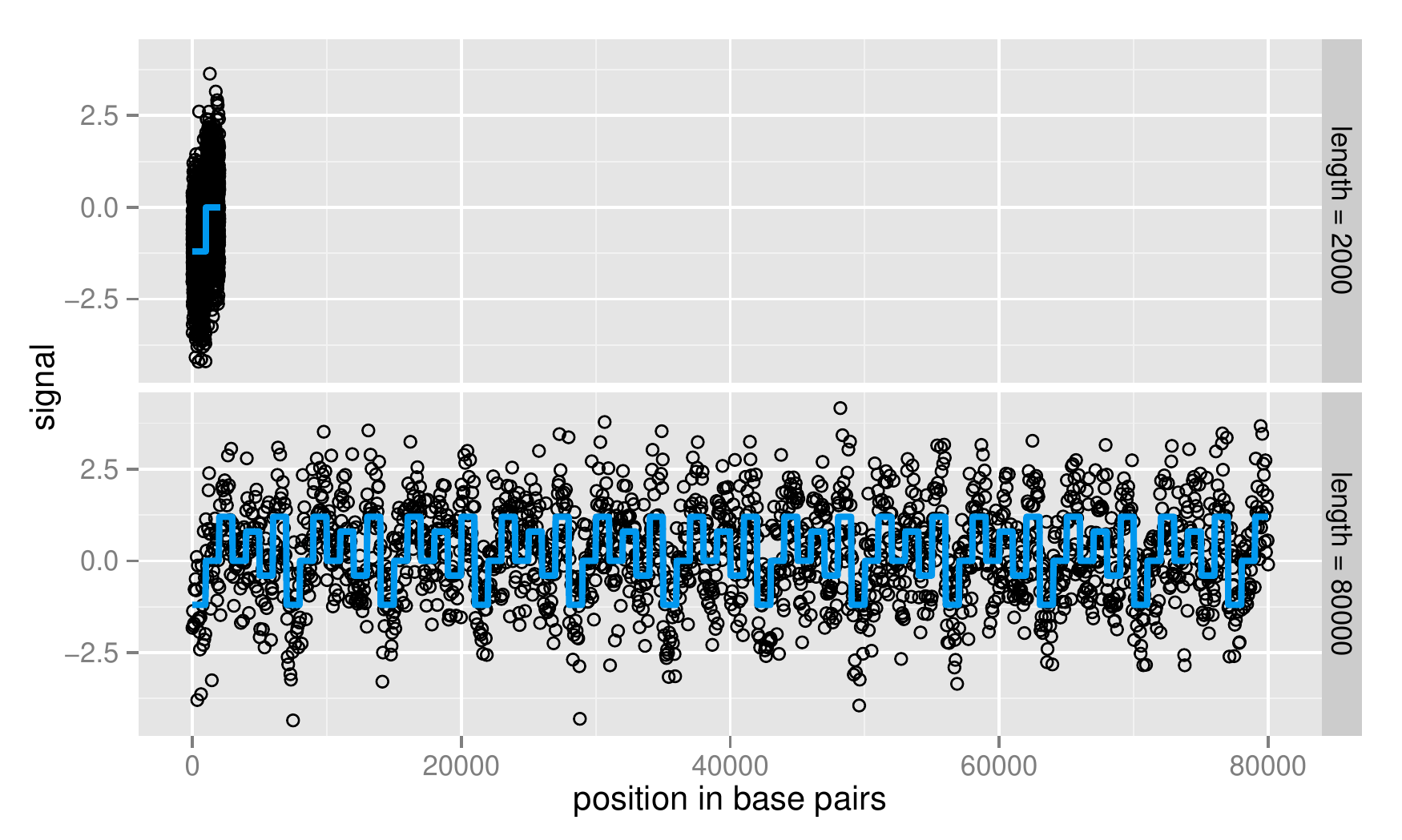}{Samples of 2 different lengths
  $l_i$ but constant number of points $d=2000$.}

\newpage

For each signal $i$, we define the penalty
\begin{equation}
  \label{eq:kstar_length}
  k_i^\beta(\lambda) = \argmin_k \lambda k l_i^\beta
+ ||\mathbf y_i - \mathbf{\hat y}_i^k||^2_2,
\end{equation}
where $l_i$ is the length of the signal in base pairs. The goal will
be to find a $\beta$ that can be used for signals of varying length.

In Figure~\ref{fig:variable-breaks-constant-size-berr}, we show
the breakpoint detection error curves for two signals and several
penalty exponents $\beta$.
These curves seem to align when $\beta=-1/2$.

  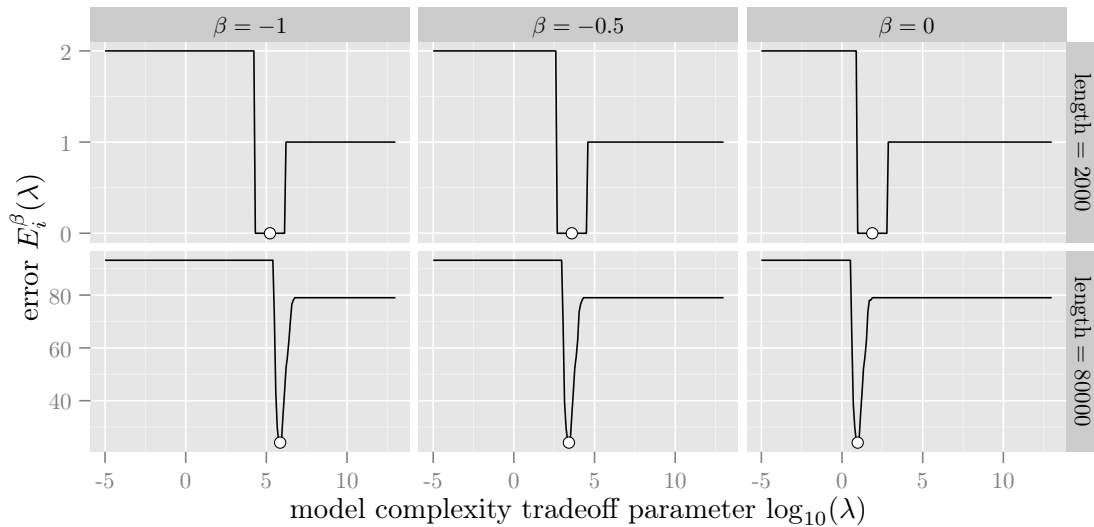
\begin{figure}[H]
    \hskip -1cm
    \input{figure-variable-breaks-constant-size-berr}
    \caption{Breakpoint detection error
  curves for several penalty exponents $\beta$ and 2 samples of
  varying length in base pairs $l_i$. The penalty contains a term
  $l_i^\beta$.}
    \label{fig:variable-breaks-constant-size-berr}
  \end{figure}

\newpage

In Figure~\ref{fig:variable-breaks-constant-size-alpha}, we plot the
train and test error curves over the entire set of simulated signals.
These curves indicate minimal breakpoint detection error at
$\beta=-1/2$, corresponding to the following penalty:
\begin{equation}
  \label{eq:kstar_length_opt}
  k_i(\lambda) = \argmin_k \frac{\lambda k}{\sqrt{l_i}}
  + ||\mathbf y_i-\mathbf{\hat y}_i^k||^2_2.
\end{equation}

  \begin{figure}[H]
    \hskip -1cm
    \input{figure-variable-breaks-constant-size-alpha}
    \caption{Train and test error curves
  for signals of different length in base pairs $l_i$. The penalty
  contains a term
  $l_i^\beta$.}
    \label{fig:variable-breaks-constant-size-alpha}
  \end{figure}
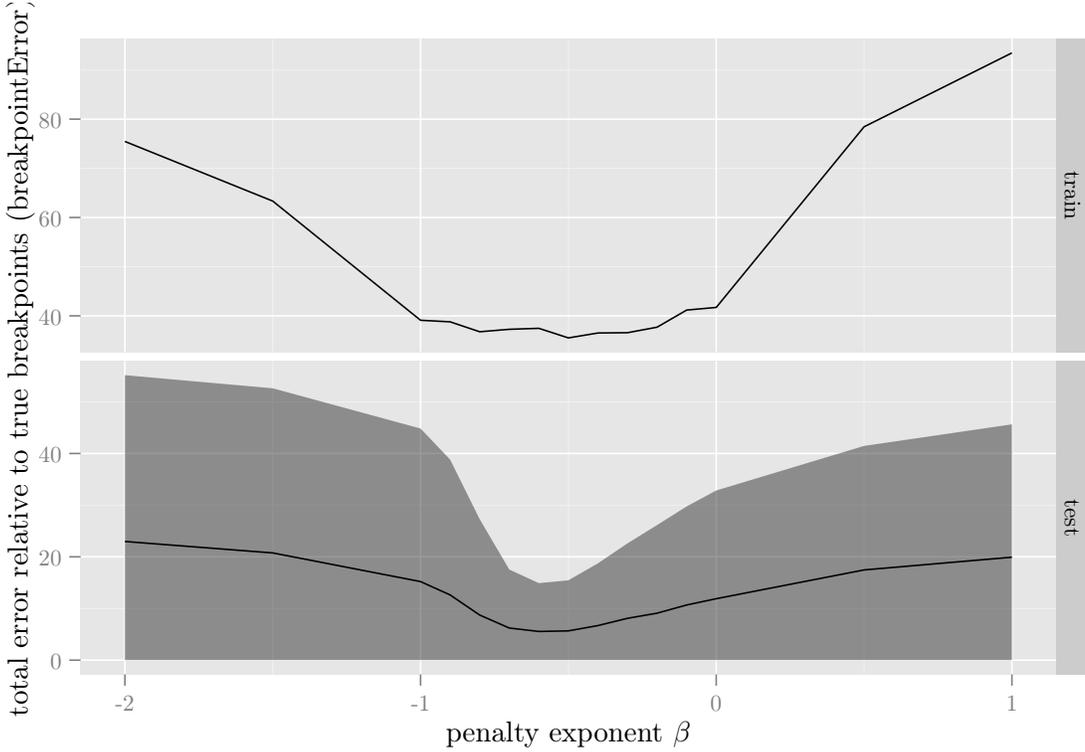

Interestingly, the $1/\sqrt{l_i}$ term that we obtain here is in good
agreement with our previous result that the optimal penalty for
variable sampling density $d_i$ should have a $\sqrt{d_i}$ term. In
particular, we can re-parameterize the problem to be in terms of the
number of points sampled per segment $\rho_i=d_i/k_i$. In
Section~\ref{variable_density} we held $k_i$ constant but in this
section we hold $d_i$ constant. In both cases we have a penalty with a
$\sqrt{\rho_i}=\sqrt{d_i/k_i}$ term.

However, we do not know the number of segments $k_i$ in advance. But
we supposed that the number of segments is proportional to the number
of base pairs $l_i$, so we can use that in the penalty. This suggests
a penalty that takes the form of $\sqrt{d_i/l_i}$. So in the next
section, we confirm that this intuition works for constructing an
optimal penalty.

\newpage
\subsection{Combining normalizations}
\label{combining_penalties}
In this section, we show that we can combine the results of the
previous sections to create composite invariant penalties. In
particular, to normalize for sampling density $d_i$ and length in base
pairs $l_i$, we need $\sqrt{d_i}$ and $1/\sqrt{l_i}$ terms in the
penalty, respectively. This suggests that when considering variable
$d_i$ and $l_i$, we need a $\sqrt{d_i/l_i}$ term in the penalty, and
in this section we show that this intuitive construction results in an
optimal penalty.

In Figure~\ref{fig:variable-size-signals}, we plot 2 signals with
different number of points $d_i$ and length in base pairs $l_i$. In
particular we tested $d_i\in\{50, \dots, 1000\}$ and $l_i\in\{200,
\dots, 1000\}$. We would like to find a penalty that allows us to
generalize model complexity tradeoff parameters $\lambda$ between
these signals.

For each signal $i$, we define the penalty
\begin{equation}
  \label{eq:kstar_composite}
  k_i^{\alpha,\beta}(\lambda) = \argmin_k \lambda k {d_i}^\alpha l_i^\beta
  + ||\mathbf y_i - \mathbf{\hat y}_i^k||^2_2,
\end{equation}
where $l_i$ is the signal length in base pairs and $d_i$ is the number
of points sampled. We will attempt to determine a pair of $\alpha$ and
$\beta$ values that allow accurate breakpoint detection in signals of
varying length and number of points sampled. Based on the results in
Sections~\ref{variable_density} and \ref{variable_size}, we expect to
find $\alpha=1/2$ and $\beta=-1/2$.

\begin{figure}[H]
  \centering
\includegraphics[width=\linewidth]{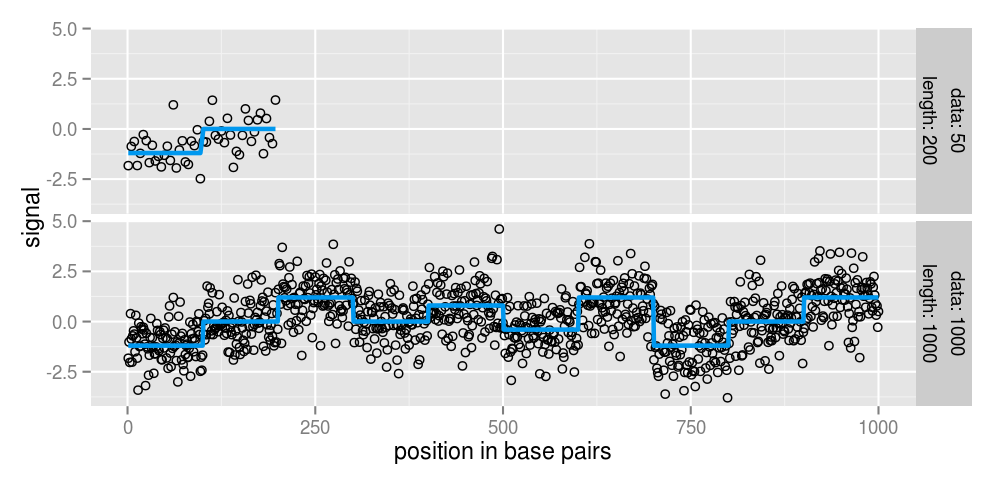}
  \caption{Two signals with a different number of
    points $d_i$ and length in base pairs $l_i$.}
\label{fig:variable-size-signals}
\end{figure}

\newpage

In Figure~\ref{fig:variable-size-error-alpha-beta}, we plot the train
and test breakpoint error functions as a function of both $\alpha$ and
$\beta$. It is clear that the minimum is achieved by penalties near
$\alpha=1/2, \beta=-1/2$, which corresponds to a penalty of

\begin{equation}
  \label{eq:kstar_composite_values}
  k_i^{\alpha,\beta}(\lambda) = \argmin_k \lambda k
  \sqrt{d_i/l_i}
  + ||\mathbf y_i - \mathbf{\hat y}_i^k||^2_2,
\end{equation}

\figpdf{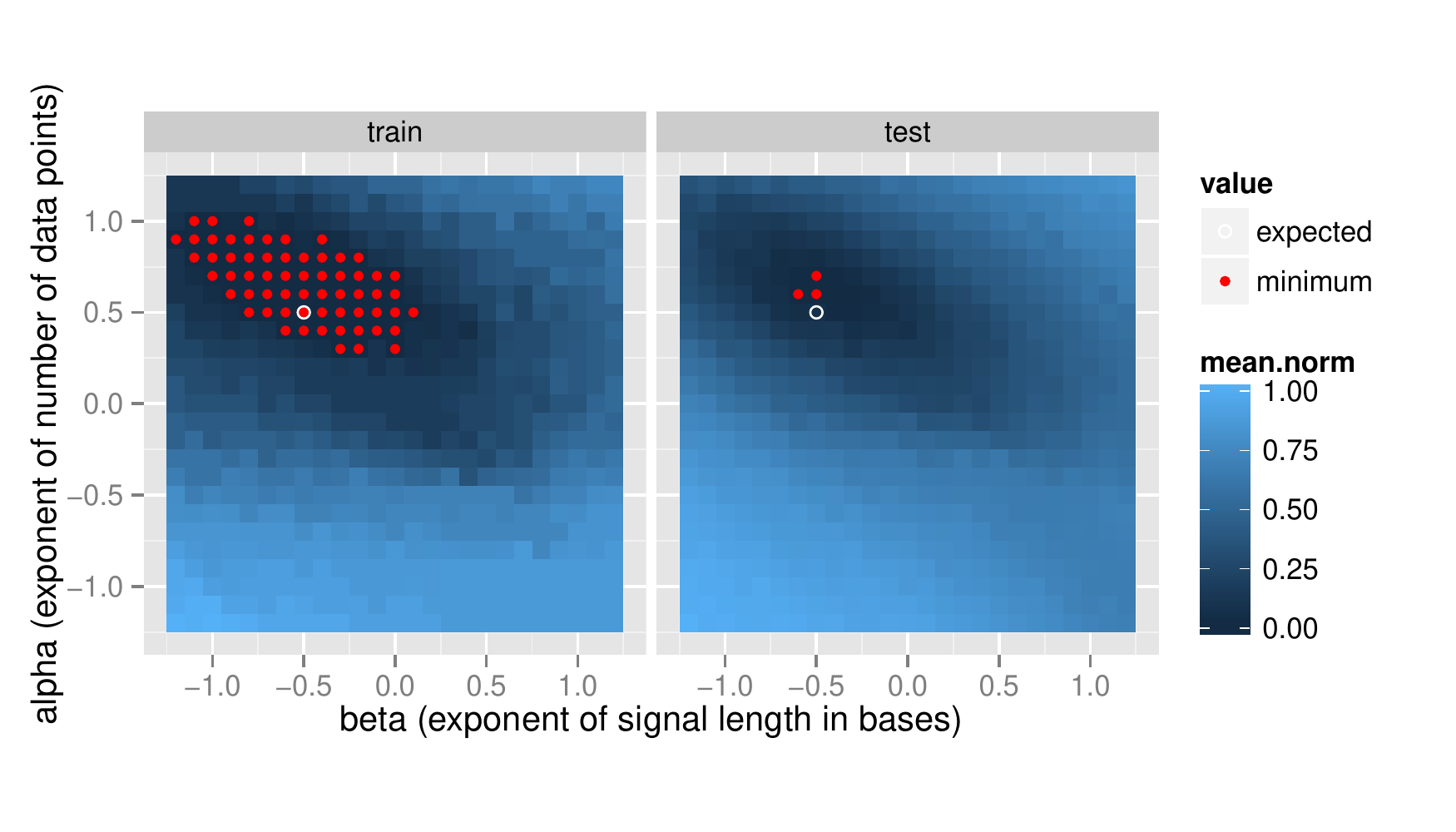}{Train and test error functions
  for several signals of varying number of points $d_i$ and length
  $l_i$. The penalty contains a term $d_i^\alpha l_i^\beta$. Mean
  error values normalized to $[0,1]$, minimum values indicated in red,
  and expected value $\alpha=1/2, \beta=-1/2$ in white. }






\newpage
\subsection{Optimal penalties for the fused lasso signal approximator}

In the previous sections, we used theoretical arguments and simulation
experiments to determine the optimal penalties for maximum likelihood
segmentation (\ref{eq:yhat^k}). In this section, we demonstrate that
the same approach can be used to find optimal penalties for another
model, the Fused Lasso Signal Approximator (FLSA).

We used the \verb|flsa| function in version 1.03 of the flsa
package from CRAN to calculate the FLSA \citep{fused-lasso-path}. Let
$\mathbf x\in\RR^d$ be the noisy copy number signal for one chromosome. The
FLSA solves the following optimization problem:
\begin{equation}
  \label{eq:flsa}
\argmin_{\mathbf m\in\RR^d} 
\frac 1 2 \sum_{j=1}^d (x_j-m_j)^2
+\lambda_1\sum_{j=1}^d|m_j|
+\lambda_2\sum_{j=1}^{d-1}|m_j-m_{j+1}|.
\end{equation}

First, we take $\lambda_1=0$ since we are concerned with breakpoint
detection, not signal sparsity. In this section, our aim is to
determine a parameterization for $\lambda_2$ that we will be able to
find similar breakpoints in signals of varying sampling density.

We use the same setup that we used to determine optimal penalties for
maximum likelihood segmentation, as described in
Section~\ref{variable_density} and shown again in
Figure~\ref{fig:variable-density-signals-flsa}.

\begin{figure}[h]
  \centering
  \includegraphics[width=\linewidth]{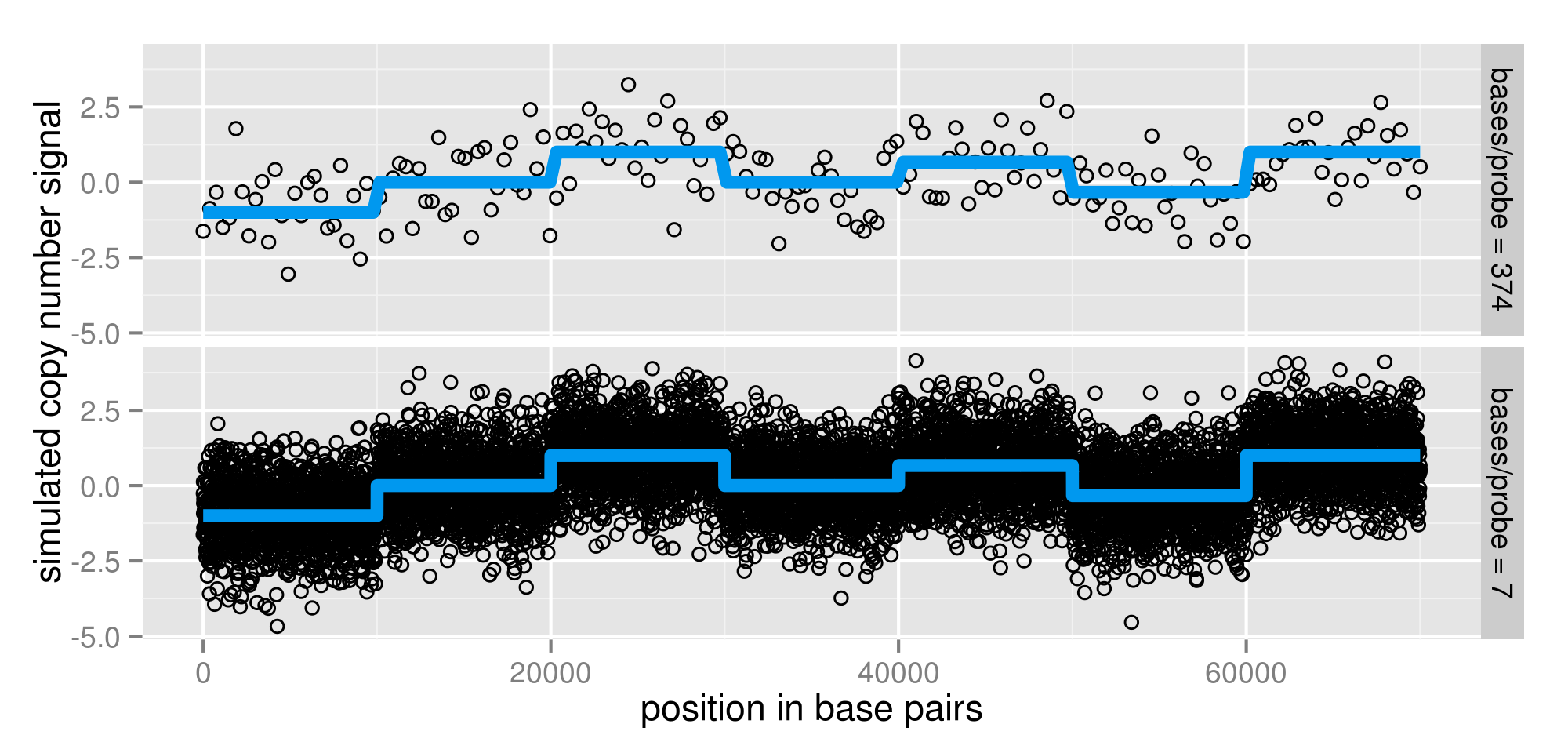}
  \caption{Simulated signals with different
  sampling density.}
  \label{fig:variable-density-signals-flsa}
\end{figure}

\newpage

In particular, for every signal $i\in\{1,\dots,z\}$, let
$\mathbf y_i\in\RR^{d_i}$ be the noisy data, sampled at positions
$\mathbf p_i\in\mathcal P^{d_i}$. To find an optimal penalty for these data,
first let $\lambda_2 = \lambda d_i^\alpha$. For each signal $i$,
exponent $\alpha\in\RR$, and tradeoff parameter $\lambda\in\RR^+$, we
define the optimal smoothing as
\begin{equation}
  \label{eq:flsa_lambda}
  \mathbf{\hat y}^{\lambda,\alpha}_i = 
\argmin_{\mathbf m\in\RR^{d_i}} 
\frac 1 2 ||\mathbf y_i-\mathbf m||_2^2
+\lambda d_i^\alpha \sum_{j=1}^{d_i-1} |m_j - m_{j+1}|.
\end{equation}

Then, we define the breakpoint detection error as a function of the
breaks in the smoothed signal:
\begin{equation}
  \label{eq:flsa_e_i_alpha}
  E_i^\alpha(\lambda) = 
E^B_{\text{exact}}
\left[
\phi\left(
\mathbf{\hat y}^{\lambda,\alpha}_i, \mathbf p_i
\right)
\right],
\end{equation}
where the breakpoint function $\phi$ is defined in
(\ref{eq:breaks_phi}).

We plot $E_i^\alpha$ for 2 signals $i$ and several penalty exponents
$\alpha$ in Figure~\ref{fig:variable-density-berr-flsa}. Note that the
functions appear to align when $\alpha=1$.

  \begin{figure}[H]
    \hskip -1cm
    \input{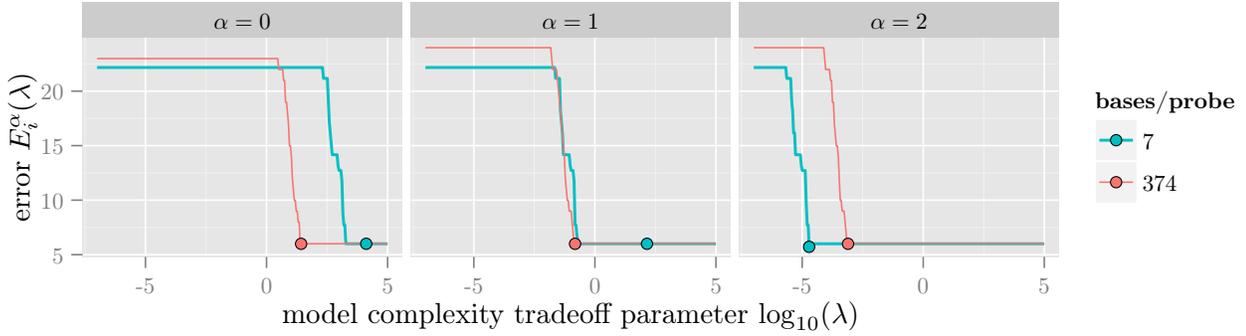}
    \caption{Model complexity breakpoint error
  functions $E_i^\alpha$.}
    \label{fig:variable-density-berr-flsa}
  \end{figure}

\newpage

To evaluate which penalty parameter $\alpha$ results in optimal
fitting and learning, we computed train error $E^*$ and TestErr as defined
in (\ref{eq:lerr_train}) and (\ref{eq:lerr_test}). These functions are
plotted in Figure~\ref{fig:variable-density-error-alpha-flsa}, and
suggest that a value of $\alpha=1$ is optimal. This analysis suggests
that taking $\lambda_2=\lambda d_i$ is optimal for breakpoint
detection using FLSA. This agrees with the observation of
\citet{HOCKING-breakpoints} that the flsa.norm penalty with a $d_i$
term works better than the un-normalized flsa penalty.

However, we obtained a different penalty ($\alpha=0.5$) in
Section~\ref{variable_density} for another model, maximum likelihood
segmentation. These differences in optimal $\alpha$ values are due to
the differences in how model complexity is measured in the two
models. Maximum likelihood segmentation measures model complexity
using the $\ell_0$ pseudo-norm of the difference vector of $\mathbf m$,
whereas the FLSA uses the $\ell_1$-norm.

We conclude by noting that this procedure could also be applied to
find penalties for FLSA that depend on other signal properties such as
length in base pairs~$l_i$. However, we did not pursue this since FLSA
does not work as well as maximum likelihood segmentation in practice
on real data \citep{HOCKING-breakpoints}.

  \begin{figure}[H]
    \hskip -1cm
    \input{figure-variable-density-error-alpha-flsa}
    \caption{Train and test error as a
  function of penalty exponent $\alpha$. The penalty has a term for
  the number of points sampled $d_i^\alpha$.}
    \label{fig:variable-density-error-alpha-flsa}
  \end{figure}
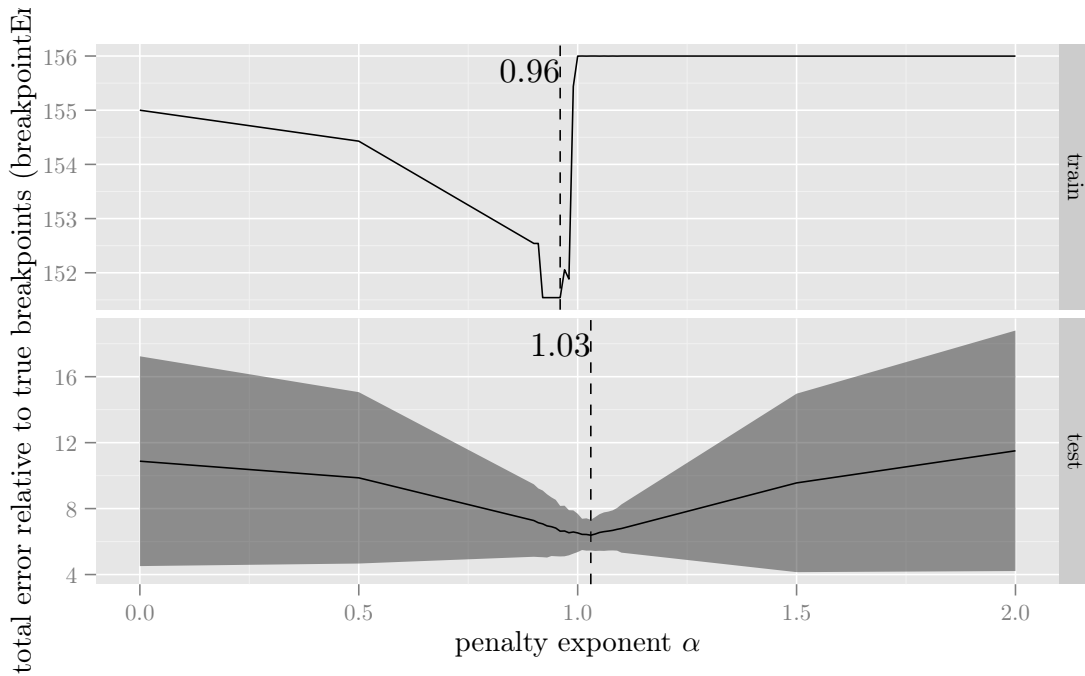

\newpage

\subsection{Applying the penalties to real data}

In Sections~\ref{variable_density}-\ref{variable_size}, we found
penalties with minimum breakpointError for simulated data with varying
number of data points sampled $d_i$ and length $l_i$ in base positions
(with $l_i$ proportional to the number of breakpoints). In
Section~\ref{combining_penalties}, we demonstrated that these results
can be combined. We also found that the optimal penalty should include
a term for the estimated variance $\hat s_i^2$
\citep{HOCKING-phd-ch4}. These results suggested the following
penalty, for every signal $i$:
\begin{equation}
  \label{eq:composite_penalty}
  k_i(\lambda) = 
  \argmin_k
  \lambda k \hat s_i^2 \sqrt{d_i/l_i}  +
  ||\mathbf y_i - \mathbf{\hat y}_i^k||^2_2
\end{equation}

In Table~\ref{tab:penalty-real-data}, we report results of using the
suggested penalties on the neuroblastoma data set described by
\citet{HOCKING-breakpoints}. The cghseg.k penalty which was found to
be the best by \citet{HOCKING-breakpoints} has a term for number of
data points sampled $d_i$ (no square root) but no terms for length
$l_i$ nor estimated variance $\hat s_i$. The penalty terms suggested
in this section do not improve breakpoint detection error in the
neuroblastoma data set. This observation suggests that distribution
that generates the real data is more complex than the simple
simulation model considered in this paper.

  \begin{table}[H]
    \begin{center}
          \input{table-penalty-real-data}
    \end{center}
    \caption{Breakpoint detection error of several
  penalties on the neuroblastoma data set, with 1 row for each
  penalty. The exponent of the number of data points $d_i$, length
  $l_i$, and variance $\hat s_i$ terms in the penalty is shown with
  the train and test annotation error (percent incorrect regions). }
    \label{tab:penalty-real-data}
  \end{table}

Practically speaking, we still would like to find a penalty with
optimal breakpoint detection for any particular real data set such as
the neuroblastoma data. 
\citet{HOCKING-penalties} achieved state-of-the-art breakpoint
detection in the neuroblastoma data set by learning
the penalty constants using a training data set of manually annotated
regions. 
For the rest
of this paper, we will discuss the relationship of the breakpointError
to these annotation-guided methods.

\newpage

\section{Annotation error functions for real data sets}
\label{sec:relaxation}

In this section, we define several annotation error functions which
can be used in real data sets (Table~\ref{tab:ann-err-funs}).  In real
data, we do not have access to the true piecewise constant signal
$\mathbf m$, nor the underlying set of breakpoints $B$. So the
breakpointError defined in the Section~3 is not readily
computable. We will first show how in real data, we can compute
another function called the incomplete annotation error. Then, we will
demonstrate its relationship to the breakpointError using the complete
annotation error function.

\begin{table}[H]
  \begin{center}
  \begin{tabular}{ccccc}
    Section & Error function & Symbol & Need & counts incorrect \\
    \hline
    \ref{sec:breakpoint_error}& breakpointError & $E^B_{\text{exact}}$ & 
    true breakpoints $B$ & guesses \\
    \ref{sec:incomplete} & incomplete annotation error & $E_{\text{incomplete}}^A$ &
    some annotations $A$ & guesses\\
    \ref{sec:complete} & complete annotation error & $E_{\text{complete}}^A$ & 
    all breakpoint annotations $A$ & guesses\\
    \ref{sec:zero-one}& 
    01 annotation error & $E_{01}^A$ & some annotations $A$ & regions
  \end{tabular}
  \end{center}
  \caption{Several breakpoint detection error functions, 
    and how much prior knowledge is needed to compute each. 
    The breakpointError needs the most prior knowledge and can only be 
    used in simulations when the true breakpoints $B$ are known. 
    In contrast, the incomplete/01 annotation error can be used in
    real data sets by using visual inspection of scatterplots to
    create annotations $A$.}
  \label{tab:ann-err-funs}
\end{table}

\subsection{Incomplete annotation error for real data}
\label{sec:incomplete}

By plotting a real data set, we can easily identify regions that
contain breakpoints by visual inspection, as shown in
Figure~\ref{fig:variable-density-annotation-cost}.

\begin{figure}[H]
  \centering
\includegraphics[width=\linewidth]{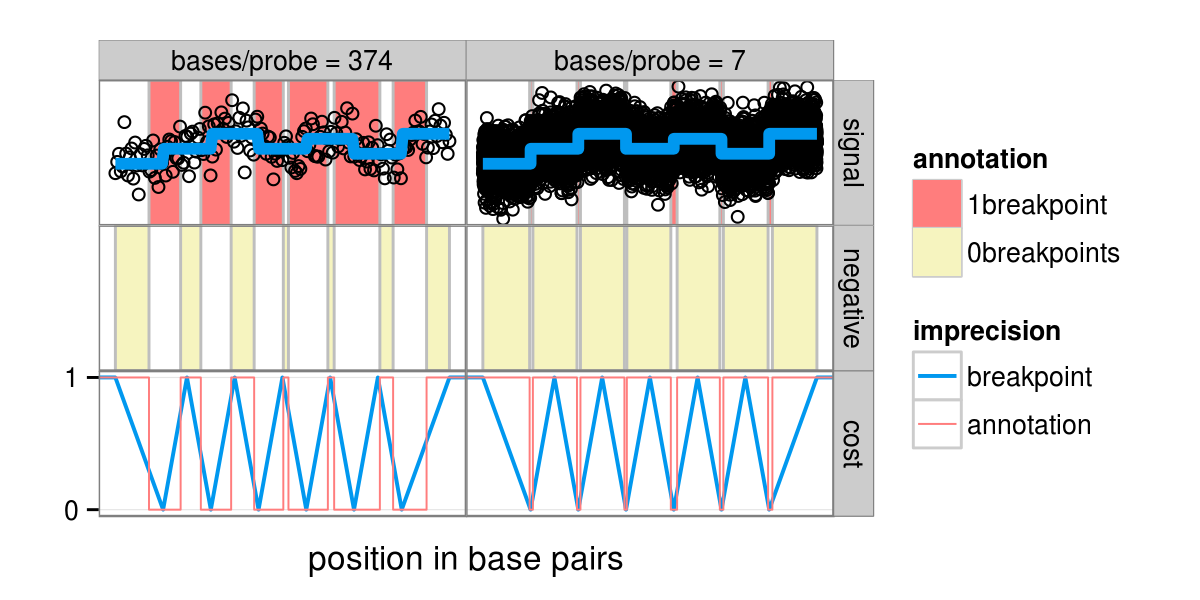}
\vskip -0.1in
  \caption{\textbf{Top}: simulated noisy
  signals (black) with their true piecewise constant signals $\mathbf m$ (blue) 
  and
  visually-determined breakpoint annotations $A$
  (red). 
\textbf{Middle}: negative annotations $A^0$ constructed
  using (\ref{eq:A^0}).
\textbf{Bottom}: breakpoint detection
  imprecision curves for the breakpointError $\ell_i$
  (\ref{eq:imprecision}) and the annotation error $\hat
  \ell_i$ (\ref{eq:hat_ell_i}).}
\label{fig:variable-density-annotation-cost}
\end{figure}


Recall that there are $P$ distinct positions in a series at which data
could be gathered, 
and that $\mathbb B=\{1,\dots,P-1\}$ is the set of all
positions after which a breakpoint is possible.

\begin{definition}
  A set of $n$ \textbf{annotations} can be written as $A=\{(a_1,
  \r_1), \dots, (a_n, \r_n)\}$. For each annotation~$i$,
  $\r_i\subset\mathbb I\mathbb B$ is an interval that defines the
  region, and $a_i\subseteq\{0,1,\dots\}$ is an interval of allowable
  breakpoint counts in this region.
\end{definition}

For
example, consider the annotated regions in
Table~\ref{tab:sample_annotations}.

\begin{table}[H]
  \begin{center}
    \begin{tabular}{ccc}
  $i$ & Allowed breakpoints $a_i$ & Region $\r_i$\\
\hline
1 & \{0\} & [5,10]\\
2 & \{1\} & [20,30]\\
3 & \{1,2,\dots\} & [40,70]\\
4 & \{0\} & [80,100]
\end{tabular}
  \end{center}
  \caption{Sample annotated regions for a signal sampled on $P=100$ base pairs. 
    An annotation $a_i$ indicates how many breakpoints are allowed in the
    corresponding region $\r_i$.
    There are 0 breaks in bases 5-10 and 80-100.
    There is exactly 1 break in bases 20-30.
    There is at least 1 break in bases 40-70.}
  \label{tab:sample_annotations}
\end{table}

Given a set of breakpoint guesses $G\subseteq \mathbb B$, we define
the annotation-dependent false positive count as
\begin{equation}
  \label{eq:FP_hat}
  \hat{ \text{FP}}(G,\r,a) =
    \big( 
|G\cap\r|-\max(a)
\big)_+
\end{equation}
where the positive part function is defined as
\begin{equation}
x_+ =
  \begin{cases}
    x & \text{ if } x > 0 \\
    0 & \text{ otherwise.}
  \end{cases}
\end{equation}
Similarly, the annotation-dependent false negative count is defined as
\begin{equation}
  \label{eq:FN_hat}
  \hat{ \text{FN}}(G,\r,a) =
  \big(
\min(a)-|G\cap\r|
\big)_+.
\end{equation}

\begin{definition}
  Let $A$ be a set of annotations and $G\subseteq \mathbb B$
  a set of breakpoint guesses. The \textbf{incomplete annotation error} is the
  count of annotation-dependent false positives and false negatives:
\begin{equation*}
  E^A_{\text{incomplete}}(G)=
    \sum_{i=1}^n 
    \hat{\textrm{FP}}(G, \r_i, a_i) + 
    \hat{\textrm{FN}}(G, \r_i, a_i).
\end{equation*}
\end{definition}

In the case of analyzing the simulated signals in the top panels of
Figure~\ref{fig:variable-density-annotation-cost}, let us consider the
set of 6 annotations $\hat A=\{(\hat a_1, \hat \r_1), \dots, (\hat
a_6, \hat \r_6)\}$ depicted using the red rectangles. These rectangles
were determined by visual inspection of the scatterplots. I used the
SegAnnDB interactive annotation web site to view the data and save a
database of 6 regions per profile \citep{SegAnnDB}. Every region $\hat
\r_i$ contains exactly 1 breakpoint, so we have $\hat a_i=\{1\}$ for every
annotation $i\in\{1, \dots, 6\}$. In real data we will probably only
be able to see a subset of the real breakpoints, but we analyze the
complete set of breakpoints in these simulated data to illustrate the
approximation induced by the annotation process.

Given any set of non-intersecting annotations $A$, we can write
$\underline r_1 < \cdots < \underline r_n$ to order the regions. Then
we can define $|A|+1$ negative annotations as
\begin{equation} 
  \label{eq:A^0}
   A^0 = \big\{ 
(0, [1,\underline r_1-1]),\ 
(0, [\overline r_1+1, \underline r_2-1]),\ 
\dots,\ 
(0, [\overline r_{n-1}+1,\underline r_n-1]),\ 
(0, [\underline r_n+1,d-1])
\big\},
\end{equation}
as drawn with yellow rectangles in the middle panels of
Figure~\ref{fig:variable-density-annotation-cost}. We will use the
complete set of annotations $\hat A\cup \hat A^0$ to define the
annotation error $E^{ \hat A\cup \hat A^0}_{\text{incomplete}}(G)$ for
breakpoint guesses $G$ given by models of these simulated signals.

\newpage

In Figure~\ref{fig:variable-density-sigerr-small}, we plot some model
selection error functions for the 2 simulated signals shown in
Figure~\ref{fig:variable-density-annotation-cost}. It is clear that
the annotation error is a good approximation of the breakpointError,
and there are several interesting observations to note.

\begin{itemize}
\item \textbf{Signal}: in these simulated data, the true piecewise
  constant signal $\mathbf m$ is available, so an efficient model
  selection procedure \citep{sylvain-survey} would select the
  estimated model which is closest to the true signal. That idea is
  illustrated in Figure~\ref{fig:motivation-modelOnly}, and can be
  used in this context by minimizing
  \begin{equation}
    \label{eq:signal_cost}
    E_{\text{signal}}(k) = \log_{10}\left[
\frac 1 d \sum_{i=1}^d(\mathbf{\hat y}_i^k - \mathbf m_i)^2
\right].
  \end{equation}
  \begin{itemize}
  \item In Figure~\ref{fig:variable-density-sigerr-small}, for the
    signal sampled at 7 bases/probe, the minimum of the Signal error
    identifies a model with 7 segments.
  \item For the signal sampled at 374 bases/probe, the minimum of the
    error identifies a model with only 5 segments.
  \end{itemize}
\item \textbf{Breakpoint}: in these simulated data, the true
  breakpoints $B$ are available, so we can compute and minimize the
  breakpointError as a consistent model selection procedure
  \citep{sylvain-survey}. For both signals, the minimum of the
  breakpointError identifies a model with 7 segments (6
  breakpoints).
\item \textbf{Annotation}: we use a set of annotated regions to
  compute the incomplete annotation error, which also identifies a
  model with 7 segments. It is clear that the annotation error is a
  good approximation of the breakpointError. In the next section, we
  explicitly demonstrate the link between the breakpointError and the
  annotation error.
\end{itemize}
\begin{figure}[H]
  \centering
  \input{figure-variable-density-sigerr-small}
  \vskip -0.5cm
  \caption{Model selection error curves for 2 simulated
    signals. Minima are highlighted using circles.
    \protect\\
    \textbf{Signal}: the log squared error $E_{\text{signal}}$ of the
    estimated signal with respect to the
    true piecewise constant signal (see text).
    \protect\\
    \textbf{Breakpoint}: exact breakpointError $E^B_{\text{exact}}$.
    \protect\\
    \textbf{Annotation}: incomplete annotation error
    $E^{\hat A\cup \hat A^0}_{\text{incomplete}}$.
}
  \label{fig:variable-density-sigerr-small}
\end{figure}

\newpage

\subsection{Link with breakpointError using complete annotation error}

\label{sec:complete}
It is clear from Figure~\ref{fig:variable-density-sigerr-small} that
the annotation error is a good approximation of the exact breakpoint
error when the annotations $A$ agree with the real
breakpoints $B$. In this section, we make this intuition precise by
showing exactly how to relax the breakpointError to obtain the
annotation error. There are two steps:
\begin{enumerate}
\item We define the complete annotation error by relaxing the
  definition of the exact \mbox{breakpointError}.
\item We show that the complete annotation error is equivalent
  to the incomplete annotation error when we have a complete set of
  annotations.
\end{enumerate}


We will define the complete annotation error as a relaxation of the
exact breakpointError. Recall from
Definition~\ref{def:exact_breakpoint_cost}, 
the exact breakpointError is
$$
  {E }_{\text{exact}}^B(G) =
  \big|G\setminus(\cup R_B)\big|
 + \sum_{i=1}^{|B|}\text{FP}(G,\r_i)+\text{FN}(G,\r_i)+I(G,\r_i,\ell_i).
$$
To define the complete annotation error, we perform two relaxations:
\begin{itemize}
\item Instead of using the true breakpoints $B$ (which are unknown in
  real data) with equations (\ref{eq:R_i}) and (\ref{eq:L_i}) to
  determine a breakpoint region $\r_i$, we use the region $\hat\r_i$
  determined by visual
  inspection.
\item Rather than the piecewise affine imprecision $\ell_i$, we use
  the zero-one imprecision $\hat \ell_i$:
\begin{equation}
  \label{eq:hat_ell_i}
  \hat\ell_i(g) = 1_{g\not\in \hat \r_i}.
\end{equation}
We show this relaxation by ploting the imprecision functions $\ell_i$
and $\hat \ell_i$ in the bottom panels of
Figure~\ref{fig:variable-density-annotation-cost}.
\end{itemize}

\begin{definition}
  Assume there are $n$ breakpoints $B_1 < \dots < B_n$, and we observe
  a set of annotations $\hat A= \{(a_1, \hat\r_1), \dots, (a_n,
  \hat\r_n)\}$ each with $a_i=1$ breakpoint, such that $B_1\in \hat
  \r_1, \dots, B_n\in\hat \r_n$. The \textbf{complete annotation
    error} of a set of breakpoint guesses $G$ is the sum of false positive and
  false negative counts:
  \begin{eqnarray*}
    E_{\textrm{complete}}^{\hat A} (G)
    &=&  \Big|G\setminus(\cup \hat A)\Big| \nonumber
    + \sum_{i=1}^{|\hat A|}
    \textrm{FP}(G,\hat\r_i)+\textrm{FN}(G,\hat\r_i)+I(G,\hat\r_i,\hat\ell_i)\\
    &=&  \Big|G\setminus(\cup \hat A)\Big|
    + \sum_{(a,\hat\r)\in\hat A}
    \textrm{FP}(G,\hat\r)+\textrm{FN}(G,\hat\r).
  \end{eqnarray*}
\end{definition}

It is clear that $E^{\hat A}_{\text{complete}}$ depends on the
annotations only through their regions. In particular, the annotated
breakpoint counts $a_i=\{1\}$ are not used in this definition, since
we assumed that each region $\hat \r_i$ contains exactly 1
breakpoint. Also, since we used the zero-one imprecision for $\hat\ell_i$, the
imprecision function $I$ is always zero.

\begin{proposition}
  Let $\hat A^0$ be a set of negative annotations as in
  (\ref{eq:A^0}). Then for a set of breakpoint guesses $G$, the
  incomplete and complete annotation error functions are equivalent:
\begin{equation*}
  E_{\text{incomplete}}^{\hat A\cup \hat A^0}(G)
=
E^{\hat A}_{\text{complete}}(G).
\end{equation*}
\end{proposition}

\begin{proof}

To see the connection between the complete and incomplete annotation
error functions, first note that
\begin{eqnarray}
\nonumber  \hat{\text{FN}}(G,\r,\{1\}) 
&=&   \big(
1-|G\cap\r|
\big)_+\\
&=&
\text{FN}(G,\r), \label{eq:fn-hat}
\end{eqnarray}
and
\begin{eqnarray}
\nonumber
  \hat{\text{FP}}(G,\r,\{1\}) 
&=&
\big( 
|G\cap\r|-1
\big)_+\\
&=& \text{FP}(G,\r). \label{eq:fp-hat}
\end{eqnarray}
For the complete annotation error we quantified the false positive
rate of the breakpoints that fall outside of the breakpoint regions
$\hat A$ using $G\setminus(\cup \hat A)$. For the incomplete
annotation error, we instead created a set of 0-breakpoint annotations
$\hat A^0$ for this purpose. Note that by construction of the negative
regions in (\ref{eq:A^0}), we have 
\begin{equation}
  \label{eq:Rcomplement}
  G\setminus \left(\cup \hat A\right)
 = 
G\cap\left(\cup \hat A^0 \right),
\end{equation}
or in words, the guesses outside of the breakpoint annotations $\hat A$
are in the negative annotations $\hat A^0$. So using
(\ref{eq:Rcomplement}), we have
\begin{eqnarray}
  \hat{\text{FP}}(G,(\cup \hat A^0),\{0\})
&=& 
|G\cap(
  \cup \hat A^0
)|\nonumber
\\
&=&
|G\setminus(\cup\hat A)|,\label{eq:fp-outside}
\end{eqnarray}
which is the first component of the complete annotation error.

Recall that $\hat A$ represents annotated regions that each contain
exactly 1 breakpoint, and $\hat A^0$ are regions with no breakpoints. So
using (\ref{eq:fn-hat}), (\ref{eq:fp-hat}), and
(\ref{eq:fp-outside}), we have that the incomplete annotation error
is equivalent to the complete error:
\begin{eqnarray}
  E_{\text{incomplete}}^{\hat A\cup \hat A^0}(G)
&=&
\sum_{(a, \r)\in \hat A^0} \hat{\text{FP}}(G,\r,a)\nonumber 
+
\sum_{(a, \r)\in\hat A} \hat{\text{FP}}(G,\r,a) + \hat{\text{FN}}(G,\r,a)
 \\
&=&
 \hat{\text{FP}}(G,\cup \hat A^0,\{0\})\nonumber 
+
\sum_{(a, \r)\in\hat A} \hat{\text{FP}}(G,\r,\{1\}) + \hat{\text{FN}}(G,\r,\{1\})
 \\
&=&
|G\setminus(\cup\hat A)|
+\nonumber
\sum_{(a, \r)\in\hat A} {\text{FP}}(G,\r) + {\text{FN}}(G,\r)\\
&=&
E^{\hat A}_{\text{complete}}(G).
\end{eqnarray}

\end{proof}

So in fact the incomplete annotation error is equivalent to the
complete error when the annotated regions $\hat A$ each contain
exactly 1 breakpoint. But we call this the incomplete error since it
is also well-defined for arbitrary sets of regions $A$.


\newpage

\subsection{Zero-one annotation error}
\label{sec:zero-one}

The incomplete annotation error counts incorrect breakpoints. In this
section, we show that by thresholding the incomplete annotation error,
we can obtain the zero-one annotation error function. This is the
original annotation error function that was introduced by
\citet{HOCKING-breakpoints}, who used it to count the number of
incorrect regions.

First, let us define the zero-one thresholding function
$t:\mathbb Z^+\rightarrow\mathbb Z^+$ as
\begin{equation}
  \label{eq:thresholding}
  t(x)=1_{x\neq 0} =
  \begin{cases}
    1 & \text{if }x\neq 0\\
    0 & \text{otherwise}.
  \end{cases}
\end{equation}

The idea of thresholding is to limit the error that any one annotation
can induce. We define the zero-one annotation error as
\begin{eqnarray}
  \label{eq:ann01err}
  E_{01}^{A}(G)
&=&\nonumber
 \sum_{(a, \r)\in A} 
t\left[\hat{\text{FP}}(G,\r,a)\right]+
t\left[\hat{\text{FN}}(G,\r,a)\right]\\
&=&\nonumber
 \sum_{(a, \r)\in A} 
1_{|G\cap\r|>\max(a)}+
1_{|G\cap \r|<\min(a)}\\
&=&
 \sum_{(a, \r)\in A} 
1_{|G\cap\r|\not\in a}.
\end{eqnarray}
So using the zero-one annotation error, we count incorrect annotated
regions instead of incorrect breakpoint guesses.  

\newpage

\subsection{Comparing annotation error functions}

In practice, we have few annotated regions per signal in real data. In
Figure~\ref{fig:variable-density-sigerr}, we show how the annotation
error is degraded as we remove annotations. In particular, it is clear
that using the thresholded zero-one annotation error significantly
degrades the approximation of the FP curve. Nevertheless, it is worth
noting that minimum of the zero-one error still uniquely identifies
the correct model with 7 segments. Even after removing many
annotations, the minimum error still identifies the correct
model, but not uniquely. 


\begin{figure}[H]
  \input{figure-variable-density-sigerr}
  \caption{Comparison of annotation error functions as the set of
    annotations changes. Minima are highlighted using circles.
    \protect\\
    \textbf{Complete}: annotation error
    for a complete set of 6 positive and 7 negative annotations.
    \protect\\
    \textbf{Zero-one}: zero-one annotation error 
    for a complete set of 6 positive and 7 negative annotations.
    \protect\\
    \textbf{Incomplete}: zero-one annotation error 
    for 3 positive and 4 negative annotations.
    \protect\\
    \textbf{Positive}: zero-one annotation error 
    for 3 positive annotations.}
  \label{fig:variable-density-sigerr}
\end{figure}

In conclusion, this section has discussed the connections between the
breakpointError and the annotation error functions. Whereas the
breakpointError is computable only when the true set of breakpoints is
known (e.g. simulated data), the annotation error is readily
computable in any data set using a set of visually determined
annotations. We showed that if the annotations are consistent with the
true breakpoints, then the annotation error function is a good
approximation of the breakpointError
(Figure~\ref{fig:variable-density-sigerr-small}). Finally, we observed
that even after thresholding and removing annotations, the annotation
error function can still be used to identify a set of minimum error
segmentation models (Figure~\ref{fig:variable-density-sigerr}).


\section{Conclusions and future work}

In this paper we defined the breakpointError, which can be used to
quantify the breakpoint detection accuracy of a segmentation model,
when the true breakpoint positions are known. In Section~4 we showed
one application of the breakpointError for determining optimal penalty
constants in several simulated data sets. In Section~5 we discussed
the relationship of the breakpointError to the annotation error, which
has been used for supervised segmentation of real data sets
\citep{HOCKING-breakpoints, HOCKING-penalties, SegAnnDB}. We showed
that the annotation error is a good approximation of the
breakpointError when the annotated regions agree with the true
breakpoints. This provides some justification for using the annotation
error in supervised analysis of real data sets.

For future work, it will be interesting to apply the breakpointError
to more realistic tasks. For example, \citet{perf-eval-framework}
proposed to evaluate breakpoint detection algorithms by adding
breakpoints and noise to real data sets. In their framework, the true
breakpoint positions are known, and a region around each breakpoint is
used to quantify the number of true and false positive breakpoint
detections. Instead of using the zero-one loss with an arbitrarily
sized region, the breakpointError could be used to more precisely
quantify breakpoint estimates, since it counts imprecision
(\ref{eq:imprecision}) in addition to false positive and false
negative breakpoint detections.

To facilitate the use of the breakpointError in future work, it is
implemented in the R package \verb|breakpointError| on R-Forge. It can
be installed in R using

\begin{verbatim}
install.packages("breakpointError", repos="http://r-forge.r-project.org")
\end{verbatim}

\textbf{Acknowledgements}: Thanks to Marco Cuturi for references about
distance functions for comparing probability distributions. Thanks to
Guillem Rigaill for helpful comments on a preliminary version of this
paper.

\bibliographystyle{abbrvnat}
\bibliography{refs}

\end{document}

%% file: figure-motivation-breakpointError.tex
\begin{tikzpicture}[x=1pt,y=1pt]
\definecolor[named]{fillColor}{rgb}{1.00,1.00,1.00}
\path[use as bounding box,fill=fillColor,fill opacity=0.00] (0,0) rectangle (433.62,144.54);
\begin{scope}
\path[clip] (  0.00,  0.00) rectangle (433.62,144.54);
\definecolor[named]{drawColor}{rgb}{1.00,1.00,1.00}
\definecolor[named]{fillColor}{rgb}{1.00,1.00,1.00}

\path[draw=drawColor,line width= 0.6pt,line join=round,line cap=round,fill=fillColor] (  0.00,  0.00) rectangle (433.62,144.54);
\end{scope}
\begin{scope}
\path[clip] ( 39.69,119.86) rectangle (230.63,132.50);
\definecolor[named]{drawColor}{rgb}{0.50,0.50,0.50}
\definecolor[named]{fillColor}{rgb}{0.80,0.80,0.80}

\path[draw=drawColor,line width= 0.2pt,line join=round,line cap=round,fill=fillColor] ( 39.69,119.86) rectangle (230.63,132.50);
\definecolor[named]{drawColor}{rgb}{0.00,0.00,0.00}

\node[text=drawColor,anchor=base,inner sep=0pt, outer sep=0pt, scale=  0.96] at (135.16,122.87) {change in mean};
\end{scope}
\begin{scope}
\path[clip] (230.63,119.86) rectangle (421.57,132.50);
\definecolor[named]{drawColor}{rgb}{0.50,0.50,0.50}
\definecolor[named]{fillColor}{rgb}{0.80,0.80,0.80}

\path[draw=drawColor,line width= 0.2pt,line join=round,line cap=round,fill=fillColor] (230.63,119.86) rectangle (421.57,132.50);
\definecolor[named]{drawColor}{rgb}{0.00,0.00,0.00}

\node[text=drawColor,anchor=base,inner sep=0pt, outer sep=0pt, scale=  0.96] at (326.10,122.87) {change in variance};
\end{scope}
\begin{scope}
\path[clip] ( 39.69, 34.03) rectangle (230.63,119.86);
\definecolor[named]{fillColor}{rgb}{1.00,1.00,1.00}

\path[fill=fillColor] ( 39.69, 34.03) rectangle (230.63,119.86);
\definecolor[named]{drawColor}{rgb}{0.98,0.98,0.98}

\path[draw=drawColor,line width= 0.6pt,line join=round] ( 39.69, 46.34) --
	(230.63, 46.34);

\path[draw=drawColor,line width= 0.6pt,line join=round] ( 39.69, 63.16) --
	(230.63, 63.16);

\path[draw=drawColor,line width= 0.6pt,line join=round] ( 39.69, 79.97) --
	(230.63, 79.97);

\path[draw=drawColor,line width= 0.6pt,line join=round] ( 39.69, 96.79) --
	(230.63, 96.79);

\path[draw=drawColor,line width= 0.6pt,line join=round] ( 39.69,113.61) --
	(230.63,113.61);

\path[draw=drawColor,line width= 0.6pt,line join=round] ( 42.17, 34.03) --
	( 42.17,119.86);

\path[draw=drawColor,line width= 0.6pt,line join=round] ( 79.36, 34.03) --
	( 79.36,119.86);

\path[draw=drawColor,line width= 0.6pt,line join=round] (116.56, 34.03) --
	(116.56,119.86);

\path[draw=drawColor,line width= 0.6pt,line join=round] (153.76, 34.03) --
	(153.76,119.86);

\path[draw=drawColor,line width= 0.6pt,line join=round] (190.95, 34.03) --
	(190.95,119.86);

\path[draw=drawColor,line width= 0.6pt,line join=round] (228.15, 34.03) --
	(228.15,119.86);
\definecolor[named]{drawColor}{rgb}{0.90,0.90,0.90}

\path[draw=drawColor,line width= 0.2pt,line join=round] ( 39.69, 37.94) --
	(230.63, 37.94);

\path[draw=drawColor,line width= 0.2pt,line join=round] ( 39.69, 54.75) --
	(230.63, 54.75);

\path[draw=drawColor,line width= 0.2pt,line join=round] ( 39.69, 71.57) --
	(230.63, 71.57);

\path[draw=drawColor,line width= 0.2pt,line join=round] ( 39.69, 88.38) --
	(230.63, 88.38);

\path[draw=drawColor,line width= 0.2pt,line join=round] ( 39.69,105.20) --
	(230.63,105.20);

\path[draw=drawColor,line width= 0.2pt,line join=round] ( 60.76, 34.03) --
	( 60.76,119.86);

\path[draw=drawColor,line width= 0.2pt,line join=round] ( 97.96, 34.03) --
	( 97.96,119.86);

\path[draw=drawColor,line width= 0.2pt,line join=round] (135.16, 34.03) --
	(135.16,119.86);

\path[draw=drawColor,line width= 0.2pt,line join=round] (172.36, 34.03) --
	(172.36,119.86);

\path[draw=drawColor,line width= 0.2pt,line join=round] (209.55, 34.03) --
	(209.55,119.86);
\definecolor[named]{drawColor}{rgb}{0.53,0.81,0.92}

\path[draw=drawColor,line width= 2.8pt,line join=round] ( 60.76, 37.94) --
	( 97.96, 37.94) --
	(135.16, 37.94) --
	(172.36, 71.57) --
	(209.55,105.20);
\definecolor[named]{drawColor}{rgb}{0.89,0.10,0.11}

\path[draw=drawColor,line width= 2.8pt,line join=round] ( 60.76,105.20) --
	( 97.96, 71.57) --
	(135.16, 37.94) --
	(172.36, 37.94) --
	(209.55, 37.94);
\definecolor[named]{drawColor}{rgb}{0.00,0.00,0.00}

\path[draw=drawColor,line width= 0.9pt,line join=round] ( 60.76,105.20) --
	( 97.96, 82.33) --
	(135.16, 48.70) --
	(172.36, 76.74) --
	(209.55,115.96);

\path[draw=drawColor,line width= 0.9pt,dash pattern=on 4pt off 4pt ,line join=round] ( 60.76, 37.94) --
	( 97.96, 48.70) --
	(135.16, 48.70) --
	(172.36, 43.11) --
	(209.55, 48.70);
\end{scope}
\begin{scope}
\path[clip] ( 39.69, 34.03) rectangle (230.63,119.86);
\definecolor[named]{drawColor}{rgb}{0.00,0.00,0.00}
\definecolor[named]{fillColor}{rgb}{0.53,0.81,0.92}

\path[draw=drawColor,line width= 0.4pt,line join=round,line cap=round,fill=fillColor] ( 44.59, 44.37) --
	( 57.92, 44.37) --
	( 60.76, 37.94) --
	( 57.92, 34.03) --
	( 44.59, 34.03) --
	cycle;
\definecolor[named]{fillColor}{rgb}{0.00,0.00,0.00}

\path[draw=drawColor,line width= 0.4pt,line join=round,line cap=round,fill=fillColor] ( 54.31, 54.70) --
	( 57.92, 54.70) --
	( 60.76, 37.94) --
	( 57.92, 44.37) --
	( 54.31, 44.37) --
	cycle;
\definecolor[named]{fillColor}{rgb}{0.89,0.10,0.11}

\path[draw=drawColor,line width= 0.4pt,line join=round,line cap=round,fill=fillColor] ( 43.90,105.20) --
	( 57.92,105.20) --
	( 60.76,105.20) --
	( 57.92, 94.87) --
	( 43.90, 94.87) --
	cycle;
\definecolor[named]{fillColor}{rgb}{0.00,0.00,0.00}

\path[draw=drawColor,line width= 0.4pt,line join=round,line cap=round,fill=fillColor] ( 51.12,115.53) --
	( 57.92,115.53) --
	( 60.76,105.20) --
	( 57.92,105.20) --
	( 51.12,105.20) --
	cycle;
\end{scope}
\begin{scope}
\path[clip] ( 39.69, 34.03) rectangle (230.63,119.86);
\definecolor[named]{drawColor}{rgb}{0.00,0.00,0.00}
\definecolor[named]{fillColor}{rgb}{0.89,0.10,0.11}

\path[draw=drawColor,line width= 0.4pt,line join=round,line cap=round,fill=fillColor] (209.55, 37.94) --
	(212.40, 44.37) --
	(226.42, 44.37) --
	(226.42, 34.03) --
	(212.40, 34.03) --
	cycle;
\definecolor[named]{fillColor}{rgb}{0.00,0.00,0.00}

\path[draw=drawColor,line width= 0.4pt,line join=round,line cap=round,fill=fillColor] (209.55, 48.70) --
	(212.40, 54.70) --
	(216.01, 54.70) --
	(216.01, 44.37) --
	(212.40, 44.37) --
	cycle;
\definecolor[named]{fillColor}{rgb}{0.53,0.81,0.92}

\path[draw=drawColor,line width= 0.4pt,line join=round,line cap=round,fill=fillColor] (209.55,105.20) --
	(212.40,109.53) --
	(225.73,109.53) --
	(225.73, 99.20) --
	(212.40, 99.20) --
	cycle;
\definecolor[named]{fillColor}{rgb}{0.00,0.00,0.00}

\path[draw=drawColor,line width= 0.4pt,line join=round,line cap=round,fill=fillColor] (209.55,115.96) --
	(212.40,119.86) --
	(219.20,119.86) --
	(219.20,109.53) --
	(212.40,109.53) --
	cycle;
\definecolor[named]{drawColor}{rgb}{1.00,1.00,1.00}

\node[text=drawColor,anchor=base east,inner sep=0pt, outer sep=0pt, scale=  1.00] at ( 57.92, 46.09) {I};

\node[text=drawColor,anchor=base west,inner sep=0pt, outer sep=0pt, scale=  1.00] at (212.40, 46.09) {I};

\node[text=drawColor,anchor=base east,inner sep=0pt, outer sep=0pt, scale=  1.00] at ( 57.92,106.92) {E};

\node[text=drawColor,anchor=base west,inner sep=0pt, outer sep=0pt, scale=  1.00] at (212.40,111.25) {E};

\node[text=drawColor,anchor=base east,inner sep=0pt, outer sep=0pt, scale=  1.00] at ( 57.92, 35.76) {FP};

\node[text=drawColor,anchor=base east,inner sep=0pt, outer sep=0pt, scale=  1.00] at ( 57.92, 96.59) {FN};

\node[text=drawColor,anchor=base west,inner sep=0pt, outer sep=0pt, scale=  1.00] at (212.40, 35.76) {FN};

\node[text=drawColor,anchor=base west,inner sep=0pt, outer sep=0pt, scale=  1.00] at (212.40,100.92) {FP};
\definecolor[named]{drawColor}{rgb}{0.50,0.50,0.50}

\path[draw=drawColor,line width= 0.6pt,line join=round,line cap=round] ( 39.69, 34.03) rectangle (230.63,119.86);
\end{scope}
\begin{scope}
\path[clip] (230.63, 34.03) rectangle (421.57,119.86);
\definecolor[named]{fillColor}{rgb}{1.00,1.00,1.00}

\path[fill=fillColor] (230.63, 34.03) rectangle (421.57,119.86);
\definecolor[named]{drawColor}{rgb}{0.98,0.98,0.98}

\path[draw=drawColor,line width= 0.6pt,line join=round] (230.63, 46.34) --
	(421.57, 46.34);

\path[draw=drawColor,line width= 0.6pt,line join=round] (230.63, 63.16) --
	(421.57, 63.16);

\path[draw=drawColor,line width= 0.6pt,line join=round] (230.63, 79.97) --
	(421.57, 79.97);

\path[draw=drawColor,line width= 0.6pt,line join=round] (230.63, 96.79) --
	(421.57, 96.79);

\path[draw=drawColor,line width= 0.6pt,line join=round] (230.63,113.61) --
	(421.57,113.61);

\path[draw=drawColor,line width= 0.6pt,line join=round] (233.11, 34.03) --
	(233.11,119.86);

\path[draw=drawColor,line width= 0.6pt,line join=round] (270.31, 34.03) --
	(270.31,119.86);

\path[draw=drawColor,line width= 0.6pt,line join=round] (307.50, 34.03) --
	(307.50,119.86);

\path[draw=drawColor,line width= 0.6pt,line join=round] (344.70, 34.03) --
	(344.70,119.86);

\path[draw=drawColor,line width= 0.6pt,line join=round] (381.90, 34.03) --
	(381.90,119.86);

\path[draw=drawColor,line width= 0.6pt,line join=round] (419.10, 34.03) --
	(419.10,119.86);
\definecolor[named]{drawColor}{rgb}{0.90,0.90,0.90}

\path[draw=drawColor,line width= 0.2pt,line join=round] (230.63, 37.94) --
	(421.57, 37.94);

\path[draw=drawColor,line width= 0.2pt,line join=round] (230.63, 54.75) --
	(421.57, 54.75);

\path[draw=drawColor,line width= 0.2pt,line join=round] (230.63, 71.57) --
	(421.57, 71.57);

\path[draw=drawColor,line width= 0.2pt,line join=round] (230.63, 88.38) --
	(421.57, 88.38);

\path[draw=drawColor,line width= 0.2pt,line join=round] (230.63,105.20) --
	(421.57,105.20);

\path[draw=drawColor,line width= 0.2pt,line join=round] (251.71, 34.03) --
	(251.71,119.86);

\path[draw=drawColor,line width= 0.2pt,line join=round] (288.91, 34.03) --
	(288.91,119.86);

\path[draw=drawColor,line width= 0.2pt,line join=round] (326.10, 34.03) --
	(326.10,119.86);

\path[draw=drawColor,line width= 0.2pt,line join=round] (363.30, 34.03) --
	(363.30,119.86);

\path[draw=drawColor,line width= 0.2pt,line join=round] (400.50, 34.03) --
	(400.50,119.86);
\definecolor[named]{drawColor}{rgb}{0.53,0.81,0.92}

\path[draw=drawColor,line width= 2.8pt,line join=round] (251.71, 37.94) --
	(288.91, 37.94) --
	(326.10, 37.94) --
	(363.30, 71.57) --
	(400.50,105.20);
\definecolor[named]{drawColor}{rgb}{0.89,0.10,0.11}

\path[draw=drawColor,line width= 2.8pt,line join=round] (251.71,105.20) --
	(288.91, 71.57) --
	(326.10, 37.94) --
	(363.30, 37.94) --
	(400.50, 37.94);
\definecolor[named]{drawColor}{rgb}{0.00,0.00,0.00}

\path[draw=drawColor,line width= 0.9pt,line join=round] (251.71,105.20) --
	(288.91, 74.94) --
	(326.10, 41.34) --
	(363.30, 77.00) --
	(400.50,108.60);

\path[draw=drawColor,line width= 0.9pt,dash pattern=on 4pt off 4pt ,line join=round] (251.71, 37.94) --
	(288.91, 41.31) --
	(326.10, 41.34) --
	(363.30, 43.37) --
	(400.50, 41.34);
\end{scope}
\begin{scope}
\path[clip] (230.63, 34.03) rectangle (421.57,119.86);
\definecolor[named]{drawColor}{rgb}{0.00,0.00,0.00}
\definecolor[named]{fillColor}{rgb}{0.53,0.81,0.92}

\path[draw=drawColor,line width= 0.4pt,line join=round,line cap=round,fill=fillColor] (235.53, 44.37) --
	(248.86, 44.37) --
	(251.71, 37.94) --
	(248.86, 34.03) --
	(235.53, 34.03) --
	cycle;
\definecolor[named]{fillColor}{rgb}{0.00,0.00,0.00}

\path[draw=drawColor,line width= 0.4pt,line join=round,line cap=round,fill=fillColor] (245.25, 54.70) --
	(248.86, 54.70) --
	(251.71, 37.94) --
	(248.86, 44.37) --
	(245.25, 44.37) --
	cycle;
\definecolor[named]{fillColor}{rgb}{0.89,0.10,0.11}

\path[draw=drawColor,line width= 0.4pt,line join=round,line cap=round,fill=fillColor] (234.84,105.20) --
	(248.86,105.20) --
	(251.71,105.20) --
	(248.86, 94.87) --
	(234.84, 94.87) --
	cycle;
\definecolor[named]{fillColor}{rgb}{0.00,0.00,0.00}

\path[draw=drawColor,line width= 0.4pt,line join=round,line cap=round,fill=fillColor] (242.06,115.53) --
	(248.86,115.53) --
	(251.71,105.20) --
	(248.86,105.20) --
	(242.06,105.20) --
	cycle;
\end{scope}
\begin{scope}
\path[clip] (230.63, 34.03) rectangle (421.57,119.86);
\definecolor[named]{drawColor}{rgb}{0.00,0.00,0.00}
\definecolor[named]{fillColor}{rgb}{0.89,0.10,0.11}

\path[draw=drawColor,line width= 0.4pt,line join=round,line cap=round,fill=fillColor] (400.50, 37.94) --
	(403.34, 44.37) --
	(417.37, 44.37) --
	(417.37, 34.03) --
	(403.34, 34.03) --
	cycle;
\definecolor[named]{fillColor}{rgb}{0.00,0.00,0.00}

\path[draw=drawColor,line width= 0.4pt,line join=round,line cap=round,fill=fillColor] (400.50, 41.34) --
	(403.34, 54.70) --
	(406.95, 54.70) --
	(406.95, 44.37) --
	(403.34, 44.37) --
	cycle;
\definecolor[named]{fillColor}{rgb}{0.53,0.81,0.92}

\path[draw=drawColor,line width= 0.4pt,line join=round,line cap=round,fill=fillColor] (400.50,105.20) --
	(403.34,106.90) --
	(416.67,106.90) --
	(416.67, 96.57) --
	(403.34, 96.57) --
	cycle;
\definecolor[named]{fillColor}{rgb}{0.00,0.00,0.00}

\path[draw=drawColor,line width= 0.4pt,line join=round,line cap=round,fill=fillColor] (400.50,108.60) --
	(403.34,117.23) --
	(410.15,117.23) --
	(410.15,106.90) --
	(403.34,106.90) --
	cycle;
\definecolor[named]{drawColor}{rgb}{1.00,1.00,1.00}

\node[text=drawColor,anchor=base east,inner sep=0pt, outer sep=0pt, scale=  1.00] at (248.86, 46.09) {I};

\node[text=drawColor,anchor=base west,inner sep=0pt, outer sep=0pt, scale=  1.00] at (403.34, 46.09) {I};

\node[text=drawColor,anchor=base east,inner sep=0pt, outer sep=0pt, scale=  1.00] at (248.86,106.92) {E};

\node[text=drawColor,anchor=base west,inner sep=0pt, outer sep=0pt, scale=  1.00] at (403.34,108.62) {E};

\node[text=drawColor,anchor=base east,inner sep=0pt, outer sep=0pt, scale=  1.00] at (248.86, 35.76) {FP};

\node[text=drawColor,anchor=base east,inner sep=0pt, outer sep=0pt, scale=  1.00] at (248.86, 96.59) {FN};

\node[text=drawColor,anchor=base west,inner sep=0pt, outer sep=0pt, scale=  1.00] at (403.34, 35.76) {FN};

\node[text=drawColor,anchor=base west,inner sep=0pt, outer sep=0pt, scale=  1.00] at (403.34, 98.29) {FP};
\definecolor[named]{drawColor}{rgb}{0.50,0.50,0.50}

\path[draw=drawColor,line width= 0.6pt,line join=round,line cap=round] (230.63, 34.03) rectangle (421.57,119.86);
\end{scope}
\begin{scope}
\path[clip] (  0.00,  0.00) rectangle (433.62,144.54);
\definecolor[named]{drawColor}{rgb}{0.00,0.00,0.00}

\node[text=drawColor,anchor=base east,inner sep=0pt, outer sep=0pt, scale=  0.96] at ( 32.57, 34.63) {0.0};

\node[text=drawColor,anchor=base east,inner sep=0pt, outer sep=0pt, scale=  0.96] at ( 32.57, 51.45) {0.5};

\node[text=drawColor,anchor=base east,inner sep=0pt, outer sep=0pt, scale=  0.96] at ( 32.57, 68.26) {1.0};

\node[text=drawColor,anchor=base east,inner sep=0pt, outer sep=0pt, scale=  0.96] at ( 32.57, 85.08) {1.5};

\node[text=drawColor,anchor=base east,inner sep=0pt, outer sep=0pt, scale=  0.96] at ( 32.57,101.89) {2.0};
\end{scope}
\begin{scope}
\path[clip] (  0.00,  0.00) rectangle (433.62,144.54);
\definecolor[named]{drawColor}{rgb}{0.00,0.00,0.00}

\path[draw=drawColor,line width= 0.6pt,line join=round] ( 35.42, 37.94) --
	( 39.69, 37.94);

\path[draw=drawColor,line width= 0.6pt,line join=round] ( 35.42, 54.75) --
	( 39.69, 54.75);

\path[draw=drawColor,line width= 0.6pt,line join=round] ( 35.42, 71.57) --
	( 39.69, 71.57);

\path[draw=drawColor,line width= 0.6pt,line join=round] ( 35.42, 88.38) --
	( 39.69, 88.38);

\path[draw=drawColor,line width= 0.6pt,line join=round] ( 35.42,105.20) --
	( 39.69,105.20);
\end{scope}
\begin{scope}
\path[clip] (  0.00,  0.00) rectangle (433.62,144.54);
\definecolor[named]{drawColor}{rgb}{0.00,0.00,0.00}

\path[draw=drawColor,line width= 0.6pt,line join=round] ( 60.76, 29.77) --
	( 60.76, 34.03);

\path[draw=drawColor,line width= 0.6pt,line join=round] ( 97.96, 29.77) --
	( 97.96, 34.03);

\path[draw=drawColor,line width= 0.6pt,line join=round] (135.16, 29.77) --
	(135.16, 34.03);

\path[draw=drawColor,line width= 0.6pt,line join=round] (172.36, 29.77) --
	(172.36, 34.03);

\path[draw=drawColor,line width= 0.6pt,line join=round] (209.55, 29.77) --
	(209.55, 34.03);
\end{scope}
\begin{scope}
\path[clip] (  0.00,  0.00) rectangle (433.62,144.54);
\definecolor[named]{drawColor}{rgb}{0.00,0.00,0.00}

\node[text=drawColor,anchor=base,inner sep=0pt, outer sep=0pt, scale=  0.96] at ( 60.76, 20.31) {1};

\node[text=drawColor,anchor=base,inner sep=0pt, outer sep=0pt, scale=  0.96] at ( 97.96, 20.31) {2};

\node[text=drawColor,anchor=base,inner sep=0pt, outer sep=0pt, scale=  0.96] at (135.16, 20.31) {3};

\node[text=drawColor,anchor=base,inner sep=0pt, outer sep=0pt, scale=  0.96] at (172.36, 20.31) {4};

\node[text=drawColor,anchor=base,inner sep=0pt, outer sep=0pt, scale=  0.96] at (209.55, 20.31) {5};
\end{scope}
\begin{scope}
\path[clip] (  0.00,  0.00) rectangle (433.62,144.54);
\definecolor[named]{drawColor}{rgb}{0.00,0.00,0.00}

\path[draw=drawColor,line width= 0.6pt,line join=round] (251.71, 29.77) --
	(251.71, 34.03);

\path[draw=drawColor,line width= 0.6pt,line join=round] (288.91, 29.77) --
	(288.91, 34.03);

\path[draw=drawColor,line width= 0.6pt,line join=round] (326.10, 29.77) --
	(326.10, 34.03);

\path[draw=drawColor,line width= 0.6pt,line join=round] (363.30, 29.77) --
	(363.30, 34.03);

\path[draw=drawColor,line width= 0.6pt,line join=round] (400.50, 29.77) --
	(400.50, 34.03);
\end{scope}
\begin{scope}
\path[clip] (  0.00,  0.00) rectangle (433.62,144.54);
\definecolor[named]{drawColor}{rgb}{0.00,0.00,0.00}

\node[text=drawColor,anchor=base,inner sep=0pt, outer sep=0pt, scale=  0.96] at (251.71, 20.31) {1};

\node[text=drawColor,anchor=base,inner sep=0pt, outer sep=0pt, scale=  0.96] at (288.91, 20.31) {2};

\node[text=drawColor,anchor=base,inner sep=0pt, outer sep=0pt, scale=  0.96] at (326.10, 20.31) {3};

\node[text=drawColor,anchor=base,inner sep=0pt, outer sep=0pt, scale=  0.96] at (363.30, 20.31) {4};

\node[text=drawColor,anchor=base,inner sep=0pt, outer sep=0pt, scale=  0.96] at (400.50, 20.31) {5};
\end{scope}
\begin{scope}
\path[clip] (  0.00,  0.00) rectangle (433.62,144.54);
\definecolor[named]{drawColor}{rgb}{0.00,0.00,0.00}

\node[text=drawColor,anchor=base,inner sep=0pt, outer sep=0pt, scale=  1.20] at (230.63,  9.03) {segments};
\end{scope}
\begin{scope}
\path[clip] (  0.00,  0.00) rectangle (433.62,144.54);
\definecolor[named]{drawColor}{rgb}{0.00,0.00,0.00}

\node[text=drawColor,rotate= 90.00,anchor=base,inner sep=0pt, outer sep=0pt, scale=  1.20] at ( 17.30, 76.95) {error};
\end{scope}
\end{tikzpicture}

%% file: figure-breakpoint-error-pieces.tex
\begin{tikzpicture}[x=1pt,y=1pt]
\definecolor[named]{fillColor}{rgb}{1.00,1.00,1.00}
\path[use as bounding box,fill=fillColor,fill opacity=0.00] (0,0) rectangle (361.35,180.67);
\begin{scope}
\path[clip] (  0.00,  0.00) rectangle (361.35,180.67);
\definecolor[named]{drawColor}{rgb}{1.00,1.00,1.00}
\definecolor[named]{fillColor}{rgb}{1.00,1.00,1.00}

\path[draw=drawColor,line width= 0.6pt,line join=round,line cap=round,fill=fillColor] (  0.00,  0.00) rectangle (361.35,180.68);
\end{scope}
\begin{scope}
\path[clip] ( 33.34, 74.30) rectangle (334.89,167.43);
\definecolor[named]{fillColor}{rgb}{0.90,0.90,0.90}

\path[fill=fillColor] ( 33.34, 74.30) rectangle (334.89,167.43);
\definecolor[named]{drawColor}{rgb}{1.00,1.00,1.00}

\path[draw=drawColor,line width= 0.6pt,line join=round] ( 33.34,143.66) --
	(334.89,143.66);

\path[draw=drawColor,line width= 0.6pt,line join=round] ( 33.34, 78.53) --
	(334.89, 78.53);

\path[draw=drawColor,line width= 0.6pt,line join=round] ( 53.28, 74.30) --
	( 53.28,167.43);

\path[draw=drawColor,line width= 0.6pt,line join=round] ( 90.66, 74.30) --
	( 90.66,167.43);

\path[draw=drawColor,line width= 0.6pt,line join=round] (152.97, 74.30) --
	(152.97,167.43);

\path[draw=drawColor,line width= 0.6pt,line join=round] (165.43, 74.30) --
	(165.43,167.43);

\path[draw=drawColor,line width= 0.6pt,line join=round] (215.27, 74.30) --
	(215.27,167.43);

\path[draw=drawColor,line width= 0.6pt,line join=round] (302.50, 74.30) --
	(302.50,167.43);

\path[draw=drawColor,line width= 0.6pt,line join=round] (314.96, 74.30) --
	(314.96,167.43);
\definecolor[named]{drawColor}{rgb}{0.97,0.46,0.43}

\node[text=drawColor,anchor=base west,inner sep=0pt, outer sep=0pt, scale=  1.29] at ( 53.28,153.48) {$\underline r_1$};

\node[text=drawColor,anchor=base west,inner sep=0pt, outer sep=0pt, scale=  1.29] at ( 90.66,153.48) {$B_1$};

\node[text=drawColor,anchor=base east,inner sep=0pt, outer sep=0pt, scale=  1.29] at (152.97,153.48) {$\overline r_1$};
\definecolor[named]{drawColor}{rgb}{0.00,0.75,0.77}

\node[text=drawColor,anchor=base west,inner sep=0pt, outer sep=0pt, scale=  1.29] at (165.43,153.48) {$\underline r_2$};

\node[text=drawColor,anchor=base west,inner sep=0pt, outer sep=0pt, scale=  1.29] at (215.27,153.48) {$B_2$};

\node[text=drawColor,anchor=base east,inner sep=0pt, outer sep=0pt, scale=  1.29] at (302.50,153.48) {$\overline r_2$};
\definecolor[named]{drawColor}{rgb}{0.97,0.46,0.43}

\path[draw=drawColor,line width= 0.4pt,line join=round,line cap=round] ( 53.28,143.66) circle (  3.20);

\path[draw=drawColor,line width= 0.4pt,line join=round,line cap=round] ( 65.74,121.95) circle (  3.20);

\path[draw=drawColor,line width= 0.4pt,line join=round,line cap=round] ( 78.20,100.24) circle (  3.20);

\path[draw=drawColor,line width= 0.4pt,line join=round,line cap=round] ( 90.66, 78.53) circle (  3.20);

\path[draw=drawColor,line width= 0.4pt,line join=round,line cap=round] (103.12, 91.56) circle (  3.20);

\path[draw=drawColor,line width= 0.4pt,line join=round,line cap=round] (115.58,104.58) circle (  3.20);

\path[draw=drawColor,line width= 0.4pt,line join=round,line cap=round] (128.04,117.61) circle (  3.20);

\path[draw=drawColor,line width= 0.4pt,line join=round,line cap=round] (140.50,130.63) circle (  3.20);

\path[draw=drawColor,line width= 0.4pt,line join=round,line cap=round] (152.97,143.66) circle (  3.20);

\path[draw=drawColor,line width= 0.4pt,line join=round,line cap=round] (165.43,143.66) circle (  3.20);

\path[draw=drawColor,line width= 0.4pt,line join=round,line cap=round] (177.89,143.66) circle (  3.20);

\path[draw=drawColor,line width= 0.4pt,line join=round,line cap=round] (190.35,143.66) circle (  3.20);

\path[draw=drawColor,line width= 0.4pt,line join=round,line cap=round] (202.81,143.66) circle (  3.20);

\path[draw=drawColor,line width= 0.4pt,line join=round,line cap=round] (215.27,143.66) circle (  3.20);

\path[draw=drawColor,line width= 0.4pt,line join=round,line cap=round] (227.73,143.66) circle (  3.20);

\path[draw=drawColor,line width= 0.4pt,line join=round,line cap=round] (240.19,143.66) circle (  3.20);

\path[draw=drawColor,line width= 0.4pt,line join=round,line cap=round] (252.65,143.66) circle (  3.20);

\path[draw=drawColor,line width= 0.4pt,line join=round,line cap=round] (265.11,143.66) circle (  3.20);

\path[draw=drawColor,line width= 0.4pt,line join=round,line cap=round] (277.57,143.66) circle (  3.20);

\path[draw=drawColor,line width= 0.4pt,line join=round,line cap=round] (290.03,143.66) circle (  3.20);

\path[draw=drawColor,line width= 0.4pt,line join=round,line cap=round] (302.50,143.66) circle (  3.20);
\definecolor[named]{drawColor}{rgb}{0.00,0.75,0.77}

\path[draw=drawColor,line width= 0.4pt,line join=round,line cap=round] ( 53.28,143.66) circle (  3.20);

\path[draw=drawColor,line width= 0.4pt,line join=round,line cap=round] ( 65.74,143.66) circle (  3.20);

\path[draw=drawColor,line width= 0.4pt,line join=round,line cap=round] ( 78.20,143.66) circle (  3.20);

\path[draw=drawColor,line width= 0.4pt,line join=round,line cap=round] ( 90.66,143.66) circle (  3.20);

\path[draw=drawColor,line width= 0.4pt,line join=round,line cap=round] (103.12,143.66) circle (  3.20);

\path[draw=drawColor,line width= 0.4pt,line join=round,line cap=round] (115.58,143.66) circle (  3.20);

\path[draw=drawColor,line width= 0.4pt,line join=round,line cap=round] (128.04,143.66) circle (  3.20);

\path[draw=drawColor,line width= 0.4pt,line join=round,line cap=round] (140.50,143.66) circle (  3.20);

\path[draw=drawColor,line width= 0.4pt,line join=round,line cap=round] (152.97,143.66) circle (  3.20);

\path[draw=drawColor,line width= 0.4pt,line join=round,line cap=round] (165.43,143.66) circle (  3.20);

\path[draw=drawColor,line width= 0.4pt,line join=round,line cap=round] (177.89,127.38) circle (  3.20);

\path[draw=drawColor,line width= 0.4pt,line join=round,line cap=round] (190.35,111.09) circle (  3.20);

\path[draw=drawColor,line width= 0.4pt,line join=round,line cap=round] (202.81, 94.81) circle (  3.20);

\path[draw=drawColor,line width= 0.4pt,line join=round,line cap=round] (215.27, 78.53) circle (  3.20);

\path[draw=drawColor,line width= 0.4pt,line join=round,line cap=round] (227.73, 87.84) circle (  3.20);

\path[draw=drawColor,line width= 0.4pt,line join=round,line cap=round] (240.19, 97.14) circle (  3.20);

\path[draw=drawColor,line width= 0.4pt,line join=round,line cap=round] (252.65,106.44) circle (  3.20);

\path[draw=drawColor,line width= 0.4pt,line join=round,line cap=round] (265.11,115.75) circle (  3.20);

\path[draw=drawColor,line width= 0.4pt,line join=round,line cap=round] (277.57,125.05) circle (  3.20);

\path[draw=drawColor,line width= 0.4pt,line join=round,line cap=round] (290.03,134.35) circle (  3.20);

\path[draw=drawColor,line width= 0.4pt,line join=round,line cap=round] (302.50,143.66) circle (  3.20);
\definecolor[named]{drawColor}{rgb}{0.97,0.46,0.43}

\path[draw=drawColor,line width= 2.6pt,line join=round] ( 33.34,143.66) --
	( 53.28,143.66) --
	( 90.66, 78.53) --
	(152.97,143.66) --
	(165.43,143.66) --
	(215.27,143.66) --
	(302.50,143.66) --
	(334.89,143.66);
\definecolor[named]{drawColor}{rgb}{0.00,0.75,0.77}

\path[draw=drawColor,line width= 1.4pt,line join=round] ( 33.34,143.66) --
	( 53.28,143.66) --
	( 90.66,143.66) --
	(152.97,143.66) --
	(165.43,143.66) --
	(215.27, 78.53) --
	(302.50,143.66) --
	(334.89,143.66);
\definecolor[named]{drawColor}{rgb}{0.97,0.46,0.43}

\node[text=drawColor,anchor=base west,inner sep=0pt, outer sep=0pt, scale=  1.29] at (109.35, 83.39) {$\ell_1 = C_{1,4,9}$};
\definecolor[named]{drawColor}{rgb}{0.00,0.75,0.77}

\node[text=drawColor,anchor=base west,inner sep=0pt, outer sep=0pt, scale=  1.29] at (240.19, 83.39) {$\ell_2 = C_{10,14,21}$};
\end{scope}
\begin{scope}
\path[clip] ( 33.34, 35.17) rectangle (334.89, 70.99);
\definecolor[named]{fillColor}{rgb}{0.90,0.90,0.90}

\path[fill=fillColor] ( 33.34, 35.17) rectangle (334.89, 70.99);
\definecolor[named]{drawColor}{rgb}{1.00,1.00,1.00}

\path[draw=drawColor,line width= 0.6pt,line join=round] ( 53.28, 35.17) --
	( 53.28, 70.99);

\path[draw=drawColor,line width= 0.6pt,line join=round] ( 90.66, 35.17) --
	( 90.66, 70.99);

\path[draw=drawColor,line width= 0.6pt,line join=round] (152.97, 35.17) --
	(152.97, 70.99);

\path[draw=drawColor,line width= 0.6pt,line join=round] (165.43, 35.17) --
	(165.43, 70.99);

\path[draw=drawColor,line width= 0.6pt,line join=round] (215.27, 35.17) --
	(215.27, 70.99);

\path[draw=drawColor,line width= 0.6pt,line join=round] (302.50, 35.17) --
	(302.50, 70.99);

\path[draw=drawColor,line width= 0.6pt,line join=round] (314.96, 35.17) --
	(314.96, 70.99);
\definecolor[named]{drawColor}{rgb}{0.00,0.60,0.94}
\definecolor[named]{fillColor}{rgb}{0.00,0.60,0.94}

\path[draw=drawColor,line width= 2.3pt,line join=round,fill=fillColor] ( 47.05, 36.80) -- ( 96.89, 36.80);

\path[draw=drawColor,line width= 2.3pt,line join=round,fill=fillColor] ( 96.89, 58.51) -- (221.50, 58.51);

\path[draw=drawColor,line width= 2.3pt,line join=round,fill=fillColor] (221.50, 69.36) -- (321.19, 69.36);
\definecolor[named]{drawColor}{rgb}{0.00,0.00,0.00}

\node[text=drawColor,anchor=base,inner sep=0pt, outer sep=0pt, scale=  0.77] at ( 53.28, 44.74) {N};

\node[text=drawColor,anchor=base,inner sep=0pt, outer sep=0pt, scale=  0.77] at ( 65.74, 44.74) {N};

\node[text=drawColor,anchor=base,inner sep=0pt, outer sep=0pt, scale=  0.77] at ( 78.20, 44.74) {N};

\node[text=drawColor,anchor=base,inner sep=0pt, outer sep=0pt, scale=  0.77] at ( 90.66, 44.74) {N};

\node[text=drawColor,anchor=base,inner sep=0pt, outer sep=0pt, scale=  0.77] at (103.12, 44.74) {N};

\node[text=drawColor,anchor=base,inner sep=0pt, outer sep=0pt, scale=  0.77] at (115.58, 44.74) {N};

\node[text=drawColor,anchor=base,inner sep=0pt, outer sep=0pt, scale=  0.77] at (128.04, 44.74) {N};

\node[text=drawColor,anchor=base,inner sep=0pt, outer sep=0pt, scale=  0.77] at (140.50, 44.74) {N};

\node[text=drawColor,anchor=base,inner sep=0pt, outer sep=0pt, scale=  0.77] at (152.97, 44.74) {N};

\node[text=drawColor,anchor=base,inner sep=0pt, outer sep=0pt, scale=  0.77] at (165.43, 44.74) {N};

\node[text=drawColor,anchor=base,inner sep=0pt, outer sep=0pt, scale=  0.77] at (177.89, 44.74) {N};

\node[text=drawColor,anchor=base,inner sep=0pt, outer sep=0pt, scale=  0.77] at (190.35, 44.74) {N};

\node[text=drawColor,anchor=base,inner sep=0pt, outer sep=0pt, scale=  0.77] at (202.81, 44.74) {N};

\node[text=drawColor,anchor=base,inner sep=0pt, outer sep=0pt, scale=  0.77] at (215.27, 44.74) {N};

\node[text=drawColor,anchor=base,inner sep=0pt, outer sep=0pt, scale=  0.77] at (227.73, 44.74) {N};

\node[text=drawColor,anchor=base,inner sep=0pt, outer sep=0pt, scale=  0.77] at (240.19, 44.74) {N};

\node[text=drawColor,anchor=base,inner sep=0pt, outer sep=0pt, scale=  0.77] at (252.65, 44.74) {N};

\node[text=drawColor,anchor=base,inner sep=0pt, outer sep=0pt, scale=  0.77] at (265.11, 44.74) {N};

\node[text=drawColor,anchor=base,inner sep=0pt, outer sep=0pt, scale=  0.77] at (277.57, 44.74) {N};

\node[text=drawColor,anchor=base,inner sep=0pt, outer sep=0pt, scale=  0.77] at (290.03, 44.74) {N};

\node[text=drawColor,anchor=base,inner sep=0pt, outer sep=0pt, scale=  0.77] at (302.50, 44.74) {N};

\node[text=drawColor,anchor=base,inner sep=0pt, outer sep=0pt, scale=  0.77] at (314.96, 44.74) {N};
\end{scope}
\begin{scope}
\path[clip] (  0.00,  0.00) rectangle (361.35,180.67);
\definecolor[named]{drawColor}{rgb}{0.50,0.50,0.50}

\node[text=drawColor,anchor=base east,inner sep=0pt, outer sep=0pt, scale=  0.87] at ( 26.23,140.37) {1};

\node[text=drawColor,anchor=base east,inner sep=0pt, outer sep=0pt, scale=  0.87] at ( 26.23, 75.24) {0};
\end{scope}
\begin{scope}
\path[clip] (  0.00,  0.00) rectangle (361.35,180.67);
\definecolor[named]{drawColor}{rgb}{0.50,0.50,0.50}

\path[draw=drawColor,line width= 0.6pt,line join=round] ( 29.07,143.66) --
	( 33.34,143.66);

\path[draw=drawColor,line width= 0.6pt,line join=round] ( 29.07, 78.53) --
	( 33.34, 78.53);
\end{scope}
\begin{scope}
\path[clip] (334.89, 74.30) rectangle (348.10,167.43);
\definecolor[named]{fillColor}{rgb}{0.80,0.80,0.80}

\path[fill=fillColor] (334.89, 74.30) rectangle (348.10,167.43);
\definecolor[named]{drawColor}{rgb}{0.00,0.00,0.00}

\node[text=drawColor,rotate=270.00,anchor=base,inner sep=0pt, outer sep=0pt, scale=  0.87] at (338.21,120.86) {error};
\end{scope}
\begin{scope}
\path[clip] (334.89, 35.17) rectangle (348.10, 70.99);
\definecolor[named]{fillColor}{rgb}{0.80,0.80,0.80}

\path[fill=fillColor] (334.89, 35.17) rectangle (348.10, 70.99);
\definecolor[named]{drawColor}{rgb}{0.00,0.00,0.00}

\node[text=drawColor,rotate=270.00,anchor=base,inner sep=0pt, outer sep=0pt, scale=  0.87] at (338.21, 53.08) {signal};
\end{scope}
\begin{scope}
\path[clip] (  0.00,  0.00) rectangle (361.35,180.67);
\definecolor[named]{drawColor}{rgb}{0.50,0.50,0.50}

\path[draw=drawColor,line width= 0.6pt,line join=round] ( 53.28, 30.90) --
	( 53.28, 35.17);

\path[draw=drawColor,line width= 0.6pt,line join=round] ( 90.66, 30.90) --
	( 90.66, 35.17);

\path[draw=drawColor,line width= 0.6pt,line join=round] (152.97, 30.90) --
	(152.97, 35.17);

\path[draw=drawColor,line width= 0.6pt,line join=round] (165.43, 30.90) --
	(165.43, 35.17);

\path[draw=drawColor,line width= 0.6pt,line join=round] (215.27, 30.90) --
	(215.27, 35.17);

\path[draw=drawColor,line width= 0.6pt,line join=round] (302.50, 30.90) --
	(302.50, 35.17);

\path[draw=drawColor,line width= 0.6pt,line join=round] (314.96, 30.90) --
	(314.96, 35.17);
\end{scope}
\begin{scope}
\path[clip] (  0.00,  0.00) rectangle (361.35,180.67);
\definecolor[named]{drawColor}{rgb}{0.50,0.50,0.50}

\node[text=drawColor,anchor=base,inner sep=0pt, outer sep=0pt, scale=  0.87] at ( 53.28, 21.48) {1};

\node[text=drawColor,anchor=base,inner sep=0pt, outer sep=0pt, scale=  0.87] at ( 90.66, 21.48) {4};

\node[text=drawColor,anchor=base,inner sep=0pt, outer sep=0pt, scale=  0.87] at (152.97, 21.48) {9};

\node[text=drawColor,anchor=base,inner sep=0pt, outer sep=0pt, scale=  0.87] at (165.43, 21.48) {10};

\node[text=drawColor,anchor=base,inner sep=0pt, outer sep=0pt, scale=  0.87] at (215.27, 21.48) {14};

\node[text=drawColor,anchor=base,inner sep=0pt, outer sep=0pt, scale=  0.87] at (302.50, 21.48) {21};

\node[text=drawColor,anchor=base,inner sep=0pt, outer sep=0pt, scale=  0.87] at (314.96, 21.48) {22};
\end{scope}
\begin{scope}
\path[clip] (  0.00,  0.00) rectangle (361.35,180.67);
\definecolor[named]{drawColor}{rgb}{0.00,0.00,0.00}

\node[text=drawColor,anchor=base,inner sep=0pt, outer sep=0pt, scale=  1.09] at (184.12,  9.94) {base position};
\end{scope}
\end{tikzpicture}

%% file: figure-variable-density-berr-k.tex
\begin{tikzpicture}[x=1pt,y=1pt]
\definecolor[named]{fillColor}{rgb}{1.00,1.00,1.00}
\path[use as bounding box,fill=fillColor,fill opacity=0.00] (0,0) rectangle (433.62,180.67);
\begin{scope}
\path[clip] (  0.00,  0.00) rectangle (433.62,180.67);
\definecolor[named]{drawColor}{rgb}{1.00,1.00,1.00}
\definecolor[named]{fillColor}{rgb}{1.00,1.00,1.00}

\path[draw=drawColor,line width= 0.6pt,line join=round,line cap=round,fill=fillColor] (  0.00,  0.00) rectangle (433.62,180.68);
\end{scope}
\begin{scope}
\path[clip] ( 38.09,154.22) rectangle (227.58,167.43);
\definecolor[named]{fillColor}{rgb}{0.80,0.80,0.80}

\path[fill=fillColor] ( 38.09,154.22) rectangle (227.58,167.43);
\definecolor[named]{drawColor}{rgb}{0.00,0.00,0.00}

\node[text=drawColor,anchor=base,inner sep=0pt, outer sep=0pt, scale=  0.87] at (132.83,157.53) {bases/probe = 374};
\end{scope}
\begin{scope}
\path[clip] (230.89,154.22) rectangle (420.37,167.43);
\definecolor[named]{fillColor}{rgb}{0.80,0.80,0.80}

\path[fill=fillColor] (230.89,154.22) rectangle (420.37,167.43);
\definecolor[named]{drawColor}{rgb}{0.00,0.00,0.00}

\node[text=drawColor,anchor=base,inner sep=0pt, outer sep=0pt, scale=  0.87] at (325.63,157.53) {bases/probe = 7};
\end{scope}
\begin{scope}
\path[clip] ( 38.09, 35.17) rectangle (227.58,154.22);
\definecolor[named]{fillColor}{rgb}{0.90,0.90,0.90}

\path[fill=fillColor] ( 38.09, 35.17) rectangle (227.58,154.22);
\definecolor[named]{drawColor}{rgb}{1.00,1.00,1.00}

\path[draw=drawColor,line width= 0.6pt,line join=round] ( 38.09, 40.31) --
	(227.58, 40.31);

\path[draw=drawColor,line width= 0.6pt,line join=round] ( 38.09, 48.01) --
	(227.58, 48.01);

\path[draw=drawColor,line width= 0.6pt,line join=round] ( 38.09, 86.50) --
	(227.58, 86.50);

\path[draw=drawColor,line width= 0.6pt,line join=round] ( 38.09,140.39) --
	(227.58,140.39);

\path[draw=drawColor,line width= 0.6pt,line join=round] ( 38.09,148.09) --
	(227.58,148.09);

\path[draw=drawColor,line width= 0.6pt,line join=round] ( 46.70, 35.17) --
	( 46.70,154.22);

\path[draw=drawColor,line width= 0.6pt,line join=round] (101.10, 35.17) --
	(101.10,154.22);

\path[draw=drawColor,line width= 0.6pt,line join=round] (218.96, 35.17) --
	(218.96,154.22);
\definecolor[named]{drawColor}{rgb}{0.97,0.46,0.43}

\path[draw=drawColor,line width= 0.6pt,line join=round] ( 46.70, 86.50) --
	( 55.77, 79.92) --
	( 64.84, 79.47) --
	( 73.90, 72.95) --
	( 82.97, 60.41) --
	( 92.04, 61.69) --
	(101.10, 49.15) --
	(110.17, 58.00) --
	(119.23, 65.70) --
	(128.30, 71.83) --
	(137.37, 79.52) --
	(146.43, 87.22) --
	(155.50, 94.92) --
	(164.57,102.62) --
	(173.63,110.32) --
	(182.70,118.02) --
	(191.76,125.71) --
	(200.83,133.41) --
	(209.90,141.11) --
	(218.96,148.81);
\definecolor[named]{fillColor}{rgb}{0.00,0.00,0.00}

\path[fill=fillColor] ( 46.70, 86.50) circle (  2.13);

\path[fill=fillColor] ( 55.77, 79.92) circle (  2.13);

\path[fill=fillColor] ( 64.84, 79.47) circle (  2.13);

\path[fill=fillColor] ( 73.90, 72.95) circle (  2.13);

\path[fill=fillColor] ( 82.97, 60.41) circle (  2.13);

\path[fill=fillColor] ( 92.04, 61.69) circle (  2.13);

\path[fill=fillColor] (101.10, 49.15) circle (  2.13);

\path[fill=fillColor] (110.17, 58.00) circle (  2.13);

\path[fill=fillColor] (119.23, 65.70) circle (  2.13);

\path[fill=fillColor] (128.30, 71.83) circle (  2.13);

\path[fill=fillColor] (137.37, 79.52) circle (  2.13);

\path[fill=fillColor] (146.43, 87.22) circle (  2.13);

\path[fill=fillColor] (155.50, 94.92) circle (  2.13);

\path[fill=fillColor] (164.57,102.62) circle (  2.13);

\path[fill=fillColor] (173.63,110.32) circle (  2.13);

\path[fill=fillColor] (182.70,118.02) circle (  2.13);

\path[fill=fillColor] (191.76,125.71) circle (  2.13);

\path[fill=fillColor] (200.83,133.41) circle (  2.13);

\path[fill=fillColor] (209.90,141.11) circle (  2.13);

\path[fill=fillColor] (218.96,148.81) circle (  2.13);
\end{scope}
\begin{scope}
\path[clip] (230.89, 35.17) rectangle (420.37,154.22);
\definecolor[named]{fillColor}{rgb}{0.90,0.90,0.90}

\path[fill=fillColor] (230.89, 35.17) rectangle (420.37,154.22);
\definecolor[named]{drawColor}{rgb}{1.00,1.00,1.00}

\path[draw=drawColor,line width= 0.6pt,line join=round] (230.89, 40.31) --
	(420.37, 40.31);

\path[draw=drawColor,line width= 0.6pt,line join=round] (230.89, 48.01) --
	(420.37, 48.01);

\path[draw=drawColor,line width= 0.6pt,line join=round] (230.89, 86.50) --
	(420.37, 86.50);

\path[draw=drawColor,line width= 0.6pt,line join=round] (230.89,140.39) --
	(420.37,140.39);

\path[draw=drawColor,line width= 0.6pt,line join=round] (230.89,148.09) --
	(420.37,148.09);

\path[draw=drawColor,line width= 0.6pt,line join=round] (239.50, 35.17) --
	(239.50,154.22);

\path[draw=drawColor,line width= 0.6pt,line join=round] (293.90, 35.17) --
	(293.90,154.22);

\path[draw=drawColor,line width= 0.6pt,line join=round] (411.76, 35.17) --
	(411.76,154.22);
\definecolor[named]{drawColor}{rgb}{0.00,0.75,0.77}

\path[draw=drawColor,line width= 0.6pt,line join=round] (239.50, 86.50) --
	(248.57, 78.83) --
	(257.63, 71.11) --
	(266.70, 63.46) --
	(275.76, 55.75) --
	(284.83, 48.25) --
	(293.90, 40.58) --
	(302.96, 48.28) --
	(312.03, 55.98) --
	(321.10, 63.68) --
	(330.16, 71.38) --
	(339.23, 79.07) --
	(348.29, 86.75) --
	(357.36, 94.47) --
	(366.43,102.17) --
	(375.49,109.87) --
	(384.56,117.54) --
	(393.63,125.24) --
	(402.69,132.94) --
	(411.76,140.64);
\definecolor[named]{fillColor}{rgb}{0.00,0.00,0.00}

\path[fill=fillColor] (239.50, 86.50) circle (  2.13);

\path[fill=fillColor] (248.57, 78.83) circle (  2.13);

\path[fill=fillColor] (257.63, 71.11) circle (  2.13);

\path[fill=fillColor] (266.70, 63.46) circle (  2.13);

\path[fill=fillColor] (275.76, 55.75) circle (  2.13);

\path[fill=fillColor] (284.83, 48.25) circle (  2.13);

\path[fill=fillColor] (293.90, 40.58) circle (  2.13);

\path[fill=fillColor] (302.96, 48.28) circle (  2.13);

\path[fill=fillColor] (312.03, 55.98) circle (  2.13);

\path[fill=fillColor] (321.10, 63.68) circle (  2.13);

\path[fill=fillColor] (330.16, 71.38) circle (  2.13);

\path[fill=fillColor] (339.23, 79.07) circle (  2.13);

\path[fill=fillColor] (348.29, 86.75) circle (  2.13);

\path[fill=fillColor] (357.36, 94.47) circle (  2.13);

\path[fill=fillColor] (366.43,102.17) circle (  2.13);

\path[fill=fillColor] (375.49,109.87) circle (  2.13);

\path[fill=fillColor] (384.56,117.54) circle (  2.13);

\path[fill=fillColor] (393.63,125.24) circle (  2.13);

\path[fill=fillColor] (402.69,132.94) circle (  2.13);

\path[fill=fillColor] (411.76,140.64) circle (  2.13);
\end{scope}
\begin{scope}
\path[clip] (  0.00,  0.00) rectangle (433.62,180.67);
\definecolor[named]{drawColor}{rgb}{0.50,0.50,0.50}

\node[text=drawColor,anchor=base east,inner sep=0pt, outer sep=0pt, scale=  0.87] at ( 30.98, 37.02) {0};

\node[text=drawColor,anchor=base east,inner sep=0pt, outer sep=0pt, scale=  0.87] at ( 30.98, 44.72) {1};

\node[text=drawColor,anchor=base east,inner sep=0pt, outer sep=0pt, scale=  0.87] at ( 30.98, 83.21) {6};

\node[text=drawColor,anchor=base east,inner sep=0pt, outer sep=0pt, scale=  0.87] at ( 30.98,137.10) {13};

\node[text=drawColor,anchor=base east,inner sep=0pt, outer sep=0pt, scale=  0.87] at ( 30.98,144.79) {14};
\end{scope}
\begin{scope}
\path[clip] (  0.00,  0.00) rectangle (433.62,180.67);
\definecolor[named]{drawColor}{rgb}{0.50,0.50,0.50}

\path[draw=drawColor,line width= 0.6pt,line join=round] ( 33.82, 40.31) --
	( 38.09, 40.31);

\path[draw=drawColor,line width= 0.6pt,line join=round] ( 33.82, 48.01) --
	( 38.09, 48.01);

\path[draw=drawColor,line width= 0.6pt,line join=round] ( 33.82, 86.50) --
	( 38.09, 86.50);

\path[draw=drawColor,line width= 0.6pt,line join=round] ( 33.82,140.39) --
	( 38.09,140.39);

\path[draw=drawColor,line width= 0.6pt,line join=round] ( 33.82,148.09) --
	( 38.09,148.09);
\end{scope}
\begin{scope}
\path[clip] (  0.00,  0.00) rectangle (433.62,180.67);
\definecolor[named]{drawColor}{rgb}{0.50,0.50,0.50}

\path[draw=drawColor,line width= 0.6pt,line join=round] ( 46.70, 30.90) --
	( 46.70, 35.17);

\path[draw=drawColor,line width= 0.6pt,line join=round] (101.10, 30.90) --
	(101.10, 35.17);

\path[draw=drawColor,line width= 0.6pt,line join=round] (218.96, 30.90) --
	(218.96, 35.17);
\end{scope}
\begin{scope}
\path[clip] (  0.00,  0.00) rectangle (433.62,180.67);
\definecolor[named]{drawColor}{rgb}{0.50,0.50,0.50}

\node[text=drawColor,anchor=base,inner sep=0pt, outer sep=0pt, scale=  0.87] at ( 46.70, 21.48) {1};

\node[text=drawColor,anchor=base,inner sep=0pt, outer sep=0pt, scale=  0.87] at (101.10, 21.48) {7};

\node[text=drawColor,anchor=base,inner sep=0pt, outer sep=0pt, scale=  0.87] at (218.96, 21.48) {20};
\end{scope}
\begin{scope}
\path[clip] (  0.00,  0.00) rectangle (433.62,180.67);
\definecolor[named]{drawColor}{rgb}{0.50,0.50,0.50}

\path[draw=drawColor,line width= 0.6pt,line join=round] (239.50, 30.90) --
	(239.50, 35.17);

\path[draw=drawColor,line width= 0.6pt,line join=round] (293.90, 30.90) --
	(293.90, 35.17);

\path[draw=drawColor,line width= 0.6pt,line join=round] (411.76, 30.90) --
	(411.76, 35.17);
\end{scope}
\begin{scope}
\path[clip] (  0.00,  0.00) rectangle (433.62,180.67);
\definecolor[named]{drawColor}{rgb}{0.50,0.50,0.50}

\node[text=drawColor,anchor=base,inner sep=0pt, outer sep=0pt, scale=  0.87] at (239.50, 21.48) {1};

\node[text=drawColor,anchor=base,inner sep=0pt, outer sep=0pt, scale=  0.87] at (293.90, 21.48) {7};

\node[text=drawColor,anchor=base,inner sep=0pt, outer sep=0pt, scale=  0.87] at (411.76, 21.48) {20};
\end{scope}
\begin{scope}
\path[clip] (  0.00,  0.00) rectangle (433.62,180.67);
\definecolor[named]{drawColor}{rgb}{0.00,0.00,0.00}

\node[text=drawColor,anchor=base,inner sep=0pt, outer sep=0pt, scale=  1.09] at (229.23,  9.94) {Number of segments $k$ in estimated maximum likelihood model};
\end{scope}
\begin{scope}
\path[clip] (  0.00,  0.00) rectangle (433.62,180.67);
\definecolor[named]{drawColor}{rgb}{0.00,0.00,0.00}

\node[text=drawColor,rotate= 90.00,anchor=base,inner sep=0pt, outer sep=0pt, scale=  1.09] at ( 18.16, 94.70) {breakpointError $\textrm{BErr}_i(k)$};
\end{scope}
\end{tikzpicture}

%% file: figure-variable-density-berr.tex
\begin{tikzpicture}[x=1pt,y=1pt]
\definecolor[named]{fillColor}{rgb}{1.00,1.00,1.00}
\path[use as bounding box,fill=fillColor,fill opacity=0.00] (0,0) rectangle (433.62,144.54);
\begin{scope}
\path[clip] (  0.00,  0.00) rectangle (433.62,144.54);
\definecolor[named]{drawColor}{rgb}{1.00,1.00,1.00}
\definecolor[named]{fillColor}{rgb}{1.00,1.00,1.00}

\path[draw=drawColor,line width= 0.6pt,line join=round,line cap=round,fill=fillColor] ( -0.00, -0.00) rectangle (433.62,144.54);
\end{scope}
\begin{scope}
\path[clip] ( 38.09,118.08) rectangle (134.89,131.29);
\definecolor[named]{fillColor}{rgb}{0.80,0.80,0.80}

\path[fill=fillColor] ( 38.09,118.08) rectangle (134.89,131.29);
\definecolor[named]{drawColor}{rgb}{0.00,0.00,0.00}

\node[text=drawColor,anchor=base,inner sep=0pt, outer sep=0pt, scale=  0.87] at ( 86.49,121.40) {$\alpha = 0$};
\end{scope}
\begin{scope}
\path[clip] (138.20,118.08) rectangle (235.00,131.29);
\definecolor[named]{fillColor}{rgb}{0.80,0.80,0.80}

\path[fill=fillColor] (138.20,118.08) rectangle (235.00,131.29);
\definecolor[named]{drawColor}{rgb}{0.00,0.00,0.00}

\node[text=drawColor,anchor=base,inner sep=0pt, outer sep=0pt, scale=  0.87] at (186.60,121.40) {$\alpha = 0.5$};
\end{scope}
\begin{scope}
\path[clip] (238.31,118.08) rectangle (335.10,131.29);
\definecolor[named]{fillColor}{rgb}{0.80,0.80,0.80}

\path[fill=fillColor] (238.31,118.08) rectangle (335.10,131.29);
\definecolor[named]{drawColor}{rgb}{0.00,0.00,0.00}

\node[text=drawColor,anchor=base,inner sep=0pt, outer sep=0pt, scale=  0.87] at (286.71,121.40) {$\alpha = 1$};
\end{scope}
\begin{scope}
\path[clip] ( 38.09, 35.17) rectangle (134.89,118.08);
\definecolor[named]{fillColor}{rgb}{0.90,0.90,0.90}

\path[fill=fillColor] ( 38.09, 35.17) rectangle (134.89,118.08);
\definecolor[named]{drawColor}{rgb}{0.95,0.95,0.95}

\path[draw=drawColor,line width= 0.3pt,line join=round] ( 38.09, 52.15) --
	(134.89, 52.15);

\path[draw=drawColor,line width= 0.3pt,line join=round] ( 38.09, 78.96) --
	(134.89, 78.96);

\path[draw=drawColor,line width= 0.3pt,line join=round] ( 38.09,105.77) --
	(134.89,105.77);

\path[draw=drawColor,line width= 0.3pt,line join=round] ( 49.82, 35.17) --
	( 49.82,118.08);

\path[draw=drawColor,line width= 0.3pt,line join=round] ( 74.27, 35.17) --
	( 74.27,118.08);

\path[draw=drawColor,line width= 0.3pt,line join=round] ( 98.71, 35.17) --
	( 98.71,118.08);

\path[draw=drawColor,line width= 0.3pt,line join=round] (123.16, 35.17) --
	(123.16,118.08);
\definecolor[named]{drawColor}{rgb}{1.00,1.00,1.00}

\path[draw=drawColor,line width= 0.6pt,line join=round] ( 38.09, 38.75) --
	(134.89, 38.75);

\path[draw=drawColor,line width= 0.6pt,line join=round] ( 38.09, 65.56) --
	(134.89, 65.56);

\path[draw=drawColor,line width= 0.6pt,line join=round] ( 38.09, 92.37) --
	(134.89, 92.37);

\path[draw=drawColor,line width= 0.6pt,line join=round] ( 62.05, 35.17) --
	( 62.05,118.08);

\path[draw=drawColor,line width= 0.6pt,line join=round] ( 86.49, 35.17) --
	( 86.49,118.08);

\path[draw=drawColor,line width= 0.6pt,line join=round] (110.93, 35.17) --
	(110.93,118.08);
\definecolor[named]{drawColor}{rgb}{0.00,0.75,0.77}

\path[draw=drawColor,line width= 1.1pt,line join=round] ( 42.49,108.62) --
	( 42.93,108.62) --
	( 43.38,108.62) --
	( 43.82,108.62) --
	( 44.26,108.62) --
	( 44.70,108.62) --
	( 45.14,108.62) --
	( 45.59,108.62) --
	( 46.03,108.62) --
	( 46.47,108.62) --
	( 46.91,108.62) --
	( 47.36,108.62) --
	( 47.80,108.62) --
	( 48.24,108.62) --
	( 48.68,108.62) --
	( 49.12,108.62) --
	( 49.57,108.62) --
	( 50.01,108.62) --
	( 50.45,108.62) --
	( 50.89,108.62) --
	( 51.34,108.62) --
	( 51.78,108.62) --
	( 52.22,108.62) --
	( 52.66,108.62) --
	( 53.10,108.62) --
	( 53.55,108.62) --
	( 53.99,108.62) --
	( 54.43,108.62) --
	( 54.87,108.62) --
	( 55.32,108.62) --
	( 55.76,108.62) --
	( 56.20,108.62) --
	( 56.64,108.62) --
	( 57.08,108.62) --
	( 57.53,108.62) --
	( 57.97,108.62) --
	( 58.41,108.62) --
	( 58.85,108.62) --
	( 59.30,108.62) --
	( 59.74,108.62) --
	( 60.18,108.62) --
	( 60.62,108.62) --
	( 61.06,108.62) --
	( 61.51,108.62) --
	( 61.95,108.62) --
	( 62.39,108.62) --
	( 62.83,108.62) --
	( 63.27,108.62) --
	( 63.72,108.62) --
	( 64.16,108.62) --
	( 64.60,108.62) --
	( 65.04,108.62) --
	( 65.49,108.62) --
	( 65.93,108.62) --
	( 66.37,108.62) --
	( 66.81,108.62) --
	( 67.25,108.62) --
	( 67.70,108.62) --
	( 68.14,108.62) --
	( 68.58,108.62) --
	( 69.02,108.62) --
	( 69.47,108.62) --
	( 69.91,108.62) --
	( 70.35,108.62) --
	( 70.79,108.62) --
	( 71.23,108.62) --
	( 71.68,108.62) --
	( 72.12,108.62) --
	( 72.56,108.62) --
	( 73.00,108.62) --
	( 73.45,108.62) --
	( 73.89,108.62) --
	( 74.33,108.62) --
	( 74.77,108.62) --
	( 75.21,108.62) --
	( 75.66,108.62) --
	( 76.10,108.62) --
	( 76.54,108.62) --
	( 76.98,108.62) --
	( 77.42,108.62) --
	( 77.87,108.62) --
	( 78.31,108.62) --
	( 78.75,108.62) --
	( 79.19,108.62) --
	( 79.64,108.62) --
	( 80.08,108.62) --
	( 80.52,108.62) --
	( 80.96,108.62) --
	( 81.40,108.62) --
	( 81.85,108.62) --
	( 82.29,108.62) --
	( 82.73,108.62) --
	( 83.17,108.62) --
	( 83.62,108.62) --
	( 84.06,108.62) --
	( 84.50,108.62) --
	( 84.94,108.62) --
	( 85.38,108.62) --
	( 85.83,108.62) --
	( 86.27,108.62) --
	( 86.71,108.62) --
	( 87.15,108.62) --
	( 87.60,108.62) --
	( 88.04,108.62) --
	( 88.48,108.62) --
	( 88.92,108.62) --
	( 89.36,108.62) --
	( 89.81,108.62) --
	( 90.25,108.62) --
	( 90.69,108.62) --
	( 91.13,103.26) --
	( 91.58, 65.75) --
	( 92.02, 49.66) --
	( 92.46, 38.94) --
	( 92.90, 38.94) --
	( 93.34, 38.94) --
	( 93.79, 38.94) --
	( 94.23, 38.94) --
	( 94.67, 38.94) --
	( 95.11, 38.94) --
	( 95.55, 38.94) --
	( 96.00, 38.94) --
	( 96.44, 38.94) --
	( 96.88, 38.94) --
	( 97.32, 38.94) --
	( 97.77, 38.94) --
	( 98.21, 38.94) --
	( 98.65, 44.28) --
	( 99.09, 54.87) --
	( 99.53, 54.87) --
	( 99.98, 54.87) --
	(100.42, 65.58) --
	(100.86, 65.58) --
	(101.30, 65.58) --
	(101.75, 65.58) --
	(102.19, 65.58) --
	(102.63, 65.58) --
	(103.07, 70.92) --
	(103.51, 70.92) --
	(103.96, 70.92) --
	(104.40, 70.92) --
	(104.84, 70.92) --
	(105.28, 70.92) --
	(105.73, 70.92) --
	(106.17, 70.92) --
	(106.61, 70.92) --
	(107.05, 70.92) --
	(107.49, 70.92) --
	(107.94, 70.92) --
	(108.38, 70.92) --
	(108.82, 70.92) --
	(109.26, 70.92) --
	(109.71, 70.92) --
	(110.15, 70.92) --
	(110.59, 70.92) --
	(111.03, 70.92) --
	(111.47, 70.92) --
	(111.92, 70.92) --
	(112.36, 70.92) --
	(112.80, 70.92) --
	(113.24, 70.92) --
	(113.68, 70.92) --
	(114.13, 70.92) --
	(114.57, 70.92) --
	(115.01, 70.92) --
	(115.45, 70.92) --
	(115.90, 70.92) --
	(116.34, 70.92) --
	(116.78, 70.92) --
	(117.22, 70.92) --
	(117.66, 70.92) --
	(118.11, 70.92) --
	(118.55, 70.92) --
	(118.99, 70.92) --
	(119.43, 70.92) --
	(119.88, 70.92) --
	(120.32, 70.92) --
	(120.76, 70.92) --
	(121.20, 70.92) --
	(121.64, 70.92) --
	(122.09, 70.92) --
	(122.53, 70.92) --
	(122.97, 70.92) --
	(123.41, 70.92) --
	(123.86, 70.92) --
	(124.30, 70.92) --
	(124.74, 70.92) --
	(125.18, 70.92) --
	(125.62, 70.92) --
	(126.07, 70.92) --
	(126.51, 70.92) --
	(126.95, 70.92) --
	(127.39, 70.92) --
	(127.83, 70.92) --
	(128.28, 70.92) --
	(128.72, 70.92) --
	(129.16, 70.92) --
	(129.60, 70.92) --
	(130.05, 70.92) --
	(130.49, 70.92);
\definecolor[named]{drawColor}{rgb}{0.97,0.46,0.43}

\path[draw=drawColor,line width= 0.6pt,line join=round] ( 42.49,114.32) --
	( 42.93,114.32) --
	( 43.38,114.32) --
	( 43.82,114.32) --
	( 44.26,114.32) --
	( 44.70,114.32) --
	( 45.14,114.32) --
	( 45.59,114.32) --
	( 46.03,114.32) --
	( 46.47,114.32) --
	( 46.91,114.32) --
	( 47.36,114.32) --
	( 47.80,114.32) --
	( 48.24,114.32) --
	( 48.68,114.32) --
	( 49.12,114.32) --
	( 49.57,114.32) --
	( 50.01,114.32) --
	( 50.45,114.32) --
	( 50.89,114.32) --
	( 51.34,114.32) --
	( 51.78,114.32) --
	( 52.22,114.32) --
	( 52.66,114.32) --
	( 53.10,114.32) --
	( 53.55,114.32) --
	( 53.99,114.32) --
	( 54.43,114.32) --
	( 54.87,114.32) --
	( 55.32,114.32) --
	( 55.76,114.32) --
	( 56.20,114.32) --
	( 56.64,114.32) --
	( 57.08,114.32) --
	( 57.53,114.32) --
	( 57.97,114.32) --
	( 58.41,114.32) --
	( 58.85,114.32) --
	( 59.30,114.32) --
	( 59.74,114.32) --
	( 60.18,114.32) --
	( 60.62,114.32) --
	( 61.06,114.32) --
	( 61.51,114.32) --
	( 61.95,114.32) --
	( 62.39,114.32) --
	( 62.83,114.32) --
	( 63.27,114.32) --
	( 63.72,114.32) --
	( 64.16,114.32) --
	( 64.60,114.32) --
	( 65.04,114.32) --
	( 65.49,114.32) --
	( 65.93,114.32) --
	( 66.37,114.32) --
	( 66.81,114.32) --
	( 67.25,114.32) --
	( 67.70,114.32) --
	( 68.14,114.32) --
	( 68.58,114.32) --
	( 69.02,114.32) --
	( 69.47,114.32) --
	( 69.91,114.32) --
	( 70.35,114.32) --
	( 70.79,114.32) --
	( 71.23,114.32) --
	( 71.68,114.32) --
	( 72.12,114.32) --
	( 72.56,114.32) --
	( 73.00,114.32) --
	( 73.45,114.32) --
	( 73.89,114.32) --
	( 74.33,114.32) --
	( 74.77,114.32) --
	( 75.21,114.32) --
	( 75.66,114.32) --
	( 76.10,114.32) --
	( 76.54,114.32) --
	( 76.98,114.32) --
	( 77.42,114.32) --
	( 77.87,114.32) --
	( 78.31,114.32) --
	( 78.75,114.32) --
	( 79.19,114.32) --
	( 79.64,114.32) --
	( 80.08,114.32) --
	( 80.52,114.32) --
	( 80.96,114.32) --
	( 81.40,114.32) --
	( 81.85,114.32) --
	( 82.29,114.32) --
	( 82.73,114.32) --
	( 83.17,114.32) --
	( 83.62,114.32) --
	( 84.06,114.32) --
	( 84.50,114.32) --
	( 84.94,114.32) --
	( 85.38,114.32) --
	( 85.83,114.32) --
	( 86.27,114.32) --
	( 86.71,114.32) --
	( 87.15,114.32) --
	( 87.60,114.32) --
	( 88.04,114.32) --
	( 88.48, 92.87) --
	( 88.92, 76.78) --
	( 89.36, 66.06) --
	( 89.81, 56.44) --
	( 90.25, 44.91) --
	( 90.69, 44.91) --
	( 91.13, 53.64) --
	( 91.58, 53.64) --
	( 92.02, 53.64) --
	( 92.46, 66.34) --
	( 92.90, 66.34) --
	( 93.34, 66.34) --
	( 93.79, 66.34) --
	( 94.23, 70.92) --
	( 94.67, 70.92) --
	( 95.11, 70.92) --
	( 95.55, 70.92) --
	( 96.00, 70.92) --
	( 96.44, 70.92) --
	( 96.88, 70.92) --
	( 97.32, 70.92) --
	( 97.77, 70.92) --
	( 98.21, 70.92) --
	( 98.65, 70.92) --
	( 99.09, 70.92) --
	( 99.53, 70.92) --
	( 99.98, 70.92) --
	(100.42, 70.92) --
	(100.86, 70.92) --
	(101.30, 70.92) --
	(101.75, 70.92) --
	(102.19, 70.92) --
	(102.63, 70.92) --
	(103.07, 70.92) --
	(103.51, 70.92) --
	(103.96, 70.92) --
	(104.40, 70.92) --
	(104.84, 70.92) --
	(105.28, 70.92) --
	(105.73, 70.92) --
	(106.17, 70.92) --
	(106.61, 70.92) --
	(107.05, 70.92) --
	(107.49, 70.92) --
	(107.94, 70.92) --
	(108.38, 70.92) --
	(108.82, 70.92) --
	(109.26, 70.92) --
	(109.71, 70.92) --
	(110.15, 70.92) --
	(110.59, 70.92) --
	(111.03, 70.92) --
	(111.47, 70.92) --
	(111.92, 70.92) --
	(112.36, 70.92) --
	(112.80, 70.92) --
	(113.24, 70.92) --
	(113.68, 70.92) --
	(114.13, 70.92) --
	(114.57, 70.92) --
	(115.01, 70.92) --
	(115.45, 70.92) --
	(115.90, 70.92) --
	(116.34, 70.92) --
	(116.78, 70.92) --
	(117.22, 70.92) --
	(117.66, 70.92) --
	(118.11, 70.92) --
	(118.55, 70.92) --
	(118.99, 70.92) --
	(119.43, 70.92) --
	(119.88, 70.92) --
	(120.32, 70.92) --
	(120.76, 70.92) --
	(121.20, 70.92) --
	(121.64, 70.92) --
	(122.09, 70.92) --
	(122.53, 70.92) --
	(122.97, 70.92) --
	(123.41, 70.92) --
	(123.86, 70.92) --
	(124.30, 70.92) --
	(124.74, 70.92) --
	(125.18, 70.92) --
	(125.62, 70.92) --
	(126.07, 70.92) --
	(126.51, 70.92) --
	(126.95, 70.92) --
	(127.39, 70.92) --
	(127.83, 70.92) --
	(128.28, 70.92) --
	(128.72, 70.92) --
	(129.16, 70.92) --
	(129.60, 70.92) --
	(130.05, 70.92) --
	(130.49, 70.92);
\definecolor[named]{drawColor}{rgb}{0.00,0.00,0.00}
\definecolor[named]{fillColor}{rgb}{0.97,0.46,0.43}

\path[draw=drawColor,line width= 0.4pt,line join=round,line cap=round,fill=fillColor] ( 90.25, 44.91) circle (  2.13);
\definecolor[named]{fillColor}{rgb}{0.00,0.75,0.77}

\path[draw=drawColor,line width= 0.4pt,line join=round,line cap=round,fill=fillColor] ( 95.11, 38.94) circle (  2.13);
\end{scope}
\begin{scope}
\path[clip] (138.20, 35.17) rectangle (235.00,118.08);
\definecolor[named]{fillColor}{rgb}{0.90,0.90,0.90}

\path[fill=fillColor] (138.20, 35.17) rectangle (235.00,118.08);
\definecolor[named]{drawColor}{rgb}{0.95,0.95,0.95}

\path[draw=drawColor,line width= 0.3pt,line join=round] (138.20, 52.15) --
	(235.00, 52.15);

\path[draw=drawColor,line width= 0.3pt,line join=round] (138.20, 78.96) --
	(235.00, 78.96);

\path[draw=drawColor,line width= 0.3pt,line join=round] (138.20,105.77) --
	(235.00,105.77);

\path[draw=drawColor,line width= 0.3pt,line join=round] (149.93, 35.17) --
	(149.93,118.08);

\path[draw=drawColor,line width= 0.3pt,line join=round] (174.38, 35.17) --
	(174.38,118.08);

\path[draw=drawColor,line width= 0.3pt,line join=round] (198.82, 35.17) --
	(198.82,118.08);

\path[draw=drawColor,line width= 0.3pt,line join=round] (223.26, 35.17) --
	(223.26,118.08);
\definecolor[named]{drawColor}{rgb}{1.00,1.00,1.00}

\path[draw=drawColor,line width= 0.6pt,line join=round] (138.20, 38.75) --
	(235.00, 38.75);

\path[draw=drawColor,line width= 0.6pt,line join=round] (138.20, 65.56) --
	(235.00, 65.56);

\path[draw=drawColor,line width= 0.6pt,line join=round] (138.20, 92.37) --
	(235.00, 92.37);

\path[draw=drawColor,line width= 0.6pt,line join=round] (162.15, 35.17) --
	(162.15,118.08);

\path[draw=drawColor,line width= 0.6pt,line join=round] (186.60, 35.17) --
	(186.60,118.08);

\path[draw=drawColor,line width= 0.6pt,line join=round] (211.04, 35.17) --
	(211.04,118.08);
\definecolor[named]{drawColor}{rgb}{0.00,0.75,0.77}

\path[draw=drawColor,line width= 1.1pt,line join=round] (142.60,108.62) --
	(143.04,108.62) --
	(143.48,108.62) --
	(143.93,108.62) --
	(144.37,108.62) --
	(144.81,108.62) --
	(145.25,108.62) --
	(145.70,108.62) --
	(146.14,108.62) --
	(146.58,108.62) --
	(147.02,108.62) --
	(147.46,108.62) --
	(147.91,108.62) --
	(148.35,108.62) --
	(148.79,108.62) --
	(149.23,108.62) --
	(149.68,108.62) --
	(150.12,108.62) --
	(150.56,108.62) --
	(151.00,108.62) --
	(151.44,108.62) --
	(151.89,108.62) --
	(152.33,108.62) --
	(152.77,108.62) --
	(153.21,108.62) --
	(153.65,108.62) --
	(154.10,108.62) --
	(154.54,108.62) --
	(154.98,108.62) --
	(155.42,108.62) --
	(155.87,108.62) --
	(156.31,108.62) --
	(156.75,108.62) --
	(157.19,108.62) --
	(157.63,108.62) --
	(158.08,108.62) --
	(158.52,108.62) --
	(158.96,108.62) --
	(159.40,108.62) --
	(159.85,108.62) --
	(160.29,108.62) --
	(160.73,108.62) --
	(161.17,108.62) --
	(161.61,108.62) --
	(162.06,108.62) --
	(162.50,108.62) --
	(162.94,108.62) --
	(163.38,108.62) --
	(163.83,108.62) --
	(164.27,108.62) --
	(164.71,108.62) --
	(165.15,108.62) --
	(165.59,108.62) --
	(166.04,108.62) --
	(166.48,108.62) --
	(166.92,108.62) --
	(167.36,108.62) --
	(167.81,108.62) --
	(168.25,108.62) --
	(168.69,108.62) --
	(169.13,108.62) --
	(169.57,108.62) --
	(170.02,108.62) --
	(170.46,108.62) --
	(170.90,108.62) --
	(171.34,108.62) --
	(171.78,108.62) --
	(172.23,108.62) --
	(172.67,108.62) --
	(173.11,108.62) --
	(173.55,108.62) --
	(174.00,108.62) --
	(174.44,108.62) --
	(174.88,108.62) --
	(175.32,108.62) --
	(175.76,108.62) --
	(176.21,108.62) --
	(176.65,108.62) --
	(177.09,108.62) --
	(177.53,108.62) --
	(177.98,108.62) --
	(178.42,108.62) --
	(178.86,108.62) --
	(179.30,108.62) --
	(179.74,108.62) --
	(180.19,108.62) --
	(180.63,108.62) --
	(181.07,108.62) --
	(181.51,103.26) --
	(181.96, 65.75) --
	(182.40, 49.66) --
	(182.84, 38.94) --
	(183.28, 38.94) --
	(183.72, 38.94) --
	(184.17, 38.94) --
	(184.61, 38.94) --
	(185.05, 38.94) --
	(185.49, 38.94) --
	(185.93, 38.94) --
	(186.38, 38.94) --
	(186.82, 38.94) --
	(187.26, 38.94) --
	(187.70, 38.94) --
	(188.15, 38.94) --
	(188.59, 38.94) --
	(189.03, 44.28) --
	(189.47, 54.87) --
	(189.91, 54.87) --
	(190.36, 54.87) --
	(190.80, 65.58) --
	(191.24, 65.58) --
	(191.68, 65.58) --
	(192.13, 65.58) --
	(192.57, 65.58) --
	(193.01, 65.58) --
	(193.45, 70.92) --
	(193.89, 70.92) --
	(194.34, 70.92) --
	(194.78, 70.92) --
	(195.22, 70.92) --
	(195.66, 70.92) --
	(196.11, 70.92) --
	(196.55, 70.92) --
	(196.99, 70.92) --
	(197.43, 70.92) --
	(197.87, 70.92) --
	(198.32, 70.92) --
	(198.76, 70.92) --
	(199.20, 70.92) --
	(199.64, 70.92) --
	(200.09, 70.92) --
	(200.53, 70.92) --
	(200.97, 70.92) --
	(201.41, 70.92) --
	(201.85, 70.92) --
	(202.30, 70.92) --
	(202.74, 70.92) --
	(203.18, 70.92) --
	(203.62, 70.92) --
	(204.06, 70.92) --
	(204.51, 70.92) --
	(204.95, 70.92) --
	(205.39, 70.92) --
	(205.83, 70.92) --
	(206.28, 70.92) --
	(206.72, 70.92) --
	(207.16, 70.92) --
	(207.60, 70.92) --
	(208.04, 70.92) --
	(208.49, 70.92) --
	(208.93, 70.92) --
	(209.37, 70.92) --
	(209.81, 70.92) --
	(210.26, 70.92) --
	(210.70, 70.92) --
	(211.14, 70.92) --
	(211.58, 70.92) --
	(212.02, 70.92) --
	(212.47, 70.92) --
	(212.91, 70.92) --
	(213.35, 70.92) --
	(213.79, 70.92) --
	(214.24, 70.92) --
	(214.68, 70.92) --
	(215.12, 70.92) --
	(215.56, 70.92) --
	(216.00, 70.92) --
	(216.45, 70.92) --
	(216.89, 70.92) --
	(217.33, 70.92) --
	(217.77, 70.92) --
	(218.22, 70.92) --
	(218.66, 70.92) --
	(219.10, 70.92) --
	(219.54, 70.92) --
	(219.98, 70.92) --
	(220.43, 70.92) --
	(220.87, 70.92) --
	(221.31, 70.92) --
	(221.75, 70.92) --
	(222.19, 70.92) --
	(222.64, 70.92) --
	(223.08, 70.92) --
	(223.52, 70.92) --
	(223.96, 70.92) --
	(224.41, 70.92) --
	(224.85, 70.92) --
	(225.29, 70.92) --
	(225.73, 70.92) --
	(226.17, 70.92) --
	(226.62, 70.92) --
	(227.06, 70.92) --
	(227.50, 70.92) --
	(227.94, 70.92) --
	(228.39, 70.92) --
	(228.83, 70.92) --
	(229.27, 70.92) --
	(229.71, 70.92) --
	(230.15, 70.92) --
	(230.60, 70.92);
\definecolor[named]{drawColor}{rgb}{0.97,0.46,0.43}

\path[draw=drawColor,line width= 0.6pt,line join=round] (142.60,114.32) --
	(143.04,114.32) --
	(143.48,114.32) --
	(143.93,114.32) --
	(144.37,114.32) --
	(144.81,114.32) --
	(145.25,114.32) --
	(145.70,114.32) --
	(146.14,114.32) --
	(146.58,114.32) --
	(147.02,114.32) --
	(147.46,114.32) --
	(147.91,114.32) --
	(148.35,114.32) --
	(148.79,114.32) --
	(149.23,114.32) --
	(149.68,114.32) --
	(150.12,114.32) --
	(150.56,114.32) --
	(151.00,114.32) --
	(151.44,114.32) --
	(151.89,114.32) --
	(152.33,114.32) --
	(152.77,114.32) --
	(153.21,114.32) --
	(153.65,114.32) --
	(154.10,114.32) --
	(154.54,114.32) --
	(154.98,114.32) --
	(155.42,114.32) --
	(155.87,114.32) --
	(156.31,114.32) --
	(156.75,114.32) --
	(157.19,114.32) --
	(157.63,114.32) --
	(158.08,114.32) --
	(158.52,114.32) --
	(158.96,114.32) --
	(159.40,114.32) --
	(159.85,114.32) --
	(160.29,114.32) --
	(160.73,114.32) --
	(161.17,114.32) --
	(161.61,114.32) --
	(162.06,114.32) --
	(162.50,114.32) --
	(162.94,114.32) --
	(163.38,114.32) --
	(163.83,114.32) --
	(164.27,114.32) --
	(164.71,114.32) --
	(165.15,114.32) --
	(165.59,114.32) --
	(166.04,114.32) --
	(166.48,114.32) --
	(166.92,114.32) --
	(167.36,114.32) --
	(167.81,114.32) --
	(168.25,114.32) --
	(168.69,114.32) --
	(169.13,114.32) --
	(169.57,114.32) --
	(170.02,114.32) --
	(170.46,114.32) --
	(170.90,114.32) --
	(171.34,114.32) --
	(171.78,114.32) --
	(172.23,114.32) --
	(172.67,114.32) --
	(173.11,114.32) --
	(173.55,114.32) --
	(174.00,114.32) --
	(174.44,114.32) --
	(174.88,114.32) --
	(175.32,114.32) --
	(175.76,114.32) --
	(176.21,114.32) --
	(176.65,114.32) --
	(177.09,114.32) --
	(177.53,114.32) --
	(177.98,114.32) --
	(178.42,114.32) --
	(178.86,114.32) --
	(179.30,114.32) --
	(179.74,114.32) --
	(180.19,114.32) --
	(180.63,114.32) --
	(181.07,114.32) --
	(181.51,114.32) --
	(181.96,114.32) --
	(182.40,114.32) --
	(182.84,114.32) --
	(183.28, 82.15) --
	(183.72, 76.78) --
	(184.17, 56.44) --
	(184.61, 44.91) --
	(185.05, 44.91) --
	(185.49, 53.64) --
	(185.93, 53.64) --
	(186.38, 53.64) --
	(186.82, 61.48) --
	(187.26, 66.34) --
	(187.70, 66.34) --
	(188.15, 66.34) --
	(188.59, 66.34) --
	(189.03, 70.92) --
	(189.47, 70.92) --
	(189.91, 70.92) --
	(190.36, 70.92) --
	(190.80, 70.92) --
	(191.24, 70.92) --
	(191.68, 70.92) --
	(192.13, 70.92) --
	(192.57, 70.92) --
	(193.01, 70.92) --
	(193.45, 70.92) --
	(193.89, 70.92) --
	(194.34, 70.92) --
	(194.78, 70.92) --
	(195.22, 70.92) --
	(195.66, 70.92) --
	(196.11, 70.92) --
	(196.55, 70.92) --
	(196.99, 70.92) --
	(197.43, 70.92) --
	(197.87, 70.92) --
	(198.32, 70.92) --
	(198.76, 70.92) --
	(199.20, 70.92) --
	(199.64, 70.92) --
	(200.09, 70.92) --
	(200.53, 70.92) --
	(200.97, 70.92) --
	(201.41, 70.92) --
	(201.85, 70.92) --
	(202.30, 70.92) --
	(202.74, 70.92) --
	(203.18, 70.92) --
	(203.62, 70.92) --
	(204.06, 70.92) --
	(204.51, 70.92) --
	(204.95, 70.92) --
	(205.39, 70.92) --
	(205.83, 70.92) --
	(206.28, 70.92) --
	(206.72, 70.92) --
	(207.16, 70.92) --
	(207.60, 70.92) --
	(208.04, 70.92) --
	(208.49, 70.92) --
	(208.93, 70.92) --
	(209.37, 70.92) --
	(209.81, 70.92) --
	(210.26, 70.92) --
	(210.70, 70.92) --
	(211.14, 70.92) --
	(211.58, 70.92) --
	(212.02, 70.92) --
	(212.47, 70.92) --
	(212.91, 70.92) --
	(213.35, 70.92) --
	(213.79, 70.92) --
	(214.24, 70.92) --
	(214.68, 70.92) --
	(215.12, 70.92) --
	(215.56, 70.92) --
	(216.00, 70.92) --
	(216.45, 70.92) --
	(216.89, 70.92) --
	(217.33, 70.92) --
	(217.77, 70.92) --
	(218.22, 70.92) --
	(218.66, 70.92) --
	(219.10, 70.92) --
	(219.54, 70.92) --
	(219.98, 70.92) --
	(220.43, 70.92) --
	(220.87, 70.92) --
	(221.31, 70.92) --
	(221.75, 70.92) --
	(222.19, 70.92) --
	(222.64, 70.92) --
	(223.08, 70.92) --
	(223.52, 70.92) --
	(223.96, 70.92) --
	(224.41, 70.92) --
	(224.85, 70.92) --
	(225.29, 70.92) --
	(225.73, 70.92) --
	(226.17, 70.92) --
	(226.62, 70.92) --
	(227.06, 70.92) --
	(227.50, 70.92) --
	(227.94, 70.92) --
	(228.39, 70.92) --
	(228.83, 70.92) --
	(229.27, 70.92) --
	(229.71, 70.92) --
	(230.15, 70.92) --
	(230.60, 70.92);
\definecolor[named]{drawColor}{rgb}{0.00,0.00,0.00}
\definecolor[named]{fillColor}{rgb}{0.97,0.46,0.43}

\path[draw=drawColor,line width= 0.4pt,line join=round,line cap=round,fill=fillColor] (184.61, 44.91) circle (  2.13);
\definecolor[named]{fillColor}{rgb}{0.00,0.75,0.77}

\path[draw=drawColor,line width= 0.4pt,line join=round,line cap=round,fill=fillColor] (185.49, 38.94) circle (  2.13);
\end{scope}
\begin{scope}
\path[clip] (238.31, 35.17) rectangle (335.10,118.08);
\definecolor[named]{fillColor}{rgb}{0.90,0.90,0.90}

\path[fill=fillColor] (238.31, 35.17) rectangle (335.10,118.08);
\definecolor[named]{drawColor}{rgb}{0.95,0.95,0.95}

\path[draw=drawColor,line width= 0.3pt,line join=round] (238.31, 52.15) --
	(335.10, 52.15);

\path[draw=drawColor,line width= 0.3pt,line join=round] (238.31, 78.96) --
	(335.10, 78.96);

\path[draw=drawColor,line width= 0.3pt,line join=round] (238.31,105.77) --
	(335.10,105.77);

\path[draw=drawColor,line width= 0.3pt,line join=round] (250.04, 35.17) --
	(250.04,118.08);

\path[draw=drawColor,line width= 0.3pt,line join=round] (274.48, 35.17) --
	(274.48,118.08);

\path[draw=drawColor,line width= 0.3pt,line join=round] (298.93, 35.17) --
	(298.93,118.08);

\path[draw=drawColor,line width= 0.3pt,line join=round] (323.37, 35.17) --
	(323.37,118.08);
\definecolor[named]{drawColor}{rgb}{1.00,1.00,1.00}

\path[draw=drawColor,line width= 0.6pt,line join=round] (238.31, 38.75) --
	(335.10, 38.75);

\path[draw=drawColor,line width= 0.6pt,line join=round] (238.31, 65.56) --
	(335.10, 65.56);

\path[draw=drawColor,line width= 0.6pt,line join=round] (238.31, 92.37) --
	(335.10, 92.37);

\path[draw=drawColor,line width= 0.6pt,line join=round] (262.26, 35.17) --
	(262.26,118.08);

\path[draw=drawColor,line width= 0.6pt,line join=round] (286.71, 35.17) --
	(286.71,118.08);

\path[draw=drawColor,line width= 0.6pt,line join=round] (311.15, 35.17) --
	(311.15,118.08);
\definecolor[named]{drawColor}{rgb}{0.00,0.75,0.77}

\path[draw=drawColor,line width= 1.1pt,line join=round] (242.71,108.62) --
	(243.15,108.62) --
	(243.59,108.62) --
	(244.04,108.62) --
	(244.48,108.62) --
	(244.92,108.62) --
	(245.36,108.62) --
	(245.80,108.62) --
	(246.25,108.62) --
	(246.69,108.62) --
	(247.13,108.62) --
	(247.57,108.62) --
	(248.01,108.62) --
	(248.46,108.62) --
	(248.90,108.62) --
	(249.34,108.62) --
	(249.78,108.62) --
	(250.23,108.62) --
	(250.67,108.62) --
	(251.11,108.62) --
	(251.55,108.62) --
	(251.99,108.62) --
	(252.44,108.62) --
	(252.88,108.62) --
	(253.32,108.62) --
	(253.76,108.62) --
	(254.21,108.62) --
	(254.65,108.62) --
	(255.09,108.62) --
	(255.53,108.62) --
	(255.97,108.62) --
	(256.42,108.62) --
	(256.86,108.62) --
	(257.30,108.62) --
	(257.74,108.62) --
	(258.19,108.62) --
	(258.63,108.62) --
	(259.07,108.62) --
	(259.51,108.62) --
	(259.95,108.62) --
	(260.40,108.62) --
	(260.84,108.62) --
	(261.28,108.62) --
	(261.72,108.62) --
	(262.16,108.62) --
	(262.61,108.62) --
	(263.05,108.62) --
	(263.49,108.62) --
	(263.93,108.62) --
	(264.38,108.62) --
	(264.82,108.62) --
	(265.26,108.62) --
	(265.70,108.62) --
	(266.14,108.62) --
	(266.59,108.62) --
	(267.03,108.62) --
	(267.47,108.62) --
	(267.91,108.62) --
	(268.36,108.62) --
	(268.80,108.62) --
	(269.24,108.62) --
	(269.68,108.62) --
	(270.12,108.62) --
	(270.57,108.62) --
	(271.01,108.62) --
	(271.45,108.62) --
	(271.89,103.26) --
	(272.34, 65.75) --
	(272.78, 49.66) --
	(273.22, 38.94) --
	(273.66, 38.94) --
	(274.10, 38.94) --
	(274.55, 38.94) --
	(274.99, 38.94) --
	(275.43, 38.94) --
	(275.87, 38.94) --
	(276.32, 38.94) --
	(276.76, 38.94) --
	(277.20, 38.94) --
	(277.64, 38.94) --
	(278.08, 38.94) --
	(278.53, 38.94) --
	(278.97, 38.94) --
	(279.41, 44.28) --
	(279.85, 54.87) --
	(280.29, 54.87) --
	(280.74, 54.87) --
	(281.18, 65.58) --
	(281.62, 65.58) --
	(282.06, 65.58) --
	(282.51, 65.58) --
	(282.95, 65.58) --
	(283.39, 65.58) --
	(283.83, 70.92) --
	(284.27, 70.92) --
	(284.72, 70.92) --
	(285.16, 70.92) --
	(285.60, 70.92) --
	(286.04, 70.92) --
	(286.49, 70.92) --
	(286.93, 70.92) --
	(287.37, 70.92) --
	(287.81, 70.92) --
	(288.25, 70.92) --
	(288.70, 70.92) --
	(289.14, 70.92) --
	(289.58, 70.92) --
	(290.02, 70.92) --
	(290.47, 70.92) --
	(290.91, 70.92) --
	(291.35, 70.92) --
	(291.79, 70.92) --
	(292.23, 70.92) --
	(292.68, 70.92) --
	(293.12, 70.92) --
	(293.56, 70.92) --
	(294.00, 70.92) --
	(294.45, 70.92) --
	(294.89, 70.92) --
	(295.33, 70.92) --
	(295.77, 70.92) --
	(296.21, 70.92) --
	(296.66, 70.92) --
	(297.10, 70.92) --
	(297.54, 70.92) --
	(297.98, 70.92) --
	(298.42, 70.92) --
	(298.87, 70.92) --
	(299.31, 70.92) --
	(299.75, 70.92) --
	(300.19, 70.92) --
	(300.64, 70.92) --
	(301.08, 70.92) --
	(301.52, 70.92) --
	(301.96, 70.92) --
	(302.40, 70.92) --
	(302.85, 70.92) --
	(303.29, 70.92) --
	(303.73, 70.92) --
	(304.17, 70.92) --
	(304.62, 70.92) --
	(305.06, 70.92) --
	(305.50, 70.92) --
	(305.94, 70.92) --
	(306.38, 70.92) --
	(306.83, 70.92) --
	(307.27, 70.92) --
	(307.71, 70.92) --
	(308.15, 70.92) --
	(308.60, 70.92) --
	(309.04, 70.92) --
	(309.48, 70.92) --
	(309.92, 70.92) --
	(310.36, 70.92) --
	(310.81, 70.92) --
	(311.25, 70.92) --
	(311.69, 70.92) --
	(312.13, 70.92) --
	(312.57, 70.92) --
	(313.02, 70.92) --
	(313.46, 70.92) --
	(313.90, 70.92) --
	(314.34, 70.92) --
	(314.79, 70.92) --
	(315.23, 70.92) --
	(315.67, 70.92) --
	(316.11, 70.92) --
	(316.55, 70.92) --
	(317.00, 70.92) --
	(317.44, 70.92) --
	(317.88, 70.92) --
	(318.32, 70.92) --
	(318.77, 70.92) --
	(319.21, 70.92) --
	(319.65, 70.92) --
	(320.09, 70.92) --
	(320.53, 70.92) --
	(320.98, 70.92) --
	(321.42, 70.92) --
	(321.86, 70.92) --
	(322.30, 70.92) --
	(322.75, 70.92) --
	(323.19, 70.92) --
	(323.63, 70.92) --
	(324.07, 70.92) --
	(324.51, 70.92) --
	(324.96, 70.92) --
	(325.40, 70.92) --
	(325.84, 70.92) --
	(326.28, 70.92) --
	(326.73, 70.92) --
	(327.17, 70.92) --
	(327.61, 70.92) --
	(328.05, 70.92) --
	(328.49, 70.92) --
	(328.94, 70.92) --
	(329.38, 70.92) --
	(329.82, 70.92) --
	(330.26, 70.92) --
	(330.70, 70.92);
\definecolor[named]{drawColor}{rgb}{0.97,0.46,0.43}

\path[draw=drawColor,line width= 0.6pt,line join=round] (242.71,114.32) --
	(243.15,114.32) --
	(243.59,114.32) --
	(244.04,114.32) --
	(244.48,114.32) --
	(244.92,114.32) --
	(245.36,114.32) --
	(245.80,114.32) --
	(246.25,114.32) --
	(246.69,114.32) --
	(247.13,114.32) --
	(247.57,114.32) --
	(248.01,114.32) --
	(248.46,114.32) --
	(248.90,114.32) --
	(249.34,114.32) --
	(249.78,114.32) --
	(250.23,114.32) --
	(250.67,114.32) --
	(251.11,114.32) --
	(251.55,114.32) --
	(251.99,114.32) --
	(252.44,114.32) --
	(252.88,114.32) --
	(253.32,114.32) --
	(253.76,114.32) --
	(254.21,114.32) --
	(254.65,114.32) --
	(255.09,114.32) --
	(255.53,114.32) --
	(255.97,114.32) --
	(256.42,114.32) --
	(256.86,114.32) --
	(257.30,114.32) --
	(257.74,114.32) --
	(258.19,114.32) --
	(258.63,114.32) --
	(259.07,114.32) --
	(259.51,114.32) --
	(259.95,114.32) --
	(260.40,114.32) --
	(260.84,114.32) --
	(261.28,114.32) --
	(261.72,114.32) --
	(262.16,114.32) --
	(262.61,114.32) --
	(263.05,114.32) --
	(263.49,114.32) --
	(263.93,114.32) --
	(264.38,114.32) --
	(264.82,114.32) --
	(265.26,114.32) --
	(265.70,114.32) --
	(266.14,114.32) --
	(266.59,114.32) --
	(267.03,114.32) --
	(267.47,114.32) --
	(267.91,114.32) --
	(268.36,114.32) --
	(268.80,114.32) --
	(269.24,114.32) --
	(269.68,114.32) --
	(270.12,114.32) --
	(270.57,114.32) --
	(271.01,114.32) --
	(271.45,114.32) --
	(271.89,114.32) --
	(272.34,114.32) --
	(272.78,114.32) --
	(273.22,114.32) --
	(273.66,114.32) --
	(274.10,114.32) --
	(274.55,114.32) --
	(274.99,114.32) --
	(275.43,114.32) --
	(275.87,114.32) --
	(276.32,114.32) --
	(276.76,114.32) --
	(277.20,114.32) --
	(277.64, 92.87) --
	(278.08, 76.78) --
	(278.53, 66.06) --
	(278.97, 56.44) --
	(279.41, 44.91) --
	(279.85, 44.91) --
	(280.29, 53.64) --
	(280.74, 53.64) --
	(281.18, 53.64) --
	(281.62, 66.34) --
	(282.06, 66.34) --
	(282.51, 66.34) --
	(282.95, 66.34) --
	(283.39, 70.92) --
	(283.83, 70.92) --
	(284.27, 70.92) --
	(284.72, 70.92) --
	(285.16, 70.92) --
	(285.60, 70.92) --
	(286.04, 70.92) --
	(286.49, 70.92) --
	(286.93, 70.92) --
	(287.37, 70.92) --
	(287.81, 70.92) --
	(288.25, 70.92) --
	(288.70, 70.92) --
	(289.14, 70.92) --
	(289.58, 70.92) --
	(290.02, 70.92) --
	(290.47, 70.92) --
	(290.91, 70.92) --
	(291.35, 70.92) --
	(291.79, 70.92) --
	(292.23, 70.92) --
	(292.68, 70.92) --
	(293.12, 70.92) --
	(293.56, 70.92) --
	(294.00, 70.92) --
	(294.45, 70.92) --
	(294.89, 70.92) --
	(295.33, 70.92) --
	(295.77, 70.92) --
	(296.21, 70.92) --
	(296.66, 70.92) --
	(297.10, 70.92) --
	(297.54, 70.92) --
	(297.98, 70.92) --
	(298.42, 70.92) --
	(298.87, 70.92) --
	(299.31, 70.92) --
	(299.75, 70.92) --
	(300.19, 70.92) --
	(300.64, 70.92) --
	(301.08, 70.92) --
	(301.52, 70.92) --
	(301.96, 70.92) --
	(302.40, 70.92) --
	(302.85, 70.92) --
	(303.29, 70.92) --
	(303.73, 70.92) --
	(304.17, 70.92) --
	(304.62, 70.92) --
	(305.06, 70.92) --
	(305.50, 70.92) --
	(305.94, 70.92) --
	(306.38, 70.92) --
	(306.83, 70.92) --
	(307.27, 70.92) --
	(307.71, 70.92) --
	(308.15, 70.92) --
	(308.60, 70.92) --
	(309.04, 70.92) --
	(309.48, 70.92) --
	(309.92, 70.92) --
	(310.36, 70.92) --
	(310.81, 70.92) --
	(311.25, 70.92) --
	(311.69, 70.92) --
	(312.13, 70.92) --
	(312.57, 70.92) --
	(313.02, 70.92) --
	(313.46, 70.92) --
	(313.90, 70.92) --
	(314.34, 70.92) --
	(314.79, 70.92) --
	(315.23, 70.92) --
	(315.67, 70.92) --
	(316.11, 70.92) --
	(316.55, 70.92) --
	(317.00, 70.92) --
	(317.44, 70.92) --
	(317.88, 70.92) --
	(318.32, 70.92) --
	(318.77, 70.92) --
	(319.21, 70.92) --
	(319.65, 70.92) --
	(320.09, 70.92) --
	(320.53, 70.92) --
	(320.98, 70.92) --
	(321.42, 70.92) --
	(321.86, 70.92) --
	(322.30, 70.92) --
	(322.75, 70.92) --
	(323.19, 70.92) --
	(323.63, 70.92) --
	(324.07, 70.92) --
	(324.51, 70.92) --
	(324.96, 70.92) --
	(325.40, 70.92) --
	(325.84, 70.92) --
	(326.28, 70.92) --
	(326.73, 70.92) --
	(327.17, 70.92) --
	(327.61, 70.92) --
	(328.05, 70.92) --
	(328.49, 70.92) --
	(328.94, 70.92) --
	(329.38, 70.92) --
	(329.82, 70.92) --
	(330.26, 70.92) --
	(330.70, 70.92);
\definecolor[named]{drawColor}{rgb}{0.00,0.00,0.00}
\definecolor[named]{fillColor}{rgb}{0.97,0.46,0.43}

\path[draw=drawColor,line width= 0.4pt,line join=round,line cap=round,fill=fillColor] (279.41, 44.91) circle (  2.13);
\definecolor[named]{fillColor}{rgb}{0.00,0.75,0.77}

\path[draw=drawColor,line width= 0.4pt,line join=round,line cap=round,fill=fillColor] (275.87, 38.94) circle (  2.13);
\end{scope}
\begin{scope}
\path[clip] (  0.00,  0.00) rectangle (433.62,144.54);
\definecolor[named]{drawColor}{rgb}{0.50,0.50,0.50}

\node[text=drawColor,anchor=base east,inner sep=0pt, outer sep=0pt, scale=  0.87] at ( 30.98, 35.46) {0};

\node[text=drawColor,anchor=base east,inner sep=0pt, outer sep=0pt, scale=  0.87] at ( 30.98, 62.27) {5};

\node[text=drawColor,anchor=base east,inner sep=0pt, outer sep=0pt, scale=  0.87] at ( 30.98, 89.08) {10};
\end{scope}
\begin{scope}
\path[clip] (  0.00,  0.00) rectangle (433.62,144.54);
\definecolor[named]{drawColor}{rgb}{0.50,0.50,0.50}

\path[draw=drawColor,line width= 0.6pt,line join=round] ( 33.82, 38.75) --
	( 38.09, 38.75);

\path[draw=drawColor,line width= 0.6pt,line join=round] ( 33.82, 65.56) --
	( 38.09, 65.56);

\path[draw=drawColor,line width= 0.6pt,line join=round] ( 33.82, 92.37) --
	( 38.09, 92.37);
\end{scope}
\begin{scope}
\path[clip] (  0.00,  0.00) rectangle (433.62,144.54);
\definecolor[named]{drawColor}{rgb}{0.50,0.50,0.50}

\path[draw=drawColor,line width= 0.6pt,line join=round] ( 62.05, 30.90) --
	( 62.05, 35.17);

\path[draw=drawColor,line width= 0.6pt,line join=round] ( 86.49, 30.90) --
	( 86.49, 35.17);

\path[draw=drawColor,line width= 0.6pt,line join=round] (110.93, 30.90) --
	(110.93, 35.17);
\end{scope}
\begin{scope}
\path[clip] (  0.00,  0.00) rectangle (433.62,144.54);
\definecolor[named]{drawColor}{rgb}{0.50,0.50,0.50}

\node[text=drawColor,anchor=base,inner sep=0pt, outer sep=0pt, scale=  0.87] at ( 62.05, 21.48) {-5};

\node[text=drawColor,anchor=base,inner sep=0pt, outer sep=0pt, scale=  0.87] at ( 86.49, 21.48) {0};

\node[text=drawColor,anchor=base,inner sep=0pt, outer sep=0pt, scale=  0.87] at (110.93, 21.48) {5};
\end{scope}
\begin{scope}
\path[clip] (  0.00,  0.00) rectangle (433.62,144.54);
\definecolor[named]{drawColor}{rgb}{0.50,0.50,0.50}

\path[draw=drawColor,line width= 0.6pt,line join=round] (162.15, 30.90) --
	(162.15, 35.17);

\path[draw=drawColor,line width= 0.6pt,line join=round] (186.60, 30.90) --
	(186.60, 35.17);

\path[draw=drawColor,line width= 0.6pt,line join=round] (211.04, 30.90) --
	(211.04, 35.17);
\end{scope}
\begin{scope}
\path[clip] (  0.00,  0.00) rectangle (433.62,144.54);
\definecolor[named]{drawColor}{rgb}{0.50,0.50,0.50}

\node[text=drawColor,anchor=base,inner sep=0pt, outer sep=0pt, scale=  0.87] at (162.15, 21.48) {-5};

\node[text=drawColor,anchor=base,inner sep=0pt, outer sep=0pt, scale=  0.87] at (186.60, 21.48) {0};

\node[text=drawColor,anchor=base,inner sep=0pt, outer sep=0pt, scale=  0.87] at (211.04, 21.48) {5};
\end{scope}
\begin{scope}
\path[clip] (  0.00,  0.00) rectangle (433.62,144.54);
\definecolor[named]{drawColor}{rgb}{0.50,0.50,0.50}

\path[draw=drawColor,line width= 0.6pt,line join=round] (262.26, 30.90) --
	(262.26, 35.17);

\path[draw=drawColor,line width= 0.6pt,line join=round] (286.71, 30.90) --
	(286.71, 35.17);

\path[draw=drawColor,line width= 0.6pt,line join=round] (311.15, 30.90) --
	(311.15, 35.17);
\end{scope}
\begin{scope}
\path[clip] (  0.00,  0.00) rectangle (433.62,144.54);
\definecolor[named]{drawColor}{rgb}{0.50,0.50,0.50}

\node[text=drawColor,anchor=base,inner sep=0pt, outer sep=0pt, scale=  0.87] at (262.26, 21.48) {-5};

\node[text=drawColor,anchor=base,inner sep=0pt, outer sep=0pt, scale=  0.87] at (286.71, 21.48) {0};

\node[text=drawColor,anchor=base,inner sep=0pt, outer sep=0pt, scale=  0.87] at (311.15, 21.48) {5};
\end{scope}
\begin{scope}
\path[clip] (  0.00,  0.00) rectangle (433.62,144.54);
\definecolor[named]{drawColor}{rgb}{0.00,0.00,0.00}

\node[text=drawColor,anchor=base,inner sep=0pt, outer sep=0pt, scale=  1.09] at (186.60,  9.94) {model complexity tradeoff parameter $\log_{10}(\lambda)$};
\end{scope}
\begin{scope}
\path[clip] (  0.00,  0.00) rectangle (433.62,144.54);
\definecolor[named]{drawColor}{rgb}{0.00,0.00,0.00}

\node[text=drawColor,rotate= 90.00,anchor=base,inner sep=0pt, outer sep=0pt, scale=  1.09] at ( 18.16, 76.63) {error $E_i^\alpha(\lambda)$};
\end{scope}
\begin{scope}
\path[clip] (  0.00,  0.00) rectangle (433.62,144.54);
\definecolor[named]{fillColor}{rgb}{1.00,1.00,1.00}

\path[fill=fillColor] (344.57, 51.18) rectangle (410.90,102.08);
\end{scope}
\begin{scope}
\path[clip] (  0.00,  0.00) rectangle (433.62,144.54);
\definecolor[named]{drawColor}{rgb}{0.00,0.00,0.00}

\node[text=drawColor,anchor=base west,inner sep=0pt, outer sep=0pt, scale=  0.87] at (348.84, 91.22) {\bfseries bases/probe};
\end{scope}
\begin{scope}
\path[clip] (  0.00,  0.00) rectangle (433.62,144.54);
\definecolor[named]{drawColor}{rgb}{1.00,1.00,1.00}
\definecolor[named]{fillColor}{rgb}{0.95,0.95,0.95}

\path[draw=drawColor,line width= 0.6pt,line join=round,line cap=round,fill=fillColor] (348.84, 71.34) rectangle (364.74, 87.24);
\end{scope}
\begin{scope}
\path[clip] (  0.00,  0.00) rectangle (433.62,144.54);
\definecolor[named]{drawColor}{rgb}{0.00,0.75,0.77}

\path[draw=drawColor,line width= 1.1pt,line join=round] (350.43, 79.29) -- (363.15, 79.29);
\end{scope}
\begin{scope}
\path[clip] (  0.00,  0.00) rectangle (433.62,144.54);
\definecolor[named]{drawColor}{rgb}{0.00,0.00,0.00}
\definecolor[named]{fillColor}{rgb}{0.00,0.75,0.77}

\path[draw=drawColor,line width= 0.4pt,line join=round,line cap=round,fill=fillColor] (356.79, 79.29) circle (  2.13);
\end{scope}
\begin{scope}
\path[clip] (  0.00,  0.00) rectangle (433.62,144.54);
\definecolor[named]{drawColor}{rgb}{1.00,1.00,1.00}
\definecolor[named]{fillColor}{rgb}{0.95,0.95,0.95}

\path[draw=drawColor,line width= 0.6pt,line join=round,line cap=round,fill=fillColor] (348.84, 55.44) rectangle (364.74, 71.34);
\end{scope}
\begin{scope}
\path[clip] (  0.00,  0.00) rectangle (433.62,144.54);
\definecolor[named]{drawColor}{rgb}{0.97,0.46,0.43}

\path[draw=drawColor,line width= 0.6pt,line join=round] (350.43, 63.39) -- (363.15, 63.39);
\end{scope}
\begin{scope}
\path[clip] (  0.00,  0.00) rectangle (433.62,144.54);
\definecolor[named]{drawColor}{rgb}{0.00,0.00,0.00}
\definecolor[named]{fillColor}{rgb}{0.97,0.46,0.43}

\path[draw=drawColor,line width= 0.4pt,line join=round,line cap=round,fill=fillColor] (356.79, 63.39) circle (  2.13);
\end{scope}
\begin{scope}
\path[clip] (  0.00,  0.00) rectangle (433.62,144.54);
\definecolor[named]{drawColor}{rgb}{0.00,0.00,0.00}

\node[text=drawColor,anchor=base west,inner sep=0pt, outer sep=0pt, scale=  0.87] at (366.73, 76.00) {7};
\end{scope}
\begin{scope}
\path[clip] (  0.00,  0.00) rectangle (433.62,144.54);
\definecolor[named]{drawColor}{rgb}{0.00,0.00,0.00}

\node[text=drawColor,anchor=base west,inner sep=0pt, outer sep=0pt, scale=  0.87] at (366.73, 60.10) {374};
\end{scope}
\end{tikzpicture}

%% file: figure-variable-density-error-train.tex
\begin{tikzpicture}[x=1pt,y=1pt]
\definecolor[named]{fillColor}{rgb}{1.00,1.00,1.00}
\path[use as bounding box,fill=fillColor,fill opacity=0.00] (0,0) rectangle (433.62,144.54);
\begin{scope}
\path[clip] (  0.00,  0.00) rectangle (433.62,144.54);
\definecolor[named]{drawColor}{rgb}{1.00,1.00,1.00}
\definecolor[named]{fillColor}{rgb}{1.00,1.00,1.00}

\path[draw=drawColor,line width= 0.6pt,line join=round,line cap=round,fill=fillColor] (  0.00, -0.00) rectangle (433.62,144.54);
\end{scope}
\begin{scope}
\path[clip] ( 42.84,118.08) rectangle (166.48,131.29);
\definecolor[named]{fillColor}{rgb}{0.80,0.80,0.80}

\path[fill=fillColor] ( 42.84,118.08) rectangle (166.48,131.29);
\definecolor[named]{drawColor}{rgb}{0.00,0.00,0.00}

\node[text=drawColor,anchor=base,inner sep=0pt, outer sep=0pt, scale=  0.87] at (104.66,121.40) {$\alpha = 0$};
\end{scope}
\begin{scope}
\path[clip] (169.79,118.08) rectangle (293.42,131.29);
\definecolor[named]{fillColor}{rgb}{0.80,0.80,0.80}

\path[fill=fillColor] (169.79,118.08) rectangle (293.42,131.29);
\definecolor[named]{drawColor}{rgb}{0.00,0.00,0.00}

\node[text=drawColor,anchor=base,inner sep=0pt, outer sep=0pt, scale=  0.87] at (231.61,121.40) {$\alpha = 0.5$};
\end{scope}
\begin{scope}
\path[clip] (296.74,118.08) rectangle (420.37,131.29);
\definecolor[named]{fillColor}{rgb}{0.80,0.80,0.80}

\path[fill=fillColor] (296.74,118.08) rectangle (420.37,131.29);
\definecolor[named]{drawColor}{rgb}{0.00,0.00,0.00}

\node[text=drawColor,anchor=base,inner sep=0pt, outer sep=0pt, scale=  0.87] at (358.55,121.40) {$\alpha = 1$};
\end{scope}
\begin{scope}
\path[clip] ( 42.84, 35.17) rectangle (166.48,118.08);
\definecolor[named]{fillColor}{rgb}{0.90,0.90,0.90}

\path[fill=fillColor] ( 42.84, 35.17) rectangle (166.48,118.08);
\definecolor[named]{drawColor}{rgb}{0.95,0.95,0.95}

\path[draw=drawColor,line width= 0.3pt,line join=round] ( 42.84, 46.47) --
	(166.48, 46.47);

\path[draw=drawColor,line width= 0.3pt,line join=round] ( 42.84, 69.14) --
	(166.48, 69.14);

\path[draw=drawColor,line width= 0.3pt,line join=round] ( 42.84, 91.80) --
	(166.48, 91.80);

\path[draw=drawColor,line width= 0.3pt,line join=round] ( 42.84,114.46) --
	(166.48,114.46);

\path[draw=drawColor,line width= 0.3pt,line join=round] ( 57.83, 35.17) --
	( 57.83,118.08);

\path[draw=drawColor,line width= 0.3pt,line join=round] ( 89.05, 35.17) --
	( 89.05,118.08);

\path[draw=drawColor,line width= 0.3pt,line join=round] (120.27, 35.17) --
	(120.27,118.08);

\path[draw=drawColor,line width= 0.3pt,line join=round] (151.49, 35.17) --
	(151.49,118.08);
\definecolor[named]{drawColor}{rgb}{1.00,1.00,1.00}

\path[draw=drawColor,line width= 0.6pt,line join=round] ( 42.84, 57.80) --
	(166.48, 57.80);

\path[draw=drawColor,line width= 0.6pt,line join=round] ( 42.84, 80.47) --
	(166.48, 80.47);

\path[draw=drawColor,line width= 0.6pt,line join=round] ( 42.84,103.13) --
	(166.48,103.13);

\path[draw=drawColor,line width= 0.6pt,line join=round] ( 73.44, 35.17) --
	( 73.44,118.08);

\path[draw=drawColor,line width= 0.6pt,line join=round] (104.66, 35.17) --
	(104.66,118.08);

\path[draw=drawColor,line width= 0.6pt,line join=round] (135.88, 35.17) --
	(135.88,118.08);
\definecolor[named]{drawColor}{rgb}{0.00,0.00,0.00}

\path[draw=drawColor,line width= 1.2pt,line join=round] ( 48.46,114.32) --
	( 49.03,114.32) --
	( 49.59,114.32) --
	( 50.16,114.32) --
	( 50.72,114.32) --
	( 51.29,114.32) --
	( 51.85,114.32) --
	( 52.42,114.32) --
	( 52.98,114.32) --
	( 53.55,114.32) --
	( 54.11,114.32) --
	( 54.68,114.32) --
	( 55.24,114.32) --
	( 55.81,114.32) --
	( 56.37,114.32) --
	( 56.93,114.32) --
	( 57.50,114.32) --
	( 58.06,114.32) --
	( 58.63,114.32) --
	( 59.19,114.32) --
	( 59.76,114.32) --
	( 60.32,114.32) --
	( 60.89,114.32) --
	( 61.45,114.32) --
	( 62.02,114.32) --
	( 62.58,114.32) --
	( 63.15,114.32) --
	( 63.71,114.32) --
	( 64.28,114.32) --
	( 64.84,114.32) --
	( 65.41,114.32) --
	( 65.97,114.32) --
	( 66.54,114.32) --
	( 67.10,114.32) --
	( 67.67,114.32) --
	( 68.23,114.32) --
	( 68.80,114.32) --
	( 69.36,114.32) --
	( 69.93,114.32) --
	( 70.49,114.32) --
	( 71.05,114.32) --
	( 71.62,114.32) --
	( 72.18,114.32) --
	( 72.75,114.32) --
	( 73.31,114.32) --
	( 73.88,114.32) --
	( 74.44,114.32) --
	( 75.01,114.32) --
	( 75.57,114.32) --
	( 76.14,114.32) --
	( 76.70,114.32) --
	( 77.27,114.32) --
	( 77.83,114.32) --
	( 78.40,114.32) --
	( 78.96,114.32) --
	( 79.53,114.32) --
	( 80.09,114.32) --
	( 80.66,114.32) --
	( 81.22,114.32) --
	( 81.79,114.32) --
	( 82.35,114.32) --
	( 82.92,114.32) --
	( 83.48,114.32) --
	( 84.04,114.32) --
	( 84.61,114.32) --
	( 85.17,114.32) --
	( 85.74,114.32) --
	( 86.30,114.32) --
	( 86.87,114.32) --
	( 87.43,114.32) --
	( 88.00,114.32) --
	( 88.56,114.32) --
	( 89.13,114.32) --
	( 89.69,114.32) --
	( 90.26,114.32) --
	( 90.82,114.32) --
	( 91.39,114.32) --
	( 91.95,114.32) --
	( 92.52,114.32) --
	( 93.08,114.32) --
	( 93.65,114.32) --
	( 94.21,114.32) --
	( 94.78,114.32) --
	( 95.34,114.32) --
	( 95.91,114.32) --
	( 96.47,114.32) --
	( 97.04,114.32) --
	( 97.60,114.32) --
	( 98.16,114.32) --
	( 98.73,114.32) --
	( 99.29,114.32) --
	( 99.86,114.32) --
	(100.42,114.32) --
	(100.99,114.32) --
	(101.55,114.32) --
	(102.12,114.32) --
	(102.68,114.32) --
	(103.25,114.32) --
	(103.81,114.32) --
	(104.38,114.32) --
	(104.94,113.46) --
	(105.51,112.55) --
	(106.07,111.43) --
	(106.64,110.30) --
	(107.20,108.49) --
	(107.77,106.87) --
	(108.33,102.70) --
	(108.90, 99.14) --
	(109.46, 93.05) --
	(110.03, 81.92) --
	(110.59, 67.06) --
	(111.16, 50.07) --
	(111.72, 42.10) --
	(112.28, 41.81) --
	(112.85, 42.87) --
	(113.41, 43.72) --
	(113.98, 45.88) --
	(114.54, 47.65) --
	(115.11, 48.07) --
	(115.67, 48.07) --
	(116.24, 48.80) --
	(116.80, 49.37) --
	(117.37, 50.26) --
	(117.93, 51.40) --
	(118.50, 52.73) --
	(119.06, 53.89) --
	(119.63, 54.98) --
	(120.19, 56.72) --
	(120.76, 57.39) --
	(121.32, 58.29) --
	(121.89, 58.52) --
	(122.45, 59.65) --
	(123.02, 60.54) --
	(123.58, 61.67) --
	(124.15, 62.58) --
	(124.71, 64.16) --
	(125.28, 65.52) --
	(125.84, 66.65) --
	(126.40, 67.55) --
	(126.97, 68.01) --
	(127.53, 68.91) --
	(128.10, 69.14) --
	(128.66, 69.36) --
	(129.23, 69.59) --
	(129.79, 69.82) --
	(130.36, 70.04) --
	(130.92, 70.27) --
	(131.49, 70.49) --
	(132.05, 70.49) --
	(132.62, 70.49) --
	(133.18, 70.49) --
	(133.75, 70.49) --
	(134.31, 70.49) --
	(134.88, 70.49) --
	(135.44, 70.49) --
	(136.01, 70.49) --
	(136.57, 70.49) --
	(137.14, 70.49) --
	(137.70, 70.49) --
	(138.27, 70.49) --
	(138.83, 70.49) --
	(139.40, 70.49) --
	(139.96, 70.49) --
	(140.52, 70.49) --
	(141.09, 70.49) --
	(141.65, 70.49) --
	(142.22, 70.49) --
	(142.78, 70.49) --
	(143.35, 70.49) --
	(143.91, 70.49) --
	(144.48, 70.49) --
	(145.04, 70.49) --
	(145.61, 70.49) --
	(146.17, 70.49) --
	(146.74, 70.49) --
	(147.30, 70.49) --
	(147.87, 70.49) --
	(148.43, 70.49) --
	(149.00, 70.49) --
	(149.56, 70.49) --
	(150.13, 70.49) --
	(150.69, 70.49) --
	(151.26, 70.49) --
	(151.82, 70.49) --
	(152.39, 70.49) --
	(152.95, 70.49) --
	(153.52, 70.49) --
	(154.08, 70.49) --
	(154.64, 70.49) --
	(155.21, 70.49) --
	(155.77, 70.49) --
	(156.34, 70.49) --
	(156.90, 70.49) --
	(157.47, 70.49) --
	(158.03, 70.49) --
	(158.60, 70.49) --
	(159.16, 70.49) --
	(159.73, 70.49) --
	(160.29, 70.49) --
	(160.86, 70.49);
\definecolor[named]{drawColor}{rgb}{1.00,0.00,0.00}
\definecolor[named]{fillColor}{rgb}{1.00,0.00,0.00}

\path[draw=drawColor,line width= 0.6pt,line join=round,fill=fillColor] (112.28, 41.81) -- ( 54.71, 41.81);

\node[text=drawColor,anchor=base west,inner sep=0pt, outer sep=0pt, scale=  1.29] at ( 54.71, 46.67) {29.4};
\end{scope}
\begin{scope}
\path[clip] (169.79, 35.17) rectangle (293.42,118.08);
\definecolor[named]{fillColor}{rgb}{0.90,0.90,0.90}

\path[fill=fillColor] (169.79, 35.17) rectangle (293.42,118.08);
\definecolor[named]{drawColor}{rgb}{0.95,0.95,0.95}

\path[draw=drawColor,line width= 0.3pt,line join=round] (169.79, 46.47) --
	(293.42, 46.47);

\path[draw=drawColor,line width= 0.3pt,line join=round] (169.79, 69.14) --
	(293.42, 69.14);

\path[draw=drawColor,line width= 0.3pt,line join=round] (169.79, 91.80) --
	(293.42, 91.80);

\path[draw=drawColor,line width= 0.3pt,line join=round] (169.79,114.46) --
	(293.42,114.46);

\path[draw=drawColor,line width= 0.3pt,line join=round] (184.78, 35.17) --
	(184.78,118.08);

\path[draw=drawColor,line width= 0.3pt,line join=round] (216.00, 35.17) --
	(216.00,118.08);

\path[draw=drawColor,line width= 0.3pt,line join=round] (247.22, 35.17) --
	(247.22,118.08);

\path[draw=drawColor,line width= 0.3pt,line join=round] (278.44, 35.17) --
	(278.44,118.08);
\definecolor[named]{drawColor}{rgb}{1.00,1.00,1.00}

\path[draw=drawColor,line width= 0.6pt,line join=round] (169.79, 57.80) --
	(293.42, 57.80);

\path[draw=drawColor,line width= 0.6pt,line join=round] (169.79, 80.47) --
	(293.42, 80.47);

\path[draw=drawColor,line width= 0.6pt,line join=round] (169.79,103.13) --
	(293.42,103.13);

\path[draw=drawColor,line width= 0.6pt,line join=round] (200.39, 35.17) --
	(200.39,118.08);

\path[draw=drawColor,line width= 0.6pt,line join=round] (231.61, 35.17) --
	(231.61,118.08);

\path[draw=drawColor,line width= 0.6pt,line join=round] (262.83, 35.17) --
	(262.83,118.08);
\definecolor[named]{drawColor}{rgb}{0.00,0.00,0.00}

\path[draw=drawColor,line width= 1.2pt,line join=round] (175.41,114.32) --
	(175.97,114.32) --
	(176.54,114.32) --
	(177.10,114.32) --
	(177.67,114.32) --
	(178.23,114.32) --
	(178.80,114.32) --
	(179.36,114.32) --
	(179.93,114.32) --
	(180.49,114.32) --
	(181.06,114.32) --
	(181.62,114.32) --
	(182.19,114.32) --
	(182.75,114.32) --
	(183.32,114.32) --
	(183.88,114.32) --
	(184.45,114.32) --
	(185.01,114.32) --
	(185.58,114.32) --
	(186.14,114.32) --
	(186.71,114.32) --
	(187.27,114.32) --
	(187.83,114.32) --
	(188.40,114.32) --
	(188.96,114.32) --
	(189.53,114.32) --
	(190.09,114.32) --
	(190.66,114.32) --
	(191.22,114.32) --
	(191.79,114.32) --
	(192.35,114.32) --
	(192.92,114.32) --
	(193.48,114.32) --
	(194.05,114.32) --
	(194.61,114.32) --
	(195.18,114.32) --
	(195.74,114.32) --
	(196.31,114.32) --
	(196.87,114.32) --
	(197.44,114.32) --
	(198.00,114.32) --
	(198.57,114.32) --
	(199.13,114.32) --
	(199.70,114.32) --
	(200.26,114.32) --
	(200.83,114.32) --
	(201.39,114.32) --
	(201.95,114.32) --
	(202.52,114.32) --
	(203.08,114.32) --
	(203.65,114.32) --
	(204.21,114.32) --
	(204.78,114.32) --
	(205.34,114.32) --
	(205.91,114.32) --
	(206.47,114.32) --
	(207.04,114.32) --
	(207.60,114.32) --
	(208.17,114.32) --
	(208.73,114.32) --
	(209.30,114.32) --
	(209.86,114.32) --
	(210.43,114.32) --
	(210.99,114.32) --
	(211.56,114.32) --
	(212.12,114.32) --
	(212.69,114.32) --
	(213.25,114.32) --
	(213.82,114.32) --
	(214.38,114.32) --
	(214.95,114.32) --
	(215.51,114.32) --
	(216.07,114.32) --
	(216.64,114.32) --
	(217.20,114.32) --
	(217.77,114.32) --
	(218.33,114.32) --
	(218.90,114.32) --
	(219.46,114.32) --
	(220.03,114.32) --
	(220.59,114.32) --
	(221.16,114.32) --
	(221.72,114.32) --
	(222.29,114.09) --
	(222.85,113.18) --
	(223.42,108.88) --
	(223.98, 99.13) --
	(224.55, 93.92) --
	(225.11, 91.42) --
	(225.68, 87.17) --
	(226.24, 76.52) --
	(226.81, 69.73) --
	(227.37, 58.60) --
	(227.94, 50.77) --
	(228.50, 42.09) --
	(229.07, 39.61) --
	(229.63, 38.94) --
	(230.19, 39.09) --
	(230.76, 40.29) --
	(231.32, 41.13) --
	(231.89, 43.56) --
	(232.45, 45.57) --
	(233.02, 47.29) --
	(233.58, 49.22) --
	(234.15, 51.28) --
	(234.71, 53.27) --
	(235.28, 55.05) --
	(235.84, 56.80) --
	(236.41, 58.80) --
	(236.97, 61.46) --
	(237.54, 62.82) --
	(238.10, 65.53) --
	(238.67, 66.21) --
	(239.23, 67.56) --
	(239.80, 68.46) --
	(240.36, 68.91) --
	(240.93, 68.91) --
	(241.49, 69.36) --
	(242.06, 69.82) --
	(242.62, 70.04) --
	(243.19, 70.49) --
	(243.75, 70.49) --
	(244.31, 70.49) --
	(244.88, 70.49) --
	(245.44, 70.49) --
	(246.01, 70.49) --
	(246.57, 70.49) --
	(247.14, 70.49) --
	(247.70, 70.49) --
	(248.27, 70.49) --
	(248.83, 70.49) --
	(249.40, 70.49) --
	(249.96, 70.49) --
	(250.53, 70.49) --
	(251.09, 70.49) --
	(251.66, 70.49) --
	(252.22, 70.49) --
	(252.79, 70.49) --
	(253.35, 70.49) --
	(253.92, 70.49) --
	(254.48, 70.49) --
	(255.05, 70.49) --
	(255.61, 70.49) --
	(256.18, 70.49) --
	(256.74, 70.49) --
	(257.31, 70.49) --
	(257.87, 70.49) --
	(258.43, 70.49) --
	(259.00, 70.49) --
	(259.56, 70.49) --
	(260.13, 70.49) --
	(260.69, 70.49) --
	(261.26, 70.49) --
	(261.82, 70.49) --
	(262.39, 70.49) --
	(262.95, 70.49) --
	(263.52, 70.49) --
	(264.08, 70.49) --
	(264.65, 70.49) --
	(265.21, 70.49) --
	(265.78, 70.49) --
	(266.34, 70.49) --
	(266.91, 70.49) --
	(267.47, 70.49) --
	(268.04, 70.49) --
	(268.60, 70.49) --
	(269.17, 70.49) --
	(269.73, 70.49) --
	(270.30, 70.49) --
	(270.86, 70.49) --
	(271.42, 70.49) --
	(271.99, 70.49) --
	(272.55, 70.49) --
	(273.12, 70.49) --
	(273.68, 70.49) --
	(274.25, 70.49) --
	(274.81, 70.49) --
	(275.38, 70.49) --
	(275.94, 70.49) --
	(276.51, 70.49) --
	(277.07, 70.49) --
	(277.64, 70.49) --
	(278.20, 70.49) --
	(278.77, 70.49) --
	(279.33, 70.49) --
	(279.90, 70.49) --
	(280.46, 70.49) --
	(281.03, 70.49) --
	(281.59, 70.49) --
	(282.16, 70.49) --
	(282.72, 70.49) --
	(283.29, 70.49) --
	(283.85, 70.49) --
	(284.42, 70.49) --
	(284.98, 70.49) --
	(285.54, 70.49) --
	(286.11, 70.49) --
	(286.67, 70.49) --
	(287.24, 70.49) --
	(287.80, 70.49);
\definecolor[named]{drawColor}{rgb}{1.00,0.00,0.00}
\definecolor[named]{fillColor}{rgb}{1.00,0.00,0.00}

\path[draw=drawColor,line width= 0.6pt,line join=round,fill=fillColor] (229.63, 38.94) -- (181.65, 38.94);

\node[text=drawColor,anchor=base west,inner sep=0pt, outer sep=0pt, scale=  1.29] at (181.65, 43.80) {16.8};
\end{scope}
\begin{scope}
\path[clip] (296.74, 35.17) rectangle (420.37,118.08);
\definecolor[named]{fillColor}{rgb}{0.90,0.90,0.90}

\path[fill=fillColor] (296.74, 35.17) rectangle (420.37,118.08);
\definecolor[named]{drawColor}{rgb}{0.95,0.95,0.95}

\path[draw=drawColor,line width= 0.3pt,line join=round] (296.74, 46.47) --
	(420.37, 46.47);

\path[draw=drawColor,line width= 0.3pt,line join=round] (296.74, 69.14) --
	(420.37, 69.14);

\path[draw=drawColor,line width= 0.3pt,line join=round] (296.74, 91.80) --
	(420.37, 91.80);

\path[draw=drawColor,line width= 0.3pt,line join=round] (296.74,114.46) --
	(420.37,114.46);

\path[draw=drawColor,line width= 0.3pt,line join=round] (311.72, 35.17) --
	(311.72,118.08);

\path[draw=drawColor,line width= 0.3pt,line join=round] (342.94, 35.17) --
	(342.94,118.08);

\path[draw=drawColor,line width= 0.3pt,line join=round] (374.16, 35.17) --
	(374.16,118.08);

\path[draw=drawColor,line width= 0.3pt,line join=round] (405.38, 35.17) --
	(405.38,118.08);
\definecolor[named]{drawColor}{rgb}{1.00,1.00,1.00}

\path[draw=drawColor,line width= 0.6pt,line join=round] (296.74, 57.80) --
	(420.37, 57.80);

\path[draw=drawColor,line width= 0.6pt,line join=round] (296.74, 80.47) --
	(420.37, 80.47);

\path[draw=drawColor,line width= 0.6pt,line join=round] (296.74,103.13) --
	(420.37,103.13);

\path[draw=drawColor,line width= 0.6pt,line join=round] (327.33, 35.17) --
	(327.33,118.08);

\path[draw=drawColor,line width= 0.6pt,line join=round] (358.55, 35.17) --
	(358.55,118.08);

\path[draw=drawColor,line width= 0.6pt,line join=round] (389.77, 35.17) --
	(389.77,118.08);
\definecolor[named]{drawColor}{rgb}{0.00,0.00,0.00}

\path[draw=drawColor,line width= 1.2pt,line join=round] (302.36,114.32) --
	(302.92,114.32) --
	(303.49,114.32) --
	(304.05,114.32) --
	(304.62,114.32) --
	(305.18,114.32) --
	(305.74,114.32) --
	(306.31,114.32) --
	(306.87,114.32) --
	(307.44,114.32) --
	(308.00,114.32) --
	(308.57,114.32) --
	(309.13,114.32) --
	(309.70,114.32) --
	(310.26,114.32) --
	(310.83,114.32) --
	(311.39,114.32) --
	(311.96,114.32) --
	(312.52,114.32) --
	(313.09,114.32) --
	(313.65,114.32) --
	(314.22,114.32) --
	(314.78,114.32) --
	(315.35,114.32) --
	(315.91,114.32) --
	(316.48,114.32) --
	(317.04,114.32) --
	(317.61,114.32) --
	(318.17,114.32) --
	(318.74,114.32) --
	(319.30,114.32) --
	(319.86,114.32) --
	(320.43,114.32) --
	(320.99,114.32) --
	(321.56,114.32) --
	(322.12,114.32) --
	(322.69,114.32) --
	(323.25,114.32) --
	(323.82,114.32) --
	(324.38,114.32) --
	(324.95,114.32) --
	(325.51,114.32) --
	(326.08,114.32) --
	(326.64,114.32) --
	(327.21,114.32) --
	(327.77,114.32) --
	(328.34,114.32) --
	(328.90,114.32) --
	(329.47,114.32) --
	(330.03,114.32) --
	(330.60,114.32) --
	(331.16,114.32) --
	(331.73,114.32) --
	(332.29,114.32) --
	(332.86,114.32) --
	(333.42,114.32) --
	(333.98,114.32) --
	(334.55,113.86) --
	(335.11,111.60) --
	(335.68,107.97) --
	(336.24,102.98) --
	(336.81,101.17) --
	(337.37, 97.09) --
	(337.94, 94.60) --
	(338.50, 93.69) --
	(339.07, 93.46) --
	(339.63, 91.42) --
	(340.20, 88.93) --
	(340.76, 87.80) --
	(341.33, 83.72) --
	(341.89, 79.41) --
	(342.46, 74.19) --
	(343.02, 73.29) --
	(343.59, 71.70) --
	(344.15, 68.01) --
	(344.72, 66.42) --
	(345.28, 65.51) --
	(345.85, 63.25) --
	(346.41, 56.74) --
	(346.98, 50.55) --
	(347.54, 45.52) --
	(348.10, 43.48) --
	(348.67, 41.93) --
	(349.23, 42.27) --
	(349.80, 50.27) --
	(350.36, 52.53) --
	(350.93, 54.39) --
	(351.49, 61.86) --
	(352.06, 64.08) --
	(352.62, 64.77) --
	(353.19, 64.77) --
	(353.75, 65.21) --
	(354.32, 65.69) --
	(354.88, 69.49) --
	(355.45, 70.32) --
	(356.01, 70.49) --
	(356.58, 70.49) --
	(357.14, 70.49) --
	(357.71, 70.49) --
	(358.27, 70.49) --
	(358.84, 70.49) --
	(359.40, 70.49) --
	(359.97, 70.49) --
	(360.53, 70.49) --
	(361.09, 70.49) --
	(361.66, 70.49) --
	(362.22, 70.49) --
	(362.79, 70.49) --
	(363.35, 70.49) --
	(363.92, 70.49) --
	(364.48, 70.49) --
	(365.05, 70.49) --
	(365.61, 70.49) --
	(366.18, 70.49) --
	(366.74, 70.49) --
	(367.31, 70.49) --
	(367.87, 70.49) --
	(368.44, 70.49) --
	(369.00, 70.49) --
	(369.57, 70.49) --
	(370.13, 70.49) --
	(370.70, 70.49) --
	(371.26, 70.49) --
	(371.83, 70.49) --
	(372.39, 70.49) --
	(372.96, 70.49) --
	(373.52, 70.49) --
	(374.09, 70.49) --
	(374.65, 70.49) --
	(375.21, 70.49) --
	(375.78, 70.49) --
	(376.34, 70.49) --
	(376.91, 70.49) --
	(377.47, 70.49) --
	(378.04, 70.49) --
	(378.60, 70.49) --
	(379.17, 70.49) --
	(379.73, 70.49) --
	(380.30, 70.49) --
	(380.86, 70.49) --
	(381.43, 70.49) --
	(381.99, 70.49) --
	(382.56, 70.49) --
	(383.12, 70.49) --
	(383.69, 70.49) --
	(384.25, 70.49) --
	(384.82, 70.49) --
	(385.38, 70.49) --
	(385.95, 70.49) --
	(386.51, 70.49) --
	(387.08, 70.49) --
	(387.64, 70.49) --
	(388.21, 70.49) --
	(388.77, 70.49) --
	(389.33, 70.49) --
	(389.90, 70.49) --
	(390.46, 70.49) --
	(391.03, 70.49) --
	(391.59, 70.49) --
	(392.16, 70.49) --
	(392.72, 70.49) --
	(393.29, 70.49) --
	(393.85, 70.49) --
	(394.42, 70.49) --
	(394.98, 70.49) --
	(395.55, 70.49) --
	(396.11, 70.49) --
	(396.68, 70.49) --
	(397.24, 70.49) --
	(397.81, 70.49) --
	(398.37, 70.49) --
	(398.94, 70.49) --
	(399.50, 70.49) --
	(400.07, 70.49) --
	(400.63, 70.49) --
	(401.20, 70.49) --
	(401.76, 70.49) --
	(402.33, 70.49) --
	(402.89, 70.49) --
	(403.45, 70.49) --
	(404.02, 70.49) --
	(404.58, 70.49) --
	(405.15, 70.49) --
	(405.71, 70.49) --
	(406.28, 70.49) --
	(406.84, 70.49) --
	(407.41, 70.49) --
	(407.97, 70.49) --
	(408.54, 70.49) --
	(409.10, 70.49) --
	(409.67, 70.49) --
	(410.23, 70.49) --
	(410.80, 70.49) --
	(411.36, 70.49) --
	(411.93, 70.49) --
	(412.49, 70.49) --
	(413.06, 70.49) --
	(413.62, 70.49) --
	(414.19, 70.49) --
	(414.75, 70.49);
\definecolor[named]{drawColor}{rgb}{1.00,0.00,0.00}
\definecolor[named]{fillColor}{rgb}{1.00,0.00,0.00}

\path[draw=drawColor,line width= 0.6pt,line join=round,fill=fillColor] (348.67, 41.93) -- (408.51, 41.93);

\node[text=drawColor,anchor=base east,inner sep=0pt, outer sep=0pt, scale=  1.29] at (408.51, 46.79) {30};
\end{scope}
\begin{scope}
\path[clip] (  0.00,  0.00) rectangle (433.62,144.54);
\definecolor[named]{drawColor}{rgb}{0.50,0.50,0.50}

\node[text=drawColor,anchor=base east,inner sep=0pt, outer sep=0pt, scale=  0.87] at ( 35.73, 54.51) {100};

\node[text=drawColor,anchor=base east,inner sep=0pt, outer sep=0pt, scale=  0.87] at ( 35.73, 77.18) {200};

\node[text=drawColor,anchor=base east,inner sep=0pt, outer sep=0pt, scale=  0.87] at ( 35.73, 99.84) {300};
\end{scope}
\begin{scope}
\path[clip] (  0.00,  0.00) rectangle (433.62,144.54);
\definecolor[named]{drawColor}{rgb}{0.50,0.50,0.50}

\path[draw=drawColor,line width= 0.6pt,line join=round] ( 38.58, 57.80) --
	( 42.84, 57.80);

\path[draw=drawColor,line width= 0.6pt,line join=round] ( 38.58, 80.47) --
	( 42.84, 80.47);

\path[draw=drawColor,line width= 0.6pt,line join=round] ( 38.58,103.13) --
	( 42.84,103.13);
\end{scope}
\begin{scope}
\path[clip] (  0.00,  0.00) rectangle (433.62,144.54);
\definecolor[named]{drawColor}{rgb}{0.50,0.50,0.50}

\path[draw=drawColor,line width= 0.6pt,line join=round] ( 73.44, 30.90) --
	( 73.44, 35.17);

\path[draw=drawColor,line width= 0.6pt,line join=round] (104.66, 30.90) --
	(104.66, 35.17);

\path[draw=drawColor,line width= 0.6pt,line join=round] (135.88, 30.90) --
	(135.88, 35.17);
\end{scope}
\begin{scope}
\path[clip] (  0.00,  0.00) rectangle (433.62,144.54);
\definecolor[named]{drawColor}{rgb}{0.50,0.50,0.50}

\node[text=drawColor,anchor=base,inner sep=0pt, outer sep=0pt, scale=  0.87] at ( 73.44, 21.48) {-5};

\node[text=drawColor,anchor=base,inner sep=0pt, outer sep=0pt, scale=  0.87] at (104.66, 21.48) {0};

\node[text=drawColor,anchor=base,inner sep=0pt, outer sep=0pt, scale=  0.87] at (135.88, 21.48) {5};
\end{scope}
\begin{scope}
\path[clip] (  0.00,  0.00) rectangle (433.62,144.54);
\definecolor[named]{drawColor}{rgb}{0.50,0.50,0.50}

\path[draw=drawColor,line width= 0.6pt,line join=round] (200.39, 30.90) --
	(200.39, 35.17);

\path[draw=drawColor,line width= 0.6pt,line join=round] (231.61, 30.90) --
	(231.61, 35.17);

\path[draw=drawColor,line width= 0.6pt,line join=round] (262.83, 30.90) --
	(262.83, 35.17);
\end{scope}
\begin{scope}
\path[clip] (  0.00,  0.00) rectangle (433.62,144.54);
\definecolor[named]{drawColor}{rgb}{0.50,0.50,0.50}

\node[text=drawColor,anchor=base,inner sep=0pt, outer sep=0pt, scale=  0.87] at (200.39, 21.48) {-5};

\node[text=drawColor,anchor=base,inner sep=0pt, outer sep=0pt, scale=  0.87] at (231.61, 21.48) {0};

\node[text=drawColor,anchor=base,inner sep=0pt, outer sep=0pt, scale=  0.87] at (262.83, 21.48) {5};
\end{scope}
\begin{scope}
\path[clip] (  0.00,  0.00) rectangle (433.62,144.54);
\definecolor[named]{drawColor}{rgb}{0.50,0.50,0.50}

\path[draw=drawColor,line width= 0.6pt,line join=round] (327.33, 30.90) --
	(327.33, 35.17);

\path[draw=drawColor,line width= 0.6pt,line join=round] (358.55, 30.90) --
	(358.55, 35.17);

\path[draw=drawColor,line width= 0.6pt,line join=round] (389.77, 30.90) --
	(389.77, 35.17);
\end{scope}
\begin{scope}
\path[clip] (  0.00,  0.00) rectangle (433.62,144.54);
\definecolor[named]{drawColor}{rgb}{0.50,0.50,0.50}

\node[text=drawColor,anchor=base,inner sep=0pt, outer sep=0pt, scale=  0.87] at (327.33, 21.48) {-5};

\node[text=drawColor,anchor=base,inner sep=0pt, outer sep=0pt, scale=  0.87] at (358.55, 21.48) {0};

\node[text=drawColor,anchor=base,inner sep=0pt, outer sep=0pt, scale=  0.87] at (389.77, 21.48) {5};
\end{scope}
\begin{scope}
\path[clip] (  0.00,  0.00) rectangle (433.62,144.54);
\definecolor[named]{drawColor}{rgb}{0.00,0.00,0.00}

\node[text=drawColor,anchor=base,inner sep=0pt, outer sep=0pt, scale=  1.09] at (231.61,  9.94) {model complexity tradeoff parameter $\log_{10}(\lambda)$};
\end{scope}
\begin{scope}
\path[clip] (  0.00,  0.00) rectangle (433.62,144.54);
\definecolor[named]{drawColor}{rgb}{0.00,0.00,0.00}

\node[text=drawColor,rotate= 90.00,anchor=base,inner sep=0pt, outer sep=0pt, scale=  1.09] at ( 18.16, 76.63) {error $E^\alpha(\lambda)$};
\end{scope}
\end{tikzpicture}

%% file: figure-variable-density-error-alpha.tex
\begin{tikzpicture}[x=1pt,y=1pt]
\definecolor[named]{fillColor}{rgb}{1.00,1.00,1.00}
\path[use as bounding box,fill=fillColor,fill opacity=0.00] (0,0) rectangle (433.62,252.94);
\begin{scope}
\path[clip] (  0.00,  0.00) rectangle (433.62,252.94);
\definecolor[named]{drawColor}{rgb}{1.00,1.00,1.00}
\definecolor[named]{fillColor}{rgb}{1.00,1.00,1.00}

\path[draw=drawColor,line width= 0.6pt,line join=round,line cap=round,fill=fillColor] (  0.00,  0.00) rectangle (433.62,252.94);
\end{scope}
\begin{scope}
\path[clip] ( 42.84,139.09) rectangle (407.16,239.70);
\definecolor[named]{fillColor}{rgb}{0.90,0.90,0.90}

\path[fill=fillColor] ( 42.84,139.09) rectangle (407.16,239.70);
\definecolor[named]{drawColor}{rgb}{0.95,0.95,0.95}

\path[draw=drawColor,line width= 0.3pt,line join=round] ( 42.84,150.85) --
	(407.16,150.85);

\path[draw=drawColor,line width= 0.3pt,line join=round] ( 42.84,185.60) --
	(407.16,185.60);

\path[draw=drawColor,line width= 0.3pt,line join=round] ( 42.84,220.35) --
	(407.16,220.35);

\path[draw=drawColor,line width= 0.3pt,line join=round] (114.60,139.09) --
	(114.60,239.70);

\path[draw=drawColor,line width= 0.3pt,line join=round] (225.00,139.09) --
	(225.00,239.70);

\path[draw=drawColor,line width= 0.3pt,line join=round] (335.40,139.09) --
	(335.40,239.70);
\definecolor[named]{drawColor}{rgb}{1.00,1.00,1.00}

\path[draw=drawColor,line width= 0.6pt,line join=round] ( 42.84,168.23) --
	(407.16,168.23);

\path[draw=drawColor,line width= 0.6pt,line join=round] ( 42.84,202.97) --
	(407.16,202.97);

\path[draw=drawColor,line width= 0.6pt,line join=round] ( 42.84,237.72) --
	(407.16,237.72);

\path[draw=drawColor,line width= 0.6pt,line join=round] ( 59.40,139.09) --
	( 59.40,239.70);

\path[draw=drawColor,line width= 0.6pt,line join=round] (169.80,139.09) --
	(169.80,239.70);

\path[draw=drawColor,line width= 0.6pt,line join=round] (280.20,139.09) --
	(280.20,239.70);

\path[draw=drawColor,line width= 0.6pt,line join=round] (390.60,139.09) --
	(390.60,239.70);
\definecolor[named]{drawColor}{rgb}{0.00,0.00,0.00}

\path[draw=drawColor,line width= 0.6pt,line join=round] ( 59.40,194.92) --
	(114.60,179.37) --
	(169.80,153.94) --
	(180.84,147.95) --
	(191.88,144.95) --
	(202.92,144.10) --
	(213.96,143.66) --
	(219.48,144.35) --
	(225.00,145.13) --
	(230.52,145.13) --
	(236.04,145.13) --
	(247.08,145.13) --
	(258.12,145.14) --
	(269.16,147.90) --
	(280.20,154.30) --
	(291.24,165.30) --
	(302.28,182.50) --
	(313.32,201.00) --
	(324.36,212.96) --
	(335.40,220.55) --
	(390.60,235.12);
\end{scope}
\begin{scope}
\path[clip] ( 42.84, 35.17) rectangle (407.16,135.78);
\definecolor[named]{fillColor}{rgb}{0.90,0.90,0.90}

\path[fill=fillColor] ( 42.84, 35.17) rectangle (407.16,135.78);
\definecolor[named]{drawColor}{rgb}{0.95,0.95,0.95}

\path[draw=drawColor,line width= 0.3pt,line join=round] ( 42.84, 57.15) --
	(407.16, 57.15);

\path[draw=drawColor,line width= 0.3pt,line join=round] ( 42.84, 91.96) --
	(407.16, 91.96);

\path[draw=drawColor,line width= 0.3pt,line join=round] ( 42.84,126.77) --
	(407.16,126.77);

\path[draw=drawColor,line width= 0.3pt,line join=round] (114.60, 35.17) --
	(114.60,135.78);

\path[draw=drawColor,line width= 0.3pt,line join=round] (225.00, 35.17) --
	(225.00,135.78);

\path[draw=drawColor,line width= 0.3pt,line join=round] (335.40, 35.17) --
	(335.40,135.78);
\definecolor[named]{drawColor}{rgb}{1.00,1.00,1.00}

\path[draw=drawColor,line width= 0.6pt,line join=round] ( 42.84, 39.74) --
	(407.16, 39.74);

\path[draw=drawColor,line width= 0.6pt,line join=round] ( 42.84, 74.55) --
	(407.16, 74.55);

\path[draw=drawColor,line width= 0.6pt,line join=round] ( 42.84,109.36) --
	(407.16,109.36);

\path[draw=drawColor,line width= 0.6pt,line join=round] ( 59.40, 35.17) --
	( 59.40,135.78);

\path[draw=drawColor,line width= 0.6pt,line join=round] (169.80, 35.17) --
	(169.80,135.78);

\path[draw=drawColor,line width= 0.6pt,line join=round] (280.20, 35.17) --
	(280.20,135.78);

\path[draw=drawColor,line width= 0.6pt,line join=round] (390.60, 35.17) --
	(390.60,135.78);
\definecolor[named]{fillColor}{rgb}{0.20,0.20,0.20}

\path[fill=fillColor,fill opacity=0.50] ( 59.40,125.89) --
	(114.60,119.45) --
	(169.80, 95.29) --
	(180.84, 80.63) --
	(191.88, 66.63) --
	(202.92, 61.28) --
	(213.96, 56.13) --
	(219.48, 53.88) --
	(225.00, 53.95) --
	(230.52, 54.45) --
	(236.04, 56.14) --
	(247.08, 72.65) --
	(258.12, 92.03) --
	(269.16,102.93) --
	(280.20,109.21) --
	(291.24,114.94) --
	(302.28,118.98) --
	(313.32,122.15) --
	(324.36,124.96) --
	(335.40,126.73) --
	(390.60,131.20) --
	(390.60, 56.52) --
	(335.40, 46.44) --
	(324.36, 44.20) --
	(313.32, 40.89) --
	(302.28, 39.74) --
	(291.24, 39.74) --
	(280.20, 39.74) --
	(269.16, 39.74) --
	(258.12, 39.74) --
	(247.08, 39.74) --
	(236.04, 39.74) --
	(230.52, 39.74) --
	(225.00, 39.74) --
	(219.48, 39.74) --
	(213.96, 39.74) --
	(202.92, 39.74) --
	(191.88, 39.74) --
	(180.84, 39.74) --
	(169.80, 39.74) --
	(114.60, 49.69) --
	( 59.40, 58.26) --
	cycle;
\definecolor[named]{drawColor}{rgb}{0.00,0.00,0.00}

\path[draw=drawColor,line width= 0.6pt,line join=round] ( 59.40, 92.07) --
	(114.60, 84.57) --
	(169.80, 65.98) --
	(180.84, 57.34) --
	(191.88, 51.20) --
	(202.92, 48.38) --
	(213.96, 46.22) --
	(219.48, 45.07) --
	(225.00, 44.87) --
	(230.52, 44.83) --
	(236.04, 45.42) --
	(247.08, 51.29) --
	(258.12, 58.97) --
	(269.16, 64.84) --
	(280.20, 69.11) --
	(291.24, 74.01) --
	(302.28, 77.93) --
	(313.32, 81.52) --
	(324.36, 84.58) --
	(335.40, 86.59) --
	(390.60, 93.86);
\end{scope}
\begin{scope}
\path[clip] (  0.00,  0.00) rectangle (433.62,252.94);
\definecolor[named]{drawColor}{rgb}{0.50,0.50,0.50}

\node[text=drawColor,anchor=base east,inner sep=0pt, outer sep=0pt, scale=  0.87] at ( 35.73,164.94) {50};

\node[text=drawColor,anchor=base east,inner sep=0pt, outer sep=0pt, scale=  0.87] at ( 35.73,199.68) {100};

\node[text=drawColor,anchor=base east,inner sep=0pt, outer sep=0pt, scale=  0.87] at ( 35.73,234.43) {150};
\end{scope}
\begin{scope}
\path[clip] (  0.00,  0.00) rectangle (433.62,252.94);
\definecolor[named]{drawColor}{rgb}{0.50,0.50,0.50}

\path[draw=drawColor,line width= 0.6pt,line join=round] ( 38.58,168.23) --
	( 42.84,168.23);

\path[draw=drawColor,line width= 0.6pt,line join=round] ( 38.58,202.97) --
	( 42.84,202.97);

\path[draw=drawColor,line width= 0.6pt,line join=round] ( 38.58,237.72) --
	( 42.84,237.72);
\end{scope}
\begin{scope}
\path[clip] (  0.00,  0.00) rectangle (433.62,252.94);
\definecolor[named]{drawColor}{rgb}{0.50,0.50,0.50}

\node[text=drawColor,anchor=base east,inner sep=0pt, outer sep=0pt, scale=  0.87] at ( 35.73, 36.45) {0};

\node[text=drawColor,anchor=base east,inner sep=0pt, outer sep=0pt, scale=  0.87] at ( 35.73, 71.26) {5};

\node[text=drawColor,anchor=base east,inner sep=0pt, outer sep=0pt, scale=  0.87] at ( 35.73,106.07) {10};
\end{scope}
\begin{scope}
\path[clip] (  0.00,  0.00) rectangle (433.62,252.94);
\definecolor[named]{drawColor}{rgb}{0.50,0.50,0.50}

\path[draw=drawColor,line width= 0.6pt,line join=round] ( 38.58, 39.74) --
	( 42.84, 39.74);

\path[draw=drawColor,line width= 0.6pt,line join=round] ( 38.58, 74.55) --
	( 42.84, 74.55);

\path[draw=drawColor,line width= 0.6pt,line join=round] ( 38.58,109.36) --
	( 42.84,109.36);
\end{scope}
\begin{scope}
\path[clip] (407.16,139.09) rectangle (420.37,239.70);
\definecolor[named]{fillColor}{rgb}{0.80,0.80,0.80}

\path[fill=fillColor] (407.16,139.09) rectangle (420.37,239.70);
\definecolor[named]{drawColor}{rgb}{0.00,0.00,0.00}

\node[text=drawColor,rotate=270.00,anchor=base,inner sep=0pt, outer sep=0pt, scale=  0.87] at (410.48,189.39) {train};
\end{scope}
\begin{scope}
\path[clip] (407.16, 35.17) rectangle (420.37,135.78);
\definecolor[named]{fillColor}{rgb}{0.80,0.80,0.80}

\path[fill=fillColor] (407.16, 35.17) rectangle (420.37,135.78);
\definecolor[named]{drawColor}{rgb}{0.00,0.00,0.00}

\node[text=drawColor,rotate=270.00,anchor=base,inner sep=0pt, outer sep=0pt, scale=  0.87] at (410.48, 85.47) {test};
\end{scope}
\begin{scope}
\path[clip] (  0.00,  0.00) rectangle (433.62,252.94);
\definecolor[named]{drawColor}{rgb}{0.50,0.50,0.50}

\path[draw=drawColor,line width= 0.6pt,line join=round] ( 59.40, 30.90) --
	( 59.40, 35.17);

\path[draw=drawColor,line width= 0.6pt,line join=round] (169.80, 30.90) --
	(169.80, 35.17);

\path[draw=drawColor,line width= 0.6pt,line join=round] (280.20, 30.90) --
	(280.20, 35.17);

\path[draw=drawColor,line width= 0.6pt,line join=round] (390.60, 30.90) --
	(390.60, 35.17);
\end{scope}
\begin{scope}
\path[clip] (  0.00,  0.00) rectangle (433.62,252.94);
\definecolor[named]{drawColor}{rgb}{0.50,0.50,0.50}

\node[text=drawColor,anchor=base,inner sep=0pt, outer sep=0pt, scale=  0.87] at ( 59.40, 21.48) {-1};

\node[text=drawColor,anchor=base,inner sep=0pt, outer sep=0pt, scale=  0.87] at (169.80, 21.48) {0};

\node[text=drawColor,anchor=base,inner sep=0pt, outer sep=0pt, scale=  0.87] at (280.20, 21.48) {1};

\node[text=drawColor,anchor=base,inner sep=0pt, outer sep=0pt, scale=  0.87] at (390.60, 21.48) {2};
\end{scope}
\begin{scope}
\path[clip] (  0.00,  0.00) rectangle (433.62,252.94);
\definecolor[named]{drawColor}{rgb}{0.00,0.00,0.00}

\node[text=drawColor,anchor=base,inner sep=0pt, outer sep=0pt, scale=  1.09] at (225.00,  9.94) {penalty exponent $\alpha$};
\end{scope}
\begin{scope}
\path[clip] (  0.00,  0.00) rectangle (433.62,252.94);
\definecolor[named]{drawColor}{rgb}{0.00,0.00,0.00}

\node[text=drawColor,rotate= 90.00,anchor=base,inner sep=0pt, outer sep=0pt, scale=  1.09] at ( 18.16,137.43) {total error relative to true breakpoints (breakpointError)};
\end{scope}
\end{tikzpicture}

%% file: figure-variable-breaks-constant-size-berr.tex
\begin{tikzpicture}[x=1pt,y=1pt]
\definecolor[named]{fillColor}{rgb}{1.00,1.00,1.00}
\path[use as bounding box,fill=fillColor,fill opacity=0.00] (0,0) rectangle (433.62,216.81);
\begin{scope}
\path[clip] (  0.00,  0.00) rectangle (433.62,216.81);
\definecolor[named]{drawColor}{rgb}{1.00,1.00,1.00}
\definecolor[named]{fillColor}{rgb}{1.00,1.00,1.00}

\path[draw=drawColor,line width= 0.6pt,line join=round,line cap=round,fill=fillColor] ( -0.00,  0.00) rectangle (433.62,216.81);
\end{scope}
\begin{scope}
\path[clip] ( 38.09,190.35) rectangle (158.91,203.56);
\definecolor[named]{fillColor}{rgb}{0.80,0.80,0.80}

\path[fill=fillColor] ( 38.09,190.35) rectangle (158.91,203.56);
\definecolor[named]{drawColor}{rgb}{0.00,0.00,0.00}

\node[text=drawColor,anchor=base,inner sep=0pt, outer sep=0pt, scale=  0.87] at ( 98.50,193.67) {$\beta = -1$};
\end{scope}
\begin{scope}
\path[clip] (162.22,190.35) rectangle (283.04,203.56);
\definecolor[named]{fillColor}{rgb}{0.80,0.80,0.80}

\path[fill=fillColor] (162.22,190.35) rectangle (283.04,203.56);
\definecolor[named]{drawColor}{rgb}{0.00,0.00,0.00}

\node[text=drawColor,anchor=base,inner sep=0pt, outer sep=0pt, scale=  0.87] at (222.63,193.67) {$\beta = -0.5$};
\end{scope}
\begin{scope}
\path[clip] (286.35,190.35) rectangle (407.16,203.56);
\definecolor[named]{fillColor}{rgb}{0.80,0.80,0.80}

\path[fill=fillColor] (286.35,190.35) rectangle (407.16,203.56);
\definecolor[named]{drawColor}{rgb}{0.00,0.00,0.00}

\node[text=drawColor,anchor=base,inner sep=0pt, outer sep=0pt, scale=  0.87] at (346.76,193.67) {$\beta = 0$};
\end{scope}
\begin{scope}
\path[clip] ( 38.09,114.42) rectangle (158.91,190.35);
\definecolor[named]{fillColor}{rgb}{0.90,0.90,0.90}

\path[fill=fillColor] ( 38.09,114.42) rectangle (158.91,190.35);
\definecolor[named]{drawColor}{rgb}{0.95,0.95,0.95}

\path[draw=drawColor,line width= 0.3pt,line join=round] ( 38.09,135.13) --
	(158.91,135.13);

\path[draw=drawColor,line width= 0.3pt,line join=round] ( 38.09,169.64) --
	(158.91,169.64);

\path[draw=drawColor,line width= 0.3pt,line join=round] ( 58.84,114.42) --
	( 58.84,190.35);

\path[draw=drawColor,line width= 0.3pt,line join=round] ( 89.35,114.42) --
	( 89.35,190.35);

\path[draw=drawColor,line width= 0.3pt,line join=round] (119.86,114.42) --
	(119.86,190.35);

\path[draw=drawColor,line width= 0.3pt,line join=round] (150.37,114.42) --
	(150.37,190.35);
\definecolor[named]{drawColor}{rgb}{1.00,1.00,1.00}

\path[draw=drawColor,line width= 0.6pt,line join=round] ( 38.09,117.87) --
	(158.91,117.87);

\path[draw=drawColor,line width= 0.6pt,line join=round] ( 38.09,152.39) --
	(158.91,152.39);

\path[draw=drawColor,line width= 0.6pt,line join=round] ( 38.09,186.90) --
	(158.91,186.90);

\path[draw=drawColor,line width= 0.6pt,line join=round] ( 43.58,114.42) --
	( 43.58,190.35);

\path[draw=drawColor,line width= 0.6pt,line join=round] ( 74.09,114.42) --
	( 74.09,190.35);

\path[draw=drawColor,line width= 0.6pt,line join=round] (104.60,114.42) --
	(104.60,190.35);

\path[draw=drawColor,line width= 0.6pt,line join=round] (135.11,114.42) --
	(135.11,190.35);
\definecolor[named]{drawColor}{rgb}{0.00,0.00,0.00}

\path[draw=drawColor,line width= 0.6pt,line join=round] ( 43.58,186.90) --
	( 44.14,186.90) --
	( 44.69,186.90) --
	( 45.24,186.90) --
	( 45.79,186.90) --
	( 46.34,186.90) --
	( 46.90,186.90) --
	( 47.45,186.90) --
	( 48.00,186.90) --
	( 48.55,186.90) --
	( 49.10,186.90) --
	( 49.65,186.90) --
	( 50.21,186.90) --
	( 50.76,186.90) --
	( 51.31,186.90) --
	( 51.86,186.90) --
	( 52.41,186.90) --
	( 52.97,186.90) --
	( 53.52,186.90) --
	( 54.07,186.90) --
	( 54.62,186.90) --
	( 55.17,186.90) --
	( 55.73,186.90) --
	( 56.28,186.90) --
	( 56.83,186.90) --
	( 57.38,186.90) --
	( 57.93,186.90) --
	( 58.49,186.90) --
	( 59.04,186.90) --
	( 59.59,186.90) --
	( 60.14,186.90) --
	( 60.69,186.90) --
	( 61.25,186.90) --
	( 61.80,186.90) --
	( 62.35,186.90) --
	( 62.90,186.90) --
	( 63.45,186.90) --
	( 64.00,186.90) --
	( 64.56,186.90) --
	( 65.11,186.90) --
	( 65.66,186.90) --
	( 66.21,186.90) --
	( 66.76,186.90) --
	( 67.32,186.90) --
	( 67.87,186.90) --
	( 68.42,186.90) --
	( 68.97,186.90) --
	( 69.52,186.90) --
	( 70.08,186.90) --
	( 70.63,186.90) --
	( 71.18,186.90) --
	( 71.73,186.90) --
	( 72.28,186.90) --
	( 72.84,186.90) --
	( 73.39,186.90) --
	( 73.94,186.90) --
	( 74.49,186.90) --
	( 75.04,186.90) --
	( 75.60,186.90) --
	( 76.15,186.90) --
	( 76.70,186.90) --
	( 77.25,186.90) --
	( 77.80,186.90) --
	( 78.35,186.90) --
	( 78.91,186.90) --
	( 79.46,186.90) --
	( 80.01,186.90) --
	( 80.56,186.90) --
	( 81.11,186.90) --
	( 81.67,186.90) --
	( 82.22,186.90) --
	( 82.77,186.90) --
	( 83.32,186.90) --
	( 83.87,186.90) --
	( 84.43,186.90) --
	( 84.98,186.90) --
	( 85.53,186.90) --
	( 86.08,186.90) --
	( 86.63,186.90) --
	( 87.19,186.90) --
	( 87.74,186.90) --
	( 88.29,186.90) --
	( 88.84,186.90) --
	( 89.39,186.90) --
	( 89.95,186.90) --
	( 90.50,186.90) --
	( 91.05,186.90) --
	( 91.60,186.90) --
	( 92.15,186.90) --
	( 92.70,186.90) --
	( 93.26,186.90) --
	( 93.81,186.90) --
	( 94.36,186.90) --
	( 94.91,186.90) --
	( 95.46,186.90) --
	( 96.02,186.90) --
	( 96.57,186.90) --
	( 97.12,186.90) --
	( 97.67,186.90) --
	( 98.22,186.90) --
	( 98.78,186.90) --
	( 99.33,186.90) --
	( 99.88,186.90) --
	(100.43,117.87) --
	(100.98,117.87) --
	(101.54,117.87) --
	(102.09,117.87) --
	(102.64,117.87) --
	(103.19,117.87) --
	(103.74,117.87) --
	(104.30,117.87) --
	(104.85,117.87) --
	(105.40,117.87) --
	(105.95,117.87) --
	(106.50,117.87) --
	(107.05,117.87) --
	(107.61,117.87) --
	(108.16,117.87) --
	(108.71,117.87) --
	(109.26,117.87) --
	(109.81,117.87) --
	(110.37,117.87) --
	(110.92,117.87) --
	(111.47,117.87) --
	(112.02,152.39) --
	(112.57,152.39) --
	(113.13,152.39) --
	(113.68,152.39) --
	(114.23,152.39) --
	(114.78,152.39) --
	(115.33,152.39) --
	(115.89,152.39) --
	(116.44,152.39) --
	(116.99,152.39) --
	(117.54,152.39) --
	(118.09,152.39) --
	(118.64,152.39) --
	(119.20,152.39) --
	(119.75,152.39) --
	(120.30,152.39) --
	(120.85,152.39) --
	(121.40,152.39) --
	(121.96,152.39) --
	(122.51,152.39) --
	(123.06,152.39) --
	(123.61,152.39) --
	(124.16,152.39) --
	(124.72,152.39) --
	(125.27,152.39) --
	(125.82,152.39) --
	(126.37,152.39) --
	(126.92,152.39) --
	(127.48,152.39) --
	(128.03,152.39) --
	(128.58,152.39) --
	(129.13,152.39) --
	(129.68,152.39) --
	(130.24,152.39) --
	(130.79,152.39) --
	(131.34,152.39) --
	(131.89,152.39) --
	(132.44,152.39) --
	(132.99,152.39) --
	(133.55,152.39) --
	(134.10,152.39) --
	(134.65,152.39) --
	(135.20,152.39) --
	(135.75,152.39) --
	(136.31,152.39) --
	(136.86,152.39) --
	(137.41,152.39) --
	(137.96,152.39) --
	(138.51,152.39) --
	(139.07,152.39) --
	(139.62,152.39) --
	(140.17,152.39) --
	(140.72,152.39) --
	(141.27,152.39) --
	(141.83,152.39) --
	(142.38,152.39) --
	(142.93,152.39) --
	(143.48,152.39) --
	(144.03,152.39) --
	(144.59,152.39) --
	(145.14,152.39) --
	(145.69,152.39) --
	(146.24,152.39) --
	(146.79,152.39) --
	(147.34,152.39) --
	(147.90,152.39) --
	(148.45,152.39) --
	(149.00,152.39) --
	(149.55,152.39) --
	(150.10,152.39) --
	(150.66,152.39) --
	(151.21,152.39) --
	(151.76,152.39) --
	(152.31,152.39) --
	(152.86,152.39) --
	(153.42,152.39);
\definecolor[named]{fillColor}{rgb}{1.00,1.00,1.00}

\path[draw=drawColor,line width= 0.4pt,line join=round,line cap=round,fill=fillColor] (105.95,117.87) circle (  2.13);
\end{scope}
\begin{scope}
\path[clip] ( 38.09, 35.17) rectangle (158.91,111.11);
\definecolor[named]{fillColor}{rgb}{0.90,0.90,0.90}

\path[fill=fillColor] ( 38.09, 35.17) rectangle (158.91,111.11);
\definecolor[named]{drawColor}{rgb}{0.95,0.95,0.95}

\path[draw=drawColor,line width= 0.3pt,line join=round] ( 38.09, 44.48) --
	(158.91, 44.48);

\path[draw=drawColor,line width= 0.3pt,line join=round] ( 38.09, 64.48) --
	(158.91, 64.48);

\path[draw=drawColor,line width= 0.3pt,line join=round] ( 38.09, 84.47) --
	(158.91, 84.47);

\path[draw=drawColor,line width= 0.3pt,line join=round] ( 38.09,104.47) --
	(158.91,104.47);

\path[draw=drawColor,line width= 0.3pt,line join=round] ( 58.84, 35.17) --
	( 58.84,111.11);

\path[draw=drawColor,line width= 0.3pt,line join=round] ( 89.35, 35.17) --
	( 89.35,111.11);

\path[draw=drawColor,line width= 0.3pt,line join=round] (119.86, 35.17) --
	(119.86,111.11);

\path[draw=drawColor,line width= 0.3pt,line join=round] (150.37, 35.17) --
	(150.37,111.11);
\definecolor[named]{drawColor}{rgb}{1.00,1.00,1.00}

\path[draw=drawColor,line width= 0.6pt,line join=round] ( 38.09, 54.48) --
	(158.91, 54.48);

\path[draw=drawColor,line width= 0.6pt,line join=round] ( 38.09, 74.47) --
	(158.91, 74.47);

\path[draw=drawColor,line width= 0.6pt,line join=round] ( 38.09, 94.47) --
	(158.91, 94.47);

\path[draw=drawColor,line width= 0.6pt,line join=round] ( 43.58, 35.17) --
	( 43.58,111.11);

\path[draw=drawColor,line width= 0.6pt,line join=round] ( 74.09, 35.17) --
	( 74.09,111.11);

\path[draw=drawColor,line width= 0.6pt,line join=round] (104.60, 35.17) --
	(104.60,111.11);

\path[draw=drawColor,line width= 0.6pt,line join=round] (135.11, 35.17) --
	(135.11,111.11);
\definecolor[named]{drawColor}{rgb}{0.00,0.00,0.00}

\path[draw=drawColor,line width= 0.6pt,line join=round] ( 43.58,107.65) --
	( 44.14,107.65) --
	( 44.69,107.65) --
	( 45.24,107.65) --
	( 45.79,107.65) --
	( 46.34,107.65) --
	( 46.90,107.65) --
	( 47.45,107.65) --
	( 48.00,107.65) --
	( 48.55,107.65) --
	( 49.10,107.65) --
	( 49.65,107.65) --
	( 50.21,107.65) --
	( 50.76,107.65) --
	( 51.31,107.65) --
	( 51.86,107.65) --
	( 52.41,107.65) --
	( 52.97,107.65) --
	( 53.52,107.65) --
	( 54.07,107.65) --
	( 54.62,107.65) --
	( 55.17,107.65) --
	( 55.73,107.65) --
	( 56.28,107.65) --
	( 56.83,107.65) --
	( 57.38,107.65) --
	( 57.93,107.65) --
	( 58.49,107.65) --
	( 59.04,107.65) --
	( 59.59,107.65) --
	( 60.14,107.65) --
	( 60.69,107.65) --
	( 61.25,107.65) --
	( 61.80,107.65) --
	( 62.35,107.65) --
	( 62.90,107.65) --
	( 63.45,107.65) --
	( 64.00,107.65) --
	( 64.56,107.65) --
	( 65.11,107.65) --
	( 65.66,107.65) --
	( 66.21,107.65) --
	( 66.76,107.65) --
	( 67.32,107.65) --
	( 67.87,107.65) --
	( 68.42,107.65) --
	( 68.97,107.65) --
	( 69.52,107.65) --
	( 70.08,107.65) --
	( 70.63,107.65) --
	( 71.18,107.65) --
	( 71.73,107.65) --
	( 72.28,107.65) --
	( 72.84,107.65) --
	( 73.39,107.65) --
	( 73.94,107.65) --
	( 74.49,107.65) --
	( 75.04,107.65) --
	( 75.60,107.65) --
	( 76.15,107.65) --
	( 76.70,107.65) --
	( 77.25,107.65) --
	( 77.80,107.65) --
	( 78.35,107.65) --
	( 78.91,107.65) --
	( 79.46,107.65) --
	( 80.01,107.65) --
	( 80.56,107.65) --
	( 81.11,107.65) --
	( 81.67,107.65) --
	( 82.22,107.65) --
	( 82.77,107.65) --
	( 83.32,107.65) --
	( 83.87,107.65) --
	( 84.43,107.65) --
	( 84.98,107.65) --
	( 85.53,107.65) --
	( 86.08,107.65) --
	( 86.63,107.65) --
	( 87.19,107.65) --
	( 87.74,107.65) --
	( 88.29,107.65) --
	( 88.84,107.65) --
	( 89.39,107.65) --
	( 89.95,107.65) --
	( 90.50,107.65) --
	( 91.05,107.65) --
	( 91.60,107.65) --
	( 92.15,107.65) --
	( 92.70,107.65) --
	( 93.26,107.65) --
	( 93.81,107.65) --
	( 94.36,107.65) --
	( 94.91,107.65) --
	( 95.46,107.65) --
	( 96.02,107.65) --
	( 96.57,107.65) --
	( 97.12,107.65) --
	( 97.67,107.65) --
	( 98.22,107.65) --
	( 98.78,107.65) --
	( 99.33,107.65) --
	( 99.88,107.65) --
	(100.43,107.65) --
	(100.98,107.65) --
	(101.54,107.65) --
	(102.09,107.65) --
	(102.64,107.65) --
	(103.19,107.65) --
	(103.74,107.65) --
	(104.30,107.65) --
	(104.85,107.65) --
	(105.40,107.65) --
	(105.95,107.65) --
	(106.50,107.65) --
	(107.05,107.65) --
	(107.61, 86.90) --
	(108.16, 57.78) --
	(108.71, 44.46) --
	(109.26, 38.82) --
	(109.81, 38.62) --
	(110.37, 40.02) --
	(110.92, 49.46) --
	(111.47, 57.62) --
	(112.02, 66.74) --
	(112.57, 71.14) --
	(113.13, 77.29) --
	(113.68, 84.49) --
	(114.23, 90.99) --
	(114.78, 92.51) --
	(115.33, 93.47) --
	(115.89, 93.47) --
	(116.44, 93.47) --
	(116.99, 93.47) --
	(117.54, 93.47) --
	(118.09, 93.47) --
	(118.64, 93.47) --
	(119.20, 93.47) --
	(119.75, 93.47) --
	(120.30, 93.47) --
	(120.85, 93.47) --
	(121.40, 93.47) --
	(121.96, 93.47) --
	(122.51, 93.47) --
	(123.06, 93.47) --
	(123.61, 93.47) --
	(124.16, 93.47) --
	(124.72, 93.47) --
	(125.27, 93.47) --
	(125.82, 93.47) --
	(126.37, 93.47) --
	(126.92, 93.47) --
	(127.48, 93.47) --
	(128.03, 93.47) --
	(128.58, 93.47) --
	(129.13, 93.47) --
	(129.68, 93.47) --
	(130.24, 93.47) --
	(130.79, 93.47) --
	(131.34, 93.47) --
	(131.89, 93.47) --
	(132.44, 93.47) --
	(132.99, 93.47) --
	(133.55, 93.47) --
	(134.10, 93.47) --
	(134.65, 93.47) --
	(135.20, 93.47) --
	(135.75, 93.47) --
	(136.31, 93.47) --
	(136.86, 93.47) --
	(137.41, 93.47) --
	(137.96, 93.47) --
	(138.51, 93.47) --
	(139.07, 93.47) --
	(139.62, 93.47) --
	(140.17, 93.47) --
	(140.72, 93.47) --
	(141.27, 93.47) --
	(141.83, 93.47) --
	(142.38, 93.47) --
	(142.93, 93.47) --
	(143.48, 93.47) --
	(144.03, 93.47) --
	(144.59, 93.47) --
	(145.14, 93.47) --
	(145.69, 93.47) --
	(146.24, 93.47) --
	(146.79, 93.47) --
	(147.34, 93.47) --
	(147.90, 93.47) --
	(148.45, 93.47) --
	(149.00, 93.47) --
	(149.55, 93.47) --
	(150.10, 93.47) --
	(150.66, 93.47) --
	(151.21, 93.47) --
	(151.76, 93.47) --
	(152.31, 93.47) --
	(152.86, 93.47) --
	(153.42, 93.47);
\definecolor[named]{fillColor}{rgb}{1.00,1.00,1.00}

\path[draw=drawColor,line width= 0.4pt,line join=round,line cap=round,fill=fillColor] (109.81, 38.62) circle (  2.13);
\end{scope}
\begin{scope}
\path[clip] (162.22,114.42) rectangle (283.04,190.35);
\definecolor[named]{fillColor}{rgb}{0.90,0.90,0.90}

\path[fill=fillColor] (162.22,114.42) rectangle (283.04,190.35);
\definecolor[named]{drawColor}{rgb}{0.95,0.95,0.95}

\path[draw=drawColor,line width= 0.3pt,line join=round] (162.22,135.13) --
	(283.04,135.13);

\path[draw=drawColor,line width= 0.3pt,line join=round] (162.22,169.64) --
	(283.04,169.64);

\path[draw=drawColor,line width= 0.3pt,line join=round] (182.97,114.42) --
	(182.97,190.35);

\path[draw=drawColor,line width= 0.3pt,line join=round] (213.48,114.42) --
	(213.48,190.35);

\path[draw=drawColor,line width= 0.3pt,line join=round] (243.98,114.42) --
	(243.98,190.35);

\path[draw=drawColor,line width= 0.3pt,line join=round] (274.49,114.42) --
	(274.49,190.35);
\definecolor[named]{drawColor}{rgb}{1.00,1.00,1.00}

\path[draw=drawColor,line width= 0.6pt,line join=round] (162.22,117.87) --
	(283.04,117.87);

\path[draw=drawColor,line width= 0.6pt,line join=round] (162.22,152.39) --
	(283.04,152.39);

\path[draw=drawColor,line width= 0.6pt,line join=round] (162.22,186.90) --
	(283.04,186.90);

\path[draw=drawColor,line width= 0.6pt,line join=round] (167.71,114.42) --
	(167.71,190.35);

\path[draw=drawColor,line width= 0.6pt,line join=round] (198.22,114.42) --
	(198.22,190.35);

\path[draw=drawColor,line width= 0.6pt,line join=round] (228.73,114.42) --
	(228.73,190.35);

\path[draw=drawColor,line width= 0.6pt,line join=round] (259.24,114.42) --
	(259.24,190.35);
\definecolor[named]{drawColor}{rgb}{0.00,0.00,0.00}

\path[draw=drawColor,line width= 0.6pt,line join=round] (167.71,186.90) --
	(168.26,186.90) --
	(168.82,186.90) --
	(169.37,186.90) --
	(169.92,186.90) --
	(170.47,186.90) --
	(171.02,186.90) --
	(171.58,186.90) --
	(172.13,186.90) --
	(172.68,186.90) --
	(173.23,186.90) --
	(173.78,186.90) --
	(174.33,186.90) --
	(174.89,186.90) --
	(175.44,186.90) --
	(175.99,186.90) --
	(176.54,186.90) --
	(177.09,186.90) --
	(177.65,186.90) --
	(178.20,186.90) --
	(178.75,186.90) --
	(179.30,186.90) --
	(179.85,186.90) --
	(180.41,186.90) --
	(180.96,186.90) --
	(181.51,186.90) --
	(182.06,186.90) --
	(182.61,186.90) --
	(183.17,186.90) --
	(183.72,186.90) --
	(184.27,186.90) --
	(184.82,186.90) --
	(185.37,186.90) --
	(185.93,186.90) --
	(186.48,186.90) --
	(187.03,186.90) --
	(187.58,186.90) --
	(188.13,186.90) --
	(188.68,186.90) --
	(189.24,186.90) --
	(189.79,186.90) --
	(190.34,186.90) --
	(190.89,186.90) --
	(191.44,186.90) --
	(192.00,186.90) --
	(192.55,186.90) --
	(193.10,186.90) --
	(193.65,186.90) --
	(194.20,186.90) --
	(194.76,186.90) --
	(195.31,186.90) --
	(195.86,186.90) --
	(196.41,186.90) --
	(196.96,186.90) --
	(197.52,186.90) --
	(198.07,186.90) --
	(198.62,186.90) --
	(199.17,186.90) --
	(199.72,186.90) --
	(200.28,186.90) --
	(200.83,186.90) --
	(201.38,186.90) --
	(201.93,186.90) --
	(202.48,186.90) --
	(203.03,186.90) --
	(203.59,186.90) --
	(204.14,186.90) --
	(204.69,186.90) --
	(205.24,186.90) --
	(205.79,186.90) --
	(206.35,186.90) --
	(206.90,186.90) --
	(207.45,186.90) --
	(208.00,186.90) --
	(208.55,186.90) --
	(209.11,186.90) --
	(209.66,186.90) --
	(210.21,186.90) --
	(210.76,186.90) --
	(211.31,186.90) --
	(211.87,186.90) --
	(212.42,186.90) --
	(212.97,186.90) --
	(213.52,186.90) --
	(214.07,186.90) --
	(214.63,117.87) --
	(215.18,117.87) --
	(215.73,117.87) --
	(216.28,117.87) --
	(216.83,117.87) --
	(217.38,117.87) --
	(217.94,117.87) --
	(218.49,117.87) --
	(219.04,117.87) --
	(219.59,117.87) --
	(220.14,117.87) --
	(220.70,117.87) --
	(221.25,117.87) --
	(221.80,117.87) --
	(222.35,117.87) --
	(222.90,117.87) --
	(223.46,117.87) --
	(224.01,117.87) --
	(224.56,117.87) --
	(225.11,117.87) --
	(225.66,117.87) --
	(226.22,152.39) --
	(226.77,152.39) --
	(227.32,152.39) --
	(227.87,152.39) --
	(228.42,152.39) --
	(228.98,152.39) --
	(229.53,152.39) --
	(230.08,152.39) --
	(230.63,152.39) --
	(231.18,152.39) --
	(231.73,152.39) --
	(232.29,152.39) --
	(232.84,152.39) --
	(233.39,152.39) --
	(233.94,152.39) --
	(234.49,152.39) --
	(235.05,152.39) --
	(235.60,152.39) --
	(236.15,152.39) --
	(236.70,152.39) --
	(237.25,152.39) --
	(237.81,152.39) --
	(238.36,152.39) --
	(238.91,152.39) --
	(239.46,152.39) --
	(240.01,152.39) --
	(240.57,152.39) --
	(241.12,152.39) --
	(241.67,152.39) --
	(242.22,152.39) --
	(242.77,152.39) --
	(243.33,152.39) --
	(243.88,152.39) --
	(244.43,152.39) --
	(244.98,152.39) --
	(245.53,152.39) --
	(246.08,152.39) --
	(246.64,152.39) --
	(247.19,152.39) --
	(247.74,152.39) --
	(248.29,152.39) --
	(248.84,152.39) --
	(249.40,152.39) --
	(249.95,152.39) --
	(250.50,152.39) --
	(251.05,152.39) --
	(251.60,152.39) --
	(252.16,152.39) --
	(252.71,152.39) --
	(253.26,152.39) --
	(253.81,152.39) --
	(254.36,152.39) --
	(254.92,152.39) --
	(255.47,152.39) --
	(256.02,152.39) --
	(256.57,152.39) --
	(257.12,152.39) --
	(257.68,152.39) --
	(258.23,152.39) --
	(258.78,152.39) --
	(259.33,152.39) --
	(259.88,152.39) --
	(260.43,152.39) --
	(260.99,152.39) --
	(261.54,152.39) --
	(262.09,152.39) --
	(262.64,152.39) --
	(263.19,152.39) --
	(263.75,152.39) --
	(264.30,152.39) --
	(264.85,152.39) --
	(265.40,152.39) --
	(265.95,152.39) --
	(266.51,152.39) --
	(267.06,152.39) --
	(267.61,152.39) --
	(268.16,152.39) --
	(268.71,152.39) --
	(269.27,152.39) --
	(269.82,152.39) --
	(270.37,152.39) --
	(270.92,152.39) --
	(271.47,152.39) --
	(272.03,152.39) --
	(272.58,152.39) --
	(273.13,152.39) --
	(273.68,152.39) --
	(274.23,152.39) --
	(274.78,152.39) --
	(275.34,152.39) --
	(275.89,152.39) --
	(276.44,152.39) --
	(276.99,152.39) --
	(277.54,152.39);
\definecolor[named]{fillColor}{rgb}{1.00,1.00,1.00}

\path[draw=drawColor,line width= 0.4pt,line join=round,line cap=round,fill=fillColor] (220.14,117.87) circle (  2.13);
\end{scope}
\begin{scope}
\path[clip] (162.22, 35.17) rectangle (283.04,111.11);
\definecolor[named]{fillColor}{rgb}{0.90,0.90,0.90}

\path[fill=fillColor] (162.22, 35.17) rectangle (283.04,111.11);
\definecolor[named]{drawColor}{rgb}{0.95,0.95,0.95}

\path[draw=drawColor,line width= 0.3pt,line join=round] (162.22, 44.48) --
	(283.04, 44.48);

\path[draw=drawColor,line width= 0.3pt,line join=round] (162.22, 64.48) --
	(283.04, 64.48);

\path[draw=drawColor,line width= 0.3pt,line join=round] (162.22, 84.47) --
	(283.04, 84.47);

\path[draw=drawColor,line width= 0.3pt,line join=round] (162.22,104.47) --
	(283.04,104.47);

\path[draw=drawColor,line width= 0.3pt,line join=round] (182.97, 35.17) --
	(182.97,111.11);

\path[draw=drawColor,line width= 0.3pt,line join=round] (213.48, 35.17) --
	(213.48,111.11);

\path[draw=drawColor,line width= 0.3pt,line join=round] (243.98, 35.17) --
	(243.98,111.11);

\path[draw=drawColor,line width= 0.3pt,line join=round] (274.49, 35.17) --
	(274.49,111.11);
\definecolor[named]{drawColor}{rgb}{1.00,1.00,1.00}

\path[draw=drawColor,line width= 0.6pt,line join=round] (162.22, 54.48) --
	(283.04, 54.48);

\path[draw=drawColor,line width= 0.6pt,line join=round] (162.22, 74.47) --
	(283.04, 74.47);

\path[draw=drawColor,line width= 0.6pt,line join=round] (162.22, 94.47) --
	(283.04, 94.47);

\path[draw=drawColor,line width= 0.6pt,line join=round] (167.71, 35.17) --
	(167.71,111.11);

\path[draw=drawColor,line width= 0.6pt,line join=round] (198.22, 35.17) --
	(198.22,111.11);

\path[draw=drawColor,line width= 0.6pt,line join=round] (228.73, 35.17) --
	(228.73,111.11);

\path[draw=drawColor,line width= 0.6pt,line join=round] (259.24, 35.17) --
	(259.24,111.11);
\definecolor[named]{drawColor}{rgb}{0.00,0.00,0.00}

\path[draw=drawColor,line width= 0.6pt,line join=round] (167.71,107.65) --
	(168.26,107.65) --
	(168.82,107.65) --
	(169.37,107.65) --
	(169.92,107.65) --
	(170.47,107.65) --
	(171.02,107.65) --
	(171.58,107.65) --
	(172.13,107.65) --
	(172.68,107.65) --
	(173.23,107.65) --
	(173.78,107.65) --
	(174.33,107.65) --
	(174.89,107.65) --
	(175.44,107.65) --
	(175.99,107.65) --
	(176.54,107.65) --
	(177.09,107.65) --
	(177.65,107.65) --
	(178.20,107.65) --
	(178.75,107.65) --
	(179.30,107.65) --
	(179.85,107.65) --
	(180.41,107.65) --
	(180.96,107.65) --
	(181.51,107.65) --
	(182.06,107.65) --
	(182.61,107.65) --
	(183.17,107.65) --
	(183.72,107.65) --
	(184.27,107.65) --
	(184.82,107.65) --
	(185.37,107.65) --
	(185.93,107.65) --
	(186.48,107.65) --
	(187.03,107.65) --
	(187.58,107.65) --
	(188.13,107.65) --
	(188.68,107.65) --
	(189.24,107.65) --
	(189.79,107.65) --
	(190.34,107.65) --
	(190.89,107.65) --
	(191.44,107.65) --
	(192.00,107.65) --
	(192.55,107.65) --
	(193.10,107.65) --
	(193.65,107.65) --
	(194.20,107.65) --
	(194.76,107.65) --
	(195.31,107.65) --
	(195.86,107.65) --
	(196.41,107.65) --
	(196.96,107.65) --
	(197.52,107.65) --
	(198.07,107.65) --
	(198.62,107.65) --
	(199.17,107.65) --
	(199.72,107.65) --
	(200.28,107.65) --
	(200.83,107.65) --
	(201.38,107.65) --
	(201.93,107.65) --
	(202.48,107.65) --
	(203.03,107.65) --
	(203.59,107.65) --
	(204.14,107.65) --
	(204.69,107.65) --
	(205.24,107.65) --
	(205.79,107.65) --
	(206.35,107.65) --
	(206.90,107.65) --
	(207.45,107.65) --
	(208.00,107.65) --
	(208.55,107.65) --
	(209.11,107.65) --
	(209.66,107.65) --
	(210.21,107.65) --
	(210.76,107.65) --
	(211.31,107.65) --
	(211.87,107.65) --
	(212.42,107.65) --
	(212.97,107.65) --
	(213.52,107.65) --
	(214.07,107.65) --
	(214.63,107.65) --
	(215.18,107.65) --
	(215.73,107.65) --
	(216.28,107.65) --
	(216.83, 83.90) --
	(217.38, 54.14) --
	(217.94, 44.46) --
	(218.49, 39.62) --
	(219.04, 38.62) --
	(219.59, 40.02) --
	(220.14, 49.10) --
	(220.70, 57.62) --
	(221.25, 66.74) --
	(221.80, 71.14) --
	(222.35, 77.29) --
	(222.90, 88.09) --
	(223.46, 90.99) --
	(224.01, 92.51) --
	(224.56, 93.47) --
	(225.11, 93.47) --
	(225.66, 93.47) --
	(226.22, 93.47) --
	(226.77, 93.47) --
	(227.32, 93.47) --
	(227.87, 93.47) --
	(228.42, 93.47) --
	(228.98, 93.47) --
	(229.53, 93.47) --
	(230.08, 93.47) --
	(230.63, 93.47) --
	(231.18, 93.47) --
	(231.73, 93.47) --
	(232.29, 93.47) --
	(232.84, 93.47) --
	(233.39, 93.47) --
	(233.94, 93.47) --
	(234.49, 93.47) --
	(235.05, 93.47) --
	(235.60, 93.47) --
	(236.15, 93.47) --
	(236.70, 93.47) --
	(237.25, 93.47) --
	(237.81, 93.47) --
	(238.36, 93.47) --
	(238.91, 93.47) --
	(239.46, 93.47) --
	(240.01, 93.47) --
	(240.57, 93.47) --
	(241.12, 93.47) --
	(241.67, 93.47) --
	(242.22, 93.47) --
	(242.77, 93.47) --
	(243.33, 93.47) --
	(243.88, 93.47) --
	(244.43, 93.47) --
	(244.98, 93.47) --
	(245.53, 93.47) --
	(246.08, 93.47) --
	(246.64, 93.47) --
	(247.19, 93.47) --
	(247.74, 93.47) --
	(248.29, 93.47) --
	(248.84, 93.47) --
	(249.40, 93.47) --
	(249.95, 93.47) --
	(250.50, 93.47) --
	(251.05, 93.47) --
	(251.60, 93.47) --
	(252.16, 93.47) --
	(252.71, 93.47) --
	(253.26, 93.47) --
	(253.81, 93.47) --
	(254.36, 93.47) --
	(254.92, 93.47) --
	(255.47, 93.47) --
	(256.02, 93.47) --
	(256.57, 93.47) --
	(257.12, 93.47) --
	(257.68, 93.47) --
	(258.23, 93.47) --
	(258.78, 93.47) --
	(259.33, 93.47) --
	(259.88, 93.47) --
	(260.43, 93.47) --
	(260.99, 93.47) --
	(261.54, 93.47) --
	(262.09, 93.47) --
	(262.64, 93.47) --
	(263.19, 93.47) --
	(263.75, 93.47) --
	(264.30, 93.47) --
	(264.85, 93.47) --
	(265.40, 93.47) --
	(265.95, 93.47) --
	(266.51, 93.47) --
	(267.06, 93.47) --
	(267.61, 93.47) --
	(268.16, 93.47) --
	(268.71, 93.47) --
	(269.27, 93.47) --
	(269.82, 93.47) --
	(270.37, 93.47) --
	(270.92, 93.47) --
	(271.47, 93.47) --
	(272.03, 93.47) --
	(272.58, 93.47) --
	(273.13, 93.47) --
	(273.68, 93.47) --
	(274.23, 93.47) --
	(274.78, 93.47) --
	(275.34, 93.47) --
	(275.89, 93.47) --
	(276.44, 93.47) --
	(276.99, 93.47) --
	(277.54, 93.47);
\definecolor[named]{fillColor}{rgb}{1.00,1.00,1.00}

\path[draw=drawColor,line width= 0.4pt,line join=round,line cap=round,fill=fillColor] (219.04, 38.62) circle (  2.13);
\end{scope}
\begin{scope}
\path[clip] (286.35,114.42) rectangle (407.16,190.35);
\definecolor[named]{fillColor}{rgb}{0.90,0.90,0.90}

\path[fill=fillColor] (286.35,114.42) rectangle (407.16,190.35);
\definecolor[named]{drawColor}{rgb}{0.95,0.95,0.95}

\path[draw=drawColor,line width= 0.3pt,line join=round] (286.35,135.13) --
	(407.16,135.13);

\path[draw=drawColor,line width= 0.3pt,line join=round] (286.35,169.64) --
	(407.16,169.64);

\path[draw=drawColor,line width= 0.3pt,line join=round] (307.09,114.42) --
	(307.09,190.35);

\path[draw=drawColor,line width= 0.3pt,line join=round] (337.60,114.42) --
	(337.60,190.35);

\path[draw=drawColor,line width= 0.3pt,line join=round] (368.11,114.42) --
	(368.11,190.35);

\path[draw=drawColor,line width= 0.3pt,line join=round] (398.62,114.42) --
	(398.62,190.35);
\definecolor[named]{drawColor}{rgb}{1.00,1.00,1.00}

\path[draw=drawColor,line width= 0.6pt,line join=round] (286.35,117.87) --
	(407.16,117.87);

\path[draw=drawColor,line width= 0.6pt,line join=round] (286.35,152.39) --
	(407.16,152.39);

\path[draw=drawColor,line width= 0.6pt,line join=round] (286.35,186.90) --
	(407.16,186.90);

\path[draw=drawColor,line width= 0.6pt,line join=round] (291.84,114.42) --
	(291.84,190.35);

\path[draw=drawColor,line width= 0.6pt,line join=round] (322.35,114.42) --
	(322.35,190.35);

\path[draw=drawColor,line width= 0.6pt,line join=round] (352.86,114.42) --
	(352.86,190.35);

\path[draw=drawColor,line width= 0.6pt,line join=round] (383.37,114.42) --
	(383.37,190.35);
\definecolor[named]{drawColor}{rgb}{0.00,0.00,0.00}

\path[draw=drawColor,line width= 0.6pt,line join=round] (291.84,186.90) --
	(292.39,186.90) --
	(292.94,186.90) --
	(293.50,186.90) --
	(294.05,186.90) --
	(294.60,186.90) --
	(295.15,186.90) --
	(295.70,186.90) --
	(296.26,186.90) --
	(296.81,186.90) --
	(297.36,186.90) --
	(297.91,186.90) --
	(298.46,186.90) --
	(299.01,186.90) --
	(299.57,186.90) --
	(300.12,186.90) --
	(300.67,186.90) --
	(301.22,186.90) --
	(301.77,186.90) --
	(302.33,186.90) --
	(302.88,186.90) --
	(303.43,186.90) --
	(303.98,186.90) --
	(304.53,186.90) --
	(305.09,186.90) --
	(305.64,186.90) --
	(306.19,186.90) --
	(306.74,186.90) --
	(307.29,186.90) --
	(307.85,186.90) --
	(308.40,186.90) --
	(308.95,186.90) --
	(309.50,186.90) --
	(310.05,186.90) --
	(310.61,186.90) --
	(311.16,186.90) --
	(311.71,186.90) --
	(312.26,186.90) --
	(312.81,186.90) --
	(313.36,186.90) --
	(313.92,186.90) --
	(314.47,186.90) --
	(315.02,186.90) --
	(315.57,186.90) --
	(316.12,186.90) --
	(316.68,186.90) --
	(317.23,186.90) --
	(317.78,186.90) --
	(318.33,186.90) --
	(318.88,186.90) --
	(319.44,186.90) --
	(319.99,186.90) --
	(320.54,186.90) --
	(321.09,186.90) --
	(321.64,186.90) --
	(322.20,186.90) --
	(322.75,186.90) --
	(323.30,186.90) --
	(323.85,186.90) --
	(324.40,186.90) --
	(324.96,186.90) --
	(325.51,186.90) --
	(326.06,186.90) --
	(326.61,186.90) --
	(327.16,186.90) --
	(327.71,186.90) --
	(328.27,117.87) --
	(328.82,117.87) --
	(329.37,117.87) --
	(329.92,117.87) --
	(330.47,117.87) --
	(331.03,117.87) --
	(331.58,117.87) --
	(332.13,117.87) --
	(332.68,117.87) --
	(333.23,117.87) --
	(333.79,117.87) --
	(334.34,117.87) --
	(334.89,117.87) --
	(335.44,117.87) --
	(335.99,117.87) --
	(336.55,117.87) --
	(337.10,117.87) --
	(337.65,117.87) --
	(338.20,117.87) --
	(338.75,117.87) --
	(339.31,117.87) --
	(339.86,152.39) --
	(340.41,152.39) --
	(340.96,152.39) --
	(341.51,152.39) --
	(342.06,152.39) --
	(342.62,152.39) --
	(343.17,152.39) --
	(343.72,152.39) --
	(344.27,152.39) --
	(344.82,152.39) --
	(345.38,152.39) --
	(345.93,152.39) --
	(346.48,152.39) --
	(347.03,152.39) --
	(347.58,152.39) --
	(348.14,152.39) --
	(348.69,152.39) --
	(349.24,152.39) --
	(349.79,152.39) --
	(350.34,152.39) --
	(350.90,152.39) --
	(351.45,152.39) --
	(352.00,152.39) --
	(352.55,152.39) --
	(353.10,152.39) --
	(353.66,152.39) --
	(354.21,152.39) --
	(354.76,152.39) --
	(355.31,152.39) --
	(355.86,152.39) --
	(356.41,152.39) --
	(356.97,152.39) --
	(357.52,152.39) --
	(358.07,152.39) --
	(358.62,152.39) --
	(359.17,152.39) --
	(359.73,152.39) --
	(360.28,152.39) --
	(360.83,152.39) --
	(361.38,152.39) --
	(361.93,152.39) --
	(362.49,152.39) --
	(363.04,152.39) --
	(363.59,152.39) --
	(364.14,152.39) --
	(364.69,152.39) --
	(365.25,152.39) --
	(365.80,152.39) --
	(366.35,152.39) --
	(366.90,152.39) --
	(367.45,152.39) --
	(368.01,152.39) --
	(368.56,152.39) --
	(369.11,152.39) --
	(369.66,152.39) --
	(370.21,152.39) --
	(370.76,152.39) --
	(371.32,152.39) --
	(371.87,152.39) --
	(372.42,152.39) --
	(372.97,152.39) --
	(373.52,152.39) --
	(374.08,152.39) --
	(374.63,152.39) --
	(375.18,152.39) --
	(375.73,152.39) --
	(376.28,152.39) --
	(376.84,152.39) --
	(377.39,152.39) --
	(377.94,152.39) --
	(378.49,152.39) --
	(379.04,152.39) --
	(379.60,152.39) --
	(380.15,152.39) --
	(380.70,152.39) --
	(381.25,152.39) --
	(381.80,152.39) --
	(382.36,152.39) --
	(382.91,152.39) --
	(383.46,152.39) --
	(384.01,152.39) --
	(384.56,152.39) --
	(385.11,152.39) --
	(385.67,152.39) --
	(386.22,152.39) --
	(386.77,152.39) --
	(387.32,152.39) --
	(387.87,152.39) --
	(388.43,152.39) --
	(388.98,152.39) --
	(389.53,152.39) --
	(390.08,152.39) --
	(390.63,152.39) --
	(391.19,152.39) --
	(391.74,152.39) --
	(392.29,152.39) --
	(392.84,152.39) --
	(393.39,152.39) --
	(393.95,152.39) --
	(394.50,152.39) --
	(395.05,152.39) --
	(395.60,152.39) --
	(396.15,152.39) --
	(396.71,152.39) --
	(397.26,152.39) --
	(397.81,152.39) --
	(398.36,152.39) --
	(398.91,152.39) --
	(399.46,152.39) --
	(400.02,152.39) --
	(400.57,152.39) --
	(401.12,152.39) --
	(401.67,152.39);
\definecolor[named]{fillColor}{rgb}{1.00,1.00,1.00}

\path[draw=drawColor,line width= 0.4pt,line join=round,line cap=round,fill=fillColor] (333.79,117.87) circle (  2.13);
\end{scope}
\begin{scope}
\path[clip] (286.35, 35.17) rectangle (407.16,111.11);
\definecolor[named]{fillColor}{rgb}{0.90,0.90,0.90}

\path[fill=fillColor] (286.35, 35.17) rectangle (407.16,111.11);
\definecolor[named]{drawColor}{rgb}{0.95,0.95,0.95}

\path[draw=drawColor,line width= 0.3pt,line join=round] (286.35, 44.48) --
	(407.16, 44.48);

\path[draw=drawColor,line width= 0.3pt,line join=round] (286.35, 64.48) --
	(407.16, 64.48);

\path[draw=drawColor,line width= 0.3pt,line join=round] (286.35, 84.47) --
	(407.16, 84.47);

\path[draw=drawColor,line width= 0.3pt,line join=round] (286.35,104.47) --
	(407.16,104.47);

\path[draw=drawColor,line width= 0.3pt,line join=round] (307.09, 35.17) --
	(307.09,111.11);

\path[draw=drawColor,line width= 0.3pt,line join=round] (337.60, 35.17) --
	(337.60,111.11);

\path[draw=drawColor,line width= 0.3pt,line join=round] (368.11, 35.17) --
	(368.11,111.11);

\path[draw=drawColor,line width= 0.3pt,line join=round] (398.62, 35.17) --
	(398.62,111.11);
\definecolor[named]{drawColor}{rgb}{1.00,1.00,1.00}

\path[draw=drawColor,line width= 0.6pt,line join=round] (286.35, 54.48) --
	(407.16, 54.48);

\path[draw=drawColor,line width= 0.6pt,line join=round] (286.35, 74.47) --
	(407.16, 74.47);

\path[draw=drawColor,line width= 0.6pt,line join=round] (286.35, 94.47) --
	(407.16, 94.47);

\path[draw=drawColor,line width= 0.6pt,line join=round] (291.84, 35.17) --
	(291.84,111.11);

\path[draw=drawColor,line width= 0.6pt,line join=round] (322.35, 35.17) --
	(322.35,111.11);

\path[draw=drawColor,line width= 0.6pt,line join=round] (352.86, 35.17) --
	(352.86,111.11);

\path[draw=drawColor,line width= 0.6pt,line join=round] (383.37, 35.17) --
	(383.37,111.11);
\definecolor[named]{drawColor}{rgb}{0.00,0.00,0.00}

\path[draw=drawColor,line width= 0.6pt,line join=round] (291.84,107.65) --
	(292.39,107.65) --
	(292.94,107.65) --
	(293.50,107.65) --
	(294.05,107.65) --
	(294.60,107.65) --
	(295.15,107.65) --
	(295.70,107.65) --
	(296.26,107.65) --
	(296.81,107.65) --
	(297.36,107.65) --
	(297.91,107.65) --
	(298.46,107.65) --
	(299.01,107.65) --
	(299.57,107.65) --
	(300.12,107.65) --
	(300.67,107.65) --
	(301.22,107.65) --
	(301.77,107.65) --
	(302.33,107.65) --
	(302.88,107.65) --
	(303.43,107.65) --
	(303.98,107.65) --
	(304.53,107.65) --
	(305.09,107.65) --
	(305.64,107.65) --
	(306.19,107.65) --
	(306.74,107.65) --
	(307.29,107.65) --
	(307.85,107.65) --
	(308.40,107.65) --
	(308.95,107.65) --
	(309.50,107.65) --
	(310.05,107.65) --
	(310.61,107.65) --
	(311.16,107.65) --
	(311.71,107.65) --
	(312.26,107.65) --
	(312.81,107.65) --
	(313.36,107.65) --
	(313.92,107.65) --
	(314.47,107.65) --
	(315.02,107.65) --
	(315.57,107.65) --
	(316.12,107.65) --
	(316.68,107.65) --
	(317.23,107.65) --
	(317.78,107.65) --
	(318.33,107.65) --
	(318.88,107.65) --
	(319.44,107.65) --
	(319.99,107.65) --
	(320.54,107.65) --
	(321.09,107.65) --
	(321.64,107.65) --
	(322.20,107.65) --
	(322.75,107.65) --
	(323.30,107.65) --
	(323.85,107.65) --
	(324.40,107.65) --
	(324.96,107.65) --
	(325.51,107.65) --
	(326.06, 77.90) --
	(326.61, 53.14) --
	(327.16, 44.46) --
	(327.71, 39.34) --
	(328.27, 38.62) --
	(328.82, 40.02) --
	(329.37, 49.46) --
	(329.92, 57.62) --
	(330.47, 66.74) --
	(331.03, 71.14) --
	(331.58, 77.29) --
	(332.13, 88.09) --
	(332.68, 92.51) --
	(333.23, 92.51) --
	(333.79, 93.47) --
	(334.34, 93.47) --
	(334.89, 93.47) --
	(335.44, 93.47) --
	(335.99, 93.47) --
	(336.55, 93.47) --
	(337.10, 93.47) --
	(337.65, 93.47) --
	(338.20, 93.47) --
	(338.75, 93.47) --
	(339.31, 93.47) --
	(339.86, 93.47) --
	(340.41, 93.47) --
	(340.96, 93.47) --
	(341.51, 93.47) --
	(342.06, 93.47) --
	(342.62, 93.47) --
	(343.17, 93.47) --
	(343.72, 93.47) --
	(344.27, 93.47) --
	(344.82, 93.47) --
	(345.38, 93.47) --
	(345.93, 93.47) --
	(346.48, 93.47) --
	(347.03, 93.47) --
	(347.58, 93.47) --
	(348.14, 93.47) --
	(348.69, 93.47) --
	(349.24, 93.47) --
	(349.79, 93.47) --
	(350.34, 93.47) --
	(350.90, 93.47) --
	(351.45, 93.47) --
	(352.00, 93.47) --
	(352.55, 93.47) --
	(353.10, 93.47) --
	(353.66, 93.47) --
	(354.21, 93.47) --
	(354.76, 93.47) --
	(355.31, 93.47) --
	(355.86, 93.47) --
	(356.41, 93.47) --
	(356.97, 93.47) --
	(357.52, 93.47) --
	(358.07, 93.47) --
	(358.62, 93.47) --
	(359.17, 93.47) --
	(359.73, 93.47) --
	(360.28, 93.47) --
	(360.83, 93.47) --
	(361.38, 93.47) --
	(361.93, 93.47) --
	(362.49, 93.47) --
	(363.04, 93.47) --
	(363.59, 93.47) --
	(364.14, 93.47) --
	(364.69, 93.47) --
	(365.25, 93.47) --
	(365.80, 93.47) --
	(366.35, 93.47) --
	(366.90, 93.47) --
	(367.45, 93.47) --
	(368.01, 93.47) --
	(368.56, 93.47) --
	(369.11, 93.47) --
	(369.66, 93.47) --
	(370.21, 93.47) --
	(370.76, 93.47) --
	(371.32, 93.47) --
	(371.87, 93.47) --
	(372.42, 93.47) --
	(372.97, 93.47) --
	(373.52, 93.47) --
	(374.08, 93.47) --
	(374.63, 93.47) --
	(375.18, 93.47) --
	(375.73, 93.47) --
	(376.28, 93.47) --
	(376.84, 93.47) --
	(377.39, 93.47) --
	(377.94, 93.47) --
	(378.49, 93.47) --
	(379.04, 93.47) --
	(379.60, 93.47) --
	(380.15, 93.47) --
	(380.70, 93.47) --
	(381.25, 93.47) --
	(381.80, 93.47) --
	(382.36, 93.47) --
	(382.91, 93.47) --
	(383.46, 93.47) --
	(384.01, 93.47) --
	(384.56, 93.47) --
	(385.11, 93.47) --
	(385.67, 93.47) --
	(386.22, 93.47) --
	(386.77, 93.47) --
	(387.32, 93.47) --
	(387.87, 93.47) --
	(388.43, 93.47) --
	(388.98, 93.47) --
	(389.53, 93.47) --
	(390.08, 93.47) --
	(390.63, 93.47) --
	(391.19, 93.47) --
	(391.74, 93.47) --
	(392.29, 93.47) --
	(392.84, 93.47) --
	(393.39, 93.47) --
	(393.95, 93.47) --
	(394.50, 93.47) --
	(395.05, 93.47) --
	(395.60, 93.47) --
	(396.15, 93.47) --
	(396.71, 93.47) --
	(397.26, 93.47) --
	(397.81, 93.47) --
	(398.36, 93.47) --
	(398.91, 93.47) --
	(399.46, 93.47) --
	(400.02, 93.47) --
	(400.57, 93.47) --
	(401.12, 93.47) --
	(401.67, 93.47);
\definecolor[named]{fillColor}{rgb}{1.00,1.00,1.00}

\path[draw=drawColor,line width= 0.4pt,line join=round,line cap=round,fill=fillColor] (328.27, 38.62) circle (  2.13);
\end{scope}
\begin{scope}
\path[clip] (  0.00,  0.00) rectangle (433.62,216.81);
\definecolor[named]{drawColor}{rgb}{0.50,0.50,0.50}

\node[text=drawColor,anchor=base east,inner sep=0pt, outer sep=0pt, scale=  0.87] at ( 30.98,114.58) {0};

\node[text=drawColor,anchor=base east,inner sep=0pt, outer sep=0pt, scale=  0.87] at ( 30.98,149.10) {1};

\node[text=drawColor,anchor=base east,inner sep=0pt, outer sep=0pt, scale=  0.87] at ( 30.98,183.61) {2};
\end{scope}
\begin{scope}
\path[clip] (  0.00,  0.00) rectangle (433.62,216.81);
\definecolor[named]{drawColor}{rgb}{0.50,0.50,0.50}

\path[draw=drawColor,line width= 0.6pt,line join=round] ( 33.82,117.87) --
	( 38.09,117.87);

\path[draw=drawColor,line width= 0.6pt,line join=round] ( 33.82,152.39) --
	( 38.09,152.39);

\path[draw=drawColor,line width= 0.6pt,line join=round] ( 33.82,186.90) --
	( 38.09,186.90);
\end{scope}
\begin{scope}
\path[clip] (  0.00,  0.00) rectangle (433.62,216.81);
\definecolor[named]{drawColor}{rgb}{0.50,0.50,0.50}

\node[text=drawColor,anchor=base east,inner sep=0pt, outer sep=0pt, scale=  0.87] at ( 30.98, 51.19) {40};

\node[text=drawColor,anchor=base east,inner sep=0pt, outer sep=0pt, scale=  0.87] at ( 30.98, 71.18) {60};

\node[text=drawColor,anchor=base east,inner sep=0pt, outer sep=0pt, scale=  0.87] at ( 30.98, 91.18) {80};
\end{scope}
\begin{scope}
\path[clip] (  0.00,  0.00) rectangle (433.62,216.81);
\definecolor[named]{drawColor}{rgb}{0.50,0.50,0.50}

\path[draw=drawColor,line width= 0.6pt,line join=round] ( 33.82, 54.48) --
	( 38.09, 54.48);

\path[draw=drawColor,line width= 0.6pt,line join=round] ( 33.82, 74.47) --
	( 38.09, 74.47);

\path[draw=drawColor,line width= 0.6pt,line join=round] ( 33.82, 94.47) --
	( 38.09, 94.47);
\end{scope}
\begin{scope}
\path[clip] (407.16,114.42) rectangle (420.37,190.35);
\definecolor[named]{fillColor}{rgb}{0.80,0.80,0.80}

\path[fill=fillColor] (407.16,114.42) rectangle (420.37,190.35);
\definecolor[named]{drawColor}{rgb}{0.00,0.00,0.00}

\node[text=drawColor,rotate=270.00,anchor=base,inner sep=0pt, outer sep=0pt, scale=  0.87] at (410.48,152.39) {$\textrm{length} = 2000$};
\end{scope}
\begin{scope}
\path[clip] (407.16, 35.17) rectangle (420.37,111.11);
\definecolor[named]{fillColor}{rgb}{0.80,0.80,0.80}

\path[fill=fillColor] (407.16, 35.17) rectangle (420.37,111.11);
\definecolor[named]{drawColor}{rgb}{0.00,0.00,0.00}

\node[text=drawColor,rotate=270.00,anchor=base,inner sep=0pt, outer sep=0pt, scale=  0.87] at (410.48, 73.14) {$\textrm{length} = 80000$};
\end{scope}
\begin{scope}
\path[clip] (  0.00,  0.00) rectangle (433.62,216.81);
\definecolor[named]{drawColor}{rgb}{0.50,0.50,0.50}

\path[draw=drawColor,line width= 0.6pt,line join=round] ( 43.58, 30.90) --
	( 43.58, 35.17);

\path[draw=drawColor,line width= 0.6pt,line join=round] ( 74.09, 30.90) --
	( 74.09, 35.17);

\path[draw=drawColor,line width= 0.6pt,line join=round] (104.60, 30.90) --
	(104.60, 35.17);

\path[draw=drawColor,line width= 0.6pt,line join=round] (135.11, 30.90) --
	(135.11, 35.17);
\end{scope}
\begin{scope}
\path[clip] (  0.00,  0.00) rectangle (433.62,216.81);
\definecolor[named]{drawColor}{rgb}{0.50,0.50,0.50}

\node[text=drawColor,anchor=base,inner sep=0pt, outer sep=0pt, scale=  0.87] at ( 43.58, 21.48) {-5};

\node[text=drawColor,anchor=base,inner sep=0pt, outer sep=0pt, scale=  0.87] at ( 74.09, 21.48) {0};

\node[text=drawColor,anchor=base,inner sep=0pt, outer sep=0pt, scale=  0.87] at (104.60, 21.48) {5};

\node[text=drawColor,anchor=base,inner sep=0pt, outer sep=0pt, scale=  0.87] at (135.11, 21.48) {10};
\end{scope}
\begin{scope}
\path[clip] (  0.00,  0.00) rectangle (433.62,216.81);
\definecolor[named]{drawColor}{rgb}{0.50,0.50,0.50}

\path[draw=drawColor,line width= 0.6pt,line join=round] (167.71, 30.90) --
	(167.71, 35.17);

\path[draw=drawColor,line width= 0.6pt,line join=round] (198.22, 30.90) --
	(198.22, 35.17);

\path[draw=drawColor,line width= 0.6pt,line join=round] (228.73, 30.90) --
	(228.73, 35.17);

\path[draw=drawColor,line width= 0.6pt,line join=round] (259.24, 30.90) --
	(259.24, 35.17);
\end{scope}
\begin{scope}
\path[clip] (  0.00,  0.00) rectangle (433.62,216.81);
\definecolor[named]{drawColor}{rgb}{0.50,0.50,0.50}

\node[text=drawColor,anchor=base,inner sep=0pt, outer sep=0pt, scale=  0.87] at (167.71, 21.48) {-5};

\node[text=drawColor,anchor=base,inner sep=0pt, outer sep=0pt, scale=  0.87] at (198.22, 21.48) {0};

\node[text=drawColor,anchor=base,inner sep=0pt, outer sep=0pt, scale=  0.87] at (228.73, 21.48) {5};

\node[text=drawColor,anchor=base,inner sep=0pt, outer sep=0pt, scale=  0.87] at (259.24, 21.48) {10};
\end{scope}
\begin{scope}
\path[clip] (  0.00,  0.00) rectangle (433.62,216.81);
\definecolor[named]{drawColor}{rgb}{0.50,0.50,0.50}

\path[draw=drawColor,line width= 0.6pt,line join=round] (291.84, 30.90) --
	(291.84, 35.17);

\path[draw=drawColor,line width= 0.6pt,line join=round] (322.35, 30.90) --
	(322.35, 35.17);

\path[draw=drawColor,line width= 0.6pt,line join=round] (352.86, 30.90) --
	(352.86, 35.17);

\path[draw=drawColor,line width= 0.6pt,line join=round] (383.37, 30.90) --
	(383.37, 35.17);
\end{scope}
\begin{scope}
\path[clip] (  0.00,  0.00) rectangle (433.62,216.81);
\definecolor[named]{drawColor}{rgb}{0.50,0.50,0.50}

\node[text=drawColor,anchor=base,inner sep=0pt, outer sep=0pt, scale=  0.87] at (291.84, 21.48) {-5};

\node[text=drawColor,anchor=base,inner sep=0pt, outer sep=0pt, scale=  0.87] at (322.35, 21.48) {0};

\node[text=drawColor,anchor=base,inner sep=0pt, outer sep=0pt, scale=  0.87] at (352.86, 21.48) {5};

\node[text=drawColor,anchor=base,inner sep=0pt, outer sep=0pt, scale=  0.87] at (383.37, 21.48) {10};
\end{scope}
\begin{scope}
\path[clip] (  0.00,  0.00) rectangle (433.62,216.81);
\definecolor[named]{drawColor}{rgb}{0.00,0.00,0.00}

\node[text=drawColor,anchor=base,inner sep=0pt, outer sep=0pt, scale=  1.09] at (222.63,  9.94) {model complexity tradeoff parameter $\log_{10}(\lambda)$};
\end{scope}
\begin{scope}
\path[clip] (  0.00,  0.00) rectangle (433.62,216.81);
\definecolor[named]{drawColor}{rgb}{0.00,0.00,0.00}

\node[text=drawColor,rotate= 90.00,anchor=base,inner sep=0pt, outer sep=0pt, scale=  1.09] at ( 18.16,112.76) {error $E_i^\beta(\lambda)$};
\end{scope}
\end{tikzpicture}

%% file: figure-variable-breaks-constant-size-alpha.tex
\begin{tikzpicture}[x=1pt,y=1pt]
\definecolor[named]{fillColor}{rgb}{1.00,1.00,1.00}
\path[use as bounding box,fill=fillColor,fill opacity=0.00] (0,0) rectangle (433.62,289.08);
\begin{scope}
\path[clip] (  0.00,  0.00) rectangle (433.62,289.08);
\definecolor[named]{drawColor}{rgb}{1.00,1.00,1.00}
\definecolor[named]{fillColor}{rgb}{1.00,1.00,1.00}

\path[draw=drawColor,line width= 0.6pt,line join=round,line cap=round,fill=fillColor] (  0.00,  0.00) rectangle (433.62,289.08);
\end{scope}
\begin{scope}
\path[clip] ( 38.09,157.16) rectangle (407.16,275.83);
\definecolor[named]{fillColor}{rgb}{0.90,0.90,0.90}

\path[fill=fillColor] ( 38.09,157.16) rectangle (407.16,275.83);
\definecolor[named]{drawColor}{rgb}{0.95,0.95,0.95}

\path[draw=drawColor,line width= 0.3pt,line join=round] ( 38.09,189.52) --
	(407.16,189.52);

\path[draw=drawColor,line width= 0.3pt,line join=round] ( 38.09,226.75) --
	(407.16,226.75);

\path[draw=drawColor,line width= 0.3pt,line join=round] ( 38.09,263.98) --
	(407.16,263.98);

\path[draw=drawColor,line width= 0.3pt,line join=round] (110.79,157.16) --
	(110.79,275.83);

\path[draw=drawColor,line width= 0.3pt,line join=round] (222.63,157.16) --
	(222.63,275.83);

\path[draw=drawColor,line width= 0.3pt,line join=round] (334.47,157.16) --
	(334.47,275.83);
\definecolor[named]{drawColor}{rgb}{1.00,1.00,1.00}

\path[draw=drawColor,line width= 0.6pt,line join=round] ( 38.09,170.90) --
	(407.16,170.90);

\path[draw=drawColor,line width= 0.6pt,line join=round] ( 38.09,208.14) --
	(407.16,208.14);

\path[draw=drawColor,line width= 0.6pt,line join=round] ( 38.09,245.37) --
	(407.16,245.37);

\path[draw=drawColor,line width= 0.6pt,line join=round] ( 54.87,157.16) --
	( 54.87,275.83);

\path[draw=drawColor,line width= 0.6pt,line join=round] (166.71,157.16) --
	(166.71,275.83);

\path[draw=drawColor,line width= 0.6pt,line join=round] (278.55,157.16) --
	(278.55,275.83);

\path[draw=drawColor,line width= 0.6pt,line join=round] (390.39,157.16) --
	(390.39,275.83);
\definecolor[named]{drawColor}{rgb}{0.00,0.00,0.00}

\path[draw=drawColor,line width= 0.6pt,line join=round] ( 54.87,236.92) --
	(110.79,214.35) --
	(166.71,169.23) --
	(177.89,168.64) --
	(189.08,164.88) --
	(200.26,165.81) --
	(211.44,166.18) --
	(222.63,162.55) --
	(233.81,164.41) --
	(245.00,164.52) --
	(256.18,166.61) --
	(267.36,173.09) --
	(278.55,174.09) --
	(334.47,242.47) --
	(390.39,270.44);
\end{scope}
\begin{scope}
\path[clip] ( 38.09, 35.17) rectangle (407.16,153.84);
\definecolor[named]{fillColor}{rgb}{0.90,0.90,0.90}

\path[fill=fillColor] ( 38.09, 35.17) rectangle (407.16,153.84);
\definecolor[named]{drawColor}{rgb}{0.95,0.95,0.95}

\path[draw=drawColor,line width= 0.3pt,line join=round] ( 38.09, 60.13) --
	(407.16, 60.13);

\path[draw=drawColor,line width= 0.3pt,line join=round] ( 38.09, 99.24) --
	(407.16, 99.24);

\path[draw=drawColor,line width= 0.3pt,line join=round] ( 38.09,138.36) --
	(407.16,138.36);

\path[draw=drawColor,line width= 0.3pt,line join=round] (110.79, 35.17) --
	(110.79,153.84);

\path[draw=drawColor,line width= 0.3pt,line join=round] (222.63, 35.17) --
	(222.63,153.84);

\path[draw=drawColor,line width= 0.3pt,line join=round] (334.47, 35.17) --
	(334.47,153.84);
\definecolor[named]{drawColor}{rgb}{1.00,1.00,1.00}

\path[draw=drawColor,line width= 0.6pt,line join=round] ( 38.09, 40.57) --
	(407.16, 40.57);

\path[draw=drawColor,line width= 0.6pt,line join=round] ( 38.09, 79.68) --
	(407.16, 79.68);

\path[draw=drawColor,line width= 0.6pt,line join=round] ( 38.09,118.80) --
	(407.16,118.80);

\path[draw=drawColor,line width= 0.6pt,line join=round] ( 54.87, 35.17) --
	( 54.87,153.84);

\path[draw=drawColor,line width= 0.6pt,line join=round] (166.71, 35.17) --
	(166.71,153.84);

\path[draw=drawColor,line width= 0.6pt,line join=round] (278.55, 35.17) --
	(278.55,153.84);

\path[draw=drawColor,line width= 0.6pt,line join=round] (390.39, 35.17) --
	(390.39,153.84);
\definecolor[named]{fillColor}{rgb}{0.20,0.20,0.20}

\path[fill=fillColor,fill opacity=0.50] ( 54.87,148.45) --
	(110.79,143.48) --
	(166.71,128.25) --
	(177.89,116.49) --
	(189.08, 93.91) --
	(200.26, 74.88) --
	(211.44, 69.73) --
	(222.63, 70.79) --
	(233.81, 77.23) --
	(245.00, 84.79) --
	(256.18, 91.71) --
	(267.36, 98.76) --
	(278.55,104.80) --
	(334.47,121.66) --
	(390.39,129.83) --
	(390.39, 40.57) --
	(334.47, 40.57) --
	(278.55, 40.57) --
	(267.36, 40.57) --
	(256.18, 40.57) --
	(245.00, 40.57) --
	(233.81, 40.57) --
	(222.63, 40.57) --
	(211.44, 40.57) --
	(200.26, 40.57) --
	(189.08, 40.57) --
	(177.89, 40.57) --
	(166.71, 40.57) --
	(110.79, 40.57) --
	( 54.87, 40.57) --
	cycle;
\definecolor[named]{drawColor}{rgb}{0.00,0.00,0.00}

\path[draw=drawColor,line width= 0.6pt,line join=round] ( 54.87, 85.49) --
	(110.79, 81.16) --
	(166.71, 70.31) --
	(177.89, 65.26) --
	(189.08, 57.64) --
	(200.26, 52.73) --
	(211.44, 51.43) --
	(222.63, 51.65) --
	(233.81, 53.66) --
	(245.00, 56.41) --
	(256.18, 58.35) --
	(267.36, 61.45) --
	(278.55, 63.83) --
	(334.47, 74.72) --
	(390.39, 79.54);
\end{scope}
\begin{scope}
\path[clip] (  0.00,  0.00) rectangle (433.62,289.08);
\definecolor[named]{drawColor}{rgb}{0.50,0.50,0.50}

\node[text=drawColor,anchor=base east,inner sep=0pt, outer sep=0pt, scale=  0.87] at ( 30.98,167.61) {40};

\node[text=drawColor,anchor=base east,inner sep=0pt, outer sep=0pt, scale=  0.87] at ( 30.98,204.84) {60};

\node[text=drawColor,anchor=base east,inner sep=0pt, outer sep=0pt, scale=  0.87] at ( 30.98,242.08) {80};
\end{scope}
\begin{scope}
\path[clip] (  0.00,  0.00) rectangle (433.62,289.08);
\definecolor[named]{drawColor}{rgb}{0.50,0.50,0.50}

\path[draw=drawColor,line width= 0.6pt,line join=round] ( 33.82,170.90) --
	( 38.09,170.90);

\path[draw=drawColor,line width= 0.6pt,line join=round] ( 33.82,208.14) --
	( 38.09,208.14);

\path[draw=drawColor,line width= 0.6pt,line join=round] ( 33.82,245.37) --
	( 38.09,245.37);
\end{scope}
\begin{scope}
\path[clip] (  0.00,  0.00) rectangle (433.62,289.08);
\definecolor[named]{drawColor}{rgb}{0.50,0.50,0.50}

\node[text=drawColor,anchor=base east,inner sep=0pt, outer sep=0pt, scale=  0.87] at ( 30.98, 37.27) {0};

\node[text=drawColor,anchor=base east,inner sep=0pt, outer sep=0pt, scale=  0.87] at ( 30.98, 76.39) {20};

\node[text=drawColor,anchor=base east,inner sep=0pt, outer sep=0pt, scale=  0.87] at ( 30.98,115.51) {40};
\end{scope}
\begin{scope}
\path[clip] (  0.00,  0.00) rectangle (433.62,289.08);
\definecolor[named]{drawColor}{rgb}{0.50,0.50,0.50}

\path[draw=drawColor,line width= 0.6pt,line join=round] ( 33.82, 40.57) --
	( 38.09, 40.57);

\path[draw=drawColor,line width= 0.6pt,line join=round] ( 33.82, 79.68) --
	( 38.09, 79.68);

\path[draw=drawColor,line width= 0.6pt,line join=round] ( 33.82,118.80) --
	( 38.09,118.80);
\end{scope}
\begin{scope}
\path[clip] (407.16,157.16) rectangle (420.37,275.83);
\definecolor[named]{fillColor}{rgb}{0.80,0.80,0.80}

\path[fill=fillColor] (407.16,157.16) rectangle (420.37,275.83);
\definecolor[named]{drawColor}{rgb}{0.00,0.00,0.00}

\node[text=drawColor,rotate=270.00,anchor=base,inner sep=0pt, outer sep=0pt, scale=  0.87] at (410.48,216.49) {train};
\end{scope}
\begin{scope}
\path[clip] (407.16, 35.17) rectangle (420.37,153.84);
\definecolor[named]{fillColor}{rgb}{0.80,0.80,0.80}

\path[fill=fillColor] (407.16, 35.17) rectangle (420.37,153.84);
\definecolor[named]{drawColor}{rgb}{0.00,0.00,0.00}

\node[text=drawColor,rotate=270.00,anchor=base,inner sep=0pt, outer sep=0pt, scale=  0.87] at (410.48, 94.51) {test};
\end{scope}
\begin{scope}
\path[clip] (  0.00,  0.00) rectangle (433.62,289.08);
\definecolor[named]{drawColor}{rgb}{0.50,0.50,0.50}

\path[draw=drawColor,line width= 0.6pt,line join=round] ( 54.87, 30.90) --
	( 54.87, 35.17);

\path[draw=drawColor,line width= 0.6pt,line join=round] (166.71, 30.90) --
	(166.71, 35.17);

\path[draw=drawColor,line width= 0.6pt,line join=round] (278.55, 30.90) --
	(278.55, 35.17);

\path[draw=drawColor,line width= 0.6pt,line join=round] (390.39, 30.90) --
	(390.39, 35.17);
\end{scope}
\begin{scope}
\path[clip] (  0.00,  0.00) rectangle (433.62,289.08);
\definecolor[named]{drawColor}{rgb}{0.50,0.50,0.50}

\node[text=drawColor,anchor=base,inner sep=0pt, outer sep=0pt, scale=  0.87] at ( 54.87, 21.48) {-2};

\node[text=drawColor,anchor=base,inner sep=0pt, outer sep=0pt, scale=  0.87] at (166.71, 21.48) {-1};

\node[text=drawColor,anchor=base,inner sep=0pt, outer sep=0pt, scale=  0.87] at (278.55, 21.48) {0};

\node[text=drawColor,anchor=base,inner sep=0pt, outer sep=0pt, scale=  0.87] at (390.39, 21.48) {1};
\end{scope}
\begin{scope}
\path[clip] (  0.00,  0.00) rectangle (433.62,289.08);
\definecolor[named]{drawColor}{rgb}{0.00,0.00,0.00}

\node[text=drawColor,anchor=base,inner sep=0pt, outer sep=0pt, scale=  1.09] at (222.63,  9.94) {penalty exponent $\beta$};
\end{scope}
\begin{scope}
\path[clip] (  0.00,  0.00) rectangle (433.62,289.08);
\definecolor[named]{drawColor}{rgb}{0.00,0.00,0.00}

\node[text=drawColor,rotate= 90.00,anchor=base,inner sep=0pt, outer sep=0pt, scale=  1.09] at ( 18.16,155.50) {total error relative to true breakpoints (breakpointError)};
\end{scope}
\end{tikzpicture}

%% file: figure-variable-density-error-alpha-flsa.tex
\begin{tikzpicture}[x=1pt,y=1pt]
\definecolor[named]{fillColor}{rgb}{1.00,1.00,1.00}
\path[use as bounding box,fill=fillColor,fill opacity=0.00] (0,0) rectangle (433.62,252.94);
\begin{scope}
\path[clip] (  0.00,  0.00) rectangle (433.62,252.94);
\definecolor[named]{drawColor}{rgb}{1.00,1.00,1.00}
\definecolor[named]{fillColor}{rgb}{1.00,1.00,1.00}

\path[draw=drawColor,line width= 0.6pt,line join=round,line cap=round,fill=fillColor] (  0.00,  0.00) rectangle (433.62,252.94);
\end{scope}
\begin{scope}
\path[clip] ( 42.84,139.09) rectangle (407.16,239.70);
\definecolor[named]{fillColor}{rgb}{0.90,0.90,0.90}

\path[fill=fillColor] ( 42.84,139.09) rectangle (407.16,239.70);
\definecolor[named]{drawColor}{rgb}{0.95,0.95,0.95}

\path[draw=drawColor,line width= 0.3pt,line join=round] ( 42.84,142.81) --
	(407.16,142.81);

\path[draw=drawColor,line width= 0.3pt,line join=round] ( 42.84,163.33) --
	(407.16,163.33);

\path[draw=drawColor,line width= 0.3pt,line join=round] ( 42.84,183.84) --
	(407.16,183.84);

\path[draw=drawColor,line width= 0.3pt,line join=round] ( 42.84,204.35) --
	(407.16,204.35);

\path[draw=drawColor,line width= 0.3pt,line join=round] ( 42.84,224.87) --
	(407.16,224.87);

\path[draw=drawColor,line width= 0.3pt,line join=round] (100.80,139.09) --
	(100.80,239.70);

\path[draw=drawColor,line width= 0.3pt,line join=round] (183.60,139.09) --
	(183.60,239.70);

\path[draw=drawColor,line width= 0.3pt,line join=round] (266.40,139.09) --
	(266.40,239.70);

\path[draw=drawColor,line width= 0.3pt,line join=round] (349.20,139.09) --
	(349.20,239.70);
\definecolor[named]{drawColor}{rgb}{1.00,1.00,1.00}

\path[draw=drawColor,line width= 0.6pt,line join=round] ( 42.84,153.07) --
	(407.16,153.07);

\path[draw=drawColor,line width= 0.6pt,line join=round] ( 42.84,173.58) --
	(407.16,173.58);

\path[draw=drawColor,line width= 0.6pt,line join=round] ( 42.84,194.10) --
	(407.16,194.10);

\path[draw=drawColor,line width= 0.6pt,line join=round] ( 42.84,214.61) --
	(407.16,214.61);

\path[draw=drawColor,line width= 0.6pt,line join=round] ( 42.84,235.12) --
	(407.16,235.12);

\path[draw=drawColor,line width= 0.6pt,line join=round] ( 59.40,139.09) --
	( 59.40,239.70);

\path[draw=drawColor,line width= 0.6pt,line join=round] (142.20,139.09) --
	(142.20,239.70);

\path[draw=drawColor,line width= 0.6pt,line join=round] (225.00,139.09) --
	(225.00,239.70);

\path[draw=drawColor,line width= 0.6pt,line join=round] (307.80,139.09) --
	(307.80,239.70);

\path[draw=drawColor,line width= 0.6pt,line join=round] (390.60,139.09) --
	(390.60,239.70);
\definecolor[named]{drawColor}{rgb}{0.00,0.00,0.00}
\definecolor[named]{fillColor}{rgb}{0.00,0.00,0.00}

\path[draw=drawColor,line width= 0.6pt,dash pattern=on 4pt off 4pt ,line join=round,fill=fillColor] (218.38,139.09) -- (218.38,239.70);

\node[text=drawColor,anchor=base east,inner sep=0pt, outer sep=0pt, scale=  1.29] at (218.38,225.41) {0.96};

\path[draw=drawColor,line width= 0.6pt,line join=round] ( 59.40,214.61) --
	(142.20,202.89) --
	(208.44,164.17) --
	(210.10,164.17) --
	(211.76,143.66) --
	(213.41,143.66) --
	(215.07,143.66) --
	(216.72,143.66) --
	(218.38,143.66) --
	(220.04,154.22) --
	(221.69,150.70) --
	(223.35,223.63) --
	(225.00,235.09) --
	(226.66,235.12) --
	(228.32,235.09) --
	(229.97,235.12) --
	(231.63,235.12) --
	(233.28,235.09) --
	(234.94,235.12) --
	(236.60,235.09) --
	(238.25,235.12) --
	(239.91,235.09) --
	(241.56,235.12) --
	(307.80,235.09) --
	(390.60,235.09);
\end{scope}
\begin{scope}
\path[clip] ( 42.84, 35.17) rectangle (407.16,135.78);
\definecolor[named]{fillColor}{rgb}{0.90,0.90,0.90}

\path[fill=fillColor] ( 42.84, 35.17) rectangle (407.16,135.78);
\definecolor[named]{drawColor}{rgb}{0.95,0.95,0.95}

\path[draw=drawColor,line width= 0.3pt,line join=round] ( 42.84, 51.31) --
	(407.16, 51.31);

\path[draw=drawColor,line width= 0.3pt,line join=round] ( 42.84, 76.28) --
	(407.16, 76.28);

\path[draw=drawColor,line width= 0.3pt,line join=round] ( 42.84,101.25) --
	(407.16,101.25);

\path[draw=drawColor,line width= 0.3pt,line join=round] ( 42.84,126.22) --
	(407.16,126.22);

\path[draw=drawColor,line width= 0.3pt,line join=round] (100.80, 35.17) --
	(100.80,135.78);

\path[draw=drawColor,line width= 0.3pt,line join=round] (183.60, 35.17) --
	(183.60,135.78);

\path[draw=drawColor,line width= 0.3pt,line join=round] (266.40, 35.17) --
	(266.40,135.78);

\path[draw=drawColor,line width= 0.3pt,line join=round] (349.20, 35.17) --
	(349.20,135.78);
\definecolor[named]{drawColor}{rgb}{1.00,1.00,1.00}

\path[draw=drawColor,line width= 0.6pt,line join=round] ( 42.84, 38.83) --
	(407.16, 38.83);

\path[draw=drawColor,line width= 0.6pt,line join=round] ( 42.84, 63.80) --
	(407.16, 63.80);

\path[draw=drawColor,line width= 0.6pt,line join=round] ( 42.84, 88.76) --
	(407.16, 88.76);

\path[draw=drawColor,line width= 0.6pt,line join=round] ( 42.84,113.73) --
	(407.16,113.73);

\path[draw=drawColor,line width= 0.6pt,line join=round] ( 59.40, 35.17) --
	( 59.40,135.78);

\path[draw=drawColor,line width= 0.6pt,line join=round] (142.20, 35.17) --
	(142.20,135.78);

\path[draw=drawColor,line width= 0.6pt,line join=round] (225.00, 35.17) --
	(225.00,135.78);

\path[draw=drawColor,line width= 0.6pt,line join=round] (307.80, 35.17) --
	(307.80,135.78);

\path[draw=drawColor,line width= 0.6pt,line join=round] (390.60, 35.17) --
	(390.60,135.78);
\definecolor[named]{drawColor}{rgb}{0.00,0.00,0.00}
\definecolor[named]{fillColor}{rgb}{0.00,0.00,0.00}

\path[draw=drawColor,line width= 0.6pt,dash pattern=on 4pt off 4pt ,line join=round,fill=fillColor] (229.97, 35.17) -- (229.97,135.78);

\node[text=drawColor,anchor=base east,inner sep=0pt, outer sep=0pt, scale=  1.29] at (229.97,121.49) {1.03};
\definecolor[named]{fillColor}{rgb}{0.20,0.20,0.20}

\path[fill=fillColor,fill opacity=0.50] ( 59.40,121.45) --
	(142.20,107.88) --
	(208.44, 72.98) --
	(210.10, 71.44) --
	(211.76, 70.65) --
	(213.41, 69.22) --
	(215.07, 68.05) --
	(216.72, 67.10) --
	(218.38, 64.81) --
	(220.04, 64.90) --
	(221.69, 63.20) --
	(223.35, 63.12) --
	(225.00, 61.85) --
	(226.66, 59.95) --
	(228.32, 60.12) --
	(229.97, 59.55) --
	(231.63, 60.58) --
	(233.28, 61.69) --
	(234.94, 62.36) --
	(236.60, 62.66) --
	(238.25, 63.14) --
	(239.91, 64.03) --
	(241.56, 65.39) --
	(307.80,107.30) --
	(390.60,131.20) --
	(390.60, 40.16) --
	(307.80, 39.74) --
	(241.56, 47.12) --
	(239.91, 47.88) --
	(238.25, 47.96) --
	(236.60, 47.92) --
	(234.94, 47.79) --
	(233.28, 47.84) --
	(231.63, 47.76) --
	(229.97, 47.91) --
	(228.32, 47.88) --
	(226.66, 48.12) --
	(225.00, 47.34) --
	(223.35, 46.75) --
	(221.69, 46.11) --
	(220.04, 45.74) --
	(218.38, 45.68) --
	(216.72, 45.75) --
	(215.07, 45.84) --
	(213.41, 45.27) --
	(211.76, 45.39) --
	(210.10, 45.46) --
	(208.44, 45.56) --
	(142.20, 43.00) --
	( 59.40, 42.03) --
	cycle;

\path[draw=drawColor,line width= 0.6pt,line join=round] ( 59.40, 81.74) --
	(142.20, 75.44) --
	(208.44, 59.27) --
	(210.10, 58.45) --
	(211.76, 58.02) --
	(213.41, 57.25) --
	(215.07, 56.95) --
	(216.72, 56.42) --
	(218.38, 55.25) --
	(220.04, 55.32) --
	(221.69, 54.65) --
	(223.35, 54.94) --
	(225.00, 54.59) --
	(226.66, 54.04) --
	(228.32, 54.00) --
	(229.97, 53.73) --
	(231.63, 54.17) --
	(233.28, 54.76) --
	(234.94, 55.07) --
	(236.60, 55.29) --
	(238.25, 55.55) --
	(239.91, 55.95) --
	(241.56, 56.26) --
	(307.80, 73.52) --
	(390.60, 85.68);
\end{scope}
\begin{scope}
\path[clip] (  0.00,  0.00) rectangle (433.62,252.94);
\definecolor[named]{drawColor}{rgb}{0.50,0.50,0.50}

\node[text=drawColor,anchor=base east,inner sep=0pt, outer sep=0pt, scale=  0.87] at ( 35.73,149.78) {152};

\node[text=drawColor,anchor=base east,inner sep=0pt, outer sep=0pt, scale=  0.87] at ( 35.73,170.29) {153};

\node[text=drawColor,anchor=base east,inner sep=0pt, outer sep=0pt, scale=  0.87] at ( 35.73,190.81) {154};

\node[text=drawColor,anchor=base east,inner sep=0pt, outer sep=0pt, scale=  0.87] at ( 35.73,211.32) {155};

\node[text=drawColor,anchor=base east,inner sep=0pt, outer sep=0pt, scale=  0.87] at ( 35.73,231.83) {156};
\end{scope}
\begin{scope}
\path[clip] (  0.00,  0.00) rectangle (433.62,252.94);
\definecolor[named]{drawColor}{rgb}{0.50,0.50,0.50}

\path[draw=drawColor,line width= 0.6pt,line join=round] ( 38.58,153.07) --
	( 42.84,153.07);

\path[draw=drawColor,line width= 0.6pt,line join=round] ( 38.58,173.58) --
	( 42.84,173.58);

\path[draw=drawColor,line width= 0.6pt,line join=round] ( 38.58,194.10) --
	( 42.84,194.10);

\path[draw=drawColor,line width= 0.6pt,line join=round] ( 38.58,214.61) --
	( 42.84,214.61);

\path[draw=drawColor,line width= 0.6pt,line join=round] ( 38.58,235.12) --
	( 42.84,235.12);
\end{scope}
\begin{scope}
\path[clip] (  0.00,  0.00) rectangle (433.62,252.94);
\definecolor[named]{drawColor}{rgb}{0.50,0.50,0.50}

\node[text=drawColor,anchor=base east,inner sep=0pt, outer sep=0pt, scale=  0.87] at ( 35.73, 35.54) {4};

\node[text=drawColor,anchor=base east,inner sep=0pt, outer sep=0pt, scale=  0.87] at ( 35.73, 60.51) {8};

\node[text=drawColor,anchor=base east,inner sep=0pt, outer sep=0pt, scale=  0.87] at ( 35.73, 85.47) {12};

\node[text=drawColor,anchor=base east,inner sep=0pt, outer sep=0pt, scale=  0.87] at ( 35.73,110.44) {16};
\end{scope}
\begin{scope}
\path[clip] (  0.00,  0.00) rectangle (433.62,252.94);
\definecolor[named]{drawColor}{rgb}{0.50,0.50,0.50}

\path[draw=drawColor,line width= 0.6pt,line join=round] ( 38.58, 38.83) --
	( 42.84, 38.83);

\path[draw=drawColor,line width= 0.6pt,line join=round] ( 38.58, 63.80) --
	( 42.84, 63.80);

\path[draw=drawColor,line width= 0.6pt,line join=round] ( 38.58, 88.76) --
	( 42.84, 88.76);

\path[draw=drawColor,line width= 0.6pt,line join=round] ( 38.58,113.73) --
	( 42.84,113.73);
\end{scope}
\begin{scope}
\path[clip] (407.16,139.09) rectangle (420.37,239.70);
\definecolor[named]{fillColor}{rgb}{0.80,0.80,0.80}

\path[fill=fillColor] (407.16,139.09) rectangle (420.37,239.70);
\definecolor[named]{drawColor}{rgb}{0.00,0.00,0.00}

\node[text=drawColor,rotate=270.00,anchor=base,inner sep=0pt, outer sep=0pt, scale=  0.87] at (410.48,189.39) {train};
\end{scope}
\begin{scope}
\path[clip] (407.16, 35.17) rectangle (420.37,135.78);
\definecolor[named]{fillColor}{rgb}{0.80,0.80,0.80}

\path[fill=fillColor] (407.16, 35.17) rectangle (420.37,135.78);
\definecolor[named]{drawColor}{rgb}{0.00,0.00,0.00}

\node[text=drawColor,rotate=270.00,anchor=base,inner sep=0pt, outer sep=0pt, scale=  0.87] at (410.48, 85.47) {test};
\end{scope}
\begin{scope}
\path[clip] (  0.00,  0.00) rectangle (433.62,252.94);
\definecolor[named]{drawColor}{rgb}{0.50,0.50,0.50}

\path[draw=drawColor,line width= 0.6pt,line join=round] ( 59.40, 30.90) --
	( 59.40, 35.17);

\path[draw=drawColor,line width= 0.6pt,line join=round] (142.20, 30.90) --
	(142.20, 35.17);

\path[draw=drawColor,line width= 0.6pt,line join=round] (225.00, 30.90) --
	(225.00, 35.17);

\path[draw=drawColor,line width= 0.6pt,line join=round] (307.80, 30.90) --
	(307.80, 35.17);

\path[draw=drawColor,line width= 0.6pt,line join=round] (390.60, 30.90) --
	(390.60, 35.17);
\end{scope}
\begin{scope}
\path[clip] (  0.00,  0.00) rectangle (433.62,252.94);
\definecolor[named]{drawColor}{rgb}{0.50,0.50,0.50}

\node[text=drawColor,anchor=base,inner sep=0pt, outer sep=0pt, scale=  0.87] at ( 59.40, 21.48) {0.0};

\node[text=drawColor,anchor=base,inner sep=0pt, outer sep=0pt, scale=  0.87] at (142.20, 21.48) {0.5};

\node[text=drawColor,anchor=base,inner sep=0pt, outer sep=0pt, scale=  0.87] at (225.00, 21.48) {1.0};

\node[text=drawColor,anchor=base,inner sep=0pt, outer sep=0pt, scale=  0.87] at (307.80, 21.48) {1.5};

\node[text=drawColor,anchor=base,inner sep=0pt, outer sep=0pt, scale=  0.87] at (390.60, 21.48) {2.0};
\end{scope}
\begin{scope}
\path[clip] (  0.00,  0.00) rectangle (433.62,252.94);
\definecolor[named]{drawColor}{rgb}{0.00,0.00,0.00}

\node[text=drawColor,anchor=base,inner sep=0pt, outer sep=0pt, scale=  1.09] at (225.00,  9.94) {penalty exponent $\alpha$};
\end{scope}
\begin{scope}
\path[clip] (  0.00,  0.00) rectangle (433.62,252.94);
\definecolor[named]{drawColor}{rgb}{0.00,0.00,0.00}

\node[text=drawColor,rotate= 90.00,anchor=base,inner sep=0pt, outer sep=0pt, scale=  1.09] at ( 18.16,137.43) {total error relative to true breakpoints (breakpointError)};
\end{scope}
\end{tikzpicture}

%% file: table-penalty-real-data.tex
\begin{tabular}{rrrrrrr}
  \hline
 & points & length & variance & train & test.mean & test.sd \\ 
  \hline
cghseg.k & 1 & 0 & 0 & 2.19 & 2.20 & 1.01 \\ 
  cghseg.k.var & 1 & 0 & 2 & 2.46 & 2.73 & 1.98 \\ 
  cghseg.k.sqrt.d & $1/2$ & 0 & 0 & 3.51 & 3.87 & 1.58 \\ 
  cghseg.k.sqrt & $1/2$ & $-1/2$ & 0 & 4.30 & 6.11 & 5.02 \\ 
  cghseg.k.sqrt.d.var & $1/2$ & 0 & 2 & 3.19 & 4.47 & 5.02 \\ 
  cghseg.k.sqrt.var & $1/2$ & $-1/2$ & 2 & 4.18 & 6.38 & 7.61 \\ 
   \hline
\end{tabular}

%% file: figure-variable-density-sigerr-small.tex
\begin{tikzpicture}[x=1pt,y=1pt]
\definecolor[named]{fillColor}{rgb}{1.00,1.00,1.00}
\path[use as bounding box,fill=fillColor,fill opacity=0.00] (0,0) rectangle (361.35,231.26);
\begin{scope}
\path[clip] (  0.00,  0.00) rectangle (361.35,231.26);
\definecolor[named]{drawColor}{rgb}{1.00,1.00,1.00}
\definecolor[named]{fillColor}{rgb}{1.00,1.00,1.00}

\path[draw=drawColor,line width= 0.6pt,line join=round,line cap=round,fill=fillColor] (  0.00,  0.00) rectangle (361.35,231.26);
\end{scope}
\begin{scope}
\path[clip] ( 38.09,204.81) rectangle (186.49,218.01);
\definecolor[named]{drawColor}{rgb}{0.50,0.50,0.50}
\definecolor[named]{fillColor}{rgb}{0.80,0.80,0.80}

\path[draw=drawColor,line width= 0.2pt,line join=round,line cap=round,fill=fillColor] ( 38.09,204.81) rectangle (186.49,218.01);
\definecolor[named]{drawColor}{rgb}{0.00,0.00,0.00}

\node[text=drawColor,anchor=base,inner sep=0pt, outer sep=0pt, scale=  0.87] at (112.29,208.12) {bases/probe = 374};
\end{scope}
\begin{scope}
\path[clip] (186.49,204.81) rectangle (334.89,218.01);
\definecolor[named]{drawColor}{rgb}{0.50,0.50,0.50}
\definecolor[named]{fillColor}{rgb}{0.80,0.80,0.80}

\path[draw=drawColor,line width= 0.2pt,line join=round,line cap=round,fill=fillColor] (186.49,204.81) rectangle (334.89,218.01);
\definecolor[named]{drawColor}{rgb}{0.00,0.00,0.00}

\node[text=drawColor,anchor=base,inner sep=0pt, outer sep=0pt, scale=  0.87] at (260.69,208.12) {bases/probe = 7};
\end{scope}
\begin{scope}
\path[clip] ( 38.09,148.26) rectangle (186.49,204.81);
\definecolor[named]{fillColor}{rgb}{1.00,1.00,1.00}

\path[fill=fillColor] ( 38.09,148.26) rectangle (186.49,204.81);
\definecolor[named]{drawColor}{rgb}{0.90,0.90,0.90}

\path[draw=drawColor,line width= 0.2pt,line join=round] ( 38.09,160.89) --
	(186.49,160.89);

\path[draw=drawColor,line width= 0.2pt,line join=round] ( 38.09,179.91) --
	(186.49,179.91);

\path[draw=drawColor,line width= 0.2pt,line join=round] ( 38.09,198.93) --
	(186.49,198.93);

\path[draw=drawColor,line width= 0.2pt,line join=round] ( 56.57,148.26) --
	( 56.57,204.81);

\path[draw=drawColor,line width= 0.2pt,line join=round] ( 80.03,148.26) --
	( 80.03,204.81);

\path[draw=drawColor,line width= 0.2pt,line join=round] ( 91.76,148.26) --
	( 91.76,204.81);

\path[draw=drawColor,line width= 0.2pt,line join=round] (168.02,148.26) --
	(168.02,204.81);
\definecolor[named]{drawColor}{rgb}{0.00,0.00,0.00}

\path[draw=drawColor,line width= 0.9pt,line join=round] ( 56.57,192.73) --
	( 62.43,187.32) --
	( 68.30,188.70) --
	( 74.17,186.54) --
	( 80.03,183.06) --
	( 85.90,186.61) --
	( 91.76,183.55) --
	( 97.63,183.90) --
	(103.49,185.82) --
	(109.36,186.54) --
	(115.23,187.51) --
	(121.09,188.02) --
	(126.96,188.74) --
	(132.82,189.18) --
	(138.69,189.52) --
	(144.55,190.06) --
	(150.42,190.20) --
	(156.28,190.63) --
	(162.15,190.98) --
	(168.02,191.24);

\path[draw=drawColor,line width= 0.4pt,line join=round,line cap=round] ( 80.03,183.06) circle (  2.13);
\end{scope}
\begin{scope}
\path[clip] ( 38.09,148.26) rectangle (186.49,204.81);
\definecolor[named]{drawColor}{rgb}{0.00,0.00,0.00}

\node[text=drawColor,anchor=base east,inner sep=0pt, outer sep=0pt, scale=  1.00] at ( 53.72,188.96) {$E$};

\node[text=drawColor,anchor=base west,inner sep=0pt, outer sep=0pt, scale=  1.00] at (170.86,187.47) {$E$};
\definecolor[named]{drawColor}{rgb}{0.50,0.50,0.50}

\path[draw=drawColor,line width= 0.6pt,line join=round,line cap=round] ( 38.09,148.26) rectangle (186.49,204.81);
\end{scope}
\begin{scope}
\path[clip] ( 38.09, 91.72) rectangle (186.49,148.26);
\definecolor[named]{fillColor}{rgb}{1.00,1.00,1.00}

\path[fill=fillColor] ( 38.09, 91.72) rectangle (186.49,148.26);
\definecolor[named]{drawColor}{rgb}{0.90,0.90,0.90}

\path[draw=drawColor,line width= 0.2pt,line join=round] ( 38.09,103.53) --
	(186.49,103.53);

\path[draw=drawColor,line width= 0.2pt,line join=round] ( 38.09,116.63) --
	(186.49,116.63);

\path[draw=drawColor,line width= 0.2pt,line join=round] ( 38.09,129.73) --
	(186.49,129.73);

\path[draw=drawColor,line width= 0.2pt,line join=round] ( 38.09,142.83) --
	(186.49,142.83);

\path[draw=drawColor,line width= 0.2pt,line join=round] ( 56.57, 91.72) --
	( 56.57,148.26);

\path[draw=drawColor,line width= 0.2pt,line join=round] ( 80.03, 91.72) --
	( 80.03,148.26);

\path[draw=drawColor,line width= 0.2pt,line join=round] ( 91.76, 91.72) --
	( 91.76,148.26);

\path[draw=drawColor,line width= 0.2pt,line join=round] (168.02, 91.72) --
	(168.02,148.26);
\definecolor[named]{drawColor}{rgb}{0.53,0.81,0.92}

\path[draw=drawColor,line width= 2.8pt,line join=round] ( 56.57,103.53) --
	( 62.43,103.53) --
	( 68.30,103.53) --
	( 74.17,103.53) --
	( 80.03,103.53) --
	( 85.90,103.53) --
	( 91.76,103.53) --
	( 97.63,106.15) --
	(103.49,108.77) --
	(109.36,111.39) --
	(115.23,114.01) --
	(121.09,116.63) --
	(126.96,119.25) --
	(132.82,121.87) --
	(138.69,124.49) --
	(144.55,127.11) --
	(150.42,129.73) --
	(156.28,132.35) --
	(162.15,134.97) --
	(168.02,137.59);
\definecolor[named]{drawColor}{rgb}{0.89,0.10,0.11}

\path[draw=drawColor,line width= 2.8pt,line join=round] ( 56.57,119.25) --
	( 62.43,116.63) --
	( 68.30,114.01) --
	( 74.17,111.39) --
	( 80.03,108.77) --
	( 85.90,106.15) --
	( 91.76,103.53) --
	( 97.63,103.53) --
	(103.49,103.53) --
	(109.36,103.53) --
	(115.23,103.53) --
	(121.09,103.53) --
	(126.96,103.53) --
	(132.82,103.53) --
	(138.69,103.53) --
	(144.55,103.53) --
	(150.42,103.53) --
	(156.28,103.53) --
	(162.15,103.53) --
	(168.02,103.53);
\definecolor[named]{drawColor}{rgb}{0.00,0.00,0.00}

\path[draw=drawColor,line width= 0.9pt,dash pattern=on 4pt off 4pt ,line join=round] ( 56.57,103.53) --
	( 62.43,103.91) --
	( 68.30,106.38) --
	( 74.17,106.78) --
	( 80.03,105.13) --
	( 85.90,108.19) --
	( 91.76,106.54) --
	( 97.63,106.93) --
	(103.49,106.93) --
	(109.36,106.40) --
	(115.23,106.40) --
	(121.09,106.40) --
	(126.96,106.40) --
	(132.82,106.40) --
	(138.69,106.40) --
	(144.55,106.40) --
	(150.42,106.40) --
	(156.28,106.40) --
	(162.15,106.40) --
	(168.02,106.40);

\path[draw=drawColor,line width= 0.9pt,line join=round] ( 56.57,119.25) --
	( 62.43,117.01) --
	( 68.30,116.86) --
	( 74.17,114.64) --
	( 80.03,110.37) --
	( 85.90,110.81) --
	( 91.76,106.54) --
	( 97.63,109.55) --
	(103.49,112.17) --
	(109.36,114.26) --
	(115.23,116.88) --
	(121.09,119.50) --
	(126.96,122.12) --
	(132.82,124.74) --
	(138.69,127.35) --
	(144.55,129.97) --
	(150.42,132.59) --
	(156.28,135.21) --
	(162.15,137.83) --
	(168.02,140.45);

\path[draw=drawColor,line width= 0.4pt,line join=round,line cap=round] ( 91.76,106.54) circle (  2.13);
\end{scope}
\begin{scope}
\path[clip] ( 38.09, 91.72) rectangle (186.49,148.26);
\definecolor[named]{drawColor}{rgb}{0.89,0.10,0.11}

\node[text=drawColor,anchor=base east,inner sep=0pt, outer sep=0pt, scale=  1.00] at ( 53.72,112.52) {FN};

\node[text=drawColor,anchor=base west,inner sep=0pt, outer sep=0pt, scale=  1.00] at (170.86, 96.29) {FN};
\definecolor[named]{drawColor}{rgb}{0.00,0.00,0.00}

\node[text=drawColor,anchor=base east,inner sep=0pt, outer sep=0pt, scale=  1.00] at ( 53.72,102.72) {$I$};

\node[text=drawColor,anchor=base east,inner sep=0pt, outer sep=0pt, scale=  1.00] at ( 53.72,122.33) {$E$};

\node[text=drawColor,anchor=base west,inner sep=0pt, outer sep=0pt, scale=  1.00] at (170.86,106.10) {$I$};

\node[text=drawColor,anchor=base west,inner sep=0pt, outer sep=0pt, scale=  1.00] at (170.86,138.46) {$E$};
\definecolor[named]{drawColor}{rgb}{0.53,0.81,0.92}

\node[text=drawColor,anchor=base east,inner sep=0pt, outer sep=0pt, scale=  1.00] at ( 53.72, 92.91) {FP};

\node[text=drawColor,anchor=base west,inner sep=0pt, outer sep=0pt, scale=  1.00] at (170.86,128.65) {FP};
\definecolor[named]{drawColor}{rgb}{0.50,0.50,0.50}

\path[draw=drawColor,line width= 0.6pt,line join=round,line cap=round] ( 38.09, 91.72) rectangle (186.49,148.26);
\end{scope}
\begin{scope}
\path[clip] ( 38.09, 35.17) rectangle (186.49, 91.72);
\definecolor[named]{fillColor}{rgb}{1.00,1.00,1.00}

\path[fill=fillColor] ( 38.09, 35.17) rectangle (186.49, 91.72);
\definecolor[named]{drawColor}{rgb}{0.90,0.90,0.90}

\path[draw=drawColor,line width= 0.2pt,line join=round] ( 38.09, 46.99) --
	(186.49, 46.99);

\path[draw=drawColor,line width= 0.2pt,line join=round] ( 38.09, 60.08) --
	(186.49, 60.08);

\path[draw=drawColor,line width= 0.2pt,line join=round] ( 38.09, 73.18) --
	(186.49, 73.18);

\path[draw=drawColor,line width= 0.2pt,line join=round] ( 38.09, 86.28) --
	(186.49, 86.28);

\path[draw=drawColor,line width= 0.2pt,line join=round] ( 56.57, 35.17) --
	( 56.57, 91.72);

\path[draw=drawColor,line width= 0.2pt,line join=round] ( 80.03, 35.17) --
	( 80.03, 91.72);

\path[draw=drawColor,line width= 0.2pt,line join=round] ( 91.76, 35.17) --
	( 91.76, 91.72);

\path[draw=drawColor,line width= 0.2pt,line join=round] (168.02, 35.17) --
	(168.02, 91.72);
\definecolor[named]{drawColor}{rgb}{0.89,0.10,0.11}

\path[draw=drawColor,line width= 2.8pt,line join=round] ( 56.57, 62.70) --
	( 62.43, 60.08) --
	( 68.30, 60.08) --
	( 74.17, 57.47) --
	( 80.03, 52.23) --
	( 85.90, 52.23) --
	( 91.76, 46.99) --
	( 97.63, 46.99) --
	(103.49, 46.99) --
	(109.36, 46.99) --
	(115.23, 46.99) --
	(121.09, 46.99) --
	(126.96, 46.99) --
	(132.82, 46.99) --
	(138.69, 46.99) --
	(144.55, 46.99) --
	(150.42, 46.99) --
	(156.28, 46.99) --
	(162.15, 46.99) --
	(168.02, 46.99);
\definecolor[named]{drawColor}{rgb}{0.53,0.81,0.92}

\path[draw=drawColor,line width= 2.8pt,line join=round] ( 56.57, 46.99) --
	( 62.43, 46.99) --
	( 68.30, 49.61) --
	( 74.17, 49.61) --
	( 80.03, 46.99) --
	( 85.90, 49.61) --
	( 91.76, 46.99) --
	( 97.63, 49.61) --
	(103.49, 52.23) --
	(109.36, 54.85) --
	(115.23, 57.47) --
	(121.09, 60.08) --
	(126.96, 62.70) --
	(132.82, 65.32) --
	(138.69, 67.94) --
	(144.55, 70.56) --
	(150.42, 73.18) --
	(156.28, 75.80) --
	(162.15, 78.42) --
	(168.02, 81.04);
\definecolor[named]{drawColor}{rgb}{0.00,0.00,0.00}

\path[draw=drawColor,line width= 0.9pt,line join=round] ( 56.57, 62.70) --
	( 62.43, 60.08) --
	( 68.30, 62.70) --
	( 74.17, 60.08) --
	( 80.03, 52.23) --
	( 85.90, 54.85) --
	( 91.76, 46.99) --
	( 97.63, 49.61) --
	(103.49, 52.23) --
	(109.36, 54.85) --
	(115.23, 57.47) --
	(121.09, 60.08) --
	(126.96, 62.70) --
	(132.82, 65.32) --
	(138.69, 67.94) --
	(144.55, 70.56) --
	(150.42, 73.18) --
	(156.28, 75.80) --
	(162.15, 78.42) --
	(168.02, 81.04);

\path[draw=drawColor,line width= 0.4pt,line join=round,line cap=round] ( 91.76, 46.99) circle (  2.13);
\end{scope}
\begin{scope}
\path[clip] ( 38.09, 35.17) rectangle (186.49, 91.72);
\definecolor[named]{drawColor}{rgb}{0.89,0.10,0.11}

\node[text=drawColor,anchor=base east,inner sep=0pt, outer sep=0pt, scale=  1.00] at ( 53.72, 53.69) {$\hat{\text{FN}}$};

\node[text=drawColor,anchor=base west,inner sep=0pt, outer sep=0pt, scale=  1.00] at (170.86, 43.22) {$\hat{\text{FN}}$};
\definecolor[named]{drawColor}{rgb}{0.00,0.00,0.00}

\node[text=drawColor,anchor=base east,inner sep=0pt, outer sep=0pt, scale=  1.00] at ( 53.72, 66.51) {$E$};

\node[text=drawColor,anchor=base west,inner sep=0pt, outer sep=0pt, scale=  1.00] at (170.86, 81.91) {$E$};
\definecolor[named]{drawColor}{rgb}{0.53,0.81,0.92}

\node[text=drawColor,anchor=base east,inner sep=0pt, outer sep=0pt, scale=  1.00] at ( 53.72, 40.87) {$\hat{\text{FP}}$};

\node[text=drawColor,anchor=base west,inner sep=0pt, outer sep=0pt, scale=  1.00] at (170.86, 69.09) {$\hat{\text{FP}}$};
\definecolor[named]{drawColor}{rgb}{0.50,0.50,0.50}

\path[draw=drawColor,line width= 0.6pt,line join=round,line cap=round] ( 38.09, 35.17) rectangle (186.49, 91.72);
\end{scope}
\begin{scope}
\path[clip] (186.49,148.26) rectangle (334.89,204.81);
\definecolor[named]{fillColor}{rgb}{1.00,1.00,1.00}

\path[fill=fillColor] (186.49,148.26) rectangle (334.89,204.81);
\definecolor[named]{drawColor}{rgb}{0.90,0.90,0.90}

\path[draw=drawColor,line width= 0.2pt,line join=round] (186.49,160.89) --
	(334.89,160.89);

\path[draw=drawColor,line width= 0.2pt,line join=round] (186.49,179.91) --
	(334.89,179.91);

\path[draw=drawColor,line width= 0.2pt,line join=round] (186.49,198.93) --
	(334.89,198.93);

\path[draw=drawColor,line width= 0.2pt,line join=round] (204.97,148.26) --
	(204.97,204.81);

\path[draw=drawColor,line width= 0.2pt,line join=round] (228.43,148.26) --
	(228.43,204.81);

\path[draw=drawColor,line width= 0.2pt,line join=round] (240.16,148.26) --
	(240.16,204.81);

\path[draw=drawColor,line width= 0.2pt,line join=round] (316.42,148.26) --
	(316.42,204.81);
\definecolor[named]{drawColor}{rgb}{0.00,0.00,0.00}

\path[draw=drawColor,line width= 0.9pt,line join=round] (204.97,192.73) --
	(210.84,187.00) --
	(216.70,184.38) --
	(222.57,180.53) --
	(228.43,177.51) --
	(234.30,170.82) --
	(240.16,150.83) --
	(246.03,153.32) --
	(251.90,155.84) --
	(257.76,157.87) --
	(263.63,158.82) --
	(269.49,160.28) --
	(275.36,160.94) --
	(281.22,161.98) --
	(287.09,162.67) --
	(292.95,163.42) --
	(298.82,163.87) --
	(304.69,164.39) --
	(310.55,165.00) --
	(316.42,165.39);

\path[draw=drawColor,line width= 0.4pt,line join=round,line cap=round] (240.16,150.83) circle (  2.13);
\end{scope}
\begin{scope}
\path[clip] (186.49,148.26) rectangle (334.89,204.81);
\definecolor[named]{drawColor}{rgb}{0.00,0.00,0.00}

\node[text=drawColor,anchor=base east,inner sep=0pt, outer sep=0pt, scale=  1.00] at (202.12,188.96) {$E$};

\node[text=drawColor,anchor=base west,inner sep=0pt, outer sep=0pt, scale=  1.00] at (319.26,161.62) {$E$};
\definecolor[named]{drawColor}{rgb}{0.50,0.50,0.50}

\path[draw=drawColor,line width= 0.6pt,line join=round,line cap=round] (186.49,148.26) rectangle (334.89,204.81);
\end{scope}
\begin{scope}
\path[clip] (186.49, 91.72) rectangle (334.89,148.26);
\definecolor[named]{fillColor}{rgb}{1.00,1.00,1.00}

\path[fill=fillColor] (186.49, 91.72) rectangle (334.89,148.26);
\definecolor[named]{drawColor}{rgb}{0.90,0.90,0.90}

\path[draw=drawColor,line width= 0.2pt,line join=round] (186.49,103.53) --
	(334.89,103.53);

\path[draw=drawColor,line width= 0.2pt,line join=round] (186.49,116.63) --
	(334.89,116.63);

\path[draw=drawColor,line width= 0.2pt,line join=round] (186.49,129.73) --
	(334.89,129.73);

\path[draw=drawColor,line width= 0.2pt,line join=round] (186.49,142.83) --
	(334.89,142.83);

\path[draw=drawColor,line width= 0.2pt,line join=round] (204.97, 91.72) --
	(204.97,148.26);

\path[draw=drawColor,line width= 0.2pt,line join=round] (228.43, 91.72) --
	(228.43,148.26);

\path[draw=drawColor,line width= 0.2pt,line join=round] (240.16, 91.72) --
	(240.16,148.26);

\path[draw=drawColor,line width= 0.2pt,line join=round] (316.42, 91.72) --
	(316.42,148.26);
\definecolor[named]{drawColor}{rgb}{0.53,0.81,0.92}

\path[draw=drawColor,line width= 2.8pt,line join=round] (204.97,103.53) --
	(210.84,103.53) --
	(216.70,103.53) --
	(222.57,103.53) --
	(228.43,103.53) --
	(234.30,103.53) --
	(240.16,103.53) --
	(246.03,106.15) --
	(251.90,108.77) --
	(257.76,111.39) --
	(263.63,114.01) --
	(269.49,116.63) --
	(275.36,119.25) --
	(281.22,121.87) --
	(287.09,124.49) --
	(292.95,127.11) --
	(298.82,129.73) --
	(304.69,132.35) --
	(310.55,134.97) --
	(316.42,137.59);
\definecolor[named]{drawColor}{rgb}{0.89,0.10,0.11}

\path[draw=drawColor,line width= 2.8pt,line join=round] (204.97,119.25) --
	(210.84,116.63) --
	(216.70,114.01) --
	(222.57,111.39) --
	(228.43,108.77) --
	(234.30,106.15) --
	(240.16,103.53) --
	(246.03,103.53) --
	(251.90,103.53) --
	(257.76,103.53) --
	(263.63,103.53) --
	(269.49,103.53) --
	(275.36,103.53) --
	(281.22,103.53) --
	(287.09,103.53) --
	(292.95,103.53) --
	(298.82,103.53) --
	(304.69,103.53) --
	(310.55,103.53) --
	(316.42,103.53);
\definecolor[named]{drawColor}{rgb}{0.00,0.00,0.00}

\path[draw=drawColor,line width= 0.9pt,dash pattern=on 4pt off 4pt ,line join=round] (204.97,103.53) --
	(210.84,103.54) --
	(216.70,103.53) --
	(222.57,103.55) --
	(228.43,103.55) --
	(234.30,103.62) --
	(240.16,103.62) --
	(246.03,103.62) --
	(251.90,103.62) --
	(257.76,103.62) --
	(263.63,103.62) --
	(269.49,103.62) --
	(275.36,103.62) --
	(281.22,103.62) --
	(287.09,103.62) --
	(292.95,103.62) --
	(298.82,103.62) --
	(304.69,103.62) --
	(310.55,103.62) --
	(316.42,103.62);

\path[draw=drawColor,line width= 0.9pt,line join=round] (204.97,119.25) --
	(210.84,116.64) --
	(216.70,114.01) --
	(222.57,111.41) --
	(228.43,108.79) --
	(234.30,106.23) --
	(240.16,103.62) --
	(246.03,106.24) --
	(251.90,108.86) --
	(257.76,111.48) --
	(263.63,114.10) --
	(269.49,116.72) --
	(275.36,119.34) --
	(281.22,121.96) --
	(287.09,124.58) --
	(292.95,127.20) --
	(298.82,129.81) --
	(304.69,132.43) --
	(310.55,135.05) --
	(316.42,137.67);

\path[draw=drawColor,line width= 0.4pt,line join=round,line cap=round] (240.16,103.62) circle (  2.13);
\end{scope}
\begin{scope}
\path[clip] (186.49, 91.72) rectangle (334.89,148.26);
\definecolor[named]{drawColor}{rgb}{0.89,0.10,0.11}

\node[text=drawColor,anchor=base east,inner sep=0pt, outer sep=0pt, scale=  1.00] at (202.12,112.52) {FN};

\node[text=drawColor,anchor=base west,inner sep=0pt, outer sep=0pt, scale=  1.00] at (319.26, 94.90) {FN};
\definecolor[named]{drawColor}{rgb}{0.00,0.00,0.00}

\node[text=drawColor,anchor=base east,inner sep=0pt, outer sep=0pt, scale=  1.00] at (202.12,102.72) {$I$};

\node[text=drawColor,anchor=base east,inner sep=0pt, outer sep=0pt, scale=  1.00] at (202.12,122.33) {$E$};

\node[text=drawColor,anchor=base west,inner sep=0pt, outer sep=0pt, scale=  1.00] at (319.26,104.71) {$I$};

\node[text=drawColor,anchor=base west,inner sep=0pt, outer sep=0pt, scale=  1.00] at (319.26,138.46) {$E$};
\definecolor[named]{drawColor}{rgb}{0.53,0.81,0.92}

\node[text=drawColor,anchor=base east,inner sep=0pt, outer sep=0pt, scale=  1.00] at (202.12, 92.91) {FP};

\node[text=drawColor,anchor=base west,inner sep=0pt, outer sep=0pt, scale=  1.00] at (319.26,128.65) {FP};
\definecolor[named]{drawColor}{rgb}{0.50,0.50,0.50}

\path[draw=drawColor,line width= 0.6pt,line join=round,line cap=round] (186.49, 91.72) rectangle (334.89,148.26);
\end{scope}
\begin{scope}
\path[clip] (186.49, 35.17) rectangle (334.89, 91.72);
\definecolor[named]{fillColor}{rgb}{1.00,1.00,1.00}

\path[fill=fillColor] (186.49, 35.17) rectangle (334.89, 91.72);
\definecolor[named]{drawColor}{rgb}{0.90,0.90,0.90}

\path[draw=drawColor,line width= 0.2pt,line join=round] (186.49, 46.99) --
	(334.89, 46.99);

\path[draw=drawColor,line width= 0.2pt,line join=round] (186.49, 60.08) --
	(334.89, 60.08);

\path[draw=drawColor,line width= 0.2pt,line join=round] (186.49, 73.18) --
	(334.89, 73.18);

\path[draw=drawColor,line width= 0.2pt,line join=round] (186.49, 86.28) --
	(334.89, 86.28);

\path[draw=drawColor,line width= 0.2pt,line join=round] (204.97, 35.17) --
	(204.97, 91.72);

\path[draw=drawColor,line width= 0.2pt,line join=round] (228.43, 35.17) --
	(228.43, 91.72);

\path[draw=drawColor,line width= 0.2pt,line join=round] (240.16, 35.17) --
	(240.16, 91.72);

\path[draw=drawColor,line width= 0.2pt,line join=round] (316.42, 35.17) --
	(316.42, 91.72);
\definecolor[named]{drawColor}{rgb}{0.89,0.10,0.11}

\path[draw=drawColor,line width= 2.8pt,line join=round] (204.97, 62.70) --
	(210.84, 60.08) --
	(216.70, 57.47) --
	(222.57, 54.85) --
	(228.43, 52.23) --
	(234.30, 49.61) --
	(240.16, 46.99) --
	(246.03, 46.99) --
	(251.90, 46.99) --
	(257.76, 46.99) --
	(263.63, 46.99) --
	(269.49, 46.99) --
	(275.36, 46.99) --
	(281.22, 46.99) --
	(287.09, 46.99) --
	(292.95, 46.99) --
	(298.82, 46.99) --
	(304.69, 46.99) --
	(310.55, 46.99) --
	(316.42, 46.99);
\definecolor[named]{drawColor}{rgb}{0.53,0.81,0.92}

\path[draw=drawColor,line width= 2.8pt,line join=round] (204.97, 46.99) --
	(210.84, 46.99) --
	(216.70, 46.99) --
	(222.57, 46.99) --
	(228.43, 46.99) --
	(234.30, 46.99) --
	(240.16, 46.99) --
	(246.03, 49.61) --
	(251.90, 52.23) --
	(257.76, 54.85) --
	(263.63, 57.47) --
	(269.49, 60.08) --
	(275.36, 62.70) --
	(281.22, 65.32) --
	(287.09, 67.94) --
	(292.95, 70.56) --
	(298.82, 73.18) --
	(304.69, 75.80) --
	(310.55, 78.42) --
	(316.42, 81.04);
\definecolor[named]{drawColor}{rgb}{0.00,0.00,0.00}

\path[draw=drawColor,line width= 0.9pt,line join=round] (204.97, 62.70) --
	(210.84, 60.08) --
	(216.70, 57.47) --
	(222.57, 54.85) --
	(228.43, 52.23) --
	(234.30, 49.61) --
	(240.16, 46.99) --
	(246.03, 49.61) --
	(251.90, 52.23) --
	(257.76, 54.85) --
	(263.63, 57.47) --
	(269.49, 60.08) --
	(275.36, 62.70) --
	(281.22, 65.32) --
	(287.09, 67.94) --
	(292.95, 70.56) --
	(298.82, 73.18) --
	(304.69, 75.80) --
	(310.55, 78.42) --
	(316.42, 81.04);

\path[draw=drawColor,line width= 0.4pt,line join=round,line cap=round] (240.16, 46.99) circle (  2.13);
\end{scope}
\begin{scope}
\path[clip] (186.49, 35.17) rectangle (334.89, 91.72);
\definecolor[named]{drawColor}{rgb}{0.89,0.10,0.11}

\node[text=drawColor,anchor=base east,inner sep=0pt, outer sep=0pt, scale=  1.00] at (202.12, 53.69) {$\hat{\text{FN}}$};

\node[text=drawColor,anchor=base west,inner sep=0pt, outer sep=0pt, scale=  1.00] at (319.26, 43.22) {$\hat{\text{FN}}$};
\definecolor[named]{drawColor}{rgb}{0.00,0.00,0.00}

\node[text=drawColor,anchor=base east,inner sep=0pt, outer sep=0pt, scale=  1.00] at (202.12, 66.51) {$E$};

\node[text=drawColor,anchor=base west,inner sep=0pt, outer sep=0pt, scale=  1.00] at (319.26, 81.91) {$E$};
\definecolor[named]{drawColor}{rgb}{0.53,0.81,0.92}

\node[text=drawColor,anchor=base east,inner sep=0pt, outer sep=0pt, scale=  1.00] at (202.12, 40.87) {$\hat{\text{FP}}$};

\node[text=drawColor,anchor=base west,inner sep=0pt, outer sep=0pt, scale=  1.00] at (319.26, 69.09) {$\hat{\text{FP}}$};
\definecolor[named]{drawColor}{rgb}{0.50,0.50,0.50}

\path[draw=drawColor,line width= 0.6pt,line join=round,line cap=round] (186.49, 35.17) rectangle (334.89, 91.72);
\end{scope}
\begin{scope}
\path[clip] (  0.00,  0.00) rectangle (361.35,231.26);
\definecolor[named]{drawColor}{rgb}{0.00,0.00,0.00}

\node[text=drawColor,anchor=base east,inner sep=0pt, outer sep=0pt, scale=  0.87] at ( 30.98,157.60) {-2};

\node[text=drawColor,anchor=base east,inner sep=0pt, outer sep=0pt, scale=  0.87] at ( 30.98,176.62) {-1};

\node[text=drawColor,anchor=base east,inner sep=0pt, outer sep=0pt, scale=  0.87] at ( 30.98,195.64) {0};
\end{scope}
\begin{scope}
\path[clip] (  0.00,  0.00) rectangle (361.35,231.26);
\definecolor[named]{drawColor}{rgb}{0.00,0.00,0.00}

\path[draw=drawColor,line width= 0.6pt,line join=round] ( 33.82,160.89) --
	( 38.09,160.89);

\path[draw=drawColor,line width= 0.6pt,line join=round] ( 33.82,179.91) --
	( 38.09,179.91);

\path[draw=drawColor,line width= 0.6pt,line join=round] ( 33.82,198.93) --
	( 38.09,198.93);
\end{scope}
\begin{scope}
\path[clip] (  0.00,  0.00) rectangle (361.35,231.26);
\definecolor[named]{drawColor}{rgb}{0.00,0.00,0.00}

\node[text=drawColor,anchor=base east,inner sep=0pt, outer sep=0pt, scale=  0.87] at ( 30.98,100.24) {0};

\node[text=drawColor,anchor=base east,inner sep=0pt, outer sep=0pt, scale=  0.87] at ( 30.98,113.34) {5};

\node[text=drawColor,anchor=base east,inner sep=0pt, outer sep=0pt, scale=  0.87] at ( 30.98,126.44) {10};

\node[text=drawColor,anchor=base east,inner sep=0pt, outer sep=0pt, scale=  0.87] at ( 30.98,139.54) {15};
\end{scope}
\begin{scope}
\path[clip] (  0.00,  0.00) rectangle (361.35,231.26);
\definecolor[named]{drawColor}{rgb}{0.00,0.00,0.00}

\path[draw=drawColor,line width= 0.6pt,line join=round] ( 33.82,103.53) --
	( 38.09,103.53);

\path[draw=drawColor,line width= 0.6pt,line join=round] ( 33.82,116.63) --
	( 38.09,116.63);

\path[draw=drawColor,line width= 0.6pt,line join=round] ( 33.82,129.73) --
	( 38.09,129.73);

\path[draw=drawColor,line width= 0.6pt,line join=round] ( 33.82,142.83) --
	( 38.09,142.83);
\end{scope}
\begin{scope}
\path[clip] (  0.00,  0.00) rectangle (361.35,231.26);
\definecolor[named]{drawColor}{rgb}{0.00,0.00,0.00}

\node[text=drawColor,anchor=base east,inner sep=0pt, outer sep=0pt, scale=  0.87] at ( 30.98, 43.70) {0};

\node[text=drawColor,anchor=base east,inner sep=0pt, outer sep=0pt, scale=  0.87] at ( 30.98, 56.79) {5};

\node[text=drawColor,anchor=base east,inner sep=0pt, outer sep=0pt, scale=  0.87] at ( 30.98, 69.89) {10};

\node[text=drawColor,anchor=base east,inner sep=0pt, outer sep=0pt, scale=  0.87] at ( 30.98, 82.99) {15};
\end{scope}
\begin{scope}
\path[clip] (  0.00,  0.00) rectangle (361.35,231.26);
\definecolor[named]{drawColor}{rgb}{0.00,0.00,0.00}

\path[draw=drawColor,line width= 0.6pt,line join=round] ( 33.82, 46.99) --
	( 38.09, 46.99);

\path[draw=drawColor,line width= 0.6pt,line join=round] ( 33.82, 60.08) --
	( 38.09, 60.08);

\path[draw=drawColor,line width= 0.6pt,line join=round] ( 33.82, 73.18) --
	( 38.09, 73.18);

\path[draw=drawColor,line width= 0.6pt,line join=round] ( 33.82, 86.28) --
	( 38.09, 86.28);
\end{scope}
\begin{scope}
\path[clip] (334.89,148.26) rectangle (348.10,204.81);
\definecolor[named]{drawColor}{rgb}{0.50,0.50,0.50}
\definecolor[named]{fillColor}{rgb}{0.80,0.80,0.80}

\path[draw=drawColor,line width= 0.2pt,line join=round,line cap=round,fill=fillColor] (334.89,148.26) rectangle (348.10,204.81);
\definecolor[named]{drawColor}{rgb}{0.00,0.00,0.00}

\node[text=drawColor,rotate=270.00,anchor=base,inner sep=0pt, outer sep=0pt, scale=  0.87] at (338.21,176.54) {Signal};
\end{scope}
\begin{scope}
\path[clip] (334.89, 91.72) rectangle (348.10,148.26);
\definecolor[named]{drawColor}{rgb}{0.50,0.50,0.50}
\definecolor[named]{fillColor}{rgb}{0.80,0.80,0.80}

\path[draw=drawColor,line width= 0.2pt,line join=round,line cap=round,fill=fillColor] (334.89, 91.72) rectangle (348.10,148.26);
\definecolor[named]{drawColor}{rgb}{0.00,0.00,0.00}

\node[text=drawColor,rotate=270.00,anchor=base,inner sep=0pt, outer sep=0pt, scale=  0.87] at (338.21,119.99) {Breakpoint};
\end{scope}
\begin{scope}
\path[clip] (334.89, 35.17) rectangle (348.10, 91.72);
\definecolor[named]{drawColor}{rgb}{0.50,0.50,0.50}
\definecolor[named]{fillColor}{rgb}{0.80,0.80,0.80}

\path[draw=drawColor,line width= 0.2pt,line join=round,line cap=round,fill=fillColor] (334.89, 35.17) rectangle (348.10, 91.72);
\definecolor[named]{drawColor}{rgb}{0.00,0.00,0.00}

\node[text=drawColor,rotate=270.00,anchor=base,inner sep=0pt, outer sep=0pt, scale=  0.87] at (338.21, 63.44) {Annotation};
\end{scope}
\begin{scope}
\path[clip] (  0.00,  0.00) rectangle (361.35,231.26);
\definecolor[named]{drawColor}{rgb}{0.00,0.00,0.00}

\path[draw=drawColor,line width= 0.6pt,line join=round] ( 56.57, 30.90) --
	( 56.57, 35.17);

\path[draw=drawColor,line width= 0.6pt,line join=round] ( 80.03, 30.90) --
	( 80.03, 35.17);

\path[draw=drawColor,line width= 0.6pt,line join=round] ( 91.76, 30.90) --
	( 91.76, 35.17);

\path[draw=drawColor,line width= 0.6pt,line join=round] (168.02, 30.90) --
	(168.02, 35.17);
\end{scope}
\begin{scope}
\path[clip] (  0.00,  0.00) rectangle (361.35,231.26);
\definecolor[named]{drawColor}{rgb}{0.00,0.00,0.00}

\node[text=drawColor,anchor=base,inner sep=0pt, outer sep=0pt, scale=  0.87] at ( 56.57, 21.48) {1};

\node[text=drawColor,anchor=base,inner sep=0pt, outer sep=0pt, scale=  0.87] at ( 80.03, 21.48) {5};

\node[text=drawColor,anchor=base,inner sep=0pt, outer sep=0pt, scale=  0.87] at ( 91.76, 21.48) {7};

\node[text=drawColor,anchor=base,inner sep=0pt, outer sep=0pt, scale=  0.87] at (168.02, 21.48) {20};
\end{scope}
\begin{scope}
\path[clip] (  0.00,  0.00) rectangle (361.35,231.26);
\definecolor[named]{drawColor}{rgb}{0.00,0.00,0.00}

\path[draw=drawColor,line width= 0.6pt,line join=round] (204.97, 30.90) --
	(204.97, 35.17);

\path[draw=drawColor,line width= 0.6pt,line join=round] (228.43, 30.90) --
	(228.43, 35.17);

\path[draw=drawColor,line width= 0.6pt,line join=round] (240.16, 30.90) --
	(240.16, 35.17);

\path[draw=drawColor,line width= 0.6pt,line join=round] (316.42, 30.90) --
	(316.42, 35.17);
\end{scope}
\begin{scope}
\path[clip] (  0.00,  0.00) rectangle (361.35,231.26);
\definecolor[named]{drawColor}{rgb}{0.00,0.00,0.00}

\node[text=drawColor,anchor=base,inner sep=0pt, outer sep=0pt, scale=  0.87] at (204.97, 21.48) {1};

\node[text=drawColor,anchor=base,inner sep=0pt, outer sep=0pt, scale=  0.87] at (228.43, 21.48) {5};

\node[text=drawColor,anchor=base,inner sep=0pt, outer sep=0pt, scale=  0.87] at (240.16, 21.48) {7};

\node[text=drawColor,anchor=base,inner sep=0pt, outer sep=0pt, scale=  0.87] at (316.42, 21.48) {20};
\end{scope}
\begin{scope}
\path[clip] (  0.00,  0.00) rectangle (361.35,231.26);
\definecolor[named]{drawColor}{rgb}{0.00,0.00,0.00}

\node[text=drawColor,anchor=base,inner sep=0pt, outer sep=0pt, scale=  1.09] at (186.49,  9.94) {segments $k$ of estimated signal};
\end{scope}
\begin{scope}
\path[clip] (  0.00,  0.00) rectangle (361.35,231.26);
\definecolor[named]{drawColor}{rgb}{0.00,0.00,0.00}

\node[text=drawColor,rotate= 90.00,anchor=base,inner sep=0pt, outer sep=0pt, scale=  1.09] at ( 18.16,119.99) {error};
\end{scope}
\end{tikzpicture}

%% file: figure-variable-density-sigerr.tex
\begin{tikzpicture}[x=1pt,y=1pt]
\definecolor[named]{fillColor}{rgb}{1.00,1.00,1.00}
\path[use as bounding box,fill=fillColor,fill opacity=0.00] (0,0) rectangle (361.35,289.08);
\begin{scope}
\path[clip] (  0.00,  0.00) rectangle (361.35,289.08);
\definecolor[named]{drawColor}{rgb}{1.00,1.00,1.00}
\definecolor[named]{fillColor}{rgb}{1.00,1.00,1.00}

\path[draw=drawColor,line width= 0.6pt,line join=round,line cap=round,fill=fillColor] (  0.00,  0.00) rectangle (361.35,289.08);
\end{scope}
\begin{scope}
\path[clip] ( 38.09,262.62) rectangle (186.49,275.83);
\definecolor[named]{drawColor}{rgb}{0.50,0.50,0.50}
\definecolor[named]{fillColor}{rgb}{0.80,0.80,0.80}

\path[draw=drawColor,line width= 0.2pt,line join=round,line cap=round,fill=fillColor] ( 38.09,262.62) rectangle (186.49,275.83);
\definecolor[named]{drawColor}{rgb}{0.00,0.00,0.00}

\node[text=drawColor,anchor=base,inner sep=0pt, outer sep=0pt, scale=  0.87] at (112.29,265.94) {bases/probe = 374};
\end{scope}
\begin{scope}
\path[clip] (186.49,262.62) rectangle (334.89,275.83);
\definecolor[named]{drawColor}{rgb}{0.50,0.50,0.50}
\definecolor[named]{fillColor}{rgb}{0.80,0.80,0.80}

\path[draw=drawColor,line width= 0.2pt,line join=round,line cap=round,fill=fillColor] (186.49,262.62) rectangle (334.89,275.83);
\definecolor[named]{drawColor}{rgb}{0.00,0.00,0.00}

\node[text=drawColor,anchor=base,inner sep=0pt, outer sep=0pt, scale=  0.87] at (260.69,265.94) {bases/probe = 7};
\end{scope}
\begin{scope}
\path[clip] ( 38.09,205.76) rectangle (186.49,262.62);
\definecolor[named]{fillColor}{rgb}{1.00,1.00,1.00}

\path[fill=fillColor] ( 38.09,205.76) rectangle (186.49,262.62);
\definecolor[named]{drawColor}{rgb}{0.90,0.90,0.90}

\path[draw=drawColor,line width= 0.2pt,line join=round] ( 38.09,211.58) --
	(186.49,211.58);

\path[draw=drawColor,line width= 0.2pt,line join=round] ( 38.09,227.73) --
	(186.49,227.73);

\path[draw=drawColor,line width= 0.2pt,line join=round] ( 38.09,243.89) --
	(186.49,243.89);

\path[draw=drawColor,line width= 0.2pt,line join=round] ( 38.09,260.04) --
	(186.49,260.04);

\path[draw=drawColor,line width= 0.2pt,line join=round] ( 56.57,205.76) --
	( 56.57,262.62);

\path[draw=drawColor,line width= 0.2pt,line join=round] ( 80.03,205.76) --
	( 80.03,262.62);

\path[draw=drawColor,line width= 0.2pt,line join=round] ( 91.76,205.76) --
	( 91.76,262.62);

\path[draw=drawColor,line width= 0.2pt,line join=round] (168.02,205.76) --
	(168.02,262.62);
\definecolor[named]{drawColor}{rgb}{0.89,0.10,0.11}

\path[draw=drawColor,line width= 2.8pt,line join=round] ( 56.57,230.96) --
	( 62.43,227.73) --
	( 68.30,227.73) --
	( 74.17,224.50) --
	( 80.03,218.04) --
	( 85.90,218.04) --
	( 91.76,211.58) --
	( 97.63,211.58) --
	(103.49,211.58) --
	(109.36,211.58) --
	(115.23,211.58) --
	(121.09,211.58) --
	(126.96,211.58) --
	(132.82,211.58) --
	(138.69,211.58) --
	(144.55,211.58) --
	(150.42,211.58) --
	(156.28,211.58) --
	(162.15,211.58) --
	(168.02,211.58);
\definecolor[named]{drawColor}{rgb}{0.53,0.81,0.92}

\path[draw=drawColor,line width= 2.8pt,line join=round] ( 56.57,211.58) --
	( 62.43,211.58) --
	( 68.30,214.81) --
	( 74.17,214.81) --
	( 80.03,211.58) --
	( 85.90,214.81) --
	( 91.76,211.58) --
	( 97.63,214.81) --
	(103.49,218.04) --
	(109.36,221.27) --
	(115.23,224.50) --
	(121.09,227.73) --
	(126.96,230.96) --
	(132.82,234.19) --
	(138.69,237.42) --
	(144.55,240.65) --
	(150.42,243.89) --
	(156.28,247.12) --
	(162.15,250.35) --
	(168.02,253.58);
\definecolor[named]{drawColor}{rgb}{0.00,0.00,0.00}

\path[draw=drawColor,line width= 0.9pt,line join=round] ( 56.57,230.96) --
	( 62.43,227.73) --
	( 68.30,230.96) --
	( 74.17,227.73) --
	( 80.03,218.04) --
	( 85.90,221.27) --
	( 91.76,211.58) --
	( 97.63,214.81) --
	(103.49,218.04) --
	(109.36,221.27) --
	(115.23,224.50) --
	(121.09,227.73) --
	(126.96,230.96) --
	(132.82,234.19) --
	(138.69,237.42) --
	(144.55,240.65) --
	(150.42,243.89) --
	(156.28,247.12) --
	(162.15,250.35) --
	(168.02,253.58);

\path[draw=drawColor,line width= 0.4pt,line join=round,line cap=round] ( 91.76,211.58) circle (  2.13);
\end{scope}
\begin{scope}
\path[clip] ( 38.09,205.76) rectangle (186.49,262.62);
\definecolor[named]{drawColor}{rgb}{0.89,0.10,0.11}

\node[text=drawColor,anchor=base east,inner sep=0pt, outer sep=0pt, scale=  1.00] at ( 53.72,220.78) {$\hat{\text{FN}}$};

\node[text=drawColor,anchor=base west,inner sep=0pt, outer sep=0pt, scale=  1.00] at (170.86,207.81) {$\hat{\text{FN}}$};
\definecolor[named]{drawColor}{rgb}{0.00,0.00,0.00}

\node[text=drawColor,anchor=base east,inner sep=0pt, outer sep=0pt, scale=  1.00] at ( 53.72,233.60) {$E$};

\node[text=drawColor,anchor=base west,inner sep=0pt, outer sep=0pt, scale=  1.00] at (170.86,252.82) {$E$};
\definecolor[named]{drawColor}{rgb}{0.53,0.81,0.92}

\node[text=drawColor,anchor=base east,inner sep=0pt, outer sep=0pt, scale=  1.00] at ( 53.72,207.81) {$\hat{\text{FP}}$};

\node[text=drawColor,anchor=base west,inner sep=0pt, outer sep=0pt, scale=  1.00] at (170.86,240.00) {$\hat{\text{FP}}$};
\definecolor[named]{drawColor}{rgb}{0.50,0.50,0.50}

\path[draw=drawColor,line width= 0.6pt,line join=round,line cap=round] ( 38.09,205.76) rectangle (186.49,262.62);
\end{scope}
\begin{scope}
\path[clip] ( 38.09,148.90) rectangle (186.49,205.76);
\definecolor[named]{fillColor}{rgb}{1.00,1.00,1.00}

\path[fill=fillColor] ( 38.09,148.90) rectangle (186.49,205.76);
\definecolor[named]{drawColor}{rgb}{0.90,0.90,0.90}

\path[draw=drawColor,line width= 0.2pt,line join=round] ( 38.09,154.71) --
	(186.49,154.71);

\path[draw=drawColor,line width= 0.2pt,line join=round] ( 38.09,170.87) --
	(186.49,170.87);

\path[draw=drawColor,line width= 0.2pt,line join=round] ( 38.09,187.02) --
	(186.49,187.02);

\path[draw=drawColor,line width= 0.2pt,line join=round] ( 38.09,203.18) --
	(186.49,203.18);

\path[draw=drawColor,line width= 0.2pt,line join=round] ( 56.57,148.90) --
	( 56.57,205.76);

\path[draw=drawColor,line width= 0.2pt,line join=round] ( 80.03,148.90) --
	( 80.03,205.76);

\path[draw=drawColor,line width= 0.2pt,line join=round] ( 91.76,148.90) --
	( 91.76,205.76);

\path[draw=drawColor,line width= 0.2pt,line join=round] (168.02,148.90) --
	(168.02,205.76);
\definecolor[named]{drawColor}{rgb}{0.89,0.10,0.11}

\path[draw=drawColor,line width= 2.8pt,line join=round] ( 56.57,174.10) --
	( 62.43,170.87) --
	( 68.30,170.87) --
	( 74.17,167.64) --
	( 80.03,161.18) --
	( 85.90,161.18) --
	( 91.76,154.71) --
	( 97.63,154.71) --
	(103.49,154.71) --
	(109.36,154.71) --
	(115.23,154.71) --
	(121.09,154.71) --
	(126.96,154.71) --
	(132.82,154.71) --
	(138.69,154.71) --
	(144.55,154.71) --
	(150.42,154.71) --
	(156.28,154.71) --
	(162.15,154.71) --
	(168.02,154.71);
\definecolor[named]{drawColor}{rgb}{0.53,0.81,0.92}

\path[draw=drawColor,line width= 2.8pt,line join=round] ( 56.57,154.71) --
	( 62.43,154.71) --
	( 68.30,157.94) --
	( 74.17,157.94) --
	( 80.03,154.71) --
	( 85.90,157.94) --
	( 91.76,154.71) --
	( 97.63,157.94) --
	(103.49,157.94) --
	(109.36,161.18) --
	(115.23,161.18) --
	(121.09,164.41) --
	(126.96,164.41) --
	(132.82,167.64) --
	(138.69,170.87) --
	(144.55,167.64) --
	(150.42,170.87) --
	(156.28,167.64) --
	(162.15,170.87) --
	(168.02,174.10);
\definecolor[named]{drawColor}{rgb}{0.00,0.00,0.00}

\path[draw=drawColor,line width= 0.9pt,line join=round] ( 56.57,174.10) --
	( 62.43,170.87) --
	( 68.30,174.10) --
	( 74.17,170.87) --
	( 80.03,161.18) --
	( 85.90,164.41) --
	( 91.76,154.71) --
	( 97.63,157.94) --
	(103.49,157.94) --
	(109.36,161.18) --
	(115.23,161.18) --
	(121.09,164.41) --
	(126.96,164.41) --
	(132.82,167.64) --
	(138.69,170.87) --
	(144.55,167.64) --
	(150.42,170.87) --
	(156.28,167.64) --
	(162.15,170.87) --
	(168.02,174.10);

\path[draw=drawColor,line width= 0.4pt,line join=round,line cap=round] ( 91.76,154.71) circle (  2.13);
\end{scope}
\begin{scope}
\path[clip] ( 38.09,148.90) rectangle (186.49,205.76);
\definecolor[named]{drawColor}{rgb}{0.89,0.10,0.11}

\node[text=drawColor,anchor=base east,inner sep=0pt, outer sep=0pt, scale=  1.00] at ( 53.72,163.92) {$\hat{\text{FN}}$};

\node[text=drawColor,anchor=base west,inner sep=0pt, outer sep=0pt, scale=  1.00] at (170.86,150.94) {$\hat{\text{FN}}$};
\definecolor[named]{drawColor}{rgb}{0.00,0.00,0.00}

\node[text=drawColor,anchor=base east,inner sep=0pt, outer sep=0pt, scale=  1.00] at ( 53.72,176.74) {$E$};

\node[text=drawColor,anchor=base west,inner sep=0pt, outer sep=0pt, scale=  1.00] at (170.86,176.74) {$E$};
\definecolor[named]{drawColor}{rgb}{0.53,0.81,0.92}

\node[text=drawColor,anchor=base east,inner sep=0pt, outer sep=0pt, scale=  1.00] at ( 53.72,150.94) {$\hat{\text{FP}}$};

\node[text=drawColor,anchor=base west,inner sep=0pt, outer sep=0pt, scale=  1.00] at (170.86,163.92) {$\hat{\text{FP}}$};
\definecolor[named]{drawColor}{rgb}{0.50,0.50,0.50}

\path[draw=drawColor,line width= 0.6pt,line join=round,line cap=round] ( 38.09,148.90) rectangle (186.49,205.76);
\end{scope}
\begin{scope}
\path[clip] ( 38.09, 92.03) rectangle (186.49,148.90);
\definecolor[named]{fillColor}{rgb}{1.00,1.00,1.00}

\path[fill=fillColor] ( 38.09, 92.03) rectangle (186.49,148.90);
\definecolor[named]{drawColor}{rgb}{0.90,0.90,0.90}

\path[draw=drawColor,line width= 0.2pt,line join=round] ( 38.09, 97.85) --
	(186.49, 97.85);

\path[draw=drawColor,line width= 0.2pt,line join=round] ( 38.09,114.00) --
	(186.49,114.00);

\path[draw=drawColor,line width= 0.2pt,line join=round] ( 38.09,130.16) --
	(186.49,130.16);

\path[draw=drawColor,line width= 0.2pt,line join=round] ( 38.09,146.31) --
	(186.49,146.31);

\path[draw=drawColor,line width= 0.2pt,line join=round] ( 56.57, 92.03) --
	( 56.57,148.90);

\path[draw=drawColor,line width= 0.2pt,line join=round] ( 80.03, 92.03) --
	( 80.03,148.90);

\path[draw=drawColor,line width= 0.2pt,line join=round] ( 91.76, 92.03) --
	( 91.76,148.90);

\path[draw=drawColor,line width= 0.2pt,line join=round] (168.02, 92.03) --
	(168.02,148.90);
\definecolor[named]{drawColor}{rgb}{0.89,0.10,0.11}

\path[draw=drawColor,line width= 2.8pt,line join=round] ( 56.57,107.54) --
	( 62.43,107.54) --
	( 68.30,104.31) --
	( 74.17,104.31) --
	( 80.03,101.08) --
	( 85.90,101.08) --
	( 91.76, 97.85) --
	( 97.63, 97.85) --
	(103.49, 97.85) --
	(109.36, 97.85) --
	(115.23, 97.85) --
	(121.09, 97.85) --
	(126.96, 97.85) --
	(132.82, 97.85) --
	(138.69, 97.85) --
	(144.55, 97.85) --
	(150.42, 97.85) --
	(156.28, 97.85) --
	(162.15, 97.85) --
	(168.02, 97.85);
\definecolor[named]{drawColor}{rgb}{0.53,0.81,0.92}

\path[draw=drawColor,line width= 2.8pt,line join=round] ( 56.57, 97.85) --
	( 62.43, 97.85) --
	( 68.30, 97.85) --
	( 74.17, 97.85) --
	( 80.03, 97.85) --
	( 85.90, 97.85) --
	( 91.76, 97.85) --
	( 97.63,101.08) --
	(103.49,101.08) --
	(109.36,104.31) --
	(115.23,104.31) --
	(121.09,104.31) --
	(126.96,107.54) --
	(132.82,107.54) --
	(138.69,110.77) --
	(144.55,107.54) --
	(150.42,110.77) --
	(156.28,107.54) --
	(162.15,110.77) --
	(168.02,110.77);
\definecolor[named]{drawColor}{rgb}{0.00,0.00,0.00}

\path[draw=drawColor,line width= 0.9pt,line join=round] ( 56.57,107.54) --
	( 62.43,107.54) --
	( 68.30,104.31) --
	( 74.17,104.31) --
	( 80.03,101.08) --
	( 85.90,101.08) --
	( 91.76, 97.85) --
	( 97.63,101.08) --
	(103.49,101.08) --
	(109.36,104.31) --
	(115.23,104.31) --
	(121.09,104.31) --
	(126.96,107.54) --
	(132.82,107.54) --
	(138.69,110.77) --
	(144.55,107.54) --
	(150.42,110.77) --
	(156.28,107.54) --
	(162.15,110.77) --
	(168.02,110.77);

\path[draw=drawColor,line width= 0.4pt,line join=round,line cap=round] ( 91.76, 97.85) circle (  2.13);
\end{scope}
\begin{scope}
\path[clip] ( 38.09, 92.03) rectangle (186.49,148.90);
\definecolor[named]{drawColor}{rgb}{0.89,0.10,0.11}

\node[text=drawColor,anchor=base east,inner sep=0pt, outer sep=0pt, scale=  1.00] at ( 53.72,104.86) {$\hat{\text{FN}}$};

\node[text=drawColor,anchor=base west,inner sep=0pt, outer sep=0pt, scale=  1.00] at (170.86, 92.03) {$\hat{\text{FN}}$};
\definecolor[named]{drawColor}{rgb}{0.00,0.00,0.00}

\node[text=drawColor,anchor=base east,inner sep=0pt, outer sep=0pt, scale=  1.00] at ( 53.72,117.68) {$E$};

\node[text=drawColor,anchor=base west,inner sep=0pt, outer sep=0pt, scale=  1.00] at (170.86,117.68) {$E$};
\definecolor[named]{drawColor}{rgb}{0.53,0.81,0.92}

\node[text=drawColor,anchor=base east,inner sep=0pt, outer sep=0pt, scale=  1.00] at ( 53.72, 92.03) {$\hat{\text{FP}}$};

\node[text=drawColor,anchor=base west,inner sep=0pt, outer sep=0pt, scale=  1.00] at (170.86,104.86) {$\hat{\text{FP}}$};
\definecolor[named]{drawColor}{rgb}{0.50,0.50,0.50}

\path[draw=drawColor,line width= 0.6pt,line join=round,line cap=round] ( 38.09, 92.03) rectangle (186.49,148.90);
\end{scope}
\begin{scope}
\path[clip] ( 38.09, 35.17) rectangle (186.49, 92.03);
\definecolor[named]{fillColor}{rgb}{1.00,1.00,1.00}

\path[fill=fillColor] ( 38.09, 35.17) rectangle (186.49, 92.03);
\definecolor[named]{drawColor}{rgb}{0.90,0.90,0.90}

\path[draw=drawColor,line width= 0.2pt,line join=round] ( 38.09, 40.99) --
	(186.49, 40.99);

\path[draw=drawColor,line width= 0.2pt,line join=round] ( 38.09, 57.14) --
	(186.49, 57.14);

\path[draw=drawColor,line width= 0.2pt,line join=round] ( 38.09, 73.30) --
	(186.49, 73.30);

\path[draw=drawColor,line width= 0.2pt,line join=round] ( 38.09, 89.45) --
	(186.49, 89.45);

\path[draw=drawColor,line width= 0.2pt,line join=round] ( 56.57, 35.17) --
	( 56.57, 92.03);

\path[draw=drawColor,line width= 0.2pt,line join=round] ( 80.03, 35.17) --
	( 80.03, 92.03);

\path[draw=drawColor,line width= 0.2pt,line join=round] ( 91.76, 35.17) --
	( 91.76, 92.03);

\path[draw=drawColor,line width= 0.2pt,line join=round] (168.02, 35.17) --
	(168.02, 92.03);
\definecolor[named]{drawColor}{rgb}{0.89,0.10,0.11}

\path[draw=drawColor,line width= 2.8pt,line join=round] ( 56.57, 50.68) --
	( 62.43, 50.68) --
	( 68.30, 47.45) --
	( 74.17, 47.45) --
	( 80.03, 44.22) --
	( 85.90, 44.22) --
	( 91.76, 40.99) --
	( 97.63, 40.99) --
	(103.49, 40.99) --
	(109.36, 40.99) --
	(115.23, 40.99) --
	(121.09, 40.99) --
	(126.96, 40.99) --
	(132.82, 40.99) --
	(138.69, 40.99) --
	(144.55, 40.99) --
	(150.42, 40.99) --
	(156.28, 40.99) --
	(162.15, 40.99) --
	(168.02, 40.99);
\definecolor[named]{drawColor}{rgb}{0.53,0.81,0.92}

\path[draw=drawColor,line width= 2.8pt,line join=round] ( 56.57, 40.99) --
	( 62.43, 40.99) --
	( 68.30, 40.99) --
	( 74.17, 40.99) --
	( 80.03, 40.99) --
	( 85.90, 40.99) --
	( 91.76, 40.99) --
	( 97.63, 40.99) --
	(103.49, 40.99) --
	(109.36, 44.22) --
	(115.23, 44.22) --
	(121.09, 44.22) --
	(126.96, 44.22) --
	(132.82, 44.22) --
	(138.69, 47.45) --
	(144.55, 44.22) --
	(150.42, 47.45) --
	(156.28, 44.22) --
	(162.15, 47.45) --
	(168.02, 47.45);
\definecolor[named]{drawColor}{rgb}{0.00,0.00,0.00}

\path[draw=drawColor,line width= 0.9pt,line join=round] ( 56.57, 50.68) --
	( 62.43, 50.68) --
	( 68.30, 47.45) --
	( 74.17, 47.45) --
	( 80.03, 44.22) --
	( 85.90, 44.22) --
	( 91.76, 40.99) --
	( 97.63, 40.99) --
	(103.49, 40.99) --
	(109.36, 44.22) --
	(115.23, 44.22) --
	(121.09, 44.22) --
	(126.96, 44.22) --
	(132.82, 44.22) --
	(138.69, 47.45) --
	(144.55, 44.22) --
	(150.42, 47.45) --
	(156.28, 44.22) --
	(162.15, 47.45) --
	(168.02, 47.45);

\path[draw=drawColor,line width= 0.4pt,line join=round,line cap=round] ( 91.76, 40.99) circle (  2.13);

\path[draw=drawColor,line width= 0.4pt,line join=round,line cap=round] ( 97.63, 40.99) circle (  2.13);

\path[draw=drawColor,line width= 0.4pt,line join=round,line cap=round] (103.49, 40.99) circle (  2.13);
\end{scope}
\begin{scope}
\path[clip] ( 38.09, 35.17) rectangle (186.49, 92.03);
\definecolor[named]{drawColor}{rgb}{0.89,0.10,0.11}

\node[text=drawColor,anchor=base east,inner sep=0pt, outer sep=0pt, scale=  1.00] at ( 53.72, 47.99) {$\hat{\text{FN}}$};

\node[text=drawColor,anchor=base west,inner sep=0pt, outer sep=0pt, scale=  1.00] at (170.86, 35.17) {$\hat{\text{FN}}$};
\definecolor[named]{drawColor}{rgb}{0.00,0.00,0.00}

\node[text=drawColor,anchor=base east,inner sep=0pt, outer sep=0pt, scale=  1.00] at ( 53.72, 60.81) {$E$};

\node[text=drawColor,anchor=base west,inner sep=0pt, outer sep=0pt, scale=  1.00] at (170.86, 60.81) {$E$};
\definecolor[named]{drawColor}{rgb}{0.53,0.81,0.92}

\node[text=drawColor,anchor=base east,inner sep=0pt, outer sep=0pt, scale=  1.00] at ( 53.72, 35.17) {$\hat{\text{FP}}$};

\node[text=drawColor,anchor=base west,inner sep=0pt, outer sep=0pt, scale=  1.00] at (170.86, 47.99) {$\hat{\text{FP}}$};
\definecolor[named]{drawColor}{rgb}{0.50,0.50,0.50}

\path[draw=drawColor,line width= 0.6pt,line join=round,line cap=round] ( 38.09, 35.17) rectangle (186.49, 92.03);
\end{scope}
\begin{scope}
\path[clip] (186.49,205.76) rectangle (334.89,262.62);
\definecolor[named]{fillColor}{rgb}{1.00,1.00,1.00}

\path[fill=fillColor] (186.49,205.76) rectangle (334.89,262.62);
\definecolor[named]{drawColor}{rgb}{0.90,0.90,0.90}

\path[draw=drawColor,line width= 0.2pt,line join=round] (186.49,211.58) --
	(334.89,211.58);

\path[draw=drawColor,line width= 0.2pt,line join=round] (186.49,227.73) --
	(334.89,227.73);

\path[draw=drawColor,line width= 0.2pt,line join=round] (186.49,243.89) --
	(334.89,243.89);

\path[draw=drawColor,line width= 0.2pt,line join=round] (186.49,260.04) --
	(334.89,260.04);

\path[draw=drawColor,line width= 0.2pt,line join=round] (204.97,205.76) --
	(204.97,262.62);

\path[draw=drawColor,line width= 0.2pt,line join=round] (228.43,205.76) --
	(228.43,262.62);

\path[draw=drawColor,line width= 0.2pt,line join=round] (240.16,205.76) --
	(240.16,262.62);

\path[draw=drawColor,line width= 0.2pt,line join=round] (316.42,205.76) --
	(316.42,262.62);
\definecolor[named]{drawColor}{rgb}{0.89,0.10,0.11}

\path[draw=drawColor,line width= 2.8pt,line join=round] (204.97,230.96) --
	(210.84,227.73) --
	(216.70,224.50) --
	(222.57,221.27) --
	(228.43,218.04) --
	(234.30,214.81) --
	(240.16,211.58) --
	(246.03,211.58) --
	(251.90,211.58) --
	(257.76,211.58) --
	(263.63,211.58) --
	(269.49,211.58) --
	(275.36,211.58) --
	(281.22,211.58) --
	(287.09,211.58) --
	(292.95,211.58) --
	(298.82,211.58) --
	(304.69,211.58) --
	(310.55,211.58) --
	(316.42,211.58);
\definecolor[named]{drawColor}{rgb}{0.53,0.81,0.92}

\path[draw=drawColor,line width= 2.8pt,line join=round] (204.97,211.58) --
	(210.84,211.58) --
	(216.70,211.58) --
	(222.57,211.58) --
	(228.43,211.58) --
	(234.30,211.58) --
	(240.16,211.58) --
	(246.03,214.81) --
	(251.90,218.04) --
	(257.76,221.27) --
	(263.63,224.50) --
	(269.49,227.73) --
	(275.36,230.96) --
	(281.22,234.19) --
	(287.09,237.42) --
	(292.95,240.65) --
	(298.82,243.89) --
	(304.69,247.12) --
	(310.55,250.35) --
	(316.42,253.58);
\definecolor[named]{drawColor}{rgb}{0.00,0.00,0.00}

\path[draw=drawColor,line width= 0.9pt,line join=round] (204.97,230.96) --
	(210.84,227.73) --
	(216.70,224.50) --
	(222.57,221.27) --
	(228.43,218.04) --
	(234.30,214.81) --
	(240.16,211.58) --
	(246.03,214.81) --
	(251.90,218.04) --
	(257.76,221.27) --
	(263.63,224.50) --
	(269.49,227.73) --
	(275.36,230.96) --
	(281.22,234.19) --
	(287.09,237.42) --
	(292.95,240.65) --
	(298.82,243.89) --
	(304.69,247.12) --
	(310.55,250.35) --
	(316.42,253.58);

\path[draw=drawColor,line width= 0.4pt,line join=round,line cap=round] (240.16,211.58) circle (  2.13);
\end{scope}
\begin{scope}
\path[clip] (186.49,205.76) rectangle (334.89,262.62);
\definecolor[named]{drawColor}{rgb}{0.89,0.10,0.11}

\node[text=drawColor,anchor=base east,inner sep=0pt, outer sep=0pt, scale=  1.00] at (202.12,220.78) {$\hat{\text{FN}}$};

\node[text=drawColor,anchor=base west,inner sep=0pt, outer sep=0pt, scale=  1.00] at (319.26,207.81) {$\hat{\text{FN}}$};
\definecolor[named]{drawColor}{rgb}{0.00,0.00,0.00}

\node[text=drawColor,anchor=base east,inner sep=0pt, outer sep=0pt, scale=  1.00] at (202.12,233.60) {$E$};

\node[text=drawColor,anchor=base west,inner sep=0pt, outer sep=0pt, scale=  1.00] at (319.26,252.82) {$E$};
\definecolor[named]{drawColor}{rgb}{0.53,0.81,0.92}

\node[text=drawColor,anchor=base east,inner sep=0pt, outer sep=0pt, scale=  1.00] at (202.12,207.81) {$\hat{\text{FP}}$};

\node[text=drawColor,anchor=base west,inner sep=0pt, outer sep=0pt, scale=  1.00] at (319.26,240.00) {$\hat{\text{FP}}$};
\definecolor[named]{drawColor}{rgb}{0.50,0.50,0.50}

\path[draw=drawColor,line width= 0.6pt,line join=round,line cap=round] (186.49,205.76) rectangle (334.89,262.62);
\end{scope}
\begin{scope}
\path[clip] (186.49,148.90) rectangle (334.89,205.76);
\definecolor[named]{fillColor}{rgb}{1.00,1.00,1.00}

\path[fill=fillColor] (186.49,148.90) rectangle (334.89,205.76);
\definecolor[named]{drawColor}{rgb}{0.90,0.90,0.90}

\path[draw=drawColor,line width= 0.2pt,line join=round] (186.49,154.71) --
	(334.89,154.71);

\path[draw=drawColor,line width= 0.2pt,line join=round] (186.49,170.87) --
	(334.89,170.87);

\path[draw=drawColor,line width= 0.2pt,line join=round] (186.49,187.02) --
	(334.89,187.02);

\path[draw=drawColor,line width= 0.2pt,line join=round] (186.49,203.18) --
	(334.89,203.18);

\path[draw=drawColor,line width= 0.2pt,line join=round] (204.97,148.90) --
	(204.97,205.76);

\path[draw=drawColor,line width= 0.2pt,line join=round] (228.43,148.90) --
	(228.43,205.76);

\path[draw=drawColor,line width= 0.2pt,line join=round] (240.16,148.90) --
	(240.16,205.76);

\path[draw=drawColor,line width= 0.2pt,line join=round] (316.42,148.90) --
	(316.42,205.76);
\definecolor[named]{drawColor}{rgb}{0.89,0.10,0.11}

\path[draw=drawColor,line width= 2.8pt,line join=round] (204.97,174.10) --
	(210.84,170.87) --
	(216.70,167.64) --
	(222.57,164.41) --
	(228.43,161.18) --
	(234.30,157.94) --
	(240.16,154.71) --
	(246.03,154.71) --
	(251.90,154.71) --
	(257.76,154.71) --
	(263.63,154.71) --
	(269.49,154.71) --
	(275.36,154.71) --
	(281.22,154.71) --
	(287.09,154.71) --
	(292.95,154.71) --
	(298.82,154.71) --
	(304.69,154.71) --
	(310.55,154.71) --
	(316.42,154.71);
\definecolor[named]{drawColor}{rgb}{0.53,0.81,0.92}

\path[draw=drawColor,line width= 2.8pt,line join=round] (204.97,154.71) --
	(210.84,154.71) --
	(216.70,154.71) --
	(222.57,154.71) --
	(228.43,154.71) --
	(234.30,154.71) --
	(240.16,154.71) --
	(246.03,157.94) --
	(251.90,161.18) --
	(257.76,161.18) --
	(263.63,164.41) --
	(269.49,164.41) --
	(275.36,167.64) --
	(281.22,164.41) --
	(287.09,164.41) --
	(292.95,164.41) --
	(298.82,167.64) --
	(304.69,167.64) --
	(310.55,167.64) --
	(316.42,170.87);
\definecolor[named]{drawColor}{rgb}{0.00,0.00,0.00}

\path[draw=drawColor,line width= 0.9pt,line join=round] (204.97,174.10) --
	(210.84,170.87) --
	(216.70,167.64) --
	(222.57,164.41) --
	(228.43,161.18) --
	(234.30,157.94) --
	(240.16,154.71) --
	(246.03,157.94) --
	(251.90,161.18) --
	(257.76,161.18) --
	(263.63,164.41) --
	(269.49,164.41) --
	(275.36,167.64) --
	(281.22,164.41) --
	(287.09,164.41) --
	(292.95,164.41) --
	(298.82,167.64) --
	(304.69,167.64) --
	(310.55,167.64) --
	(316.42,170.87);

\path[draw=drawColor,line width= 0.4pt,line join=round,line cap=round] (240.16,154.71) circle (  2.13);
\end{scope}
\begin{scope}
\path[clip] (186.49,148.90) rectangle (334.89,205.76);
\definecolor[named]{drawColor}{rgb}{0.89,0.10,0.11}

\node[text=drawColor,anchor=base east,inner sep=0pt, outer sep=0pt, scale=  1.00] at (202.12,163.92) {$\hat{\text{FN}}$};

\node[text=drawColor,anchor=base west,inner sep=0pt, outer sep=0pt, scale=  1.00] at (319.26,148.90) {$\hat{\text{FN}}$};
\definecolor[named]{drawColor}{rgb}{0.00,0.00,0.00}

\node[text=drawColor,anchor=base east,inner sep=0pt, outer sep=0pt, scale=  1.00] at (202.12,176.74) {$E$};

\node[text=drawColor,anchor=base west,inner sep=0pt, outer sep=0pt, scale=  1.00] at (319.26,174.54) {$E$};
\definecolor[named]{drawColor}{rgb}{0.53,0.81,0.92}

\node[text=drawColor,anchor=base east,inner sep=0pt, outer sep=0pt, scale=  1.00] at (202.12,150.94) {$\hat{\text{FP}}$};

\node[text=drawColor,anchor=base west,inner sep=0pt, outer sep=0pt, scale=  1.00] at (319.26,161.72) {$\hat{\text{FP}}$};
\definecolor[named]{drawColor}{rgb}{0.50,0.50,0.50}

\path[draw=drawColor,line width= 0.6pt,line join=round,line cap=round] (186.49,148.90) rectangle (334.89,205.76);
\end{scope}
\begin{scope}
\path[clip] (186.49, 92.03) rectangle (334.89,148.90);
\definecolor[named]{fillColor}{rgb}{1.00,1.00,1.00}

\path[fill=fillColor] (186.49, 92.03) rectangle (334.89,148.90);
\definecolor[named]{drawColor}{rgb}{0.90,0.90,0.90}

\path[draw=drawColor,line width= 0.2pt,line join=round] (186.49, 97.85) --
	(334.89, 97.85);

\path[draw=drawColor,line width= 0.2pt,line join=round] (186.49,114.00) --
	(334.89,114.00);

\path[draw=drawColor,line width= 0.2pt,line join=round] (186.49,130.16) --
	(334.89,130.16);

\path[draw=drawColor,line width= 0.2pt,line join=round] (186.49,146.31) --
	(334.89,146.31);

\path[draw=drawColor,line width= 0.2pt,line join=round] (204.97, 92.03) --
	(204.97,148.90);

\path[draw=drawColor,line width= 0.2pt,line join=round] (228.43, 92.03) --
	(228.43,148.90);

\path[draw=drawColor,line width= 0.2pt,line join=round] (240.16, 92.03) --
	(240.16,148.90);

\path[draw=drawColor,line width= 0.2pt,line join=round] (316.42, 92.03) --
	(316.42,148.90);
\definecolor[named]{drawColor}{rgb}{0.89,0.10,0.11}

\path[draw=drawColor,line width= 2.8pt,line join=round] (204.97,107.54) --
	(210.84,107.54) --
	(216.70,104.31) --
	(222.57,104.31) --
	(228.43,101.08) --
	(234.30,101.08) --
	(240.16, 97.85) --
	(246.03, 97.85) --
	(251.90, 97.85) --
	(257.76, 97.85) --
	(263.63, 97.85) --
	(269.49, 97.85) --
	(275.36, 97.85) --
	(281.22, 97.85) --
	(287.09, 97.85) --
	(292.95, 97.85) --
	(298.82, 97.85) --
	(304.69, 97.85) --
	(310.55, 97.85) --
	(316.42, 97.85);
\definecolor[named]{drawColor}{rgb}{0.53,0.81,0.92}

\path[draw=drawColor,line width= 2.8pt,line join=round] (204.97, 97.85) --
	(210.84, 97.85) --
	(216.70, 97.85) --
	(222.57, 97.85) --
	(228.43, 97.85) --
	(234.30, 97.85) --
	(240.16, 97.85) --
	(246.03, 97.85) --
	(251.90,101.08) --
	(257.76,101.08) --
	(263.63,101.08) --
	(269.49,101.08) --
	(275.36,101.08) --
	(281.22,101.08) --
	(287.09,101.08) --
	(292.95,101.08) --
	(298.82,101.08) --
	(304.69,101.08) --
	(310.55,101.08) --
	(316.42,104.31);
\definecolor[named]{drawColor}{rgb}{0.00,0.00,0.00}

\path[draw=drawColor,line width= 0.9pt,line join=round] (204.97,107.54) --
	(210.84,107.54) --
	(216.70,104.31) --
	(222.57,104.31) --
	(228.43,101.08) --
	(234.30,101.08) --
	(240.16, 97.85) --
	(246.03, 97.85) --
	(251.90,101.08) --
	(257.76,101.08) --
	(263.63,101.08) --
	(269.49,101.08) --
	(275.36,101.08) --
	(281.22,101.08) --
	(287.09,101.08) --
	(292.95,101.08) --
	(298.82,101.08) --
	(304.69,101.08) --
	(310.55,101.08) --
	(316.42,104.31);

\path[draw=drawColor,line width= 0.4pt,line join=round,line cap=round] (240.16, 97.85) circle (  2.13);

\path[draw=drawColor,line width= 0.4pt,line join=round,line cap=round] (246.03, 97.85) circle (  2.13);
\end{scope}
\begin{scope}
\path[clip] (186.49, 92.03) rectangle (334.89,148.90);
\definecolor[named]{drawColor}{rgb}{0.89,0.10,0.11}

\node[text=drawColor,anchor=base east,inner sep=0pt, outer sep=0pt, scale=  1.00] at (202.12,104.86) {$\hat{\text{FN}}$};

\node[text=drawColor,anchor=base west,inner sep=0pt, outer sep=0pt, scale=  1.00] at (319.26, 92.03) {$\hat{\text{FN}}$};
\definecolor[named]{drawColor}{rgb}{0.00,0.00,0.00}

\node[text=drawColor,anchor=base east,inner sep=0pt, outer sep=0pt, scale=  1.00] at (202.12,117.68) {$E$};

\node[text=drawColor,anchor=base west,inner sep=0pt, outer sep=0pt, scale=  1.00] at (319.26,117.68) {$E$};
\definecolor[named]{drawColor}{rgb}{0.53,0.81,0.92}

\node[text=drawColor,anchor=base east,inner sep=0pt, outer sep=0pt, scale=  1.00] at (202.12, 92.03) {$\hat{\text{FP}}$};

\node[text=drawColor,anchor=base west,inner sep=0pt, outer sep=0pt, scale=  1.00] at (319.26,104.86) {$\hat{\text{FP}}$};
\definecolor[named]{drawColor}{rgb}{0.50,0.50,0.50}

\path[draw=drawColor,line width= 0.6pt,line join=round,line cap=round] (186.49, 92.03) rectangle (334.89,148.90);
\end{scope}
\begin{scope}
\path[clip] (186.49, 35.17) rectangle (334.89, 92.03);
\definecolor[named]{fillColor}{rgb}{1.00,1.00,1.00}

\path[fill=fillColor] (186.49, 35.17) rectangle (334.89, 92.03);
\definecolor[named]{drawColor}{rgb}{0.90,0.90,0.90}

\path[draw=drawColor,line width= 0.2pt,line join=round] (186.49, 40.99) --
	(334.89, 40.99);

\path[draw=drawColor,line width= 0.2pt,line join=round] (186.49, 57.14) --
	(334.89, 57.14);

\path[draw=drawColor,line width= 0.2pt,line join=round] (186.49, 73.30) --
	(334.89, 73.30);

\path[draw=drawColor,line width= 0.2pt,line join=round] (186.49, 89.45) --
	(334.89, 89.45);

\path[draw=drawColor,line width= 0.2pt,line join=round] (204.97, 35.17) --
	(204.97, 92.03);

\path[draw=drawColor,line width= 0.2pt,line join=round] (228.43, 35.17) --
	(228.43, 92.03);

\path[draw=drawColor,line width= 0.2pt,line join=round] (240.16, 35.17) --
	(240.16, 92.03);

\path[draw=drawColor,line width= 0.2pt,line join=round] (316.42, 35.17) --
	(316.42, 92.03);
\definecolor[named]{drawColor}{rgb}{0.89,0.10,0.11}

\path[draw=drawColor,line width= 2.8pt,line join=round] (204.97, 50.68) --
	(210.84, 50.68) --
	(216.70, 47.45) --
	(222.57, 47.45) --
	(228.43, 44.22) --
	(234.30, 44.22) --
	(240.16, 40.99) --
	(246.03, 40.99) --
	(251.90, 40.99) --
	(257.76, 40.99) --
	(263.63, 40.99) --
	(269.49, 40.99) --
	(275.36, 40.99) --
	(281.22, 40.99) --
	(287.09, 40.99) --
	(292.95, 40.99) --
	(298.82, 40.99) --
	(304.69, 40.99) --
	(310.55, 40.99) --
	(316.42, 40.99);
\definecolor[named]{drawColor}{rgb}{0.53,0.81,0.92}

\path[draw=drawColor,line width= 2.8pt,line join=round] (204.97, 40.99) --
	(210.84, 40.99) --
	(216.70, 40.99) --
	(222.57, 40.99) --
	(228.43, 40.99) --
	(234.30, 40.99) --
	(240.16, 40.99) --
	(246.03, 40.99) --
	(251.90, 40.99) --
	(257.76, 40.99) --
	(263.63, 40.99) --
	(269.49, 40.99) --
	(275.36, 40.99) --
	(281.22, 40.99) --
	(287.09, 40.99) --
	(292.95, 40.99) --
	(298.82, 40.99) --
	(304.69, 40.99) --
	(310.55, 40.99) --
	(316.42, 44.22);
\definecolor[named]{drawColor}{rgb}{0.00,0.00,0.00}

\path[draw=drawColor,line width= 0.9pt,line join=round] (204.97, 50.68) --
	(210.84, 50.68) --
	(216.70, 47.45) --
	(222.57, 47.45) --
	(228.43, 44.22) --
	(234.30, 44.22) --
	(240.16, 40.99) --
	(246.03, 40.99) --
	(251.90, 40.99) --
	(257.76, 40.99) --
	(263.63, 40.99) --
	(269.49, 40.99) --
	(275.36, 40.99) --
	(281.22, 40.99) --
	(287.09, 40.99) --
	(292.95, 40.99) --
	(298.82, 40.99) --
	(304.69, 40.99) --
	(310.55, 40.99) --
	(316.42, 44.22);

\path[draw=drawColor,line width= 0.4pt,line join=round,line cap=round] (240.16, 40.99) circle (  2.13);

\path[draw=drawColor,line width= 0.4pt,line join=round,line cap=round] (246.03, 40.99) circle (  2.13);

\path[draw=drawColor,line width= 0.4pt,line join=round,line cap=round] (251.90, 40.99) circle (  2.13);

\path[draw=drawColor,line width= 0.4pt,line join=round,line cap=round] (257.76, 40.99) circle (  2.13);

\path[draw=drawColor,line width= 0.4pt,line join=round,line cap=round] (263.63, 40.99) circle (  2.13);

\path[draw=drawColor,line width= 0.4pt,line join=round,line cap=round] (269.49, 40.99) circle (  2.13);

\path[draw=drawColor,line width= 0.4pt,line join=round,line cap=round] (275.36, 40.99) circle (  2.13);

\path[draw=drawColor,line width= 0.4pt,line join=round,line cap=round] (281.22, 40.99) circle (  2.13);

\path[draw=drawColor,line width= 0.4pt,line join=round,line cap=round] (287.09, 40.99) circle (  2.13);

\path[draw=drawColor,line width= 0.4pt,line join=round,line cap=round] (292.95, 40.99) circle (  2.13);

\path[draw=drawColor,line width= 0.4pt,line join=round,line cap=round] (298.82, 40.99) circle (  2.13);

\path[draw=drawColor,line width= 0.4pt,line join=round,line cap=round] (304.69, 40.99) circle (  2.13);

\path[draw=drawColor,line width= 0.4pt,line join=round,line cap=round] (310.55, 40.99) circle (  2.13);
\end{scope}
\begin{scope}
\path[clip] (186.49, 35.17) rectangle (334.89, 92.03);
\definecolor[named]{drawColor}{rgb}{0.89,0.10,0.11}

\node[text=drawColor,anchor=base east,inner sep=0pt, outer sep=0pt, scale=  1.00] at (202.12, 47.99) {$\hat{\text{FN}}$};

\node[text=drawColor,anchor=base west,inner sep=0pt, outer sep=0pt, scale=  1.00] at (319.26, 35.17) {$\hat{\text{FN}}$};
\definecolor[named]{drawColor}{rgb}{0.00,0.00,0.00}

\node[text=drawColor,anchor=base east,inner sep=0pt, outer sep=0pt, scale=  1.00] at (202.12, 60.81) {$E$};

\node[text=drawColor,anchor=base west,inner sep=0pt, outer sep=0pt, scale=  1.00] at (319.26, 60.81) {$E$};
\definecolor[named]{drawColor}{rgb}{0.53,0.81,0.92}

\node[text=drawColor,anchor=base east,inner sep=0pt, outer sep=0pt, scale=  1.00] at (202.12, 35.17) {$\hat{\text{FP}}$};

\node[text=drawColor,anchor=base west,inner sep=0pt, outer sep=0pt, scale=  1.00] at (319.26, 47.99) {$\hat{\text{FP}}$};
\definecolor[named]{drawColor}{rgb}{0.50,0.50,0.50}

\path[draw=drawColor,line width= 0.6pt,line join=round,line cap=round] (186.49, 35.17) rectangle (334.89, 92.03);
\end{scope}
\begin{scope}
\path[clip] (  0.00,  0.00) rectangle (361.35,289.08);
\definecolor[named]{drawColor}{rgb}{0.00,0.00,0.00}

\node[text=drawColor,anchor=base east,inner sep=0pt, outer sep=0pt, scale=  0.87] at ( 30.98,208.29) {0};

\node[text=drawColor,anchor=base east,inner sep=0pt, outer sep=0pt, scale=  0.87] at ( 30.98,224.44) {5};

\node[text=drawColor,anchor=base east,inner sep=0pt, outer sep=0pt, scale=  0.87] at ( 30.98,240.59) {10};

\node[text=drawColor,anchor=base east,inner sep=0pt, outer sep=0pt, scale=  0.87] at ( 30.98,256.75) {15};
\end{scope}
\begin{scope}
\path[clip] (  0.00,  0.00) rectangle (361.35,289.08);
\definecolor[named]{drawColor}{rgb}{0.00,0.00,0.00}

\path[draw=drawColor,line width= 0.6pt,line join=round] ( 33.82,211.58) --
	( 38.09,211.58);

\path[draw=drawColor,line width= 0.6pt,line join=round] ( 33.82,227.73) --
	( 38.09,227.73);

\path[draw=drawColor,line width= 0.6pt,line join=round] ( 33.82,243.89) --
	( 38.09,243.89);

\path[draw=drawColor,line width= 0.6pt,line join=round] ( 33.82,260.04) --
	( 38.09,260.04);
\end{scope}
\begin{scope}
\path[clip] (  0.00,  0.00) rectangle (361.35,289.08);
\definecolor[named]{drawColor}{rgb}{0.00,0.00,0.00}

\node[text=drawColor,anchor=base east,inner sep=0pt, outer sep=0pt, scale=  0.87] at ( 30.98,151.42) {0};

\node[text=drawColor,anchor=base east,inner sep=0pt, outer sep=0pt, scale=  0.87] at ( 30.98,167.58) {5};

\node[text=drawColor,anchor=base east,inner sep=0pt, outer sep=0pt, scale=  0.87] at ( 30.98,183.73) {10};

\node[text=drawColor,anchor=base east,inner sep=0pt, outer sep=0pt, scale=  0.87] at ( 30.98,199.89) {15};
\end{scope}
\begin{scope}
\path[clip] (  0.00,  0.00) rectangle (361.35,289.08);
\definecolor[named]{drawColor}{rgb}{0.00,0.00,0.00}

\path[draw=drawColor,line width= 0.6pt,line join=round] ( 33.82,154.71) --
	( 38.09,154.71);

\path[draw=drawColor,line width= 0.6pt,line join=round] ( 33.82,170.87) --
	( 38.09,170.87);

\path[draw=drawColor,line width= 0.6pt,line join=round] ( 33.82,187.02) --
	( 38.09,187.02);

\path[draw=drawColor,line width= 0.6pt,line join=round] ( 33.82,203.18) --
	( 38.09,203.18);
\end{scope}
\begin{scope}
\path[clip] (  0.00,  0.00) rectangle (361.35,289.08);
\definecolor[named]{drawColor}{rgb}{0.00,0.00,0.00}

\node[text=drawColor,anchor=base east,inner sep=0pt, outer sep=0pt, scale=  0.87] at ( 30.98, 94.56) {0};

\node[text=drawColor,anchor=base east,inner sep=0pt, outer sep=0pt, scale=  0.87] at ( 30.98,110.71) {5};

\node[text=drawColor,anchor=base east,inner sep=0pt, outer sep=0pt, scale=  0.87] at ( 30.98,126.87) {10};

\node[text=drawColor,anchor=base east,inner sep=0pt, outer sep=0pt, scale=  0.87] at ( 30.98,143.02) {15};
\end{scope}
\begin{scope}
\path[clip] (  0.00,  0.00) rectangle (361.35,289.08);
\definecolor[named]{drawColor}{rgb}{0.00,0.00,0.00}

\path[draw=drawColor,line width= 0.6pt,line join=round] ( 33.82, 97.85) --
	( 38.09, 97.85);

\path[draw=drawColor,line width= 0.6pt,line join=round] ( 33.82,114.00) --
	( 38.09,114.00);

\path[draw=drawColor,line width= 0.6pt,line join=round] ( 33.82,130.16) --
	( 38.09,130.16);

\path[draw=drawColor,line width= 0.6pt,line join=round] ( 33.82,146.31) --
	( 38.09,146.31);
\end{scope}
\begin{scope}
\path[clip] (  0.00,  0.00) rectangle (361.35,289.08);
\definecolor[named]{drawColor}{rgb}{0.00,0.00,0.00}

\node[text=drawColor,anchor=base east,inner sep=0pt, outer sep=0pt, scale=  0.87] at ( 30.98, 37.70) {0};

\node[text=drawColor,anchor=base east,inner sep=0pt, outer sep=0pt, scale=  0.87] at ( 30.98, 53.85) {5};

\node[text=drawColor,anchor=base east,inner sep=0pt, outer sep=0pt, scale=  0.87] at ( 30.98, 70.00) {10};

\node[text=drawColor,anchor=base east,inner sep=0pt, outer sep=0pt, scale=  0.87] at ( 30.98, 86.16) {15};
\end{scope}
\begin{scope}
\path[clip] (  0.00,  0.00) rectangle (361.35,289.08);
\definecolor[named]{drawColor}{rgb}{0.00,0.00,0.00}

\path[draw=drawColor,line width= 0.6pt,line join=round] ( 33.82, 40.99) --
	( 38.09, 40.99);

\path[draw=drawColor,line width= 0.6pt,line join=round] ( 33.82, 57.14) --
	( 38.09, 57.14);

\path[draw=drawColor,line width= 0.6pt,line join=round] ( 33.82, 73.30) --
	( 38.09, 73.30);

\path[draw=drawColor,line width= 0.6pt,line join=round] ( 33.82, 89.45) --
	( 38.09, 89.45);
\end{scope}
\begin{scope}
\path[clip] (334.89,205.76) rectangle (348.10,262.62);
\definecolor[named]{drawColor}{rgb}{0.50,0.50,0.50}
\definecolor[named]{fillColor}{rgb}{0.80,0.80,0.80}

\path[draw=drawColor,line width= 0.2pt,line join=round,line cap=round,fill=fillColor] (334.89,205.76) rectangle (348.10,262.62);
\definecolor[named]{drawColor}{rgb}{0.00,0.00,0.00}

\node[text=drawColor,rotate=270.00,anchor=base,inner sep=0pt, outer sep=0pt, scale=  0.87] at (338.21,234.19) {Complete};
\end{scope}
\begin{scope}
\path[clip] (334.89,148.90) rectangle (348.10,205.76);
\definecolor[named]{drawColor}{rgb}{0.50,0.50,0.50}
\definecolor[named]{fillColor}{rgb}{0.80,0.80,0.80}

\path[draw=drawColor,line width= 0.2pt,line join=round,line cap=round,fill=fillColor] (334.89,148.90) rectangle (348.10,205.76);
\definecolor[named]{drawColor}{rgb}{0.00,0.00,0.00}

\node[text=drawColor,rotate=270.00,anchor=base,inner sep=0pt, outer sep=0pt, scale=  0.87] at (338.21,177.33) {Zero-one};
\end{scope}
\begin{scope}
\path[clip] (334.89, 92.03) rectangle (348.10,148.90);
\definecolor[named]{drawColor}{rgb}{0.50,0.50,0.50}
\definecolor[named]{fillColor}{rgb}{0.80,0.80,0.80}

\path[draw=drawColor,line width= 0.2pt,line join=round,line cap=round,fill=fillColor] (334.89, 92.03) rectangle (348.10,148.90);
\definecolor[named]{drawColor}{rgb}{0.00,0.00,0.00}

\node[text=drawColor,rotate=270.00,anchor=base,inner sep=0pt, outer sep=0pt, scale=  0.87] at (338.21,120.47) {Incomplete};
\end{scope}
\begin{scope}
\path[clip] (334.89, 35.17) rectangle (348.10, 92.03);
\definecolor[named]{drawColor}{rgb}{0.50,0.50,0.50}
\definecolor[named]{fillColor}{rgb}{0.80,0.80,0.80}

\path[draw=drawColor,line width= 0.2pt,line join=round,line cap=round,fill=fillColor] (334.89, 35.17) rectangle (348.10, 92.03);
\definecolor[named]{drawColor}{rgb}{0.00,0.00,0.00}

\node[text=drawColor,rotate=270.00,anchor=base,inner sep=0pt, outer sep=0pt, scale=  0.87] at (338.21, 63.60) {Positive};
\end{scope}
\begin{scope}
\path[clip] (  0.00,  0.00) rectangle (361.35,289.08);
\definecolor[named]{drawColor}{rgb}{0.00,0.00,0.00}

\path[draw=drawColor,line width= 0.6pt,line join=round] ( 56.57, 30.90) --
	( 56.57, 35.17);

\path[draw=drawColor,line width= 0.6pt,line join=round] ( 80.03, 30.90) --
	( 80.03, 35.17);

\path[draw=drawColor,line width= 0.6pt,line join=round] ( 91.76, 30.90) --
	( 91.76, 35.17);

\path[draw=drawColor,line width= 0.6pt,line join=round] (168.02, 30.90) --
	(168.02, 35.17);
\end{scope}
\begin{scope}
\path[clip] (  0.00,  0.00) rectangle (361.35,289.08);
\definecolor[named]{drawColor}{rgb}{0.00,0.00,0.00}

\node[text=drawColor,anchor=base,inner sep=0pt, outer sep=0pt, scale=  0.87] at ( 56.57, 21.48) {1};

\node[text=drawColor,anchor=base,inner sep=0pt, outer sep=0pt, scale=  0.87] at ( 80.03, 21.48) {5};

\node[text=drawColor,anchor=base,inner sep=0pt, outer sep=0pt, scale=  0.87] at ( 91.76, 21.48) {7};

\node[text=drawColor,anchor=base,inner sep=0pt, outer sep=0pt, scale=  0.87] at (168.02, 21.48) {20};
\end{scope}
\begin{scope}
\path[clip] (  0.00,  0.00) rectangle (361.35,289.08);
\definecolor[named]{drawColor}{rgb}{0.00,0.00,0.00}

\path[draw=drawColor,line width= 0.6pt,line join=round] (204.97, 30.90) --
	(204.97, 35.17);

\path[draw=drawColor,line width= 0.6pt,line join=round] (228.43, 30.90) --
	(228.43, 35.17);

\path[draw=drawColor,line width= 0.6pt,line join=round] (240.16, 30.90) --
	(240.16, 35.17);

\path[draw=drawColor,line width= 0.6pt,line join=round] (316.42, 30.90) --
	(316.42, 35.17);
\end{scope}
\begin{scope}
\path[clip] (  0.00,  0.00) rectangle (361.35,289.08);
\definecolor[named]{drawColor}{rgb}{0.00,0.00,0.00}

\node[text=drawColor,anchor=base,inner sep=0pt, outer sep=0pt, scale=  0.87] at (204.97, 21.48) {1};

\node[text=drawColor,anchor=base,inner sep=0pt, outer sep=0pt, scale=  0.87] at (228.43, 21.48) {5};

\node[text=drawColor,anchor=base,inner sep=0pt, outer sep=0pt, scale=  0.87] at (240.16, 21.48) {7};

\node[text=drawColor,anchor=base,inner sep=0pt, outer sep=0pt, scale=  0.87] at (316.42, 21.48) {20};
\end{scope}
\begin{scope}
\path[clip] (  0.00,  0.00) rectangle (361.35,289.08);
\definecolor[named]{drawColor}{rgb}{0.00,0.00,0.00}

\node[text=drawColor,anchor=base,inner sep=0pt, outer sep=0pt, scale=  1.09] at (186.49,  9.94) {segments $k$ of estimated signal};
\end{scope}
\begin{scope}
\path[clip] (  0.00,  0.00) rectangle (361.35,289.08);
\definecolor[named]{drawColor}{rgb}{0.00,0.00,0.00}

\node[text=drawColor,rotate= 90.00,anchor=base,inner sep=0pt, outer sep=0pt, scale=  1.09] at ( 18.16,148.90) {error};
\end{scope}
\end{tikzpicture}